\documentclass{amsart}

\usepackage{amssymb}
\usepackage{amsmath}
\usepackage{mathrsfs}
\usepackage{comment}
\usepackage[english]{babel}
\usepackage[utf8]{inputenc}
\usepackage[T1]{fontenc}
\usepackage{tikz}
\usetikzlibrary{cd}
\usepackage{youngtab}
\usepackage[footnotesize]{caption}

\usepackage[top=3cm,bottom=2cm,left=3cm,right=3cm,marginparwidth=1.75cm]{geometry}

\usepackage{graphicx}
\usepackage[colorinlistoftodos]{todonotes}
\usepackage[hypertexnames=false,pdftex,backref=false %=page,
 	pdfpagemode=UseNone,
 	breaklinks=true,
 	extension=pdf,
 	colorlinks=true,
 	linkcolor=blue,
 	citecolor=blue,
 	urlcolor=blue,
]{hyperref}
\usepackage[sort]{cite}
\usepackage{enumitem}
\setlist[1]{itemsep=3pt}
\usepackage[all,cmtip]{xy}
\usepackage[capitalise]{cleveref}

\usepackage{xltabular}

\crefname{equation}{}{}

\usepackage{autonum} %disabled since confusing
\usepackage{setspace}

\newcommand{\N}{\mathbb{N}}

\newcommand{\Q}{\mathbb{Q}}
\newcommand{\R}{\mathbb{R}}
\newcommand{\C}{\mathbb{C}}
\newcommand{\EE}{\mathbb{E}}
\newcommand{\KK}{\mathscr{K}}
\newcommand{\ZZ}{\mathcal{Z}}
\newcommand{\VZ}{\widehat{\mathcal{Z}}}
\newcommand{\A}{\mathcal A}
\newcommand{\VA}{\widehat{\mathcal{A}}}
\newcommand{\GZ}{\mathcal{G}}
\newcommand{\hG}{\widetilde{G}}
\newcommand{\VGZ}{\widehat{\mathcal{G}}}
\newcommand{\bo}{\mathbf{o}}
\newcommand{\VGZo}{\widehat{\mathcal{G}}_{o}}
\newcommand{\GZo}{\GZ_{o}}
\newcommand{\M}{\mathcal{I}}
\newcommand{\CGZ}{\mathcal{Q}}
\newcommand{\cN}{\mathcal{N}}
\newcommand{\orf}{\varpi}

\newcommand{\z}{Z}
\newcommand{\cz}{K}

\newcommand{\proj}{\mathbb{P}}
\newcommand{\cM}{\mathcal{M}}

\newcommand{\vol}{\mathrm{vol}}
\newcommand{\MV}{\mathrm{MV}}
\newcommand{\sgn}{\mathrm{sgn}}

\newcommand{\Kl}{\mathrm{Kl}}

\newcommand{\eps}{\varepsilon}
\newcommand{\ot}{\otimes}
\newcommand{\Gl}{\mathrm{GL}}
\newcommand{\Sl}{\mathrm{SL}}
\newcommand{\ti}{\times}
\newcommand{\Sc}{\mathscr{S}}
\newcommand{\LR}{\mathrm{LR}}
\newcommand{\la}{\lambda}

\newcommand{\Vect}{V}

\newcommand{\simto}{\xrightarrow[]{\sim}}
\newcommand{\val}{\mathrm{Val}_+}
\newcommand{\sval}{\mathrm{Val}_{+,\mathrm{sm}}}
\newcommand{\set}[2]{\left\{#1\, |\, #2\right\}}

\DeclareMathOperator{\Span}{span}

\DeclareMathOperator{\Hom}{Hom}

\newcommand{\Lg}{\mathfrak{g}}
\newcommand{\Lh}{\mathfrak{h}}

\newcommand{\id}{\mathrm{Id}}
\newcommand{\ori}{\mathrm{ori}}
\newcommand{\GL}{\mathrm{GL}}
\newcommand{\SO}{\mathrm{SO}}
\newcommand{\OO}{\mathrm{O}}
\newcommand{\oo}{\mathrm{o}}
\newcommand{\UU}{\mathrm{U}}

\newcommand{\End}{\mathrm{End}}

\newcommand{\HE}{{\mathrm{H}^+_{\EE}}}
\newcommand{\HEo}{\mathrm{H}_{\EE}}
\newcommand{\HDR}{\mathrm{H}_{\mathrm{DR}}}
\newcommand{\HdR}{\mathrm{H}_{c}}
\newcommand{\RP}{\mathbb{R}\mathrm{P}}
\newcommand{\CP}{\mathbb{C}\mathrm{P}}

\newcommand{\Osm}{\Omega^{\mathrm{sm}}}
\renewcommand{\b}{\beta}
\newcommand{\g}{\gamma}
\newcommand{\cB}{\mathcal{B}}
\newcommand{\cP}{\mathcal{P}}

\newtheorem{theorem}{Theorem}
\newtheorem{proposition}[theorem]{Proposition}
\newtheorem{corollary}[theorem]{Corollary}
\newtheorem{lemma}[theorem]{Lemma}
\theoremstyle{definition} 
\newtheorem{definition}[theorem]{Definition}

\newtheorem{example}[theorem]{Example}
\theoremstyle{remark} 
\newtheorem{remark}[theorem]{Remark}

\newcommand{\be}{\begin{equation}}
\newcommand{\ee}{\end{equation}}
\numberwithin{equation}{section}
\numberwithin{theorem}{section}

\title[Probabilistic intersection theory]{Probabilistic intersection theory in Riemannian homogeneous spaces}
\author{Paul Breiding}
\address{FB 6 Mathematik/Informatik/Physik, 
Universität Osnabrück, Germany}
\email{pbreiding@uni-osnabrueck.de}
\thanks{P.~Breiding is funded by the Deutsche Forschungsgemeinschaft (DFG) – Project 445466444.}
\author{Peter B\"urgisser}
\address{Institute of Mathematics, Technische Universit\"at Berlin, Germany} 
\email{pbuerg@math.tu-berlin.de}
\thanks{P.~B\"urgisser is funded by the ERC under the European Union's Horizon 2020 research and innovation programme 
(grant agreement no. 787840).} 
\author{Antonio Lerario}
\address{SISSA, Trieste, Italy}
\email{lerario@sissa.it}
\thanks{A. Lerario is funded by (i) the PRIN project “Optimal transport: new challenges across analysis and geometry” funded by the Italian Ministry of University and Research; (ii) the Knut and Alice Wallenberg Foundation.}
\author{L\'eo Mathis}
\address{Institut für Mathematik, 
Goethe-Universität Frankfurt, Germany}
\email{mathis@mathematik.uni-frankfurt}
\thanks{L.~Mathis is funded by Deutsche Forschungsgemeinschaft (DFG) grant BE 2484/10-1}

\date{\today}

\begin{document}

\begin{abstract}
Let $M=G/H$ be a Riemannian homogeneous space, where $G$ is a compact Lie group with closed subgroup $H$. 
Classical intersection theory states that the de Rham cohomology ring of $M$
describes the signed count of intersection points of submanifolds 
$Y_1, \ldots, Y_s$ of $M$ in general position, 
when the codimensions add up to $\dim M$. 

We introduce the 
\emph{probabilistic intersection ring} $\HEo(M)$, 
whose multiplication describes 
the \emph{unsigned} count of intersection points, 
when the $Y_i$ are randomly moved by independent uniformly random elements of $G$. 
The probabilistic intersection ring~$\HEo(M)$ 
has the structure of a graded commutative and associative real Banach algebra.
It is defined as a quotient of the \emph{ring of Grassmann zonoids} 
of a fixed cotangent space $V$ of $M$.
The latter was introduced by the authors in [Adv.~Math.~402, 2022].
There is a close connection to valuations of 
convex bodies: $\HEo(M)$ can be interpreted as a subspace of the space of 
translation invariant, even, continuous valuations on $V$, whose multiplication coincides with 
Alesker's multiplication for smooth valuations.

We describe the ring structure of the probabilistic intersection ring for spheres, real projective space and
complex projective space, relying on Fu [J.~Diff.~Geo.~72(3), 2006] for the latter case. 
From this, we derive an interesting probabilistic intersection formula in complex projective space. 
Finally, we initiate the investigation of the probabilistic intersection ring for real Grassmannians, outlining the construction of a probabilistic version of Schubert Calculus.  %$G(d,n)$. 
\end{abstract}

\keywords{intersection theory, homogeneous space, cohomology ring, valuations on convex bodies, 
zonoids, mixed volume, exterior algebra, Grassmann manifolds, metric algebraic geometry}

%\subjclass{
\subjclass[2020]{
Primary 
52A22, %Random convex sets and integral geometry (aspects of convex geometry) [See also 53C65, 60D05
52A39, %mixed volume
43A85, %harmonic analysis on homogeneous space 
53C65, %Integral geometry 
60D05; %Geometric probability and stochastic geometry 
Secondary 
14C17, %intersection theory, characteristic classes, intersection multiplicities in algebraic geometry 
} 

\maketitle

%%%%%%%%
\bigskip
\section{Introduction}

Modern intersection theory is built on the idea of associating cycles with their cohomology classes 
through Poincaré duality. This approach transforms the geometric problem of determining the intersection of cycles 
into an algebraic one—computing the product of their Poincaré duals in the cohomology ring of the ambient space.

Consider, for example, the problem of counting the number of lines that intersect four generic lines 
in three--dimensional complex projective space. 
In this case, the relevant ambient space is the Grassmannian of lines in complex projective space, 
and the condition that a line intersects a given one is described by a Schubert variety within it. 
Geometrically, the set of lines that intersect all four given lines corresponds to the intersection 
of their associated Schubert varieties. Determining the number of points in this intersection reduces 
to computing the product of the corresponding cohomology classes in the cohomology ring of the Grassmannian. 
Computations of these type are known as Schubert calculus and, in this particular example, the answer one gets is 2.

Over the real numbers, the situation differs significantly. Even when the four defining lines are in general position, 
the number of real intersecting lines can be either 2 or 0. While it is still possible to formulate the corresponding Schubert problem, the absence of orientation—a key feature in the complex case—prevents this from being defined over the integers. Motivated by this, B\"urgisser and Lerario, in~\cite{PSC}, proposed a probabilistic version of the problem: 
What is the expected number of real lines intersecting four \emph{random} lines in 3--dimensional real projective space? 
In this setting, randomness is introduced by moving the lines using random and independent elements from the orthogonal group, which acts naturally on the Grassmannian. In \cite{PSC}, the authors discovered that this expected number equals 
(up to a constant) the volume of a special convex body, which they called the Segre zonoid. 
One can imagine similar  and more general questions involving random incidence problems, aiming at a probabilistic version of Schubert calculus, as proposed in \cite{PSC}.

In analogy with the complex (oriented) case, from this example a natural question arises: is there a framework, such as a suitable ring, that can systematically handle computations of this type and explain the appearance of convex geometry in the probabilistic version of the problem? 

The goal of this paper is to provide an affirmative answer to this question. To achieve this, we introduce the concept of \emph{probabilistic intersection ring}, a structure designed to formalize the algebraic operations necessary for probabilistic intersection questions in the context of compact homogeneous spaces. We will explain now the main ideas of this construction, aiming, in the rest of the introduction, for an informal presentation. Throughout, we guide the reader to the relevant sections of the paper for precise results and  further details.

\subsection{Algebras of zonoids}

The definition of the probabilistic intersection ring uses the %so-called 
\emph{zonoid algebra}, 
introduced by the authors in \cite{BBLM}. The elements of this algebra are convex bodies of a special type, called \emph{zonoids}. They are Hausdorff limits of Minkowski sums of segments \cite{bible}; 
see~\cref{se:zonoids}. 
Given a real vector space $V$ of dimension $n$, it is useful to define  \emph{virtual zonoids} 
as formal differences of zonoids:
 they form a real vector space denoted~$\VZ(V)$. 
 Moreover, $\VZ_o(V)$ denotes the subspace of centered virtual zonoids.  
 The vector space decomposition 
 $\VZ(V) \simeq \VZ_o(V) \oplus V$ 
 simply expresses that every zonoid can be shifted to a centered zonoid,
 see \eqref{eq:Z-iso}.  
 Consider now the space
 \begin{equation}\label{eq:zonoisointro}
  \VA(V) := \bigoplus_{d=0}^{n}\VZ(\Lambda^d V) \simeq \left(
   \bigoplus_{d=0}^{n}\VZ_o(\Lambda^d V) \right)\oplus \left(\bigoplus_{d=0}^{n}\Lambda^d V\right).
\end{equation}

In \cite{BBLM} we proved that it is possible to endow the centered part of $\VA(V)$ 
with a ring structure that turns it into a graded, commutative algebra. 
This ring structure can be extended also to the non--centered part, 
essentially using the isomorphism~\eqref{eq:zonoisointro}, 
as explained in \cref{se:zon-in-ext-power}. The resulting algebra is  the zonoid algebra.  
We call the product operation the ``wedge'' and denote it with the symbol~``$\wedge$''. 
This expresses the fact that if $K_1$ and $K_2$ are zonoids, one should think of 
$K_1\wedge K_2$ as convex body analogue of the wedge of differential forms. 
The degree of the wedge of two zonoids is the sum of their individual degrees. 
This ring operation enjoys some interesting properties, allowing to reinterpret classical operations on convex bodies. For example, if $K_1, \ldots, K_n\in \mathcal{Z}(\R^n)$, 
then $K_1\wedge\ldots\wedge K_n$ is a segment of length 
$\mathrm{MV}(K_1, \ldots, K_n)$ (the mixed volume).
We review the details of this construction in \cref{sec:algebra}.

In fact, only a subalgebra of this will be relevant for us. 
We say that a centered zonoid  $Z\in \ZZ_o(\Lambda^dV)$ 
is  \emph{Grassmannian of degree $d$},
if it is the Hausdorff limit of Minkowski sums of
segments of the form $ [-\xi, \xi]$,  
where $\xi\in \Lambda^dV$ is a simple vector; 
see~\cref{re:alt-char-GZ}.
Noncentered Grassmann zonoids are obtained by shifting with a fixed vector in~$\Lambda^dV$. 
We denote by $\VGZ^d (V)$ the set of virtual Grassmann zonoids of degree $d$ in $V$. 
They define a graded subalgebra of $\VA(V)$, denoted by
\[
 \VGZ(V) := \bigoplus_{d=0}^{n}\VGZ^d (V),
 \]  
and called \emph{Grassmann zonoid algebra}; see \cref{sec:grass_algebra}. 
We denote by $ \M^d(V)$ the linear span of Grassmann zonoids in $\VGZ^d(V)$, 
whose support function vanishes on unit simple vectors (recall that a Grassmann zonoid lives in $\Lambda^d V$). 
The direct sum
$ \M(V):=\bigoplus_{d=0}^{n}\M^d(V) $
is an ideal of the Grassmann zonoid algebra $\VGZ(V)$, 
as we show in~\cref{sec:cosine}. 
The \emph{algebra of classes of Grassmann zonoids} of~$V$ is defined as the quotient algebra
$$
 \CGZ(V) := \VGZ(V)/\M(V).
$$ 
 All these constructions have natural functorial properties, see \cref{pro:funct}. 
 
In \cref{sec:valuation} we will show how these notions are linked to the theory of 
translation invariant valuations on convex bodies, 
which has expanded significantly in recent decades, see \cite[Chap.~6]{bible} 
and \cite{bernig-survey:12,fu-survey:14}. 
For instance, the algebra of smooth even valuations with Alesker's product can be identified 
with a dense subalgebra of the algebra $\CGZ_o(V)$ of centered Grassmann zonoid classes 
(see \cref{th:crofton-1} and \cref{cor:crofton-1}).

%%%
\subsection{The probabilistic intersection ring} 
We are now ready to define our main object of interest: the probabilistic intersection ring. 
Let us first recall the framework that we will adopt (see \cref{se:setting} for further details). 
We consider a homogeneous space $ M := G/H $, where $ G $ is a compact Lie group and $ H $ a closed subgroup. 
We assume that $ M $ is endowed with a $G$--invariant Riemannian structure, 
allowing us to use geometric concepts like ``volume'' for submanifolds. 
Let $ \pi: G \to M $ be the quotient map and $ e \in G $ the identity element. 
We define $ \bo := \pi(e) \in M $. 
The relevant space $V$ from the previous sections is now the cotangent space of $M$ at $\bo$:
\[
V := T^*_\bo M.
\] 
The isotropy action of $ H $ on $ V $ induces an action of $ H $ on the exterior algebra of $V$, 
and similarly on the algebra $\CGZ(V)$ of Grassmann zonoid classes and its subspace of centered elements $\CGZ_o(V)$.
The \emph{probabilistic intersection ring of $M$} is then defined as the ring of $H$-invariant elements in $\CGZ(V)$:
\[
\HE(M) := \CGZ(V)^H.
\]
We also define the 
\emph{centered probabilistic intersection ring} and the set of $H$--invariants of the exterior algebra as, respectively:
\[
 \HEo(M) := \CGZ_o(V)^H,\quad
 \HdR(M) := \Lambda(V)^H.
\]
This gives a direct decomposition
\[\label{eq:decointro}
 \HE(M)\simeq \HEo(M)\oplus \HdR(M),
\]
which reflects the fact that every zonoid is the sum of a centered zonoid and its center. 
We will discuss the geometric meaning of the two summands in 
\cref{sec:PIP} and  \cref{sec:PIRC} below. Before doing so, 
we explain the process of associating elements of this ring to submanifolds of $M$, 
which parallels the classical association of cohomology classes to submanifolds , 
and then explain how to use this framework for probabilistic intersection problems, 
thus explaining the name of this ring.

\subsection{The class of a submanifold in the probabilistic intersection ring}

Let $Y \subseteq M$ be a submanifold, possibly with singularities (see \cref{def:stratified}) 
and of codimension $0\leq d\leq n=\dim M$. 
Assume also that $Y$ is cooriented, meaning its normal bundle is oriented. 
In this case, we first construct a zonoid $Z(Y) \in \VZ(\Lambda^d V)$, as follows. 
At each point $y \in Y$, 
we represent the conormal space of $Y$ by 
the simple unit vector
$\nu(y) :=  v_1(y)\wedge\ldots\wedge v_d(y)$, 
where $v_1(y), \ldots, v_d(y)$ is an oriented orthonormal basis of this space. 
Using the group action, we can transport $y$ to $\bo$ and thus 
assume that $\nu(y) \in \Lambda^d V$
(since we are going to average over the stabilizer group $H$, 
the choice of the transporting group element is irrelevant).
Using the notation $[a,b]$ for the segment from $a$ to $b$ \eqref{notation_segments}, the Grassmann zonoid $Z(Y)$~is:
\[\label{eq:intro-def-Z(Y)}
Z(Y) :=  \frac{1}{\vol(M)} \, \int_H \int_Y [0, h \cdot \nu(y)] \, \mathrm{d}y \, \mathrm{d}h,
\]
where $\mathrm{d}y$ and $\mathrm{d}h$
denote integration with respect to the induced Riemannian metric on the smooth points of~$Y$ and 
the normalized Haar measure of $H$, respectively. 
This construction essentially averages the segments over the stabilizer action and the manifold $Y$. 
Intuitively, this is possible, because the integral is the limit of sums, and zonoids are defined 
as limits of Minkowski sums of segments. 
This informal definition is rigorously justified in \cref{def:KZ_Y}, 
using an effective description of zonoids in terms of random variables, as introduced by Vitale~\cite{vitale}. 
For further details on this description, we refer to~\cite[\S 2.2]{BBLM} and \cref{se:Zon-RV}. 
For a cooriented submanifold $Y$ of codimension $d$, 
we then define its \emph{oriented Grassmann class} as:
\[
 [Y]^+_{\EE} := [Z(Y)] \in (\HE)^d(M),
\]
and we denote its \emph{centered class} as:\enlargethispage{\baselineskip}
\[
 [Y]_c := 2 \, c(Z(Y)) \in \HdR^d(M),
\]
where $c(Z(Y)) \in \Lambda^d V$ represents the center of the zonoid; see \cref{def:classY}.

In order to define the Grassmann class of a submanifold $Y$ that is not necessarily coorientable, 
we first construct a centered zonoid $K(Y) \in \VZ_o(\Lambda^d V)$. 

Here, the unit vector $\nu(y) \in \Lambda^d V$ is defined only up to a sign,  
because there is no preferred orientation of the normal space. 
Slightly abusing notation, 
we define in analogy with~\eqref{eq:intro-def-Z(Y)}, 
\[
 K(Y) :=  \frac{1}{\vol(M)} \,\int_H \int_Y \int_{\{\pm 1\}} 
    \tfrac12 [- h \cdot \eps \nu(y), h \cdot \eps \nu(y)] 
     \, \mathrm{d}\eps \, \mathrm{d}y \, \mathrm{d}h ,
\]
where $\mathrm{d}\eps$ denotes integration with respect to a uniform element in $\{\pm 1\}$.
The \emph{Grassmann class} of $Y$ is now defined as:
\[
[Y]_{\EE} := [K(Y)] \in \HEo^d(M).
\]
Here again, the reader should look at \cref{def:KZ_Y} and \cref{def:classY} for more details.
Note that, when $ Y $ is cooriented, the relation $Z(Y) = K(Y) + c(Z(Y))$ 
leads to the following decomposition in the probabilistic intersection ring \eqref{eq:decointro}:
\[
  [Y]^+_{\EE} = [Y]_{\EE} + \tfrac{1}{2}[Y]_c.
\]

The construction of the probabilistic intersection ring is functorial in the following sense. 
Let $M_1=G_1/H_1$ and $M_2=G_2/H_2$ be two homogeneous spaces as above. 
We say that $f:M_1\to M_2$ is a morphism of homogeneous spaces, if this map is induced by a 
Lie group homomorphism $\rho: G_1\to G_2$ 
such that $\rho(H_1)\subseteq H_2$; see \cref{def:morph-HS}.
In \cref{prop: pullback of pir}, we show that a morphism $f:M_1\to M_2$ 
of homogeneous spaces induces a graded algebra homomorphism
\begin{equation}
    f^*:\HE(M_2)\to \HE(M_1),
\end{equation}
that respects the decomposition \eqref{eq:decointro}. Moreover, if $\rho(G_1)=G_2$, 
then the Grassmann class of a preimage is the pullback of the Grassmann class, i.e.,  
$f^*[Y]_{\EE}=[f^{-1}(Y)]_\EE$ and, in the cooriented case, 
$f^*[Y]^+_{\EE}=[f^{-1}(Y)]^+_\EE$, see \cref{th:pullback}.

So far, we have introduced the ring  and Grassmann classes of submanifolds as its elements. 
We proceed with explaining the geometric meaning of the process of taking the product of these classes.

%%%
\subsection{Probabilistic intersection problems}\label{sec:PIP}

Consider the following question, generalizing the four lines problem above: 
given submanifolds $Y_1, \ldots, Y_s$ of the homogeneous space $M=G/H$ with 
$\sum_i \mathrm{codim}(Y_i)=\dim(M)$, 
what is the \emph{average} number of points 
in the intersection $g_1Y_1\cap \cdots \cap g_sY_s$, 
where the elements $g_1, \ldots, g_s\in G$ are chosen independently and at \emph{random} 
from the uniform distribution on $G$? We call this a probabilistic intersection problem. 
Strictly speaking, the problem amounts to  compute the integral
\be\label{eq:IGFintro}\
 \EE\, \#\left(g_1Y_1\cap \ldots \cap g_s Y_s\right) 
   :=  \frac{1}{\vol(G^s)} \,\int_{G^s}\#\left(g_1Y_1\cap \ldots \cap g_s Y_s\right)\mathrm{d}g_1\cdots \mathrm{d}g_s,
\ee
and falls in the subject of integral geometry~\cite{santalo:76, Howard}. 
 
One of our main results, \cref{thm:PSC}, 
allows to give an algebraic formulation of probabilistic intersection problems: 
the product of the Grassmann classes 
$[Y_1]_{\EE}\cdot \ldots \cdot [Y_s]_{\EE}$ 
equals a segment of length 
$\tfrac{1}{\vol(M)}\, \mathbb{E}\#\left(g_1Y_1\cap \cdots \cap g_s Y_s\right).$
In the special case where the $Y_i$ are hypersurfaces, 
we can express the average number of intersection points in terms of the mixed volume of the 
Grassmann zonoids $K(Y_i)$, which are now convex bodies in $V\simeq \R^n$ (recall $n=\dim(M)$):
\[
\label{eq:hyperEintro}
   \EE\,\#\left (g_1Y_1\cap\cdots\cap g_nY_n\right)
       =\vol(M)\cdot n!\cdot \mathrm{MV}(K(Y_1), \ldots, K(Y_n));
\]
see  \cref{coro:hypersurfaces}.

For  instance, in the four lines problem, $M=\OO(4)/(\OO(2)\times \OO(2))$
is the real Grassmannian of lines in projective space, $n=4$, and
the $Y_i$ can be taken to be the (same) Schubert variety of codimension one,
denoted $\Omega_{\Box}$ in~\cref{sec:psc}. 
Its zonoid $K(\Omega_{\Box})\subseteq V\simeq\R^2\otimes\R^2$ 
is the Segre zonoid, up to a constant factor, see~\cite{PSC}. 
The expectation of the number of real lines intersecting four random lines is therefore 
determined by the volume of the Segre zonoid.
The probabilistic intersection ring explains that this is not merely a counting formula: 
it encodes the geometry of the probabilistic intersection problem in a ring theoretical framework!

\begin{remark}
The map $\{\mathrm{submanifolds}\} \to \{\textrm{classes of zonoids}\}$ 
introduced in this paper is analogous in spirit to the map 
$\{\mathrm{subvarieties}\} \to \{\textrm{classes of polytopes}\}$ 
that arises from the embedding of the Chow ring of a toric variety into 
McMullen's polytope algebra \cite{FultonSturmfels}. At an abstract level, 
both constructions define a relationship of the form $\{\textrm{cycles}\} \to \{\textrm{convex sets}\}$,
with a group action playing a role in each case. In fact, our construction also resembles the one of the ring of conditions by De Concini and Procesi \cite{DCP}, which also belongs to Newton polyhedra theory \cite{KazarnovskiiKhovanskii}.
Formulating a unifying framework that encompasses both perspectives presents~a significant theoretical challenge. 
At present, no clear approach to establishing such a framework is~evident.
\end{remark}

%%%
\subsection{The probabilistic intersection ring and the cohomology ring}\label{sec:PIRC}

When $M$ is an oriented symmetric space, we will establish in \cref{th:HdOC} that, within the decomposition \eqref{eq:decointro}, the right-hand summand can be identified with the de Rham cohomology of $M$:  
\[\label{eq:isoDRc}
 \HdR(M) \simeq \mathrm{H}_{\mathrm{DR}}(M).
\]
In fact, this follows from fundamental results due to Elie Cartan~\cite{CartanE:29}, 
which imply that the de Rham cohomology of $M$ can be identified with the ring of $H$--invariants 
of~$\Lambda (V)$.
Furthermore, in \cref{propo:poincare}, we will show that if $Y$ is a coorientable submanifold, 
then under the isomorphism \eqref{eq:isoDRc}, 
the Grassmann class $[Y]_c$ corresponds to the Poincaré dual of $Y$. 
This result manifests itself in signed counting formulas (see \cref{thm:general2}). 
For example, if $Y_1, \dots, Y_n$ are cooriented hypersurfaces 
intersecting transversely, then their intersection number satisfies (\cref{coro:general2}):  
\[
\sum_{x\in Y_1\cap\cdots\cap Y_n } i_x(Y_1,\dots,Y_n;M)  
  = \vol(M) \cdot \det\left(c(Z_1) \wedge \cdots \wedge c(Z_n)\right),
\]
where, for taking the determinant, we are assuming the isomorphism $V\simeq \R^n$ given by the orientation.

%%%
\subsection{The probabilistic intersection ring of the complex projective space}

The probabilistic intersection ring of the sphere $S^n:=\OO(n+1)/\OO(n)$ and the real projective space 
are described in~\cref{sec_HE_sphere}.
Let us discuss now the application of the theory developed so far in the case of the complex projective space:
\[ 
 \CP^n:=\UU(n+1)/(\UU(1)\times \UU(n)).
\]

By~\eqref{eq:isoDRc}, we have 
$\HdR(\CP^n)\simeq \mathrm{H}_\mathrm{DR}(\CP^n)$ 
and 
 \[
 \HEo(\CP^n)\simeq \CGZ_o(\C^n)^{\UU(n)} 
\]
turns out to be the same as the algebra 
of $U(n)$--invariant 
translation invariant valuations on $\C^n$, 
see \cref{cor:crofton-1} and \cref{prop:findimtransact}.

The cohomological component is discussed in \cref{sec:cohomologyCPn}. 
Here, we focus on the structure of the centered probabilistic intersection ring  $\HEo(\CP^n)$. 
Our results can be thought as reformulations in our language of results obtained 
by Alesker~\cite{alesker:01,aleskerHLT}, Fu~\cite{FuUnval}, and 
Bernig and Fu~\cite{bernig-fu:11} in the context of valuations. 
However, our derivation is new and self--contained. (While we rely on 
Alesker’s dimension bound~\cite[Thm.~6.1]{alesker:01}, this step can also be done independently; 
see \cref{se:Kangle}.)

In degree one, $\HEo(\CP^n)$ is spanned by the class $\beta_n:=[B_{2n}]\in \HEo^1(\CP^n)$, 
where $B_{2n}\subseteq \C^n$ is the unit ball. In degree two, besides $\beta_n\wedge \beta_n$, 
we also have the Grassmann class of hyperplanes $\gamma_n:=[\CP^{n-1}]_\EE\in \HEo^2(\CP^n)$. 
The Grassmann classes of algebraic submanifolds are multiples of powers of this class (\cref{propo:subm_CPn}). 
The centered probabilistic intersection ring is generated by the classes~$\beta_n$ and $\gamma_n$ 
(see~\cref{th:basis-An}(1)):
\[ 
  \HEo(\CP^n)=\R[\beta_n, \gamma_n] .
\]
We prove the Hard--Lefschetz property 
and Hodge--Riemann relations for $\HEo(\CP^n)$ in \cref{th:basis-An}, using the
positive definiteness of the matrix of moments of the Beta distribution. 
The ideal of relations between the generators $\beta_n, \gamma_n$ is derived in \cref{se:genFU}. 
The description of these relations is nontrivial (\cref{pro:gen-alg}).
Using functoriality of the construction with respect to the inclusion $\CP^n\hookrightarrow \CP^{n+1}$, 
these simplify in the limit and one gets a polynomial ring (see \cref{eq:invlim-HECPn}):
$
 \varprojlim \HEo(\CP^n) \simeq \R[\b,\g].
$

To conclude our study of complex projective space, 
we derive an interesting probabilistic self--intersection formula 
for a real codimension--two submanifold $Y\subseteq \CP^n$. 
In order to state it, denote 
\[ 
 d_Y:=\EE\# \left(Y\cap g \CP^{n-1}\right)\quad \textrm{and}\quad  \Delta_Y:=d_Y-\frac{ \vol(Y)}{ \vol(\CP^{n-1})}.
\]
In \cref{cor:selfintcodim2CPn} we prove that 
\begin{equation}\label{eq:expand}
  \EE \#(g_1 Y\cap\ldots\cap g_n Y) =\sum_{k=0}^{\left\lfloor\tfrac{n}{2}\right\rfloor}
      \binom{n}{2k}\binom{2k}{k}\left(\frac{n}{2(n-1)}\right)^{2k} \, d_Y^{n-2k} \Delta_Y^{2k}.
\end{equation}
If $Y$ is an algebraic irreducible set, then $d_Y=\deg(Y)$ 
and $\Delta_Y=0$, as can be shown by integral geometry, see~\eqref{eq:degisvol}.
In this case the theorem gives 
$ \EE\#(g_1 Y\cap\cdots\cap g_n Y)=(\deg Y)^n$, which is true by Bézout's Theorem.
When $Y$ is ``close'' to a complex irreducible hypersurface, 
the quantity $\Delta_Y$ is ``small'' and we may interpret \eqref{eq:expand} 
as a perturbation of Bézout.

%%%
\subsection{Probabilistic Schubert Calculus}

To conclude the introduction, let us go back to the origin of this story: 
the request for a ring--theoretical framework for the probabilistic Schubert calculus envisioned in \cite{PSC}. 
We shall now outline the main ideas from \cref{sec:psc}, 
which aims at using the above results for introducing such framework.  Here, we are working with the homogeneous space
\[     
 G(k,m):=\OO(k+m)/\left(\OO(k)\times \OO(m)\right),
\]
and the space $V\simeq \R^k\otimes \R^m$, with the left--right action of $\OO(k)\times \OO(m)$ on it. 
If $k,m>1$, this action is not transitive on the unit sphere of $V$, therefore
$\HEo(G(k,m))$ is an infinite dimensional real vector space 
 (\cref{cor:Grassinfinitedim}). 
This makes the study of this ring considerably more complicated than the case of complex projective spaces.

Motivated by the classical framework, for every Young diagram $\lambda\subseteq [k]\times [m]$, 
we denote by $\Omega_\lambda\subseteq G(k,m)$ the corresponding Schubert variety; see \eqref{eq:schubertcycles}. 
\emph{Probabilistic Schubert calculus} refers to computing in the subring of $\HEo(G(k,m))$ 
generated by the Grassmann classes $[\Omega_\lambda]_\EE$, 
where $\lambda$~ranges over all partitions with $\lambda \subseteq [k]\times [m]$. 
Given diagrams $\lambda_1, \ldots, \lambda_s$ such that $\sum |\lambda_i|=km$ 
(i.e., such that the codimensions of the corresponding Schubert varieties add up to 
the dimension of the Grassmannian), 
a  \emph{probabilistic Schubert problem} (see \cref{def:PS-P}) asks for the expectation
\[
 E(\lambda_1, \ldots, \lambda_s):=\EE\#\left(g_1\Omega_{\lambda_1}\cap \cdots \cap g_s\Omega_{\lambda_s}\right) .
\]
We know very little on these numbers, except for the case $k=1$ or $m=1$ (where they all equal $1$) 
and the case $\lambda_1=\cdots=\lambda_{km}= \Box$, 
where their asymptotic behavior is known for $k$ fixed and 
$m\to \infty$ (see \cite{PSC, lerario-mathis:21}). 
For instance, in the case $k=2$, as $m\to \infty$ 
we have \cite[Theorem 6.8]{PSC}\footnote{The constants look different from \cite[Theorem 6.8]{PSC}, because there the asymptotic is with respect to $n=m+1$.}:
\[ 
 E(\Box, \ldots, \Box)=\frac{2}{3}\pi^{-1/2}\left(\frac{\pi^2}{4}\right)^mm^{-1/2}\left(1 + O(m^{-1})\right).
\]

The asymptotic expansions are computed using the fact that these numbers are (up to constants) 
the volumes of the (Segre) zonoids $B^k\otimes B^m$, 
which are associated to the variety~$\Omega_{\Box}$. 
In general, zonoids associated to Schubert varieties admit a simple description, as explained in \cref{thm_schubert_zonoid_simple}. We call these zonoids \emph{Schubert zonoids} and denote them by $K_\lambda:=K(\Omega_{\lambda})$. 
Their special role manifests itself in the 
orthogonal decomposition 
$$
 \Lambda^d(\R^k\ot\R^m) = \bigoplus_{\lambda} V_\lambda,
$$
where $V_\lambda$ denotes the linear span of $K_\lambda$, 
and the sum runs over all Young diagrams $\la\subseteq [k]\times [m]$ with $d$~boxes (\cref{th:Lambda-decomp}). 
The proof relies on representation theory: we show that the complexifications 
of the $V_\lambda$ are the (multiplicity free) irreducible submodules of the complex exterior power 
$\Lambda_\C^d(\C^k\ot\C^m)$ with respect to the action of $\Gl(\C^k)\ti\Gl(\C^m)$
(\cref{cor:C-Lambda-decomp}). 

The fact that $\HEo(G(k,m))$ is infinite dimensional reflects 
in a ring structure which is difficult to analyze. 
For instance, the product of Grassmann Schubert classes 
$[\Omega_\Box]_\EE\cdot [\Omega_\Box]_\EE$
in $\HEo(G(2,2))$
is \emph{not} a linear combination of the Grassmann Schubert classes
of the Young diagrams ${\tiny\yng(2)}$ and ${\tiny\Yvcentermath1\yng(1,1)}$ (\cref{prop:no real pieri for classes}).
As a first step toward understanding how Schubert zonoids multiply, 
we study the splitting of the wedge product $
 V_\la \wedge V_\mu := \Span\big\{ u\wedge v \mid u\in V_\lambda, v\in V_\mu \big\} $. 
 In \cref{th:wedge-decomp} we prove that, for partitions $\lambda$ and $\mu$ contained in $[k] \ti [m]$, 
 we have 
\begin{equation}\label{eq:facinating}
 V_\lambda \wedge V_\mu = \bigoplus_{\nu} V_\nu,
\end{equation}
where the sum runs exactly over those Young diagrams $\nu$ in $[k] \ti [m]$, 
which appear in the decomposition of the product of the corresponding classes 
of complex Schubert varieties  
in the cohomology ring of complex Grassmannians, see~\eqref{eq:mult-CSsch}.  
(Concretely, this means positivity of the Littlewood--Richardson coefficients of $\la,\mu,\nu$.)
The formula~\eqref{eq:facinating}
implies nontrivial intersection rules for real Schubert varieties,
see \cref{cor:prob-of-intersecting-Schub-var}. 
Despite this, the true nature of the probabilistic intersection ring of real Grassmannian, 
still remains to be discovered. 
In the language of valuations, this amounts to 
determine the algebra of translation invariant valuations on $\R^k \ot \R^m$ 
that are invariant under the action of the group $\OO(k)\ti \OO(m)$.

%%%
\subsection*{Organization}

In \cref{se:zonoids} we recall basic facts on zonoids from~\cite{BBLM} and extend them 
to the noncentered case. \cref{sec:algebra} is devoted to the construction 
and detailed discussion of the Grassmann zonoid algebra and their quotient algebra~$\CGZ(V)$ 
of classes of Grassmann zonoids. The goal of~\cref{sec:valuation} is to explain that classes 
of centered Grassmann zonoids determine even translation invariant valuations on $V$ 
and that the zonoid wedge multiplication 
is compatible with Alesker's multiplication of valuations.
After all these preparations, we introduce the probabilistic intersection ring in~\cref{sec:pit} 
and prove the main theorem (\cref{thm:PSC}). 
The connection to cohomology is explained in~\cref{sec:classical}.
The remaining two sections illustrate the general theory: 
\cref{sec:complex} gives a detailed description of the probabilistic intersection ring of complex projective spaces.
Finally, \cref{sec:psc} takes first steps towards probabilistic Schubert calculus.

\subsection*{Acknowledgments}
We thank Andreas Bernig and Jan Kotrbat\'{y} 
for interesting discussions. We are particularly thankful to Askold Khovanskii, 
whose series of talk in SISSA in 2019 inspired many ideas of this paper.
We especially thank Mateusz Michalek for a discussion at the Combinatorial Coworking Space 2024, 
which led to the combinatorial proof of~\cref{prop:computealphak}. 
A.~Lerario acknowledges the support of the INdAM group GNSAGA.

%%%%%%%%
\bigskip
\section{Basics on zonoids}\label{se:zonoids}
 
We recall basic facts on centered zonoids from~\cite{BBLM,Mathis_Stecconi}
and extend them to the noncentered case, which will be necessary for capturing cohomology.
Throughout the paper, $V$~denotes a Euclidean vector space of dimension $n$. 
While it is convenient to assume the choice of an inner product,
most of our constructions do not depend on it; this will we pointed out at various places.

\subsection{Centered, noncentered and virtual zonoids}\label{se:c-nc-zonoids}

A \emph{segment} in $V$ is the convex set hull 
\begin{equation}\label{notation_segments}
 [a,b] := \{ta + (1-t)b\mid 0\leq t\leq 1\} 
\end{equation}
of two points $a,b \in V$. 
A \emph{zonotope} $Z\subseteq V$ is the Minkowski sum of finitely many line segments. 
\emph{Zonoids} arise as limits of zonotopes in the Hausdorff metric.
The unit ball of $V$ 
is an example of a zonoid, which is not a zonotope (see \eqref{ex2} below).
Zonoids have a \emph{center} $c(Z)$ of symmetry (see \cite[p.~192]{bible}).
It can be obtained as 
\[\label{eq:centerzonoid}
 c(Z) := \frac{1}{\vol(Z)} \int_Z x\,\mathrm d x ; 
\]
where $\mathrm d x$ denotes the 
Lebesgue measure on the affine span of $Z$. 
We call 
$Z - c(Z)$ the \emph{centered version} of $Z$; 
it is centrally symmetric with respect to $c(Z)$.
The zonoid is called \emph{centered} if $c(Z) =0$.
Note that $c([a,b]) = \tfrac12 (a+b)$.
We define 
\[\label{eq:zonoidsdef}
 \ZZ(V):= \{Z\subseteq V\mid Z \text{ is a zonoid}\}, \quad 
 \ZZ_o(V) := \{K\in \ZZ(V)\mid K \text{ is centered}\} 
\]
and adopt the convention that $Z$ denotes zonoids and $K$ centered ones.
Both $\ZZ(V)$ and $\ZZ_o(V)$ are closed in the Hausdorff topology. 
They are semigroups with respect to Minkowski addition. 
More precisely, we may view them as \emph{semi-vector spaces},  
since there is a scalar multiplication of their elements with nonnegative reals, 
which is compatible with the addition.
We have the semigroup isomorphism  
\begin{equation}\label{eq:Z-iso}
 \ZZ(V) \simto \ZZ_o(V) \oplus V ,\quad 
  Z \mapsto (Z- c(Z),\ 2 \cdot c(Z))
\end{equation}
The reason for putting the factor $2$ 
is to obtain the formula in~\eqref{eq:VZ-iso-E}.

Every $Z\in \ZZ(V)$ is uniquely determined by its \emph{support function} 
\[\label{eq:supportfctdef}h_Z(u):=\max\{\langle u, x\rangle \mid x\in Z\}.\]
We note the following formulas for the support functions of segments
\begin{equation}\label{eq:supp-fct-segm}
  h_{[0,x]}(u) = \max\{0, \langle u,x\rangle\}, \quad 
  h_{[-x,x]}(u) = \vert\langle u,x\rangle\vert .
\end{equation}
A zonoid $Z$ is centered iff $h_Z$ is an even function. 
Moreover, $h_{K + L} = h_K + h_L$ 
for $Z,L\in \ZZ(V)$ and~$h_{\lambda K} = \lambda h_K, \lambda \geq 0$;
see \cite[Theorem 1.7.5]{bible}.
This implies that the Minkowski sum satisfies the cancellation property:
$K+M=L+M$ implies $K=L$. 
Consequently, we can embed the semigroup $\ZZ(V)$ in the (Grothendieck) group~$\VZ(V)$
of\emph{ virtual zonoids}, whose elements are formal differences of elements from $\VZ(V)$. 
Since we have scalar multiplication, $\VZ(V)$ is a real vector space. 
Similarly, we define the subvector space $\VZ_o(V)\subseteq \VZ(V)$.
We associate with a formal difference $Z_1-Z_2\in\VZ(V)$ 
the support function $h_{Z_1-Z_2}:=h_{Z_1} - h_{Z_2}$. 
\cref{eq:Z-iso} extends to a linear isomorphism 
\begin{equation}\label{eq:VZ-iso} 
  \VZ(V)  \simeq \VZ_o(V) \oplus V . 
\end{equation}
Suppose $\varphi\colon V\to W$ is a linear map, 
where $W$ is a Euclidean space and denote by 
$\varphi^T\colon W\to V$ its adjoint map.
The image $\varphi(Z)$ of a zonoid $Z\in\ZZ(V)$ 
is again a zonoid; moreover, $\varphi$ maps the center $c(Z)$ 
to the center of $\varphi(Z)$. 
The support function of the image $\varphi(K)$ of $K\in\ZZ(V)$ 
satisfies (see \cite[Prop.~2.1(4)]{BBLM}) 
\begin{equation}\label{eq:supp-fct-im}
 h_{\varphi(K)} = h_K \circ \varphi^T .
\end{equation}
The semilinear map $\ZZ(V)\to\ZZ(W),\, K\mapsto \varphi(K)$
induces a linear map 
$\varphi_*\colon \VZ(V)\to \VZ(W)$. 
Clearly,
$\varphi_*$~maps $\VZ_o(V)$ to $\VZ_o(W)$.
This assignment is functorial.

%%%
\subsection{Zonoids and random vectors}\label{se:Zon-RV}

Random vectors provide a convenient and flexible way to represent zonoids; 
see \cite[Section 2.3]{BBLM}.
In the sequel, $\xi\in V$ denotes a random vector taking values in~$V$. 
We assume $\xi$ to be \emph{integrable}, which means $\EE\|\xi\| < \infty$. 
This notion is independent of the choice of the inner product.

The intuition is to take expectations of the random segments  $[0,\xi]$ and $\tfrac12[-\xi,\xi]$,
which we call the \emph{Vitale zonoids}  of $\xi$. 
The shortest way to introduce them rigorously is to define them through their support functions, 
following~\cref{eq:supp-fct-segm}:
\begin{equation}\label{eq:defZK}
 \z(\xi)\in \ZZ(V), \quad h_{\z(\xi)}(u) := \EE \max\{0, \langle u,\xi \rangle\}, \quad
 \cz(\xi)\in \ZZ_o(V), \quad  h_{\cz(\xi)}(u) := \tfrac{1}{2}\EE \vert\langle u,\xi \rangle\vert.
\end{equation}
Note that $K(\xi)$ is centered.  
Vitale~\cite[Theorem 3.1]{vitale} showed that every zonoid is of the form $Z = v + Z(\xi)$ 
for some integrable random vector $\xi$ and $v\in V$. 
Using $Z(\xi) = \tfrac12 \EE \xi + K(\xi)$, we can write the zonoid as
$Z = v  + \tfrac12 \EE \xi + K(\xi)$ and hence $c(Z) = v  + \tfrac12 \EE \xi$.
Therefore, every centered zonoid is of the form $K(\xi)$;
however, $\xi$ is not uniquely determined by the zonoid.

\begin{example}
Let $\xi\in\R$ denote the random variable taking the values $a<b$
with equal probability. Then, $K(\xi) = \tfrac12 [-r,r]$ with 
$r:=\tfrac12 (|a| + |b|)$, since by \eqref{eq:supp-fct-segm} and \eqref{eq:defZK}, 
$$
 h_{\cz(\xi)}(u) = \tfrac12 \EE | \xi u| = \tfrac14 (|au| + |bu|) = \tfrac12 h_{[-r,r]}(u) .
$$
\end{example}

We note that $\z(\xi)=\cz(\xi)$ if $\xi\sim -\xi$, i.e., $-\xi$ has the same distribution as~$\xi$. 
The important equation 
\begin{equation}\label{eq:Z=K+12E}
 \z(\xi) = \cz(\xi) + \tfrac{1}{2} \EE (\xi) , 
\end{equation}
follows, using $\max\{0,a\} = \tfrac12 \big( |a| + a \big)$,
from 
$$
 h_{\z(\xi)}(u) =\EE \max\{0, \langle u,\xi \rangle\} =
   \EE \tfrac{1}{2} \vert\langle u,\xi \rangle\vert + \EE \tfrac12 \langle  u,\xi \rangle = 
    h_{\cz(\xi)}(u) + h_{\tfrac12 \EE \xi}(u) .
$$
By~\eqref{eq:Z=K+12E}, the center of $Z(\xi)$ satisfies 
\begin{equation}\label{eq:c=12E}
 c(Z(\xi)) = \tfrac 12 \EE (\xi) .
\end{equation}

Using the above introduced notations, the isomorphism in~\cref{eq:Z-iso} nicely reads as 
\begin{equation}\label{eq:VZ-iso-E} 
  \ZZ(V) \simto \ZZ_o(V) \oplus V ,\quad   Z(\xi) \mapsto (K(\xi),\ \EE(\xi)) .
\end{equation}

\begin{example}\label{ex1}
Consider a random vector $\xi\in\mathbb R^2$ with two states $(2,2)$ and $(2,0)$, 
each taken with probability $\tfrac{1}{2}$. 
The support function of $\z(\xi)$ is 
$h_{\z(\xi)}(u) =  \max\{0,\langle (1,1),u\rangle\} + \max\{0,\langle (1,0),u\rangle\}$. 
The latter is the support function of the parallelogram with vertices $(0,0),(1,1),(1,0)$ and $(2,1)$, 
see~\Cref{fig1}. 
We have $ \EE (\xi) = (2,1)$ and 
the centered zonoid $\cz(\xi)$ is the translate of the parallelogram $\z(\xi)$ by $-\tfrac{1}{2}\EE \xi$.
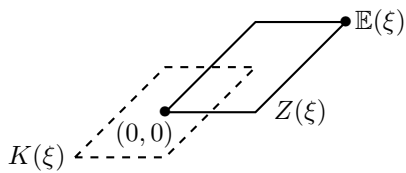
\begin{figure}
\begin{center}
\begin{tikzpicture}[scale = 1.2]
\draw[thick] (0,0) -- (1,1) -- (2,1) -- (1,0) -- cycle;
\draw[thick, dashed] (-1,-0.5) -- (0,0.5) -- (1,0.5) -- (0,-0.5) -- cycle;
\draw (1.1,0) node[right] {$\z(\xi)$};
\draw (-1,-0.5) node[left] {$\cz(\xi)$};
\draw (0,0) node {$\bullet$};
\draw (0.2,0) node[below left] {$(0,0)$};
\draw (2,1) node {$\bullet$};
\draw (2,1) node[right] {$\EE(\xi)$};
\end{tikzpicture}
\end{center}
\caption{\label{fig1}
The Vitale zonoids of a random vector $\xi$ with two states.} 
\end{figure}
This example can be generalized to any random vector $\xi\in V$ with finitely many states. 
Such a vector will always define  a zonotope. From this perspective, passing from zonotopes 
to zonoids corresponds to passing from random vectors with finitely many states 
to general integrable random vectors in $V$.
\end{example}

In~\cite[Def.~2.10]{BBLM} we defined the \emph{length} 
of a centered zonoid by 
\begin{equation}\label{eq:ell} 
 \ell(\cz(\xi)):=\EE\|\xi\| 
\end{equation}
and showed that this extends to a well defined linear functional $\ell:\VZ_o(V)\to \R$. 
The length depends on the choice of the inner product.

We further extend this in a translation invariant way on noncentered zonoids
to obtain the linear functional 
$$
 \ell:\VZ(V)\to \R,\; \ell(Z) := \ell \big(Z -c(Z)\big) ,
$$
compare~\eqref{eq:Z=K+12E}. 
The length $\ell(Z)$ equals the first intrinsic volume of $Z$,
see~\cite[Thm.~5.2]{BBLM}. It is thus monotone with respect to inclusion.
We also note that 
\begin{equation}\label{eq:E-inequ}
\|c(Z)\| \le  \tfrac 12 \ell (Z) .
\end{equation}
for any $Z\in \ZZ(V)$ with $0\in  Z$. 
This holds because of the general inequality 
$\|\EE (\xi) \| \le \EE \|\xi\|$ 
and since~$Z$ has form $Z=Z(\xi)$.

\begin{example}
The distribution of the standard Gaussian vector $\xi\in V$ is rotationally invariant.
Hence $\z(\xi)$ is rotationally invariant and thus a multiple of  the centered unit ball~$B_n$.
By \cite[Eq.~(2.9)]{BBLM}, 
\begin{equation}\label{ex2} 
 \z(\xi) = \cz(\xi)=(2\pi)^{-1/2}\, B_n ,\quad 
\ell(B_n) = \sqrt{2\pi}\, \EE(\|\xi\|) = 2\sqrt{\pi}\, \frac{\Gamma(\tfrac{n+1}{2})}{\Gamma(\tfrac{n}{2})} .
\end{equation}
\end{example}
For later use, we record the formulas for 
the volume of unit spheres and unit balls:
\begin{equation}\label{volballsphere} 
 \kappa_n := \vol(B_n) = \frac{1}{n} \, \vol(S^{n-1}) = \frac{2\pi^{\tfrac{n}{2}}}{n\Gamma(n)} .
\end{equation}

\begin{example}\label{ex:ZR}
For $V=\R$, the semigroup isomorphism~\eqref{eq:VZ-iso-E} specializes to 
$\ZZ(\R) \simto \R_+ \oplus \R$, which sends 
the interval $[a,b]$ of length $\ell := b-a$ to the pair consisting of the 
centered interval $\tfrac12 [-\ell,\ell]$ and the center $\tfrac12 (a+b)$. 
This extends to the linear isomorphism
$$
 \VZ(\R) \simto \R\oplus \R,\quad Z \mapsto (\ell(Z), 2\, c(Z)) .
$$
\end{example}

The representation of zonoids by random vectors is compatible with the functoriality: 
we have 
$$
 \varphi(Z(\xi)) = Z(\varphi(\xi)), \quad \varphi(K(\xi)) = K(\varphi(\xi)) . 
$$
Moreover, we have a natural action of the group $\Gl(V)$ on zonoids and thus on $\VZ(V)$. 
The subspace $\VZ_o(V)$ is $\Gl(V)$--invariant. 
This action is nicely compatible with the description of zonoids 
by random variables: one easily shows that
$gZ(\xi) = Z(g\xi)$ and $gK(\xi) = K(g\xi)$
if $\xi\in V$ is an integrable random variable and $g\in\Gl(V)$.

%%%
\subsection{Pairing of zonoids}\label{se:pairing} 

There is a notion of tensor product of zonoids~\cite{Aubrun2016},
which implies that any multilinear map 
$M\colon V_1\times\ldots\times V_p \to W$ 
of finite dimensional real vector spaces 
induces a multilinear map 
$\widehat{M}\colon \VZ_o(V_1)\times\ldots\times \VZ_o(V_p) \to \VZ_o(W)$, 
characterized by 
\[\label{eq:FTZCdef}\widehat{M}(K(\xi_1,\ldots,\xi_p)) = K(M(\xi_1,\ldots,\xi_p))\]
for independent, integrable random vectors $\xi_i\in V_i$;
see~\cite[\S 3--\S 4]{BBLM}.

We apply this general principle to the 
given inner product $V\times V \to \R, (x,z)\mapsto \langle x,z\rangle$ on $V$.
Using the identification $\VZ_o(\R)=\R$, we obtain 
a symmetric bilinear map
\begin{equation}\label{eq:def-pairing}
   \VZ_o (V)\times \VZ_o (V)\to \R,\quad  (K,L) \mapsto \langle K,L\rangle ,
\end{equation}
which is characterized by 
\begin{equation}\label{eq:pairingdef}
    \langle \cz(\xi), \cz(\zeta) \rangle =\EE \,| \langle \xi, \zeta\rangle| ,
\end{equation}
where $\xi,\zeta\in V$ are independent, integrable random vectors. 
In particular, taking $\eta$ with constant value $z\in V$, 
we get with~\eqref{eq:defZK}, 
\begin{equation}\label{eq:pairing-1}
  \langle K, \tfrac{1}{2}[-z,z] \rangle = 2\,h_{\cz}(z) .
\end{equation}
This shows that our pairing on $\VZ_o (V)$ is nondegenerate. 
The pairing depends on the choice of the inner product.

By taking expectations in \eqref{eq:pairing-1}, 
we obtain a useful formula: 
for an integrable random $\eta\in V$:
\begin{equation}\label{eq:pairing-2}
 \langle K, K(\eta)\rangle = 2\, \EE h_K(\eta) .
\end{equation}

\begin{lemma}\label{le:adjoint}
Suppose $\varphi\colon V\to W$ is a linear map of Euclidean spaces with 
adjoint map $\varphi^T\colon W \to V$, i.e., 
$\langle \varphi (v), w \rangle = \langle v, \varphi^T (w)\rangle$ 
for all $v\in V$ and $w\in W$. Then, 
$\langle \varphi_* (K), L \rangle = \langle K, (\varphi^T)_* (L)\rangle$ 
for all $K\in \ZZ_o(V)$ and $L\in \ZZ_o(W)$. 
\end{lemma}

\begin{proof}
By bilinearity, it suffices to prove this for zonoids $K\in \ZZ_o(V)$ and $L\in \ZZ_o(W)$, 
say $K=K(\xi)$ and $L=K(\zeta)$ for integrable random variables $\xi\in V$ and $\zeta\in W$. 
The assertion follows from \eqref{eq:def-pairing}, noting that 
$\varphi(K(\xi)) = K(\varphi(\xi))$ and $\varphi^T(K(\zeta)) = K(\varphi^T(\zeta))$.
\end{proof}

We remark that 
the pairing on $\VZ_o (V)$ can be extended in a unique way 
to a translation invariant pairing of $\VZ(V)$. 

%%%
\subsection{Zonoids as measures}
\label{se:zon-meas-top}

The space $\ZZ_o(V)$ of centered zonoids 
can be identified with the space $\cM_+(\proj(V))$ of measures~$\mu$ on real projective space~$\proj(V)$, 
see \cite[\S 3.5]{bible}. We briefly explain this correspondence. 
The \emph{cosine transform} of $\mu$ 
is the continuous function 
$\cos(\mu)\colon V \to \R$ defined by  
\begin{equation}\label{cosine_transform}
  \cos(\mu)(u) :=  \int_{[x]\in\mathbb P(V)} \frac{|\langle u,x\rangle |}{\|x\|} \, \mathrm d \mu(\mathrm [x]) , 
   \quad u \in V .
\end{equation}
For example, 
$\cos(\delta_{[x]})(u) = |\langle u,x\rangle |/\|x\|$
for the Dirac measure $\delta_{[x]}$. 
In fact, $\cos(\mu)$ could be defined more generally as a function on the dual space of $V$.

By~\cite[Thm. 3.53--3.5.4]{bible} the cosine transform
$\cos(\mu)$ is the support function of a uniquely determined centered zonoid $K_\mu\in \ZZ_o(V)$; 
moreover, any $K\in \ZZ_o(V)$ is obtained this way. 
One calls this measure $\mu_K$ the \emph{generating measure} of $K$. 
This defines the linear isomorphism 
\begin{equation}\label{eq:VZ0-to-cM}
  \VZ_o(V) \simto \cM(\proj(V)), \quad K_1 - K_2 \mapsto \mu_{K_1} -  \mu_{K_2},
\end{equation}
from the vector space $\VZ_o(V)$ of virtual centered zonoids in~$V$
to the space $\cM(\proj(V))$ of signed measures on $\proj(V)$.
The viewpoint of random vectors leads to the following description of the 
inverse of this linear isomorphism from~\cite[Prop.~2.29]{BBLM}.

\begin{proposition}\label{pro:X2measure}
Suppose $\xi\in V$ is an integrable random variable 
with probability measure $\nu$ on $V$. 
Then the generating measure $\mu_K$ of $K=K(\xi)$ satisfies, 
for all continuous functions $f$ on $\proj(V)$, 
$$
 \int_{\proj(V)} f \, \mathrm d\mu_K \ =\ 
 \int_{V} \|x\| \,  f([x]) \, \mathrm d\nu (x) ,
$$
where we formally set $f([0]):=0$. 
Thus $\mu_K$ is the pushforward measure with respect to the canonical map 
$V\setminus 0 \to \proj(V)$
of the measure~$\nu'$, defined by 
$\nu'(A) := \int_A \|x\| \, \mathrm d\nu(x)$.
In particular, if $\xi$ only takes values in the unit sphere, then 
$\mu_K$ is the pushforward of $\nu$.
\end{proposition}

\begin{remark}\label{re:funct-measure}
Recall the natural action of the group $\Gl(V)$ on 
centered zonoids and thus on $\VZ_o(V)$. 
We note that the induced $\Gl(V)$--action on $\cM(\proj(V))$ 
induced by \eqref{eq:VZ0-to-cM} is not so easily described, 
which indicates the advantage of working with zonoids 
over working with measures. 
This can be seen from the description of the linear map 
$\varphi_*\colon\cM(\proj(V))\to\cM(\proj(W))$
induced by a linear map $\varphi\colon V \to W$ 
when expressed in terms of measures; 
see the proof of~\cref{le:FA-lin-map}.
\end{remark}

%%%
\subsection{Topologies and continuity}\label{se:topologies}

One can endow the infinite dimensional vector space $\VZ_o(V)$ 
with several topologies. 
In the paper~\cite{BBLM}, the focus was on the weak--$\ast$ topology,  
for which several continuity properties were proven. 
This paper studied the norm of a virtual zonoid $Z_1-Z_2$
defined as the Hausdorff distance of the zonoids $Z_1$ and $Z_2$; 
see~\cite[Thm.~2.16]{BBLM}.
However, the corresponding normed vector space is not complete, 
which makes this norm less useful; see~\cite[Prop.~2.24]{BBLM}.
In addition, the tensor product of virtual zonoids is not a continuous bilinear map
for this norm.

Here, we show that a different choice of norm in fact turns $\VZ_o(V)$ into a 
Banach space. Moreover, the tensor product of virtual zonoids becomes 
a continuous bilinear map. 
For this purpose, we rely on the measure theoretic viewpoint of~\cref{se:zon-meas-top}.

By the Riesz representation theorem, 
the space $\cM(\proj(V))$ of real measures on the compact space~$\proj(V)$ 
can be identified with the dual space of the Banach space $C(\proj(V))$ 
of real valued continuous functions on~$\proj(V)$ with respect to 
the sup--norm. This way, $\cM(\proj(V))$ is a Banach space with 
respect to the dual norm of the sup-norm, 
which is called \emph{total variation norm}, see~\cite[Chap.~6]{rudin:87}.
Via the isomorphism~\eqref{eq:VZ0-to-cM},
$\VZ_o(V)$ gets the structure of a Banach space; 
we call the resulting norm the \emph{total variation norm of virtual zonoids} 
and simply denote it by~$\|\ \|$.
Moreover, $\cM(\proj(V))$ carries the \emph{weak--$\ast$ topology},
which via~\eqref{eq:VZ0-to-cM} define a topology on~$\VZ_o(V)$
that we will call the weak--$\ast$ topology on~$\VZ_o(V)$.
Let us point out that while the total variation norm depends 
on the chosen inner product, the induced topology does not.

Let $\varphi\colon V\to W$ be a linear map of Euclidean spaces.
We assigned to $\varphi$ the induced linear map $\varphi_*\colon \VZ_o(V) \to \VZ_o(W)$
in a functorial way. The next result shows continuity with respect 
to both topologies. 
 
\begin{proposition}\label{le:FA-lin-map}
Let $\varphi\colon V\to W$ be a linear map of Euclidean spaces with operator norm~$\|\varphi\|$.
Then the operator norm $\|\varphi_*\|$ of the induced linear map 
satisfies $\|\varphi_*\| = \|\varphi\|$. 
Hence $\varphi_*$ is a continuous linear map of Banach spaces. 
Moreover, $\varphi_*$ is sequentially continuous with respect to the weak--$\ast$ topology.
\end{proposition}

\begin{proof}
The assertion on the weak--$\ast$ continuity of $\varphi_*$
was proven in~\cite[Lemma~2.31]{BBLM}. We argue here in the same way, 
relying on the following functional analytic description of 
the linear map $\widetilde{\varphi}_*\colon \cM(\proj(V)) \to \cM(\proj(W))$
corresponding to $\varphi_*$ via the isomorphisms~\eqref{eq:VZ0-to-cM}.
For this, we first assign to~$\varphi$ the linear map
\begin{equation}\label{eq:MC}
 \varphi^*\colon C(\proj(W)) \to C(\proj(V)),\: f \mapsto \varphi^* f ,
\end{equation} 
where $\varphi^* f \in C(\proj(V))$ is defined by 
by setting for a representative $x\in V$: 
\begin{equation}\label{eq:def-linear-map}
 (\varphi^* f)([x]) := \left\{
   \begin{array}{ll}
   \frac{\|\varphi(x)\|}{\|x\|}\, f([\varphi(x)])  & \mbox{if $\varphi(x)\ne 0$,} \\
               0               & \mbox{otherwise.} 
   \end{array}\right.
\end{equation} 
Clearly, $\varphi^*$ is well defined,
continuous on $\proj(V)$, and 
linear.
We have  
$\|\varphi^* f\|_\infty \le \|\varphi\| \cdot \| f\|_\infty$ 
with equality holding for $f=1$, 
which shows that $\|\varphi^* \| = \|\varphi\|$. 
Finally, the proof of~\cite[Lemma~2.31]{BBLM} reveals that 
the dual map 
of $\varphi^*$ equals $\widetilde{\varphi}_*$.
Therefore, 
$\|\widetilde{\varphi}_*\| = \|\varphi^*\| = \|\varphi\|$,
which proves the stated norm equality.
\end{proof}

Our next goal is show that the weak--$\ast$ continuity statement in~\cite[Thm.~4.1]{BBLM}
also applies to the Banach space topology. 
According to this theorem,  a multilinear map
$\psi\colon V_1 \ti\ldots \ti V_p \to W$ of Euclidean spaces induces 
a multilinear map 
$\widehat{\psi} \colon \VZ_o(V_1) \ti\ldots \ti \VZ_o(V_p) \to \VZ_o(W)$ 
with the property that~$\widehat{\psi}(K(\xi_1),\ldots,K(\xi_p)) = K(\psi(\xi_1,\ldots,\xi_p))$
for any independent integrable random variables $\xi_i$ taking values in $V_i$.
Let us define 
$\|\psi\| := \|\varphi\|$, where 
$\varphi\colon V_1 \ot\ldots \ot V_p \to W$ 
is the unique linear map 
such that $\psi(v_1,\ldots,v_p) = \varphi(v_1 \ot\ldots\ot v_p)$.

\begin{proposition}\label{pro:main-zon-banach}
We have 
$\| \widehat{\psi}(K_1,\ldots,K_p)\| \le \|\psi\|\, \|K_1\| \cdot\ldots\cdot \|K_p\|$ 
for $K_i\in \VZ_o(V_i)$.  
Hence $\widehat{\psi}$ is a continuous multilinear map of Banach spaces.
\end{proposition}

\begin{proof}
The functoriality of the zonoid construction implies 
$\widehat{\psi}(K_1,\ldots,K_p) = \varphi_*(K_1 \ot\ldots \ot K_p)$, 
where $K_1\ot\cdot\ldots\cdot \ot K_p$ denotes 
the tensor product of the zonoids $K_i$, see~\cite[Def.~3.1]{BBLM}.
Therefore,  
$$
 \| \widehat{\psi}(K_1,\ldots,K_p)\| \le \|\varphi_*\| \| K_1 \ot \ldots \ot K_p \|.
$$
Moreover, $\|\varphi_*\| = \|\varphi\| = \|\psi\|$ by~\cref{le:FA-lin-map}
and the definition of $\|\psi\|$.
It therefore suffices to prove that 
$$
 \| K_1 \ot \ldots \ot K_p \| = \| K_1\| \cdot\ldots \cdot \| K_p \| . 
$$
For doing so, we assume $p=2$ and write now $V=V_1$ and $V_2=W$ to simplify notation.

Put $X:= \proj(V)$ and $Y:=\proj(W)$. 
The construction of the product measure~\cite[Chap.~7]{rudin:87} leads to a bilinear map
\begin{equation}\label{eq:prod-of-meas}
 \cM(X)\times \cM(Y) \to \cM(X\times Y), \quad (\mu,\nu)\mapsto \mu\ti \nu , \quad 
\end{equation}
where the real measure $\mu\ti \nu\in \cM(X\times Y)$ is characterized by 
$\langle \mu \ti \nu, \phi \ot \psi \rangle = \langle \mu, \phi\rangle \langle \nu, \psi  \rangle$.
Here we remind the reader that we view a measure $\mu\in\cM(X)$ 
as a continuous linear functional sending $\phi \in C(X)$ to $\langle \mu, \varphi\rangle$.
With Fubini's Theorem, one shows that 
\begin{equation}\label{eq:TV-product}
 \|\mu \ti \nu\| = \|\mu\| \cdot \|\nu\| .
\end{equation}
The Segre map
$$
 X \ti Y \to \proj(V \ot W),\, ([x],[y]) \mapsto [x\ot y] .
$$
is a homeomorphism on its image. Therefore, its pushforward 
$$
 \cM(X \ti Y) \to \cM(\proj(V \ot W))
$$
preserves the norm. 
The composition 
$\Psi\colon\cM(X) \ti \cM(Y) \to \cM(\proj(V \ot W))$ 
of the above two maps therefore satisfies
$\|\Psi(\mu,\nu)\| = \|\mu\| \cdot \|\nu\|$. 
Hence $\Psi$ is continuous. 
Finally, \cite[Lemma~3.8]{BBLM} shows that 
$\Psi$ corresponds to the the tensor product map 
$\VZ_o(V) \ti \VZ_o(W) \to \VZ_o(V \ot W)$ 
of zonoids 
via the isomorphism~\cref{eq:VZ0-to-cM}.
\end{proof}

\cref{pro:main-zon-banach} implies that the bilinear pairing \eqref{eq:def-pairing} 
is continuous with respect to the Banach space topology.
By~\cite[Thm~4.1]{BBLM}, it is also weak--$\ast$ sequentially continuous.

The following density result will be used later.

\begin{lemma}\label{le:EinSbar}
Let $M\subseteq V$ be a subset and let $S\subseteq \ZZ_o(V)$ 
denote the set of of finite nonnegative linear combinations of centered 
segments $[-v,v]$, where $v\in M$.
Suppose $\xi\in V$ is an integrable random variable 
taking values in $M$. Then $K(\xi)$ is in the sequential closure of $S$
with respect to the weak--$\ast$ topology.
\end{lemma}
\begin{proof}
Let $\xi_1,\xi_2,\ldots$ be an independent sequence of i.i.d.\ copies of $\xi$.
Then the averages 
$$
 \eta_N := \tfrac{1}{2N}\big( [-\xi_1,\xi_1] + \ldots + [-\xi_N,\xi_N]\big)
$$
are in $S$. On the other hand, 
by the strong law of large numbers in~\cite{art-vit:75}, 
the zonoid $K(\eta_N)$ converges to $K(\xi)$ in the Hausdorff metric. 
Finally, \cite[Thm~2.26]{BBLM} states that Hausdorff topology 
on $\ZZ_o(V)$ coincides with the weak--$\ast$ topology.
\end{proof}

%%%%%%%%
\bigskip
\section{Zonoid algebras}\label{sec:algebra}

The paper~\cite{BBLM} focused on centered zonoids and defined the 
zonoid algebra and Grassmann zonoid algebra only in this framework. 
Here we generalize these constructions by considering also noncentered zonoids.
This generalizations also appears in~\cite{Mathis_Stecconi}. 

%%%
\subsection{Exterior algebra and Hodge star}\label{se:EP-HS}

Let us review some known facts about the exterior algebra, e.g., 
see~\cite[Chap.~XIX]{lang:02}.   
For a real vector space $V$, we consider the exterior power 
\begin{equation}\label{eq:def-exterior-algebra} 
\Lambda (V) = \bigoplus_{d=0}^n\Lambda^d V
\end{equation}
of $V$, which is a noncommutative, associative real algebra 
with respect to the wedge product. 
This construction is \emph{functorial}: 
a linear map $\varphi\colon V\to W$ induces for $d\in\N$ linear maps
\begin{equation}\label{eq:Lambda^d}
 \Lambda^d\varphi\colon\Lambda^d V \to \Lambda^d W , 
\end{equation}
which map $v_1\wedge\cdots\wedge v_d$ to $\varphi(v_1)\wedge\cdots\wedge \varphi(v_d)$.
These maps combine to a graded algebra homomorphism $\Lambda(V) \to \Lambda (W)$.
In particular, there is a natural action of $\Gl(V)$ on $\Lambda^d(V)$, 
which is irreducible~\cite{fulton-harris:91}. 

We also recall~\cite{Kozlov} that an inner product on $V$ induces 
an inner product on $\Lambda^d V$ defined by 
$\langle v_1\wedge\cdots\wedge v_d, w_1\wedge\cdots\wedge w_d\rangle := \det [\langle v_i,w_j\rangle ]_{i,j\le d}$. 
Note that 
$e_{i_1}\wedge\cdots\wedge e_{i_d}$, for $1\le i_1< \ldots <i_d\le n$,  
are orthonormal when $(e_i)$ is an orthonormal basis of $V$.
We extend the inner product to the exterior algebra by viewing $\bigoplus_d\Lambda^d V$ 
as an orthogonal decomposition.
Taking the adjoint commutes with taking the $d$--th exterior power of $\varphi$:
\begin{equation}\label{eq:Lambda-adjoint}
 \Lambda^d(\varphi^T) = \big(\Lambda^d\varphi\big)^T .
\end{equation}

By an \emph{orientation} of $V$ we understand 
a choice among the two unit elements $\orf\in \Lambda^nV$ 
satisfying $\langle \orf, \orf \rangle = 1$.
This defines the isomorphism 
$\Lambda^nV\simeq \R, \, \alpha \mapsto \langle  \alpha,\orf \rangle$.

Let us recall some facts about the \emph{Hodge star operation} on $\Lambda(V)$.
This is the collection of linear maps~$\star\colon\Lambda^d V \to \Lambda^{n-d} V$ 
characterized by
\begin{equation}\label{def_hodge_star}
 \forall x,y\in \Lambda^dV: \quad 
\langle x, y\rangle := \langle \orf, x\wedge \star y\rangle .
\end{equation}
If $e_1,\ldots,e_n$ is any positively oriented orthonormal basis of $V$, 
then $\star (e_1\wedge\ldots\wedge e_d)= e_{d+1}\wedge\ldots \wedge e_n$.
This implies that the Hodge star operations are linear isometries.
Moreover, they are  involutions up to signs: we have 
$\star \star  x=(-1)^d x$ for all $x\in\Lambda^d\Vect$.
By a \emph{simple vector} of degree~$d$ we understand a vector 
of the form $v_1\wedge \cdots \wedge v_d$.
The Hodge star maps simple vectors to simple ones. 
Moreover, one can check that the Hodge star respects the functoriality: 
if $\varphi\colon V\to W$ is an orientation preserving linear map, then 
$(\Lambda^{n-d}\varphi) (\star w) = \star (\Lambda^{d}\varphi) (w)$ 
for $w\in \Lambda^d V$. 

We note the following useful identity:
for $x\in\Lambda^{d_1}V$, $y \in \Lambda^{d_2}V$ ,and 
$z\in\Lambda^{d_1+d_2}V$, we have 
\begin{equation}\label{eq_hodge}
 |\langle x\wedge y, z\rangle|=|\langle x, \star  (y\wedge \star  z) \rangle|.
\end{equation}
Indeed, applying \eqref{def_hodge_star} twice we get 
$$
   \langle x\wedge y,  z\rangle  
 = \pm \langle \orf, x\wedge y\wedge \star  z \rangle 
 =  \pm \langle x, \star( y \wedge \star z ) \rangle .
$$

%%%
\subsection{Zonoids in exterior powers and wedge multiplication}\label{se:zon-in-ext-power}

For $0\leq d\leq n$ we will study zonoids in the degree~$d$ part $\Lambda^d V$ 
and call them \emph{zonoids of degree $d$}. 
The corresponding semi--vector space of zonoids 
and vector space of virtual zonoids of degree~$d$
will be denoted as  
\[\label{eq:semivector}
 \A^d(V) := \ZZ(\Lambda^d V), \quad 
 \VA^d(V) := \VZ(\Lambda^d V) .
\]
Similarly, we denote by $\A^d_o(V)$ and $\VA_o^d(V)$ 
the corresponding subspaces of centered zonoids and virtual zonoids of degree~$d$, 
respectively. 

\begin{example}\label{ex:A0n}
In degree~$0$ we have $\Lambda^0(V)=\R$. 
Thus 
$\A^0_o(V)=\ZZ_o(\Lambda^0(V)) = \ZZ(\R) = \R_+$
using the identification made after~\eqref{eq:VZ-iso-E}. 
This leads to the identification 
$\VA^0_o(V) = \R$.  
Similarly, in degree~$n$, an orientation $\orf$
spans the one-dimensional space $\Lambda^n(V)$.
This leads to $\VA_o^n(V) = \VZ_o(\Lambda^n(V)) =\R$,
in which $a\ge 0$ labels the centered  interval 
$\tfrac{a}{2} [-\orf,\orf]$
(note that this identification does not depend on the choice of $\orf$), 
while $-a$ labels its additive inverse in $\VA_o^n(V)$. 
\end{example}

By~\eqref{eq:VZ-iso-E} 
we have the linear isomorphism 
\begin{equation}\label{eq:AZk-iso}
 \VA^d(V) \stackrel{\sim}{\longrightarrow} \VA^d_o(V) \oplus \Lambda^d V ,\quad 
  Z(\xi) \mapsto  (K(\xi),\ \EE(\xi)) .
\end{equation}
In~\cite{BBLM}, the wedge multiplication of a centered zonoid of degree~$d_1$ 
with a centered zonoid of degree~$d_2$ was defined, turning 
the direct sum 
$$
  \VA_o(V) := \bigoplus_{k=0}^{n}\VA_o^d(V) = \bigoplus_{k=0}^{n}\VZ_o(\Lambda^d V)
$$
into a graded, commutative algebra, which we 
call the \emph{centered zonoid algebra of $V$}. 
This multiplication has an intuitive description 
in terms of Vitale zonoids as follows.
Let $d_1+ d_2 \le n$ and consider independent, integrable random vectors 
$\xi_1\in\Lambda^{d_1}V$ and $\xi_2\in\Lambda^{d_2}V$.
We define the wedge multiplications
\begin{equation}
 \cz(\xi_1)\wedge \cz(\xi_2) := \cz(\xi_1\wedge \xi_2)\in\ZZ_o(\Lambda^{d_1+d_2}V) ,\quad 
 \z(\xi_1)\wedge \z(\xi_2) := \z(\xi_1\wedge \xi_2)\in\ZZ(\Lambda^{d_1+d_2}V) .
\end{equation}
\cref{le:prod-well-defined} below guarantees that this is well defined. 
This extends the definition in~\cite{BBLM} to noncentered zonoids.
Note that there is no sign change when reversing the order of multiplication! 

\begin{lemma}\label{le:prod-well-defined}
Let $d_1+ d_2 \le n$ and consider independent, integrable random vectors 
$\xi_1\in\Lambda^{d_1}V$ and $\xi_2\in\Lambda^{d_2}V$.
Then $\xi_1\wedge \xi_2$ is integrable and  
the zonoid $\z(\xi_1\wedge \xi_2)$ in $\Lambda^{d_1+d_2}V$ 
only depends on the zonoids $\z(\xi_1)$ and $\z(\xi_2)$.
Similarly, the centered zonoid $\cz(\xi_1\wedge \xi_2)$ only depends on $\cz(\xi_1)$ and~$\cz(\xi_2)$. 
We have 
$$
 Z(\xi_1\wedge \xi_2) = K(\xi_1\wedge \xi_2) + \frac12 \EE (\xi_1) \wedge \EE (\xi_2) .
$$
\end{lemma}

\begin{proof}
Suppose that $\z(\xi_1)=\z(\zeta_1)$ and $\z(\xi_2)= \z(\zeta_2)$ 
with integrable and independent random variables 
$\zeta_1\in\Lambda^{d_1}V$ and $\zeta_2\in\Lambda^{d_2}V$.
The fact that \eqref{eq:AZk-iso} is an isomorphism implies $\EE (\xi_i) = \EE (\zeta_i)$.
According to~\cite[Thm.~4.5]{BBLM}, we have 
$K(\xi_1\wedge \xi_2) = K(\zeta_1\wedge \zeta_2)$.
Moreover, by independence, 
$$
 \EE(\xi_1\wedge \xi_2) = \EE(\xi_1)\wedge \EE(\xi_2) 
  = \EE(\zeta_1) \wedge \EE(\zeta_2) = \EE(\zeta_1\wedge \zeta_2) .
$$
\cref{eq:Z=K+12E} now gives
$$
 \z(\xi_1\wedge \xi_2) = \cz(\xi_1\wedge \xi_2) + \tfrac{1}{2} \EE (\xi_1\wedge \xi_2) 
  = \cz(\zeta_1\wedge \zeta_2) + \tfrac{1}{2} \EE (\zeta_1\wedge \zeta_2) = \z(\zeta_1\wedge \zeta_2) ,
$$
completing the proof. 
\end{proof}

The  wedge multiplication defines the structure of a graded, commutative real algebra
\[\label{eq:noncentzonoidalgebra}
 \VA(V) := \bigoplus_{d=0}^n \VA^d(V) ,
\]
which we call the 
\emph{noncentered zonoid algebra of $V$}. 
Moreover, the linear isomorphism 
\begin{equation}\label{eq:VA-iso}
 \VA(V) \simto \VA_o(V) \oplus  \Lambda(V) ,\quad 
  \z(\xi) \mapsto (\cz(\xi), \EE(\xi)) ,
\end{equation}
resulting from~\eqref{eq:AZk-iso}, 
is an isomorphism of graded algebras. 
(Note that the multiplication on the right hand side is defined componentwise.)

\begin{example}\label{ex:A0An}
Using the identification $\VA^0_o(V) = \VA_o^n(V) =\R$ in~\cref{ex:A0n}, 
one can check that the multiplication 
$\VA^0_o(V) \times \VA^n_o(V) \to \A^n_o(V)$ 
corresponds to the multiplication of $\R$. 
\end{example}

By the functoriality of the exterior power, 
the linear maps $\Lambda^d\varphi\colon\Lambda^d V \to \Lambda^d W$, 
induced by a linear map $\varphi\colon V\to W$, induce linear maps 
$$
 (\Lambda^d\varphi)_*\colon \VZ(\Lambda^d V) \to \VZ(\Lambda^d W),  
$$
of the spaces virtual zonoids. It is immediate from the definition that 
these maps combine to a graded algebra homomorphism 
\begin{equation} \label{eq:Lambda*}
 \Lambda(\varphi)_* \colon \VA(V) \to \VA(W) 
\end{equation}
that maps $\VA_o(V)$ to $\VA_o(W)$. 
This shows that the construction of the zonoid algebra is functorial.
In particular, there is an induced $\Gl(V)$--action on $\VA_o^d(V)$. 

\begin{proposition}\label{pro:top-GZ}
The wedge multiplication 
$\VA^{d_1}_o(V) \ti \VA^{d_2}_o(V) \to \VA^{d_1+d_2}_o(V), \, (K,L)\mapsto K\wedge L$
satisfies
$\| K \wedge L \| \le \|K \| \cdot \|L\|$. 
Hence $\VA_o(V)$ is a graded Banach algebra.
\end{proposition}

\begin{proof}
The wedge multiplication map is obtained as the composition of the bilinear tensor product map 
$$
 \VZ_o(\Lambda^{d_1} V) \ti \VZ_o(\Lambda^{d_2} V) \to \VZ_o(\Lambda^{d_1}V \ot \Lambda^{d_2} V),
  \, (K,L) \mapsto K \ot L
$$
with the linear map 
$$
 \VZ_o(\Lambda^{d_1} V \ot \Lambda^{d_2} V) \to \VZ_o(\Lambda^{d_1+d_2} V) 
$$
induced by the antisymmetrization, which is an orthogonal projection. 
Now we apply~\cref{le:FA-lin-map} and~\cref{pro:main-zon-banach}.
\end{proof}

\begin{remark}\label{re:inner-product}
The choice of a different inner product on $V$ leads to a different norm on~$\VA_o(V)$,
which the same underlying topology. However, the length notion 
and the related (mixed) volumes depend on the chosen inner product.
\end{remark}

One motivation for introducing the wedge multiplication of zonoids is that 
the \emph{mixed volume} of zonoids $\z_1, \ldots, \z_n\in\ZZ(V)$  
can be expressed in terms of the length of their wedge products:
\begin{equation} \label{eq:MV}
 \ell(\z_1\wedge\cdots \wedge \z_n)= n!\cdot \mathrm{MV}(\z_1, \ldots, \z_n) .
\end{equation}
This was shown in \cite[Theorem 5.1]{BBLM} for centered zonoids. 
The extension to noncentered zonoids follows with 
the extension of~\cref{le:prod-well-defined} to $n$ factors, 
using that the mixed volume is translation invariant.
As a consequence, we get 
\begin{equation} \label{eq:vol}
\vol(Z) = \frac{1}{n!}\, \ell (Z^{\wedge n}),\quad \text{ for } \z\in\ZZ(V).
\end{equation}
We also note that the $d$-th \emph{intrinsic} volume $V_d(Z)$ of $Z$ 
can be expressed as (see \cite[\S 5.1]{BBLM}
\begin{equation} \label{eq:intr-vol}
 V_d(Z) = \binom{n}{d} \frac{1}{\kappa_{n-d}} \, \mathrm{MV}(Z[d],B[n-d]) .
\end{equation}

We can apply the Hodge star operation $\star\colon\Lambda^d V\to \Lambda^{n-d} V$
(see \cref{se:EP-HS}) 
to a zonoid $K\in \ZZ(\Lambda^d V)$ to define its \emph{Hodge dual} 
$\star K \in\ZZ(\Lambda^{n-d} V)$. 
This is a zonoid, which is centered if $K$ is so. 
The support functions are related as 
\begin{equation}\label{eq:support-fct-dual}
 h_{\star K}(\star u) = h_K(u) , 
\end{equation}
since the Hodge star operation preserves inner product. 
We combine the Hodge star operations to a linear involution 
$
 \star\colon \VA(V)\to \VA(V) ,
$
satisfying for an integrable random vector $\xi\in \Lambda^dV$
\begin{equation}\label{hodge_on_random_vectors}
\star  \z(\xi) = \z(\star  \xi)\quad\text{and}\quad \star  \cz(\xi) = \cz(\star  \xi) .
\end{equation}
Since the Hodge star operation preserves inner product of vectors,
we see from~\eqref {eq:ell} 
that the Hodge star operation on zonoids preserves their length.
For the same reason, the Hodge star operation on zonoids 
preserves the pairing on $\VA_o^d(V)$ introduced in~\eqref{eq:pairingdef}:
for $K,L \in \VA_o^d(V)$ we have 
\begin{equation}\label{eq:star-preserves-pairing}
 \langle K , L \rangle = \langle \star K , \star L \rangle .
\end{equation}
We extend these pairings to a bilinear pairing on the direct sum $\bigoplus_d \VA_o^d(V)$ 
by requiring $\langle K , L \rangle =0$ if~$K$ and $L$ are homogeneous 
of different degrees.
With this convention, we obtain from \eqref{def_hodge_star} for any~$K_1,K_2\in\VA_o(V)$ 
\begin{equation}\label{eq:move-star}
 \langle \star K_1, K_2 \rangle = \langle \tfrac12 [-\orf,\orf], K_1 \wedge K_2 \rangle,
\end{equation}
where, as before, $\orf\in \Lambda^nV$ denotes an orientation.

The Hodge star induces a dual operation to the wedge product of the zonoid algebra:
for centered $K_1,K_1\in\VA_o(V)$, we define their \emph{convolution product}
\begin{equation}\label{def:coprod}
    K_1\vee K_2 := \star  \left((\star  K_1)\wedge (\star  K_2)\right) .
\end{equation}
The convolution respects the dual grading: 
if $K_i$ are zonoids of degree $n-d_i$, 
then $K_1\vee K_2$ is of degree $n-(d_1+d_2)$. Moreover, 
the convolution product turns $\VA_o(V)$ into an associative algebra
whose neutral element is the Hodge dual of $1$,
which is the segment $\tfrac12 [-\orf,\orf]$. 
We note the following identity
$$
  \langle K_1\wedge K_2, K_3\rangle = \langle K_1,\star  K_2 \vee K_3 \rangle .
$$
Indeed, by \eqref{eq_hodge}, 
$
 \langle K_1\wedge K_2, K_3\rangle 
 = \langle K_1, \star ( K_2 \wedge \star K_3) \rangle$ and the latter is equal to $\langle K_1, \star K_2 \vee K_3\rangle$.

%%%
\subsection{Grassmann zonoid algebra}\label{sec:grass_algebra} 

We denote by $G(d,V)$ the \emph{Grassmannian} of $d$-dimensional subspaces of~$V$
and shall view it 
as a closed subset of~$\proj(\Lambda^d V)$ via the Pl\"ucker embedding: 
\begin{equation}\label{eq:pluecker-emb}
 G(d,V)\hookrightarrow \proj(\Lambda^d V), \quad 
  E=\Span\{v_1,\ldots,v_d\} \mapsto [v_1\wedge\cdots\wedge v_d] .
\end{equation}
Vectors in $\Lambda^d V$ of the form in $v_1\wedge\cdots\wedge v_d$ 
are called \emph{simple}.
Suppose $\xi\in \Lambda^d V$ is an integrable random vector 
that almost surely takes values in the cone of simple vectors. 
Then we call $\cz(\xi)$ a \emph{centered Grassmann zonoid} of degree~$d$. 
Thus the measure defining  $\cz(\xi)$ is a measure on $\proj(\Lambda^d V)$
that actually is supported on $G(d,V)$.  
We denote by $\GZ_o^d(V)$ the set of centered Grassmann zonoids 
and write $\VGZ_o^d(V)$ for its linear span. 
It is obvious that the space $\VGZ_o^d(V)$ is $\Gl(V)$--invariant. 

The measure theoretic interpretation of Grassmann zonoids 
is provided by the linear isomorphism
\begin{equation}\label{eq:CGZ-to-measures}
 \VGZ_o^d(V) \simto \cM(G(d,V)) ,
\end{equation}
which results from restricting the right hand side of~\eqref{eq:VZ0-to-cM} 
from measures on $\proj(\Lambda^d V)$ to measures on~$G(d,V)$.
Since $\cM(G(d,V))$ is a closed subset of $\cM(\proj \Lambda^d V)$
with respect to the weak--$\ast$ topology and the Banach space topology,
it follows that $\VGZ_o^d(V)$ is closed in $\VA_o(V)$ 
with respect to the weak--$\ast$ topology and Banach space topology. 
In particular, \cref{pro:top-GZ} implies that $\VGZ_o^d(V)$ is a Banach subalgebra of $\VA_o(V)$.

\begin{remark}\label{re:alt-char-GZ}
Alternatively, we may define Grassmann zonoids in $\GZ_o^d(V)$ as the 
Hausdorff limit of Minkowski sums of 
segments of the form $[-\xi, \xi]$, where $\xi\in \Lambda^dV$ is a simple vector; 
see~\cite[Lemma 2.2.12]{mathis-thesis:22}.
The proof relies on the fact that for sequences of zonoids,
convergence in the Hausdorff metric and weak-$\ast$ convergence are 
equivalent \cite[Thm.~2.26]{BBLM}.
\end{remark}

By \cite[Proposition 4.9]{BBLM}, 
the sum of  centered Grassmann zonoids of the same degree is a centered Grassmann zonoid.
Moreover, the product of 
two centered Grassmann zonoids is a centered Grassmann zonoid. Consequently, 
$\VGZ_o(V):=\bigoplus_{d=0}^{n}\VGZ_o^d(V)$
is a graded subalgebra of $\VA_o(V)$.

We now extend this to uncentered zonoids. 
By a \emph{Grassmann zonoid} of degree~$d$ 
we understand a zonoid of the form $Z=K+p$, 
where $K\in \GZ_o^d(V)$ and  $p$ is a (not necessarily simple) vector in $\Lambda^d V$.
Non simple centers are allowed 
since we want to include all zonoids of the form $\z(\xi) = \cz(\xi) + \tfrac{1}{2}\EE \xi$,  
where $\xi\in \Lambda^dV$ is almost surely simple. 
We denote by $\GZ^d(V)$ the set of Grassmann zonoids,
by $\VGZ^d(V)$ its linear span, 
and thus obtain the graded subalgebra 
$$ 
 \VGZ(V) := \bigoplus_{d=0}^{n}\VGZ^d(V) 
$$
of~$\VA(V)$, called 
\emph{Grassmann zonoid algebra}. 
Moreover, the isomorphism \eqref{eq:VA-iso} specializes to the graded algebra isomorphism
\begin{equation}\label{eq:VG-iso}
 \VGZ(V) \simto \VGZ_o(V)\oplus  \Lambda(V) .
\end{equation}

We now show that, in terms of measures on Grassmannians, 
the wedge multiplication of centered Grassmann zonoids is given by the pushforward 
of a scaled product measure with respect to the addition map of subspaces
$$
 \alpha\colon G(d_1,V)\times G(d_2,V) \setminus\Delta \to G(d_1+d_2,V),\quad 
  (E_1,E_2) \mapsto E_1 + E_2 ,
$$
where $\Delta$ comprises the $(E_1,E_2)$ such that $E_1\cap E_2 \ne 0$. 
The scaling function 
$$
 \sigma\colon G(d_1,V)\times G(d_2,V) \to [0,1],\quad \sigma(E_1,E_2) := \| E_1 \wedge E_2\|
$$
is defined by 
$\| E_1 \wedge E_2\| := \|v_1\wedge\ldots\wedge v_{d_1} \wedge w_1\wedge\ldots\wedge w_{d_2}\|$,
where $v_{d_1},\ldots,v_{d_1}$ and $w_1,\ldots,w_{d_2}$ are orthonormal bases of 
$E_1$ and $E_2$, respectively, see \cite[(3.3)]{PSC}.
Note that $\Delta$ is the zero set of $\sigma$.

\begin{proposition}\label{pro:wedge-measure}
Let $\mu_i\in\cM(G(d_i,V))$ be the defining measure of $K_i\in\VGZ^{d_i}(V)$, 
for $i=1,2$. Then the defining measure $\mu$ of $K_1\wedge K_2$ 
is given by the pushforward measure of 
$\sigma \alpha$. In other words,
for all continuous functions $f$ on $G(d_1,V)\times G(d_2,V)$,  
$$
 \int f \, \mathrm d\mu = \int_{G(d_1,V)\times G(d_2,V) \setminus\Delta} \sigma(E_1,E_2) \,  
   f(E_1+E_2)\, \mathrm d (\mu_1\times \mu_2)(E_1,E_2) .
$$
\end{proposition}

\begin{proof}
We view $\mu_i$ as an even real measure supported on the oriented Grassmannian $\hG(d_i,V)$, 
viewed as a subset of $\Lambda^{d_i}V$.
Since both sides of the asserted equality are bilinear in $\mu_1$ and $\mu_2$, 
we may assume that the $\mu_i$ are probability distributions. 
Say $\mu_1$ and $\mu_2$ are the probability distributions of 
independent, integrable random variables $\xi_1\in\Lambda^{d_1}V$ 
and  $\xi_2\in\Lambda^{d_2}V$, respectively.  
The probability measure of $\xi_1\wedge \xi_2$ is given by 
the pushforward 
$\wedge_*(\mu_1\times\mu_2)$, where 
$\wedge\colon\Lambda^{d_1}V \times \Lambda^{d_2}V \to \Lambda^{d_1+d_2}V,\, (x,y)\mapsto x\wedge y$. 

By definition, $\mu$ is the defining measure of $\xi_1\wedge \xi_2$. Its support is contained in 
$\proj(\Lambda^{d_1+d_2}V)$. 
By~\cref{pro:X2measure}, $\mu$ is characterized,
for  all continuous functions $f$ on $\proj(\Lambda^{d_1+d_2}V)$, by 
$$
 \int f \, \mathrm d\mu = \int_{\Lambda^{d_1+d_2}V} \|z\| f([z])\, \mathrm d \wedge_* (\mu_1\wedge \mu_2) 
  = \int_{\Lambda^{d_1}V\times \Lambda^{d_2}V}  \|x\wedge y\| \, f([x\wedge y])\, \mathrm d(\mu_1\times \mu_2) (x,y) .
$$
The integral at the right only runs over simple vectors
$x=x_1\wedge \ldots \wedge x_{d_1}\ne 0$ and $y=y_1\wedge \ldots\wedge y_{d_2}\ne 0$.
The assertion follows by expressing the right integral as an integral over 
$\proj(\Lambda^{d_1}V) \times\proj(\Lambda^{d_2}V)$,
noting that if the span $E_1$ of $x_1,\ldots,x_{d_1}$ equals $[x]$ and  
the span $E_2$ of $y_1,\ldots,y_{d_2}$ equals $[y]$, 
then the sum~$E_1+E_2$ equals $[x\wedge y]$.
\end{proof}

The \emph{exponential function} maps centered zonoids to centered Grassmann zonoids 
and is defined as:
\begin{equation}\label{eq:expzono}
\exp\colon \VZ_o(V) \to \VGZ_0(V), \quad L\mapsto e^L := \sum_{d=0}^n\tfrac{1}{d!} \, L^{\wedge d}.
\end{equation}

For example, if $L= \tfrac12 [-a,a]$ for $a\in V$, then 
$L\wedge L = \tfrac12 [-a\wedge a,a\wedge a] = 0$, 
and similarly for higher wedge powers. 
Therefore, $\exp(L) = 1 + L$.

\begin{proposition}\label{pro:exp}
The exponential function is $\GL(V)$--equivariant.
It defines an injective group homomorphism
from the additive group $\VZ_o(V)$
to the multiplicative group of units of $\VGZ_0(V)$. 
The span of the image of $\exp$ is sequentially dense in $\VGZ_o(V)$ 
with respect to the weak--$\ast$ topology. 
\end{proposition}

\begin{proof}
The equivariance of $\exp$ is clear. 
For the injectivity, 
note that the image of $\exp$ is contained in the multiplicative subgroup 
$\{\sum_{d=0}^n A_d \mid A_d \in \VZ_o^d(V),  A_0 = 1\}$ 
of $\VGZ_0(V)$. 
On this subgroup, we have the logarithm map defined by
$\log (1+A) := \sum_{d=1}^n \tfrac{1}{d}(-1)^{d-1} A^{\wedge d}$, 
which satisfies $\log(\exp(L)) = L$. 
This proves the injectivity. 

For the assertion on the density, we first show that 
the image of $\exp$ contains all centered line segments $\tfrac12 [-b,b]$, 
where $b:=a_1\wedge\cdots\wedge a_m$ with $a_i \in V$ and $1\leq m\leq n=\dim V$.
If $m=0$ or $m=1$ the assertion follows from $\exp(L) = 1 + L$ as we have shown above. 
When $m>1$ we put $L_i := \tfrac12 [-a_i,a_i]$, $1\leq i\leq m$, and we observe that 
$e^{L_1+\ldots + L_m}= e^{L_1} \wedge\cdots \wedge e^{L_m } = (1+ L_1)\wedge\cdots\wedge (1+L_m)$ 
equals the segment~$\tfrac12 [-b,b]$, up to zonoids of degree less than~$m$. 
The claim follows then by induction on $m$.
Let now $M_d\subseteq\Lambda^d V$ denote the set of simple vectors
and $S\subseteq \ZZ_o(\Lambda^d V)$ 
the set of finite  finite nonnegative linear combinations of centered 
segments $[-b,b]$, where $b\in M$.
The claim implies that $S$ is contained in the image of $\exp$. 
\cref{le:EinSbar} implies that $\GZ_0(V)$ is contained in the closure 
of the image of $\exp$ with respect to the weak--$\ast$ topology. Since the latter 
is a vector space, it also contains $\VGZ_0(V)$, which completes the proof.
\end{proof}

%%%
\subsection{Algebra of classes of Grassmann zonoids}\label{sec:cosine}

The pairing introduced in \cref{se:pairing} provides us with nondegenerate 
symmetric bilinear maps
$\VA_o^d (V)\times \VA_o^d (V)\to \R$ 
of the zonoid algebra of $V$, for each degree~$d$. 
However, the 
restriction to the subspace $\VGZ_o^d (V)$ of virtual Grassmann zonoids,
may not be nondegenerate anymore. 
We therefore consider its kernel 
\begin{equation}\label{eq:KMV}
 \M^d(V) := \{ K\in\VGZ_o^d(V) \mid \forall L\in\VGZ_o^d(V): \langle K,L\rangle = 0 \} .
\end{equation}
We note that $\M^d(V)$ is mapped to $\M^{n-d}(V)$ by the Hodge star operation
since  the Hodge star preserves the pairing and maps simple vectors to simple vectors.

We can give a different, more specific characterization of  $\M^d(V)$ as follows. 

\begin{proposition}\label{propo:cocara}
Let $K\in\VGZ^d(V)$. Then $K\in\M^d(V)$ if and only if 
$h_K(w) = 0$ for all simple vectors~$w=v_1\wedge \cdots \wedge v_d\in\Lambda^d V$. 
\end{proposition}

\begin{proof}
If $K\in\M^d(V)$, then $K$ is centered by definition, 
and we have by \eqref{eq:pairing-1} for a simple $w\in\Lambda^d V$ 
$$
 2 h_K(w)  =   \langle K, \tfrac{1}{2}[-w,w] \rangle = 0 ,
$$
which vanishes since the segment $\tfrac{1}{2}[-w,w]$ is a Grassmann zonoid. 

Conversely, suppose $K\in\VGZ^d(V)$ is such that $h_K$ vanishes on simple vectors.
We first show that $K$ is centered.
So we write $Z=K+p$ with $p\in \Lambda^dV$ and $K\in \GZ_o^d(V)$. 
The support function of $Z$ is
$h_Z(u)=h_K(u)+\langle u, p\rangle.$
For any simple vector $w\in\Lambda^d V$ we have by assumption
$h_K(w)+\langle p,w \rangle = 0$ and 
$h_K(-w)+\langle p, -w \rangle = 0$.
Using that $h_K$ is even, we infer that $\langle p,w \rangle = 0$.
Since this holds for all simple vectors $w$, 
and they span $\Lambda^dV$, we infer that $p=0$.
So we showed that $K$ is centered. 
Now we apply \eqref{eq:pairing-2}:
$$
 \langle K, K(\zeta)\rangle = 2\, \EE h_K(\zeta) = 0
$$ 
for any integrable random variable $\zeta\in\Lambda^d V$ taking 
values in simple vectors.
This proves that $K$ indeed lies in $\M^d(V)$.
\end{proof}

\begin{remark}\label{re:GCOS}
Recall from \cref{se:zon-meas-top} that $h_K$ is the cosine transform 
of the generating measure $\mu_K$ of the centered zonoid $K$.
Thus the map sending the virtual Grassmann zonoid~$K$ 
to the restriction of~$h_K$ to the set of simple vectors
can be seen as the restriction of the cosine transform~\eqref{cosine_transform}  
to the Grassmannian. In the literature this is called the 
\emph{generalized cosine--transform}.
With this interpretation, \cref{propo:cocara}
says that $\M^d(V)$ equals the 
kernel of the generalized cosine transform.  
\end{remark}

It follows from~\cref{ex:A0An} that 
$\M^0(V) = \M^n(V) = 0$. 
Moreover, $\M^1(V) = \M^{n-1}(V)=0$, 
which expresses that the kernel of the cosine transform is trivial.
This fact goes at least back to Blaschke; see 
the discussion of Thm.~3.5.4 in Schneider~\cite{bible}.
By contrast, Goodey and Howard showed in~\cite[Theorem 2.1]{GWPOFI} 
that $\M^d(V)$ is non trivial if $1 < d < n$, 
disproving a conjecture previously stated by Matheron~\cite{matheronFR}.

We prove now that the direct sum
$
 \M(V):=\bigoplus_{d=0}^{n}\M^d(V) 
$
is an ideal of the Grassmann zonoid algebra~$\VGZ(V)$. 

\begin{proposition}\label{propo:ideal}
$\M(V)$ is a graded ideal of the Grassmann zonoid algebras $\VGZ_o(V)$ and  $\VGZ(V)$.
Moreover, $\M^d(V)$ is closed in~$\VGZ^d(V)$ with respect to the Banach space 
topology and sequentially closed with respect to the weak-$\ast$ topology.
\end{proposition}

\begin{proof}
By \eqref{eq:AZk-iso} we have an algebra isomorphism
$\VA(V) \simeq \VA_o(V) \oplus  \Lambda(V)$ 
with  componentwise multiplication on the right hand side.
Therefore, it suffices to show that $\M(V)$ is an ideal of $\VGZ_o(V)$.

We need to prove that $K\wedge L\in\M^{d+e}(V)$ for 
$K\in\M^d(V)$ and $L\in\VGZ_o^{e}(V)$.
Since any virtual zonoid is a difference of zonoids, 
it suffices to show this for zonoids $K$ and $L$.
So let us write $K=K(\xi)$ and $L=K(\zeta)$ with 
independent, integrable random vectors $\xi$ and $\zeta$
taking values in simple vectors almost surely. 
Then $K\wedge L = K(\xi\wedge \zeta)$ by the definition of the 
wedge multiplication. 
We can write the support function of $K\wedge L$ as 
$
 2 h_{K\wedge L}(u) =  \EE | \langle \xi \wedge \zeta , u \rangle | .
$
Using the identity~\eqref{eq_hodge} we rewrite 
$
 |\langle \xi\wedge \zeta, u\rangle| = |\langle \zeta, \star  (\xi\wedge \star  u) \rangle| .
$
Therefore,
$$
 2 h_{K\wedge L}(u) =  \EE_\xi \, \EE_\zeta \, |\langle \zeta, \star  (\xi\wedge \star  u) \rangle| 
    =  2 \EE_{\xi} h_{K(\zeta)}  (\star (\xi\wedge \star u) ) .
$$ 
Since $\xi$ takes simple values and $u$ is simple, 
$\star (\xi\wedge \star  u)$ is simple. 
It follows that $h_{K\wedge L}(u) =0$. Consequently, we have that $K\wedge L \in \M^{d+e}(V)$.

The closedness of $\M^d(V)$ 
in the Banach space topology follows from the continuity of the bilinear pairing 
(see~\cref{se:topologies}) and~\eqref{eq:KMV}. 
The sequential closedness of $\M^d(V)$ follows from 
the continuity of the bilinear pairing in the weak-$\ast$ topology;
see~\cite[Thm.~4.1]{BBLM}.  
\end{proof}

Recall from \eqref{eq:Lambda^d} that a linear map $\varphi\colon V \to W$ 
induces the linear maps $\Lambda^d(\varphi)_*\colon \Lambda^d(V) \to \Lambda^d(W)$.  

\begin{lemma}\label{le:I-lin-map-compatible}
$\Lambda^d(\varphi)_*$ maps $I^d(V)$ to $I^d(W)$. 
In particular, $I^d(V)$ is $\Gl(V)$--invariant. 
\end{lemma}

\begin{proof}
Put $\psi:=\Lambda^d\varphi$. 
By~\eqref{eq:supp-fct-im}, the support function of the image $\psi(K)$ of 
$K\in\ZZ(\Lambda^dV)$ is given by 
$h_K \circ \psi^T$. 
This also holds for virtual zonoids $K$.
On the other hand, $\psi^T$ maps simple vectors to simple vectors, 
since $\psi^T =\Lambda^d(\varphi^T)$ by \eqref{eq:Lambda-adjoint}.
Therefore, if $Z\in\M^d(V)$, 
the support function of $\psi_*(K)$ vanishes on simple vectors. 
By \cref{propo:cocara}, this means $\psi_*(K)\in \M^d(W)$.
\end{proof}

We can finally define the central object of our investigations.

\begin{definition}\label{def:ACGZ}
The \emph{algebra of classes of Grassmann zonoids} of $V$ is defined as the quotient algebra
$$
 \CGZ(V) := \VGZ(V)/\M(V).
$$ 
It has the $\Gl(V)$--invariant subalgebra 
$$
 \CGZ_o := \VGZ_o(V)/\M(V)
$$ 
that we call the 
\emph{subalgebra of classes of centered Grassmann zonoids}.
\end{definition}

It is important that $\CGZ_o(V)$ has the structure of a Banach algebra, 
induced by the quotient norm of the Banach algebra $\VGZ_o(V)$; 
note that $\M(V)$ is a closed ideal of $\VGZ_o(V)$ by~\cref{propo:ideal}.
Moreover, we call weak--$\ast$ topology the quotient topology of 
the weak--$\ast$ topology on $\VGZ_o(V)$. 
Also, while the Banach space norm of~$\VA(Z)$ and $\VGZ(V)$ 
depend on the inner product, the resulting topologies does not; 
see~\cref{re:inner-product}.

More specifically, we can write 
\begin{equation}
 \CGZ_o(V)  = \bigoplus_{d=0}^n \CGZ_o^d(V)  =\bigoplus_{d=0}^n \VGZ_o^d(V)/\M^d(V) ,
\end{equation}
which is a graded commutative and associative algebra. 
It will be convenient to denote the product by a dot instead of a wedge. 
Its elements are classes of centered Grassmann zonoids of $V$.
The $\Gl(V)$--action on $\VGZ(V)$ induces a $\Gl(V)$--action on $\CGZ(V)$ 
for which $\CGZ_o(V)$ is $\Gl(V)$--invariant.
Moreover, by \eqref{eq:VG-iso}, $\CGZ(V)$ is the direct sum of $\CGZ_o(V)$ 
and the exterior algebra of $V$:  
\begin{equation}\label{eq:ACGZ}
 \CGZ(V) = \CGZ_o(V) \oplus \Lambda(V) .
\end{equation}

Let us convince ourself that our constructions are functorial. 
First, we note that $\VA(Z),\VGZ(V),\M(V)$ do not 
depend on the choice of an inner product of $V$. 
In fact, the inner product is only needed for defining the
length functional, the pairing and the Hodge star. 
A linear map $\varphi\colon V\to W$, 
by~\eqref{eq:Lambda*}
and~\cref{le:FA-lin-map}, induces a continuous
graded algebra homomorphism
$\Lambda(\varphi)_* \colon \VA(V) \to \VA(W)$. 
It is immediate from the definition that 
this maps the subalgebra 
$\VGZ(V)$ to $\VGZ(W)$ and $\VGZ_o(V)$ to $\VGZ_o(W)$.
Moreover, since $\Lambda(\varphi)_*$ maps $I(V)$ to $I(W)$ 
by~\cref{le:I-lin-map-compatible}, we obtain an induced 
graded algebra homomorphism (denoted by the same symbol)
\begin{equation} \label{eq:CGZ-funct}
\CGZ(\varphi):=\Lambda(\varphi)_* \colon \CGZ(V) \to \CGZ(W) ,
\end{equation}
which, according to~\eqref{eq:ACGZ},  
maps $\CGZ_o(V) \to \CGZ_o(W)$ and restricts to 
$\CGZ(\varphi)\colon \Lambda(V) \to \Lambda(W)$.

We concisely summarize these findings as follows. 
 
\begin{proposition}\label{pro:funct}
$\CGZ$ defines a covariant functor from the category of finite dimensional 
vector spaces to the category of graded associative and commutative algebras.
\end{proposition}

\begin{corollary}\label{cor:inj-surj}
An injective linear map $\varphi\colon V\to W$ 
induces an injective graded algebra morphism 
$\CGZ(\varphi)_* \colon \CGZ(V) \to \CGZ(W)$, 
while a surjective $\varphi$ induces a surjective  $\CGZ(\varphi)_*$.
\end{corollary}

\begin{proof}
An injective linear map $\varphi$ of vector spaces has a left inverse,
hence, by functoriality, $\Lambda(\varphi)_*$ has a left inverse 
and is injective. Similarly for the second assertion.
\end{proof}

Suppose $H$ is a closed subgroup of the orthogonal group of $V$. 
Our main focus is on studying the subalgebra $\CGZ(V)^H$ of $H$--invariants of $\CGZ_o(V)$,
which decomposes as  
\begin{equation}\label{eq:CGZH}
 \CGZ(V)^H = \CGZ_o(V)^H \oplus \Lambda(V)^H .
\end{equation}
In fact, we will see in~\cref{sec:prob_int_th}, that the multiplication in the subring $\CGZ_o(V)^H$ 
captures the volume of random intersections in a compact homogeneous Riemannian manifold~$M$. 
In that setting, $V$ will be the tangent space at a fixed point of $M$ and 
$H$ is the group of isometries fixing that point.
Moreover, the subalgebra $\Lambda(V)^H$ is isomorphic to the 
de Rham cohomology algebra when $M$ is an oriented symmetric space; see \cref{sec:classical}.
For this description we also need the length functional introduced in~\eqref{eq:ell}.
We now show that it is well defined on $\CGZ(V)$. 

\begin{corollary}\label{propo:loriker}
The linear functional $\ell$ vanishes on the ideal $\M(V)$ 
and hence induces a linear functional on $\CGZ(V)$. 
\end{corollary}

\begin{proof} 
Let $K\in\M^d(V)$ with $d\le n$. Then 
$K\wedge B^{n-d} \in\M^n(V)$ since $\M(V)$ is an ideal by \cref{propo:ideal}. 
The assertion follows from \cref{lem:lengthwithballs} below, 
using that $\M^n(V)=0$.
\end{proof} 

The following explicit computation of the length of the intersection of a Grassmann zonoid
with a wedge power of the unit ball~$B_n$ of $\R^n$ will be of great importance. 
A special case of this result was stated in \cite[Cor.~5.3]{BBLM} without proof. 
We will use this result in connection with the formulas for the volume of unit spheres from \eqref{volballsphere}. 

\begin{lemma}\label{lem:lengthwithballs} 
For a  Grassmann zonoid $Z\in\GZ^d(\R^n)$ and $0\le i \le n-d$, we have 
\begin{equation}
 \ell\big(Z\wedge B^{\wedge i}\big) =  
 \frac{(n-d)!}{(n-d-i)!} \ \frac{\kappa_{n-d}}{\kappa_{n-d-i}} \, \ell(Z) .
\end{equation}
Here $B$ denotes the unit ball in $\R^n$ and $\kappa_{n}$ its volume. 
\end{lemma}

\begin{proof}
By translation invariance, it suffices to show the assertion for a centered~$Z$.
So assume that we have
$Z=\cz(\zeta_1\wedge\cdots\wedge \zeta_d)$ 
with an integrable, simple random vector 
$\zeta_1\wedge\cdots\wedge \zeta_d\in\Lambda^d \R^n$. 
We also write $B=K(\xi)$, where $\xi$ is a centered Gaussian with variance $2\pi$; see~\eqref{ex2}. 
If $\xi_1,\ldots,\xi_{i}\in\R^n$ are i.i.d.\ copies of~$\xi$, 
independent of $\zeta_1\wedge\cdots\wedge \zeta_d$, 
then by the definition of the wedge product
$$ 
 Z\wedge B^{\wedge (n-d)} = K(\zeta_1\wedge\cdots\wedge \zeta_d\wedge \xi_1\wedge\cdots\wedge \xi_{i}).
$$
By \eqref{eq:ell}, we write the length as the expectation
$$
 \ell\big(Z\wedge B^{\wedge i }\big) 
  =\EE\, \|\zeta_1\wedge\cdots\wedge \zeta_d\wedge \xi_1\wedge\cdots\wedge \xi_{i}\|  .
$$ 
For the moment, think of $\zeta_1,\ldots,\zeta_d$ as being fixed and denote by 
$e_1,\ldots,e_{d}$ an orthonormal frame of the span of $\zeta_1,\ldots,\zeta_d$. 
Then, 
$\zeta_1\wedge\cdots\wedge \zeta_d 
  = \pm \| \zeta_1\wedge\cdots\wedge \zeta_d \| \cdot e_1\wedge\cdots\wedge e_d$, 
and hence 
$$
 \zeta_1\wedge\cdots\wedge \zeta_d\wedge \xi_1\wedge\cdots\wedge \xi_{i}
 = \pm \| \zeta_1\wedge\cdots\wedge \zeta_d \| \, e_1\wedge\cdots\wedge e_d \wedge \xi_1\wedge\cdots\wedge \xi_{i} .
$$ 
We have 
$\|e_1\wedge\cdots\wedge e_d\wedge \xi_1\wedge\cdots\wedge \xi_{i}\| 
  = \| \pi(\xi_1)\wedge\cdots\wedge \pi(\xi_{i})\|,$
where~$\pi$ denotes the orthogonal projection onto the orthogonal complement of 
the span of $e_{1},\ldots,e_d$. 
Taking the expectation over the $\xi_j$ gives 
$$
 \EE\| \pi(\xi_1)\wedge\cdots\wedge \pi(\xi_{i})\|=\ell(\pi(B)^{\wedge i})
  = \frac{(n-d)!}{(n-d-i)!} \, \frac{\kappa_{n}}{\kappa_{n-d-i}} ,
$$ 
where we used~\eqref{eq:intr-vol}
for the last equality, noting that $\pi(B)$ is a unit ball of dimension~$n-d$. 
We conclude that, using the independence of the $\zeta_i$ and the $\xi_j$, 
\begin{align}
 \EE\|\zeta_1\wedge\cdots\wedge \zeta_d\wedge \xi_1\wedge\cdots\wedge \xi_{i}\|  
  &=\EE\|\zeta_1\wedge\cdots\wedge \zeta_d\| \cdot  \EE\| \pi(\xi_1)\wedge\cdots\wedge \pi(\xi_{i})\| \\
  &= \EE\|\zeta_1\wedge\cdots\wedge \zeta_d\| \cdot \frac{(n-d)!}{(n-d-i)!} \, \frac{\kappa_{n}}{\kappa_{n-d-i}} 
% (n-k)!\, \kappa_{n-d} . {\tt ADJUST}
\end{align}
Finally, $\ell(Z)=\| \zeta_1\wedge\cdots\wedge \zeta_d\|$ by \eqref{eq:ell}.
\end{proof}

%%%%%%%%
\bigskip
\section{Link to valuations}\label{sec:valuation}

There is a close connection of our framework to 
the theory of valuations on convex bodies. 
In the past decades, this theory has been largely expanded, 
with important contributions by Alesker; see~\cite[Chap.~6]{bible} 
and~\cite{bernig-survey:12,fu-survey:14} for overviews. 
We show that the algebra of smooth even valuations with Alesker's product
can be interpreted as a dense subalgebra of the algebra $\CGZ_o(V)$ of classes 
of centered Grassmann zonoids; 
see \cref{th:crofton-1} and also \cref{cor:crofton-1}.
The results of this section originally appeared in the PhD thesis of the fourth named author, 
see \cite[Chap.~3]{mathis-thesis:22}.

%%%
\subsection{Facts on even valuations}\label{sec:valuations}

We collect here some known facts: more details and proofs 
can be found in~\cite[Chap.~6]{bible}. 
$\KK(V)$ denotes the set of compact convex subsets of $V$, 
also called \emph{convex bodies}. 
We endow $\KK(V)$ with the Hausdorff topology. 
It contains $\ZZ(V)$ as a closed subset and it is known that $\ZZ(V)\ne \KK(V)$ iff $n=\dim V>2$. 
Moreover, we denote by $\KK_o(V)$ the set of centered $K\in\KK(V)$. 
A \emph{valuation} is a map $\phi:\KK(V)\to\R$ 
such that for all convex bodies $K,L\in\KK(V)$, 
$$
    \phi(K)+\phi(L) = \phi(K\cup L)+\phi(K\cap L) ,
$$
whenever $K\cup L$ is a convex body. 
One calls a valuation $\phi$ \emph{even} if $\phi((-1)K)=\phi(K)$ for all~$K\in\KK(V)$. 
Moreover, $\phi$ is called \emph{homogeneous of degree~$d$} 
if $\phi( tK)= t^d\phi(K)$ for all~$K\in\KK(V)$ and $t\ge 0$. 
For instance, the $n$--dimensional volume $\vol_n$ is an even homogeneous valuation of degree~$n$.  

The object of our study is the space $\val(V)$ of 
\emph{translation invariant, continuous, even valuations} and its subspaces $\val^d$ 
of homogeneous degree~$d$ valuations.
Every $\phi$ uniquely decomposes as a sum of homogeneous valuations, thus 
$$
 \val = \bigoplus_{d=0}^n \val^d .
$$
$\val^0$ consists of the constant functions and $\val^n$ of the multiples of $\vol_n$.
$\val$ is a Banach space, when endowed with the (standard) norm 
\begin{equation}\label{eq:def-val-norm}
 \|\phi\|  := \sup\{ |\phi(K)| \mid K \subseteq B \},
\end{equation} 
where $B=B(V)\subset V$ is the unit ball in $V$.
For example, 
a choice $L_1,\ldots,L_{n-d}\in \KK_0(V)$
of convex bodies  defines $\nu_{L_1,\ldots,L_{n-d}} \in \val^d$ by  
\begin{equation}\label{eq:def-nu}
 \nu_{L_1,\ldots,L_{n-d}}(K) := \binom{n}{d}\,  \MV(K[d],L_1,\ldots,L_{n-d}) .
\end{equation}
For $L\in \KK_0(V)$ we have the valuation $\vol_n(\cdot+L)$, which,  
by the multilinearity of the mixed volume, has the  decomposition 
$\vol_n(K+L) = \sum_{d=0}^n\nu_{L[d]}(K)$.  

For proceeding, we also need to know that a valuation $\phi\in\val^d$ 
is uniquely described by its \emph{Klain Function},  
see~\cite[Theorem~6.4.11]{bible}. This function is defined by 
\[\label{def:klain}
 \Kl(\phi)\colon G(d,V) \to \R, \quad \phi(K) = \Kl(\phi)(E) \cdot \vol_d(K) 
 \quad \text{ for all } K\in\KK(E) ,
\]
which is well defined since the restriction of $\phi\in\val^d$ to a $d$--dimensional subspace $E\in G(d,V)$ 
is a real multiple of the volume function on $E$. 

\begin{corollary}\label{coro:equalval}
Two even valuations are equal if and only if they take the same values on zonoids.
\end{corollary}

\begin{proof}
Assume that $\phi,\psi\in \val^d$ take the same values on zonoids. 
Then we have $\phi(B_E) = \psi(B_E)$ for all $E\in G(d,V)$,  
where $B_E$ denotes the unit ball of $E$. 
\eqref{def:klain} implies that 
$$
  \Kl(\phi)(E) = \frac{\phi(B_E)}{\vol_d(B_E)} = \frac{\psi(B_E)}{\vol_d(B_E)} = \Kl(\psi)(E) .
$$ 
Hence $\Kl(\phi)=\Kl(\psi)$. 
Since an even valuation is determined by its Klain function \cite[Theorem~6.4.11]{bible}, 
the conclusion follows (the choice of $B_E$ in this proof is arbitrary; we can replace it 
by any zonoid contained in $E$).
\end{proof}

\subsection{Crofton map}\label{se:crofton_map}

Given $K\in \KK(V)$ and $F\in G(d,V)$ we denote by $K|F$ 
the orthogonal projection of $K$ onto $F$. 
Note that if $K\subseteq E$ for some $E\in G(d,V)$, then
$\vol_d(K|F) = \vol_d(K) \cdot |\langle E,F\rangle|$. 
The \emph{Crofton map} assigns to a real measure $\mu$ on the Grassmannian $G(d,V)$
the valuation $\phi_\mu$, defined by 
\begin{equation}\label{eq:defphiA}
  \phi_\mu(K):= \int_{G(d,V)}\,\vol_d(K|F)\, \mathrm d \mu(F),\quad K\in\KK(V) .
\end{equation}   
It is straightforward to check that indeed $\phi_\mu\in\val^d$;
$\mu$ is called a \emph{Crofton measure} of $\phi_\mu$. 
If $\mu$ has a $C^\infty$--smooth density on $G(k,d)$, then  
one calls the resulting $\phi_\mu$ \emph{smooth}.
Let us denote by $\sval^d$ the subspace of such valuations.
For smooth valuations, Alesker~\cite{alesker:04a} defined a multiplication $\cdot$,
which turns $\bigoplus_{d=0}^n\sval^d$ into a graded associate commutative algebra. 
We will see below that this algebra is closely related to the algebra of classes 
of centered Grassmann zonoids. 

The situation simplifies when interpreting the Crofton map in terms of Grassmann zonoids. 
For this, we again view the Grassmannian $G(d,V)$ as a closed subset of 
$\proj(\Lambda^d V)$ via the Pl\"ucker embedding. 
Doing so, the linear isomorphism~\eqref{eq:VZ0-to-cM} 
applied to $\Lambda^d V$ restricts to a linear isomorphism 
\begin{equation}\label{}
  \VGZ_o^d(V) \simto \cM(G(d,V)), \; L \mapsto \mu_L .
\end{equation}
When composing this with the above Crofton map, we arrive at: 

\begin{definition}\label{def:Phi}
The \emph{Crofton map for Grassmann zonoids} is the graded linear map 
\[\label{eq:crofton}
 \Phi: \VGZ_o(V) \to \val(V), \; L \mapsto \phi_L := \phi_{\mu_L}.
\]
\end{definition}

We can concisely express this Crofton map in terms of the 
exponential function of zonoids (see \cref{pro:exp})
and the pairing of \cref{se:pairing}.

\begin{proposition}\label{pro:crofton-exp}
For $L\in\VGZ_o(V)$ and  $K\in\ZZ_o(V)$ we have 
$\phi_L(K) = \langle L, e^K \rangle$. 
\end{proposition}

\begin{proof}
By bilinearity, it suffices to show for a zonoid $L\in\GZ_o^d(V)$ of degree $d$ 
that for all $K\in \GZ_o(V)$
$$
 \phi_L(K) = \frac{1}{d!} \langle L, K^{\wedge d} \rangle .
$$ 
Let $L=K(\zeta_1\wedge\ldots \wedge \zeta_d)$ 
with integrable random variables $\zeta_i \in V$. 
Let $K=K(\xi)$ for a integrable random vector $\xi\in V$ 
and pick independent random variables $\xi_1,\ldots,\xi_d$ with 
the same distribution as~$\xi$. Then $K^{\wedge d} = K(\xi_1\wedge\ldots \wedge \xi_d )$. 
By \eqref{eq:ell} and \eqref{eq:vol} we have 
$$
 d!\, \vol_d (K\mid E)  = \ell \big( \pi_E(K)^{\wedge d} \big) 
  =  \EE_\xi \| \pi_E(\xi_1) \wedge\cdots  \wedge \pi_E(\xi_d) \| ,
$$
where $\pi_E$ denotes the  orthogonal projection onto $E\in G(d,V)$. 
If $E$ is spanned by $\zeta_1,\ldots,\zeta_d$, then one checks that 
$\| \pi_E(\xi_1) \wedge\cdots  \wedge \pi_E(\xi_d) \| 
 = |\langle \xi_1\wedge\cdots \wedge \xi_d, \zeta_1\wedge\cdots \wedge \zeta_d \rangle|$. 
We conclude that 
$$
 \phi_L(K) = \int_{G(d,V)} \vol_d (K\mid E) \, \mathrm d\mu_L(E)
 =
 \frac{1}{d!} \int_{G(d,V)} \EE_\xi 
 |\langle \xi_1\wedge\cdots \wedge \xi_d, \zeta_1\wedge\cdots \wedge \zeta_d \rangle| \, \mathrm d\mu_L(E).
$$
By \cref{pro:X2measure} the latter is equal to 
$\frac{1}{d!}\, \EE_{\xi,\zeta} |
  \langle \xi_1\wedge\ldots \wedge \xi_d, \zeta_1\wedge\ldots \wedge \zeta_d \rangle| 
       = \frac{1}{d!}\, \langle  K^{\wedge d}, L \rangle$. 
\end{proof}

Let us compute the Klain function of the valuation $\phi_{L}$ defined by  the 
Grassmann zonoid $L\in\GZ_o^d(V)$ via the Crofton map. 

\begin{proposition}\label{propo:klainval}
Let $L\in\GZ_o^d(V)$ be a Grassmann zonoid. 
Then the Klain function of the valuation~$\phi_{L}$ defined by $L$ via the Crofton map is given by
$$
 \Kl(\phi_{L})(E) = h_L(E) \quad \mbox{for $E\in G(d,V)$.}
$$
\end{proposition}

\begin{proof}
Taking for $K$ in \eqref{eq:defphiA} a convex body $K_E$ contained in $E\in G(d,V)$, 
we get 
$$
 \phi_{\mu_L}(K_E) = \int_{G(d,V)} \vol_d(K_E|F) \; \mathrm d \mu_{L}(F) 
  = \vol_d(K_E) \cdot \int_{G(d,V)} |\langle E,F\rangle|\; \mathrm d \mu_{L}(F), 
$$
The integral on the right hand side equals the cosine transform \eqref{cosine_transform} of the measure $\mu_{L}$, 
evaluated at~$E$. 
Therefore, the Klain function of  $\phi_{\mu_L}$ is given by 
the support function 
$h_L\colon\Lambda^d V \to \R$ of~$L\in \GZ_o^d(V)$, 
restricted to (the cone over) $G(d,V)$.
\end{proof}

We define an action of $\Gl(V)$ on $\val(V)$ by 
$(g \phi)(K) := \phi (g^T K)$. 
Also, recall the natural action of $\Gl(V)$ on $\GZ_o(V)$. 
For these actions, we have the following important insight.

\begin{corollary}\label{cor:crofton-equivariant}
The Crofton map $\Phi$ for Grassmann zonoids is equivariant for the $\Gl(V)$--actions.
\end{corollary}

\begin{proof} 
We need to prove that for every $K\in \KK(V)$ we have $ \phi_{gL}(K)=(g\phi_L)(K)$. 
By \cref{coro:equalval}, it is enough to prove this identity in the case $K$ is a zonoid.
For $L\in\GZ_o^d(V)$, $K\in\ZZ_o^d(V)$, and $g\in\Gl(V)$, we have by \cref{pro:crofton-exp}
$$
 \phi_{gL}(K) = \langle gL, e^K \rangle = \langle L, g^Te^K \rangle 
  = \langle L, e^{g^TK} \rangle = \phi_L(g^T K) = (g\phi_L)(K),
$$
where we used \cref{le:adjoint} for the second equality and the equivariance of $\exp$ 
for the third equality.
\end{proof}

Let us show that the valuation $\nu_{L_1,\ldots,L_{n-d}}$ from~\eqref{eq:def-nu} 
defined in terms of mixed volumes arises via the Crofton map. 

\begin{corollary}\label{cor:crofton-exp}
For $L,L_i\in\ZZ_o(V)$ we have 
$$
 \Phi(\star e^L) = \vol_n(\cdot + L),\quad \text{and}\quad 
 \frac{1}{(n-d)!}\Phi(\star (L_1\wedge\ldots\wedge L_{n-d})) = \nu_{L_1,\ldots,L_{n-d}}.
$$  
\end{corollary}

\begin{proof}
As before, by \cref{coro:equalval}, it is enough to prove the identities 
of valuations for arguments that are zonoids.
Let $\orf\in\Lambda^n(V)$ denote an orientation. By \eqref{eq:move-star} we have 
$$
 \langle \star e^L, e^K \rangle =\langle \tfrac12 [-\orf,\orf], e^L \wedge e^K \rangle = 
 \langle \tfrac12 [-\orf,\orf], e^{L+K} \rangle=\frac{1}{n!}\langle \tfrac12 [-\orf,\orf], (L+K)^{\wedge n} \rangle.
$$
By~\eqref{eq:MV} this equals 
$ \frac{1}{n!} \ell\big( (L+K)^{\wedge n} \big) = \vol_n(L+K)$, 
which implies the first assertion by~\cref{pro:crofton-exp}.
Similarly,  we put $L:= L_1\wedge\ldots\wedge L_{n-d}$ and write 
\begin{equation}
\begin{split}
 \phi_{\star L}(K) &= \langle \star L, e^K \rangle \\&= \frac{1}{d!}\, \langle \star L, K^{\wedge d} \rangle 
   \\ &= \frac{1}{d!}\, \langle \tfrac12 [-\orf,\orf], L \wedge K^{\wedge d} \rangle \\
    &= \frac{1}{d!}\, \ell (L \wedge K^{\wedge d} \rangle 
        \\&= \frac{n!}{d!}\, \MV(K[d], L_1,\ldots,L_{n-d}) ,
\end{split}
\end{equation}
which shows the second assertion. 
\end{proof}

\begin{remark}
If $A\in\GZ_o^d(V)$ satisfies 
$\star A =K(\xi_1)\wedge\cdots\wedge K(\xi_{n-d})$
with integrable random vectors $\xi_i\in V$, 
then \cref{cor:crofton-exp} implies 
\begin{equation}
     \phi_A(K)  = \frac{n!}{d!}\, \EE\, \MV(K[d],[0,\xi_{1}],\ldots,[0,\xi_{n-d}]).
\end{equation}
This is a special case of Nguyen--Bac Dang and Jian Xiao's~\cite{DangXiao} 
definition of a ``$\mathcal{P}-$positive'' valuations 
using Radon measures on tuples of convex bodies.
\end{remark}

For stating the next result, we recall from~\cite{bernig-survey:12} 
that there is an involution of the space $\val$, called \emph{Alesker--Fourier transformation}, 
consisting of linear isomorphisms
$\val^d\to \val^{n-d}, \phi \mapsto \star \phi$, 
that are characterized by 
\begin{equation}\label{eq:Klain-dual}
 \Kl(\star\phi) (E^\perp) = \Kl(\phi) (E) . 
\end{equation}
The \emph{convolution product}, denoted by $*$, of two valuations $\phi_1,\phi_2 \in \val$, 
such that $\star \phi_1, \star \phi_2$ have smooth generating measures, is then defined as 
\begin{equation}\label{eq:conv_prod}\phi_1 \ast \phi_2 := \star (\star \phi_1\cdot \star\phi_2),\end{equation}
where $\cdot$ denotes the Alesker product of smooth valuations \cite{alesker:04a}.

The following result summarizes the main properties of the Crofton map
and closely connects the centered Grassmann zonoid algebra to (even) valuations. 

\begin{theorem}\label{th:crofton-1}
The Crofton map $\Phi\colon\VGZ_o(V) \to \val(V)$ has the following properties: 
\begin{enumerate}
\item $\Phi$ is $\Gl(V)$--equivariant. 
\item \label{item2} The kernel of $\Phi$ equals the the ideal $\M(V)$ defined in \eqref{eq:KMV}.
\item The image of $\Phi$ is a dense subspace of $\val$ for the standard norm~\eqref{eq:def-val-norm}.
\item The Alesker--Fourier transform of $\Phi(L)$ 
is given by $\star \Phi(L) = \Phi(\star L)$.
\item \label{item5} $\Phi(L_1 \wedge L_2)= \Phi(L_1) \cdot \Phi(L_2)$ equals 
the Alesker product of $\Phi(L_1)$ and $\Phi(L_1)$,
if the $L_i \in \VGZ^{d_i}(V)$ have smooth generating measures. 
\item 
$\Phi(L_1 \vee L_2)= \Phi(L_1) \ast \Phi(L_2)$ equals the convolution product of $\Phi(L_1)$ and $\Phi(L_1)$,
if $L_i \in \VGZ^{d_i}(V)$ have smooth generating measures. 

\item If we endow $\val(V)$ with the standard norm topology (see~\eqref{eq:def-val-norm}), 
then $\Phi$ continuous for the Banach space topology of $\VGZ_o(V)$.

\end{enumerate}
\end{theorem}

\begin{proof}
\begin{enumerate}
\item This is \cref{cor:crofton-equivariant}.

\item  We have $\phi_L=0$ iff $\Kl(\phi_{\mu_L})=0$. 
By \cref{propo:klainval} this means $h_L(E)=0$ 
for all $E\in G(d,V)$, which means $L\in\M^d(V)$ by~\cref{propo:cocara}.
This shows the assertion. 

\item The image of $\Phi$ is $\Gl(V)$--invariant by $(1)$. 
Now we apply Alesker's irreducibility theorem~\cite{alesker:01} ,
which states that any $\Gl(V)$-invariant subspace of $\val$ 
is dense in this Banach space.

\item We check that $\star \Phi(L)$ and $\Phi(\star L)$ have the same Klain function.
By~\cref{propo:klainval} and \eqref{eq:support-fct-dual} we have 
$$
 \Kl(\Phi(\star L)) (E^\perp) = h_{\star L}( E^\perp) = h_{L}( E) = \Kl(\Phi(L))(E) .
$$ 
On the other hand, by~\eqref{eq:Klain-dual} 
$$\Kl(\phi) (E)  = \Kl(\star\phi) (E^\perp),$$
which proves the assertion.

\item
According to \cite[\S 1.2.2]{BFconvo}, for even valuations, 
the Alesker product has the following simple description in terms of Crofton measures.
Suppose the valuation $\phi_i\in\sval^{d_i}$ has a Crofton measure $\mu_i$, for $i=1,2$. 
Then the Alesker product $\phi_1 \cdot \phi_2$ is given by the 
pushforward measure
$\alpha_* (\sigma (\mu_1 \times \mu_2))$, where 
$\sigma(E_1,E_2) := \| E_1 \wedge E_2\|$
and $\alpha(E_1,E_2):= E_1+E_2$.
\cref{pro:wedge-measure} proves that 
$\Phi(L_1 \wedge L_2)= \Phi(L_1) \cdot \Phi(L_2)$. 

\item The last assertion follows with the previous parts, 
using the definitions of $\vee$ and $\ast$ via duality, 
see~\eqref{def:coprod}. 

\item From \eqref{eq:defphiA} we get for a convex body $K$ contained in the unit ball in~$V$ that 
\begin{equation}
  |\phi_\mu(K)| \le \int_{G(d,V)} |\vol_d(K|F)| \, \mathrm d |\mu|(F) \ \le C \, \|\mu\| ,
\end{equation}  
where $C$ denotes the volume of unit balls in $d$-dimensional subspaces of $V$.
\end{enumerate}
\end{proof}

\begin{corollary}\label{cor:crofton-1}
The Crofton map for Grassmann zonoids $\Phi$ induces an injective graded linear map 
$$
\CGZ_o(V) \hookrightarrow \val(V) ,
$$ 
by which $\CGZ_o(V)$ can be interpreted as a dense subspace of the space $\val(V)$,
which contains the smooth even valuations,  
and whose multiplication coincides with Alesker's multiplication of smooth valuations.
\end{corollary}

\begin{proof}
This follows from \cref{item2} and \cref{item5} of \cref{th:crofton-1} and 
the definition of $\CGZ_o(V)$.
\end{proof}

For illustrating the elegance of our framework, let us reprove a known fact about the convolution product of valuations~\eqref{eq:conv_prod}.

\begin{corollary}\label{prop:convofvol}
If $L,L'\in\ZZ_o(V)$, then 
$\vol_n(\cdot+L) * \vol_n(\cdot+L')=\vol_n(\cdot+L+L')$.
\end{corollary}

\begin{proof} 
By \cref{cor:crofton-exp}, $\Phi(\star e^L) = \vol_n(\cdot + L)$. \cref{th:crofton-1} (6) implies 
$$
 \vol_n(\cdot + L) * \vol_n(\cdot + L') = \Phi(\star e^L) * \Phi(\star e^{L'}) = \Phi(\star e^{L} \vee \star  e^{L'}).
$$
By definition \eqref{def:coprod} of the convolution product for zonoids, we have
$\star e^{L} \vee \star  e^{L'} =\star \big(e^{L} \wedge e^{L'}\big) = \star e^{L+ L'}$, so 
$\vol_n(\cdot + L) * \vol_n(\cdot + L')= \Phi(  \star e^{L+ L'}) = \vol_n(\cdot + L+L')$. 
\end{proof}

Recall that we have an induced  $\Gl(V)$--action on $\CGZ_o(V)$, 
since the ideal $\M(V)$ is $\Gl(V)$--invariant (\cref{le:I-lin-map-compatible}).
Hence the Crofton map $\Phi$ defines an injective, graded $\Gl(V)$-equivariant linear map
\begin{equation}\label{eq:CGT-in-Val}
 \CGZ_o(V) = \GZ_o(V)/\M(V) \hookrightarrow \val(V) .
\end{equation}
Consider the subalgebra  $\CGZ_{o,\mathrm{sm}}(V)$ of $\CGZ_o(V)$ 
consisting of the classes of zonoids with a smooth generating measure.
Then the restriction of the above map yields a $\Gl(V)$-equivariant algebra isomorphism
\begin{equation}
\CGZ_{o,\mathrm{sm}}(V) \simto \sval(V).
\end{equation}

\begin{remark}
In view of Alesker's irreducibility theorem~\cite{alesker:01}, 
it would be interesting to investigate to what extent $\CGZ_o(V)$ is irreducible.
More specifically, for which topologies on $\CGZ_o(V)$ 
are nonzero $\Gl(V)$--invariant subspaces of $\CGZ_o(V)$ dense?
\end{remark}

We use the link to valuations to derive an interesting corollary, which  
shows that in most cases of interest, $\CGZ(V)^H$ will be an infinite dimensional Banach algebra. 
Hence methods from functional analysis are required for its investigation.

\begin{corollary}\label{prop:findimtransact}
Suppose $H$ is a closed subgroup of the orthogonal group $\OO(V)$ of $V$. 
Then $\CGZ(V)^H$ is a finite dimensional vector space if and only if 
the group $H$ acts transitively on projective space $\mathbb P(V)$. 
\end{corollary}

\begin{proof} 
Taking $H$--invariants in the equivariant embedding~\eqref{eq:CGT-in-Val} induced by $\Phi$, 
we get  
\begin{equation}
 \Phi(\CGZ_o(V)^H)) = \Phi(\CGZ_o(V))^H \subseteq \val(V)^H .
\end{equation}
By~\cref{th:crofton-1}(3), we know that 
$\Phi(\CGZ_o(V))$ is dense in $\val(H)$ 
with respect to the standard norm. 
Consider the linear projection
$\pi_H\colon\val(V)\to\val(V)^H$ obtained by averaging over~$H$ 
with respect to the normalized Haar measure: 
$
 \pi_H(\phi) := \int_H h\phi \, \mathrm dh .
$
One checks that the standard norm \eqref{eq:def-val-norm} satisfies
$\| \pi_H(\phi)\| \le \| \phi\|$, hence $\pi_H$ is continuous. 
By equivariance, $\pi_H(\Phi(\CGZ_o(V))) = \Phi(\CGZ_o(V)^H)$.
This implies that $\Phi(\CGZ_o(V)^H)$ is dense in $\val(V)^H$. 
Therefore, $\CGZ_o(V)^H$ is finite dimensional iff $\val(V)^H$ 
is finite dimensional. 

Let $H'$ denote the subgroup of $\OO(V)$ generated by $H$ and $-\id$. 
Then all $H'$--invariant valuations are even. 
A result of Alesker~\cite{alesker:00}, see also~\cite[Theorem~6.5.3]{bible}, 
tells us that $\val(V)^H$ is finite dimensional iff the group $H'$ acts transitively 
on the unit sphere of $V$. This is equivalent to $H$ acting transitively on 
projective space $\mathbb P(V)$. 
\end{proof}

\begin{remark}\label{re:alesker-fin-dim}
Alesker~\cite{alesker:04a} showed that 
$\val(V)^H=\sval(V)^H$ if $\CGZ(V)^H$ is finite dimensional.
In this case, \eqref{eq:CGT-in-Val} gives an isomorphism 
$\CGZ_o(V)^H \simeq \val(V)^H$.
\end{remark}

%%%%%%%%
\bigskip
\section{Probabilistic intersection theory in Riemannian homogeneous spaces}\label{sec:prob_int_th}

In this section we introduce the probabilistic intersection ring $\HE(M)$ of a Riemannian homogeneous space~$M$.
The multiplication in this ring models the intersection of randomly moved submanifolds $Y_i$ of $M$.
The expected volume of such random intersections is obtained by applying the length functional of zonoids. 
In particular, this gives the expected number of intersection points 
when the codimensions of $Y_i$ add up to $\dim M$. 
Later, in  \Cref{sec:classical}, we complete the picture by showing that, 
under some additional assumptions in $M$, the de Rham cohomology ring $\HdR(M)$ 
can be seen as a subring of $\HE(M)$. This amounts to looking at the centers of the Grassmann zonoids of $V$. 
The multiplication in $\HdR(M)$ describes the signed count of intersections, as it is known from 
classical intersection theory. 

%%%
\subsection{Setting}\label{se:setting}

Throughout, we assume the setting of a compact \emph{Riemannian homogeneous space}, 
i.e., we will work with a smooth, compact Riemannian manifold $M$ 
on which a compact Lie group~$G$ acts transitively by isometries. 
This is the same as requiring that the manifold $M$ is a $G$--homogeneous space, 
which is endowed with a $G$--invariant Riemannian metric.

In our context, Riemannian homogeneous spaces arise in the following way (e.g., see~\cite[Chap.~9]{lee:13}).
We have a compact Lie group~$G$ with a closed subgroup $H$ and form the quotient 
$$
 M := G/H,
$$
endowed with the smooth manifold structure such that 
the canonical map $\pi\colon G\to M$ is a submersion. 
Denoting by $e=e_G$ the identity element of $G$, we have the distinguished point  
\begin{equation}\label{def_distinguished_point}
 \bo := \pi(e_G),
\end{equation}
Then, the left translations by elements $g\in G$ induce an action of $G$ on $M$, 
denoted by $a_g:M\to M$, and $M$ is a $G$--homogeneous space.

In order to put on $M$ a $G$--invariant Riemannian metric we proceed as follows. 
Denote by $\Lg$ and~$\Lh$ the Lie algebras of $G$ and $H$, respectively. 
For every $g\in G$ we denote by $\alpha_g:G\to G$ 
the inner automorphism $x\mapsto gxg^{-1}$, and by 
$\mathrm{Ad}_G:G\to \mathrm{GL}(\mathfrak{g})$ 
the adjoint representation
\[ 
 \mathrm{Ad}_G(g):= D_{e}\alpha_g .
\]
Since $H$ is compact, the Lie algebra $\Lh$ admits a complement $\mathfrak{m}$ in $\Lg$, 
which is invariant for the restriction of the adjoint representation to $H$, 
i.e., we can write
\[ \label{eq:reductive}
 \Lg=\Lh\oplus \mathfrak{m}\quad \textrm{with}\quad \mathrm{Ad}_G(h)\mathfrak{m}
   =\mathfrak{m}\quad \text{ for all } h\in H.
\]
The differential $D_e\pi$ restricts to an isomorphism on $\mathfrak{m}$, 
which can therefore be naturally identified with~$T_\bo M$. 
The condition that $\mathfrak{m}$ is $\mathrm{Ad}_G(H)$--invariant ensures  
that this induces a well defined representation, 
denoted $\mathrm{Ad}_{G/H}:H\to \mathrm{GL}(\mathfrak{m})$. 
The $G$--invariant Riemannian metrics on $M$ are in one--to--one correspondence with 
the $\mathrm{Ad}_{G/H}$--invariant scalar products on $\mathfrak{m}$ (see \cite[Theorem~5.27]{metric}). 

Given one such scalar product $\langle\cdot, \cdot\rangle$, the $G$--invariant Riemannian structure on~$M$ 
is defined as follows. Given a point $x\in M$ and two vectors $v_1, v_2\in T_{x}M$, 
let $g\in G$ such that $a_g(x)=\bo$. Then set
\[ 
 \langle v_1, v_2\rangle_x:=\langle D_e\pi^{-1}D_xa_g v_1,\ D_e\pi^{-1}D_xa_g v_2\rangle.
 \]
The $\mathrm{Ad}_G(H)$--invariance of $\langle \cdot, \cdot\rangle $ ensures that 
this is well defined 
(i.e., independent on the choice of the element $g$ such that $a_g(x)=\bo$).

For instance, any $\mathrm{Ad}_G(H)$--invariant scalar product on~$\Lg$ 
restricts to an $\mathrm{Ad}_{G/H}$--invariant scalar product on $\mathfrak{m}$ 
and therefore, gives rise to a $G$--invariant Riemannian metric on $M$. 
In particular, this is true for a bi--invariant Riemannian metric on $G$.

We will not use this formal language and simply call $M=G/H$ a Riemannian homogeneous space in the following. 
Nevertheless, we must keep in mind that the group action 
(the same manifold can be a Riemannian homogeneous space for different compact Lie groups $G$) 
and the choice of the $G$--invariant scalar product matters. 
The relevance of these choices is illustrated in~\cref{re:Daction} and at the end of~\cref{sec_HE_sphere}.

Finally, in order to connect with the theory developed in the previous part of the paper, 
we single out the vector space of relevance for us, namely the vector space
\begin{equation}\label{def_V}
 V := \big(\Lg/\Lh\big)^* = T^*_\bo M,
\end{equation}
i.e., the cotangent space of $M$ at $\bo$.
We denote $n:=\dim M =\dim V$. Using the decomposition \eqref{eq:reductive}, 
we can identify $V\simeq \mathfrak{m}^*$.

Note that the action of $H$ on $M$ leaves $\bo$ invariant, 
and induces an action of $H$ on $V$, which we call \emph{coisotropy representation}. 
This is the homomorphism $H\to \Gl(V)$ given by
\begin{equation}\label{coiso_representation}
 h\mapsto \big(D_{\mathbf{o}}a_{h^{-1}}\big)^*,
\end{equation}
where $a_g:M\to M$ for $g\in G$  denotes the action induced on $M$ 
by left translation with~$g$. To simplify notation, we will denote this action 
$H\times V\mapsto V,\; (h,v) \mapsto h v$. 

The coisotropy representation and the dual of the $\mathrm{Ad}_{G/H}$ representation are, in fact, 
equivalent representations and sometimes for us it will be useful to freely switch from one to the other. 
This is done via the following commutative diagram (see \cite[Lemma 5.30]{metric}):
$$\begin{tikzcd}
V \arrow[dd, "(D_e\pi)^*"'] \arrow[rr, "(D_\mathbf{o}a_{h^{-1}})^*"] &  & V \arrow[dd, "(D_e\pi)^*"] \\
&  &         \\
\mathfrak{m}^* \arrow[rr, "\mathrm{Ad}_{G/H}(h)^*"]  &  & \mathfrak{m}^*           
\end{tikzcd}$$

%%%
\subsection{Probabilistic intersection ring of Riemannian homogeneous space}\label{sec:probabilisticintersectionring}

The only information relevant for the construction of the probabilistic intersection ring 
is the Euclidean vector space~$V$ from~\eqref{def_V}, together with the isometric linear action of the 
compact Lie group $H$ on $V$, the coisotropy representation~\eqref{coiso_representation}. 
(We also point out that choice of the inner product of $V$ is not relevant for constructing the ring; 
see~\cref{re:inner-product}.)
Since $H$ acts isometrically on $V$, it follows that $H$ acts 
isometrically on $\Lambda(V)$ with respect to the induced inner product,
via $h(v_1\wedge\cdots\wedge v_d) := (hv_1)\wedge\cdots\wedge (hv_d)$; see~\cref{se:EP-HS}. 
Moreover, the action of $H$ on $V$ is grading preserving.
We obtain an induced action of $H$ on the algebra $\CGZ(V)$ of classes 
of Grassmann zonoids of $V$; see \cref{def:ACGZ}. Thus, we arrive at the following central definition.

\begin{definition}[Probabilistic intersection ring]\label{def:HE}
We call 
$$ 
 \HE (M) := \CGZ(V)^H.
$$
the \emph{probabilistic intersection ring} of $M$, and we call
$$
 \HEo (M) := \CGZ_o(V)^H = \VGZ_o(V)^H\,/\,\M(V)^H
$$
the
\emph{centered probabilistic intersection ring} of $M$.  Moreover, we denote 
$$
 \HdR (M) :=\Lambda(V)^H
$$
for the  $H$--invariants of the exterior algebra.
\end{definition}

The subsequent results will justify the namings of these rings. 

In \cref{sec:DR}, we will show that the subalgebra $\HdR (M)$ 
can be identified with the de Rham cohomology algebra of the $M$,
provided $M$ is a symmetric space and $G$ is connected. 
Elements of this algebra are centers 
of $H$--invariant Grassmann zonoids of~$V$.
The algebra decomposition~\eqref{eq:CGZH} thus reads as 
\begin{equation}\label{eq:H-decomp}
\HE(M)\quad  =  \underbrace{\HEo(M)}_{\text{probabilistic intersection}} 
           \oplus \underbrace{\HdR(M)}_{\text{de Rham cohomology}}.
\end{equation}
The decomposition~\eqref{eq:H-decomp} reflects the decomposition of $Z$ 
as a sum of a centered $H$--invariant Grassmann zonoid $K$ plus its center, 
which is an $H$--invariant vector in $\Lambda^dV$.
\cref{thm:PSC} will explain the probabilistic interpretation of 
the subalgebra $\HEo (M)$ of centered $H$--invariant Grassmann zonoids of~$V$. 
\cref{thm:general2} in the next section will explain the cohomological interpretation of 
the subalgebra~$\HdR(M)$. 

\begin{remark}\label{re:Daction}
Even though not reflected in our notation, we still emphasize that 
the probabilistic intersection ring $\HE(M)$ 
depends on the action of $G$ on $M$ in a strong way.
For instance, writing the odd-dimensional sphere as~$S^{2n+1} =\SO(2n+2)/\SO(2n+1)$ leads to the algebra 
$\CGZ(\R^{2n})^{\SO(n+1)}$ (analyzed in~\cref{sec_HE_sphere}), 
while writing $S^{2n+1} =\UU(n+1)/\UU(n)$ leads to the very different algebra
$\CGZ(\C^{n}\oplus\R)^{\UU(n)}$.
(Its subalgebra $\CGZ(\C^{n})^{\UU(n)}$ is analyzed in~\cref{sec:complex}.)
\end{remark}

%%%
\subsection{Zonoids and Grassmann classes associated to a submanifold}\label{sec:submanifold} 

In this section
we introduce a new key concept by associating 
Grassmann zonoids with stratified sets in a Riemannian homogeneous space $M=G/H$.

\begin{definition}
\label{def:stratified} 
A closed subset $Y\subseteq M$ is
a \emph{stratified set} if it can be written as a finite disjoint union
$$
  Y=\bigsqcup_{j=1}^s Y_j,
$$
of  smooth submanifolds~$Y_j$ of $M$
such that for all $i\neq j$,
whenever $Y_i\cap \overline{Y_j}\neq \emptyset$, then 
$Y_i\subseteq \overline{Y_j}$ and $\dim Y_i< \dim Y_j $. 
The sets $Y_j$ are called the \emph{strata} of the stratification. 
The dimension $\dim Y$ of~$Y$ 
is defined as the maximum of the dimensions of its strata. 
The subset $Y^{\mathrm{sing}}\subseteq Y$ of \emph{singular points} of~$Y$
is defined as 
the union of the strata of dimension strictly less than $\dim Y$.
The subset $Y^{\mathrm{sm}} := Y\setminus Y^{\mathrm{sing}}$ 
is a smooth submanifold of $Y$ and 
is called the set of \emph{smooth points} of $Y$. 
(Notice that $Y^{\mathrm{sing}}$ and $Y^{\mathrm{sm}}$ 
depend on the chosen stratification.) 
\end{definition}

\begin{example}
  The Schubert varieties of a Grassmann manifold provide examples  
  for stratified sets, e.g., see~\cite{milnor-stasheff}.   
\end{example}

In the following, we associate with a stratified set~$Y\subseteq M$ of codimension $d$
a centered zonoid $K(Y)$ in $\Lambda^d(V)$, which is invariant under the action of~$H$. 
As before, $V:=T^*_\bo M$ denotes the cotangent space of $M$ at $\bo$. 
The formal definition of this associated zonoid comes in \cref{def:KZ_Y} below.
If the normal bundle of~$Y$ is oriented,
we can also associate with $Y$ a (noncentered) $H$--invariant zonoid $Z(Y)$ in $\Lambda^d(V)$.
Note that the degrees of these Grassmann zonoids are equal to the codimension~$d$ of $Y$ in $M$. 

The goal of this section is to prove that the multiplication of the classes of 
the zonoids $K(Y)$ in the ring~$\HEo(M)$
describes the expected volume of the intersection of randomly moved copies of submanifolds~$Y$. 
In the case where $M$ is symmetric and orientable, we will show in~\cref{sec:classical} that
twice the center $c(Z_Y) \in \HdR(M)$ gives the de Rham cohomology class of $Y$,
and that the multiplication in the ring $\HdR(M)$ corresponds to the multiplication 
of de Rham classes.

First, recall the notion of \emph{conormal bundle} of a submanifold $Y\subseteq M$, 
denoted by $NY\subseteq T^*M$,
\[ 
 NY:=\{\eta\in T^*M \mid \forall v\in TY:\ \eta(v)=0 \}.
\]
Using the ($G$--invariant) Riemannian metric on $M$, 
the conormal bundle $NY$ can be identified with the normal bundle of $Y$ in $M$, 
denoted by $TY^\perp$, which consists of the orthogonal complement to~$TY$. 
In the sequel, it will be sometimes convenient to do this identification.

We call a codimension~$d$ submanifold $Y\subseteq M$ \emph{coorientable}, 
if there is a nowhere vanishing continuous section $Y\to \Lambda^d(NY)$.
Such a section, if it exists, is called a \emph{coorientation} of~$Y$. 
We will say that a stratified set is cooriented if 
the submanifold of its smooth points is so.

Similarly, we  say that  a stratified set $Y\subseteq M$ has finite volume, 
and write $\mathrm{vol}(Y)<\infty$,  if the set of its smooth points has finite Riemannian volume. 
In this case, we can turn $Y$ into a probability space, defining
$\mathrm{Prob}(U):=\mathrm{vol}(U)\,/\,\mathrm{vol}(Y)$ 
for every measurable subset $U\subseteq Y$.
Moreover, the notion of a uniformly random variable on $Y$ is well defined.

In the next definition, we translate elements from $\Lambda(T_y^*M)$, 
for various $y\in M$, to $\Lambda(V)=\Lambda(T^*_{\bo}M)$ 
by choosing a measurable map 
\begin{equation}\label{eq:def-tau}
  \tau :M\to G\quad \textrm{such that}\quad \tau(y)\cdot y= \bo 
   \quad \textrm{for all $y\in M$} . 
\end{equation} 
Note that  $\bo \in h\tau(y)Y$ if $y\in Y$ and $h\in H$.
If $M=G$, we must take $\tau(y)=y^{-1}$. 
(The zonoids constructed in the end will not depend on the choice of $\tau$.)

In order to simplify notations, from now on it will be simpler to denote the induced action of $g\in G$ on $M$ 
simply by $gx:=a_g(x).$ Similarly, we denote the action induced by $g\in G$ on $T^*M$, 
via the adjoint of its differential, still by multiplication on the left by $g$. 
This means that, for a covector $\eta\in T^*_{y}M$, we use the notation
(which is consistent with~\eqref{coiso_representation})
\[ 
  g \eta := (D_{gy}a_{g^{-1}})^*\eta.
\]
In this way, if $Y\subseteq M$ is a submanifold passing through~$y$ then, 
at the level of cotangent spaces, the action of $\tau(y)$ maps the conormal space $N_yY$ of $Y$ at~$y$ 
to the conormal space of $\tau(y)Y$ at~$\bo$.

With this notation in mind, we now associate with a stratified set $Y$ certain Grassmann zonoids. 

\begin{definition}\label{def:KZ_Y}
Let $Y\subseteq M$ be a stratified set of codimension $d$ and of finite volume. 
For $y\in Y^\mathrm{sm}$, using the Pl\"ucker embedding and 
up to a sign, we can interpret the conormal space $N_yY$ as a simple vector 
$\pm \nu_Y(y)\in \Lambda^d (T_y^*M)$ of norm one. 
\begin{enumerate}
\item  We define the integrable random vector 
\[
 \xi_Y := \varepsilon \, \nu_{h\tau(y)Y}(\bo)  \in \Lambda^d(V),
\]
where $(h,y,\varepsilon)\in H\times Y\times \{-1,1\}$ is a uniform random variable. 
The \emph{centered Grassmann zonoid associated with $Y$} is defined by: 
$$
 K(Y):= \frac{\mathrm{vol}(Y)}{\mathrm{vol}(M)} \, K(\xi_Y)\in\GZ^d_o(V)^H.
$$
\item If moreover $Y$ is cooriented,
then the sign of $\nu(y)\in \Lambda^d V$ can be chosen consistently 
(agreeing with the coorientation)
and we have the well defined random variable
\[\
 \xi^+_Y :=  \nu_{h\tau(y)Y}(\bo)   \in \Lambda^d(V) .
\]
The \emph{(noncentered) Grassmann zonoid associated with $Y$} is defined by: 
\[ 
  Z(Y) := \frac{\mathrm{vol}(Y)}{\mathrm{vol}(M)} \, Z(\xi^+_Y)\in\GZ^d(V)^H.
\]
\end{enumerate}
\end{definition}

The zonoids $K(Y)$ and $Z(Y)$ do not depend on the choice of the measurable map $\tau:M\to G$, 
since we take expectations over the stabilizer group~$H$. For the same reason, 
these zonoids are $H$--invariant. 
However, the zonoid $Z(Y)$ depends on the choice of a coorientation of $Y$.
We have the Minkowski sum $Z(Y) =K(Y) + c(Z(Y))$,
see \eqref{eq:Z=K+12E} and \eqref{eq:c=12E}. 
We also observe also that $\|\xi_Y\|=1=\|\xi_Y^+ \|$ almost surely, and therefore
$$
    \ell(K(Y))=\ell(Z(Y)) = \frac{\vol(Y)}{\vol(M)}.
$$
The introduction of the random variable $\varepsilon$ 
in~\cref{def:KZ_Y}
may be be viewed as a probabilistic analogue of the fact 
that a manifold is always $\mathbb Z_2$-orientable.

\begin{example}
Let $\Lambda^n V = \R\orf$ with $\orf$ of norm one.
In the special case where $Y=\{p\}$ is a single point, we get for 
its centered Grassmann zonoid the centered interval
$$
 K(\{p\}) = 
   \frac{1}{\vol(M)} \ \frac{1}{2}\, [-\orf,\orf] .
$$
Note that $ K(\{p\})$ is a zonoid of degree $n=\dim V$, because $\{p\}$ has codimension $n$.
If the action of $H$ on $\Lambda^nV$ changes sign, then $\EE(\xi_{\{p\}})=0$ and 
$Z(\{p\}) = K(\{p\})$.  
Otherwise, the Grassmann zonoid of~$\{p\}$ equals the interval 
$$
 Z(\{p\}) = \frac{1}{\vol(M)}\ [0,\orf].
$$ 
\end{example}

Next, we focus on the classes of these zonoids in the probabilistic intersection ring of $M$.
We call them  \emph{Grassmann classes}. 

\begin{definition}\label{def:classY}
Let $Y\subseteq M$ be a stratified set of codimension $d$ and of finite volume. 
We define the \emph{Grassmann class} of $Y$ in $M$, 
\[
  [Y]_{\EE} := [K(Y)]\in \HEo^d(M).
\]
If $Y$ is cooriented, we define also its \emph{oriented Grassmann class}, 
\[ 
  [Y]^+_{\EE} := [Z(Y)]\in \HE^d(M). 
\]
and we denote 
\begin{equation}\label{eq:[Y]_c}
 [Y]_c := 2\, c(Z(Y))\in\HdR^d (M).
\end{equation}
\end{definition}

When $Y$ is cooriented, then $Z(Y) = K(Y) + c(Z(Y))$ implies the following decomposition 
in the probabilistic intersection ring:
\[\label{eq:Ydeco}
  [Y]^+_{\EE}=[Y]_{\EE}+\tfrac{1}{2}[Y]_c .
\]

In \cref{propo:poincare} we will show that 
$[Y]_c \in \HdR^d(V)$ 
can be interpreted as the de Rham cohomology class of $Y$
when $M$ is an orientable symmetric space. 
In fact, $[Y]_c $ equals the Poincar\'e dual of $Y$.
In the sequel, for simplicity of notation, we drop the $\wedge$ 
when writing the multiplication in $\HE(M)$. 
Note that, under the isomorphism \eqref{eq:AZk-iso}, the multiplication induced on $\HE(M)$ 
by the wedge multiplication in the zonoid algebra becomes 
\begin{equation}\label{def_mult_HE}
 [Y_1]_{\EE}^+\cdot [Y_2]^+_{\EE}
     =\left([Y_1]_\EE+\tfrac{1}{2}[Y_1]_c\right)\cdot\left([Y_2]_\EE+\tfrac{1}{2}[Y_2]_c\right)
 =[Y_1]_{\EE}\cdot [Y_2]_{\EE} + \tfrac{1}{2}\left([Y_1]_c\cdot [Y_2]_c\right).
\end{equation}

There is an important class of submanifolds for which the associated zonoids are easier 
to describe and to compute, namely \emph{cohomogeneous} submanifolds.

\begin{definition}[Cohomogeneous submanifold]\label{def:cohomogeneous}
A submanifold $Y\subseteq M$ is called \emph{cohomogeneous} 
if $G$~acts transitively on the set of conormal spaces of~$Y$,
see~\cite{PSC,kohn:21}. That is, for every $y_1, y_2\in Y$ 
there exists $g\in G$ such that 
$gy_1=y_2$ and $g\,N_{y_1}Y = N_{y_2}Y$.  
\end{definition}

If a submanifold $Y\subseteq M$ is cohomogeneous, the random vectors in \cref{def:KZ_Y} 
take a simpler form since they can be defined 
independently on the choice of a point $y\in Y$.  
More precisely, for every $y\in Y$, we can define the random vectors
\begin{equation}\label{random_vectors_in_cohom_case} 
 \xi_Y(y) :=   \varepsilon\, \nu_{h\tau(y)Y}(y) \in \Lambda^d(V)
      \quad \textrm{and}\quad 
       \xi_{Y}^+ :=  \nu_{h\tau(y)Y}(y)\in \Lambda^d(V)
\end{equation}
where $(h,y,\varepsilon)\in H\times Y\times \{-1,1\}$ is a uniform random variable. 
Then, for every fixed $y\in Y$, we have 
$K(\xi_Y) = K(\xi_Y(y))$ and $Z(\xi_Y^+)=Z(\xi_Y^+(y))$ 
and hence 
\[\label{formula_cohom_zonoid} 
  K(Y) = \frac{\mathrm{vol}(Y)}{\mathrm{vol}(M)} \, K(\xi_Y(y)) \quad \textrm{and} \quad 
 Z(Y) =  \frac{\mathrm{vol}(Y)}{\mathrm{vol}(M)} \, Z(\xi_Y^+(y)).
\]

Any submanifold $Y$ of a sphere $S^n$ or real projective space $\RP^n$ is cohomogeneous, 
since $\SO(n)$ acts transitively on $G(d,n)$. 
Further examples of cohomogeneous submanifolds are submanifolds of complex projective space 
with constant K\"ahler angles (see \cref{se:Kangle}) 
and Schubert varieties in Grassmannians (see \cref{sec:schubert}).

\begin{remark}\label{re:cohomog-distr} 
If $Y\subseteq M$ is a cohomogeneous submanifold, then 
the measure on $G(d,V)$ corresponding via~\eqref{eq:CGZ-to-measures} 
to the centered Grassmann zonoid $K(Y)\in\GZ^d_o(V)$, 
has a natural description. 
Recall the map $\tau$ from~\eqref{eq:def-tau} and
consider the collection of conormal spaces of $Y$,
$$
 \cN(Y) := \{ N_{\bo}\big(h\tau(y)Y\big) \mid h\in H, y\in Y \big\}\subseteq G(d,V),
$$
that are shifted in all possible ways to the distinguished point~$\mathbf o$.
Note that $\cN(Y)$ is a closed $H$--invariant subset of $G(d,V)$. 
By definition, $Y\subseteq M$ is a cohomogeneous iff $\cN(Y)$ is an $H$--orbit.
In this case, $\cN(Y)$ is a compact submanifold of $G(d,V)$
carrying a volume measure $\mu_Y$. 
By our construction, $K(Y)$ corresponds to the measure
$\frac{1}{\vol(M)}\, \mu_Y$.
\end{remark}

The span of the zonoid $K(Y)$ associated with a cohomogeneous submanifold $Y$ 
has the following description.

\begin{proposition}\label{pro:spanK(Y)}
Let $Y\subseteq M$ be a cohomogeneous submanifold of codimension~$d$
such that $\bo\in Y$ and let $\nu\in\Lambda^d (V)$ encode the normal 
space of $Y$ at $\bo$. Then the span of $K(Y)$ equals the span of 
the orbit $H\nu$.
\end{proposition}

\begin{proof}
This is a consequence of~\cite[Lemma~2.15]{BBLM}.
\end{proof}

%%%%
\subsection{The probabilistic intersection ring of the sphere}\label{sec_HE_sphere}

We now explicitly describe the probabilistic intersection ring of 
the $n$--dimensional sphere $S^n= \SO(n+1)/\SO(n)$. 
Note that, by \cref{prop:findimtransact} we already know that $\HE(S^n)$ is finite dimensional. 

\begin{proposition}\label{pro:HE-sphere}
The probabilistic intersection ring of the sphere $S^n= \SO(n+1)/\SO(n)$ 
is generated as a commutative algebra, by the Grassmann class 
$[B] \in \HE^1(S^n)$  of the unit ball $B\subseteq V$, 
and by the orientation class $[\varpi] \in\HE^n(S^n)$.
They satisfy the relations $[B]^{n+1}=0$ 
and $[\varpi]^2=0$. 
Moreover, the Grassmann class of a submanifold $Y\subseteq S^n$ of codimension~$d$, $0< d<n$
is given by 
\begin{equation}\label{eq:KZ-S^n}
 [Y]_{\EE} =\frac{\vol(Y)}{\vol(S^n)} \cdot \frac{1}{\ell(B^{\wedge d})}\cdot [B]^{d},
\end{equation}
and $[Y]_{\EE}=[Y]_{\EE}^+$ when $Y$ is cooriented. 
\end{proposition}

\begin{proof}
We identify the tangent space at an arbitrary point of $S^n$ with $V=\R^n$ 
and identify the isometry group of this point with $H:=\SO(n)$. 
Let us first compute $\HEo(S^n)$. 
The unit ball $B$ of $V$ is $H$--invariant. Hence its $d$--fold wedge product 
$B^{\wedge d}\in \GZ_o^d(V)$ is $H$--invariant as well. 
In fact, it spans $\VGZ_o^d(V)^H$: 
we claim that 
\begin{equation}\label{eq:GSO(2n)}
 \VGZo^d(V)^{H} = \R B^{\wedge d} .
\end{equation}
For this, it is enough to prove $\VGZ_o^d(V)^H$ is one dimensional. 
A quick way to see this is as follows: 
$\cM(G(d,V))^H$ is one dimensional, consisting of the multiples of the uniform measure, 
since $H$ acts transitively on the Grassmannian $G(d,V)$.
We now use the isomorphism in~\eqref{eq:CGZ-to-measures}, 
which is equivariant under the isometric action of $H$.
Passing to $H$--invariants, this yields
$
 \VGZ_o^d(V)^H \simto \cM(G(d,V))^H ,
$
therefore, $\VGZ_o^d(V)^H$ is indeed one dimensional. 
Moreover, $B^{\wedge d}\not\in \M^d(V)$, hence $\M(V)^H=0$. 
This provides the desired description (note that  $\beta^{n+1} =0$) 
$$
 \HEo(S^n)=\VGZ_o(V)^H/\M(V)^H = \VGZ_o(V)^H = \bigoplus_{d=0}^n\,  \R\, [B]^{d}.
$$

We now analyze $\HdR(S^n)$. It suffices to show that $(\Lambda^d V)^H=0$ for $0<d<n$. 
Expand $w\in (\Lambda^d V)^H$
$$
 w= \sum_{i_1<\ldots< i_d} c_{i_1\ldots i_d} \, e_{i_1}\wedge \ldots\wedge e_{i_d} .
$$ 
Consider $g=\mathrm{diag}(-1,1,\ldots,1,-1) \in H=\SO(n)$. 
Applying $g$ to $w$  changes the sign of the coefficient $c_{1\ldots d}$.
On the one hand, $gw=w$, since $g\in H=\SO(n)$. 
Therefore, $c_{1\ldots d}=0$.
Similarly, one shows that all coefficients vanish. 
\end{proof}

From~\cref{pro:HE-sphere} we get the following explicit description of graded algebras, 
with $\beta,\gamma$ denoting formal variables: 
\begin{equation}\label{eq:HE SO}
\HEo(S^n) = \R[\beta]/(\beta^{n+1}) = \bigoplus_{d=0}^n \R \, \beta^d ,\quad
\HdR(S^n) = \R[\gamma]/(\gamma^2) = \R \oplus\; \bigoplus_{d=1}^{n-1} 0\; \oplus \R \, \gamma 
\end{equation}
as a graded algebras with $\deg \beta =1$ and $\deg\gamma=n$. Now comes an important observation. 
If we represent instead $S^n = \OO(n+1)/\OO(n)$ and 
follow the proof of \cref{pro:HE-sphere}, we get the same result except that $(\Lambda^n V)^H=0$. 
So $\HEo(S^n)$ remains the same, but we obtain 
$\HdR(S^n) = \R = \R \bigoplus\; \bigoplus_{d=1}^{n} 0$. 
In particular, this shows that the probabilistic intersection ring~$\HE(M)$ 
depends on the choice of the homogeneous space structure on~$M$.

\begin{remark}\label{re:OvsSO}
In \cref{sec:classical}, we  show that the subalgebra~$\HdR (M)$  
can be identified with the de Rham cohomology algebra of the $M$,
provided $M$ is a orientable symmetric space and $G$ is connected. 
This holds in the first case, when $S^n = \SO(n+1)/\SO(n)$, 
but not in the second case $S^n = \OO(n+1)/\OO(n)$. 
Thus, the probabilistic intersection ring of $\SO(n+1)/\SO(n)$ recovers 
the de Rham cohomology of the sphere, 
but the probabilistic intersection ring of $\OO(n+1)/\OO(n)$ does not. 
\end{remark}

For real projective space $M=\RP^n$,  the analogous computation 
yields the same result for $\HEo(\RP^n)$. However, there is 
a change for $\HdR(\RP^n)$: while for odd~$n$ the same result is 
obtained, the degree $n$ part of $\HdR(\RP^n)$ vanishes if $n$ is even, 
which is consistent with the fact that $\RP^n$ is not orientable in this case. 
For showing this, one expresses $\RP^n=\SO(n+1)/H_n$ with the subgroup $H_n$ 
fixing the first standard basis vector. Then 
$$
 V=\mathfrak{so}(n+1)/\mathfrak{so}_n\simeq 
  \left\{ \begin{bmatrix} 0 & v^T\\ v & 0\end{bmatrix} \,,\bigg| v\in \R^n \right\} , \quad 
 H = \left\{ \begin{bmatrix} \det h & 0\\ 0 & h\end{bmatrix} \,\bigg|\, h\in \OO(n) \right\} \simeq \OO(n) .
$$
Thus $\OO(n)$ acts here on $\R^n$ with a twist: $(h,v)\mapsto \det(h)\, hv$.
Therefore, for $\orf:= e_2\wedge\ldots\wedge e_{n+1}$, we have 
$h \orf  = \det(h)^{n+1} \orf$. 
Hence $\orf$ is $\OO(n)$--invariant iff $n$ is even.

%%%%%
\subsection{Probabilistic intersection theory}\label{sec:pit}

Let $Y_1, \ldots, Y_s\subseteq M$ be stratified subsets of $M$. 
We move these submanifolds in $M$ by choosing independent uniformly random group elements $g_i\in G$ 
and consider the intersection $g_1Y_1\cap\cdots\cap g_s Y_s$ of 
the randomly moved copies $g_iY_i$ of $Y_i$. 
Our main result here shows that the multiplication of 
the (classes of) centered Grassmann zonoids $K_{Y_i}$  
models this intersection and 
allows to express the expected volume of such random intersections.
The proof of this result is a consequence of the general integral geometry result 
in~\cite[Theorem A.2]{PSC} and \cite{Howard}.

\begin{theorem}\label{thm:PSC}
Let $Y_1, \ldots, Y_s\subseteq M$ be stratified subsets of codimensions $d_i$ and finite volume, 
such that $d:=d_1+\ldots +d_s \le n=\dim M$, and $1\le s\le n$. 
Denote by 
$$
  [Y_1]_{\EE}, \ldots, [Y_s]_{\EE}\in \HE(M)
$$ 
their Grassmann classes; see \cref{def:classY}.
Let $g_1, \ldots, g_s\in G$ be independent and uniform. 
Then the product of the classes $ [Y_1]_{\EE}, \ldots, [Y_s]_{\EE}$ in the 
centered probabilistic intersection ring $\HEo(M)$ determines 
the expected volume of the random intersection 
$g_1 Y_1\cap\cdots \cap g_s Y_s$. It is given by the length of the product of these classes:
\begin{equation}
\EE\, \vol_{n-d}(g_1 Y_1\cap\cdots \cap g_s Y_s) =\mathrm{vol}(M)\cdot \ell( [Y_1]_{\EE} \cdots [Y_s]_{\EE}).
\end{equation}
We note that the submanifolds $g_1 Y_1,\ldots, g_s Y_s$  intersect transversally almost surely. 
In particular, if $\sum_{i=1}^sd_i=n$, we get 
$$ 
 \EE\#\,\left (g_1Y_1\cap\cdots\cap g_sY_s\right) 
    = \mathrm{vol}(M)\cdot \ell( [Y_1]_{\EE} \cdots [Y_s]_{\EE}).
$$
\end{theorem}

\begin{proof}
Let us rephrase~\cite[Theorem A.2]{PSC} 
using our terminology, writing $H$ instead of $K$ and $s$ for~$m$.
Let $y_1,\ldots,y_s \in Y$. Denote the zonoids 
$$
 K_i := \frac{\mathrm{vol}(Y)}{\mathrm{vol}(M)} \, K(\xi_Y(y_i)) ,
$$
where the random variable $\xi_Y(y_i)$ 
is defined in~\eqref{random_vectors_in_cohom_case}.
Tracing the definitions, we see that the quantity $\sigma_H(y_1,\ldots,y_s)$ 
defined in~\cite{PSC} equals here 
$$
 \sigma_H(y_1,\ldots,y_s) = \ell \big(K_1 \wedge\ldots\wedge K_s \big) ,
$$ 
Moreover, we note that  
$$
 \int_{Y_1\times\cdots \times Y_s} \sigma_H(y_1,\ldots,y_s)\; \mathrm dy_1\cdots \mathrm dy_s 
 = \vol(Y_1) \cdots \vol(Y_s) \; \EE \sigma_H(y_1,\ldots,y_s) .
$$ 
Combining these observations with~\cite[Theorem A.2]{PSC}, 
the assertion follows.
\end{proof}

In the special case where the $Y_i$ are hypersurfaces, 
we can express the average number of intersection points in terms of the mixed volume of the 
Grassmann zonoids $K(Y_i)$. 

\begin{corollary}\label{coro:hypersurfaces}
Assume that $Y_1, \ldots, Y_n\subseteq M$ are hypersurfaces of finite volume. Then,
\begin{equation}
\label{eq:hyperE}
   \EE\,\#\left (g_1Y_1\cap\cdots\cap g_nY_n\right)
       =\vol(M)\cdot n!\cdot \mathrm{MV}(K(Y_1), \ldots, K(Y_n)) 
\end{equation}
for independent and uniform $g_1, \ldots, g_n\in G$.
\end{corollary}

\begin{proof}
The statement follows from~\cref{thm:PSC} together with \eqref{eq:MV}.
\end{proof}

In the case of a sphere $M=S^n$, \cref{thm:PSC} and \eqref {eq:KZ-S^n} imply that 
for submanifolds $Y_1,\ldots,Y_s\subseteq M$ of codimension $d_i$ 
with $d:=d_1+\ldots +d_s \le n$, $s\le n$, we have 
$$
 \frac{1}{\vol(M)}\EE\#\,\left (g_1Y_1\cap\cdots\cap g_sY_s\right) = 
\frac{\ell(B^{\wedge d})}{\ell(B^{\wedge d_1}) \cdots 
      \ell(B^{\wedge d_s})}\; \prod_{i=1}^s \frac{\vol(Y_i)}{\vol(M)} .
$$
The multinomial style coefficient on the right hand side is an average scaling factor 
in the sense of \cite[Def.~3.10]{PSC}. 
Using \eqref{ex2}, on can check that it equals 
$\EE \|L_1 \wedge\cdots\wedge L_s\|$, for independent uniform random $L_i \in G(d_i,V)$. 

%%%
\subsection{Pullback of probabilistic intersection rings}

In this section we define the notion of a morphism between homogeneous spaces 
and show that it induces a morphism between the corresponding probabilistic intersection rings
in the reverse direction. In other words, we will exhibit a a contravariant functor $M\mapsto \HE(M)$.

\begin{definition}\label{def:morph-HS}
Let $\rho: G_1\to G_2$ be a Lie group morphism 
between compact Lie groups and let $H_1\subseteq G_1$ and $H_2\subseteq G_2$ be closed subgroups 
such that $\rho(H_1)\subseteq H_2$. 
Consider the homogeneous spaces $M_1:=G_1/H_1$ and $M_2:=G_2/H_2$. 
A \emph{morphism (of homogeneous spaces)} 
is a smooth map $f\colon M_1\to M_2$ such that $f(\mathbf{o}_{M_1})=\mathbf{o}_{M_2}$ 
and such that $f(g_1\cdot p)=\rho(g_1)\cdot f(p)$ 
for all $p\in M_1$ and $g_1\in G_1$. 
\end{definition} 

The canonical projection $\pi\colon G\to M=G/H$ is a morphism of homogeneous spaces: to see this,
take $G_1=G_2=G$, $\rho$ the identity, $H_1=H$ and $H_2=\{e\}$. 
Moreover, any Lie group morphism $\rho\colon G_1 \to G_2$ is a morphism of homogeneous spaces
(take $H_1=H_2=\{e\}$). 

In general, a morphism~$f$ makes the following diagram commute:
\begin{equation}\label{diagram:morphism}
\begin{tikzcd}
 0\ar[r] & H_1 \ar[r]\ar[d,"\rho|_{H_1}"] & G_1 \ar[d,"\rho"]    \ar[r, "\pi_1"]   & M_1 \ar[r]\ar[d,"f"]   & 0\\
 0\ar[r] & H_2 \ar[r]                     & G_2     \ar[r, "\pi_2"]          & M_2 \ar[r]              & 0
\end{tikzcd} ,   
\end{equation}
where the rows are short exact sequences given by the inclusion of subgroups and the canonical projections. 
Indeed, if $f$ is such a morphism, then we have for all $g_1\in G_1$
\begin{equation}
    f(\pi_1(g_1))=f(g_1\cdot\mathbf{o}_{M_1})=\rho(g_1)\cdot\mathbf{o}_{M_2} ,
\end{equation}
which completely determines the map $f$ and makes the diagram~\eqref{diagram:morphism} commutes.
We note that $f$ is injective iff $\rho(H_1)=H_2$. This is equivalent to $f$ being an immersion.
Similarly $f$ is surjective iff it is a submersion, which means that 
$\Lg_2 =D_\bo\rho(\Lg_1) + \Lh_2$.

There are many natural examples for morphisms of homogeneous spaces.
For instance, let $k\le k'$ and $m\le m'$. Then the embedding of Grassmannians
$G(k,m)\hookrightarrow G(k',m')$ is a morphism of homogeneous space
in a natural way. Of course, this also applies to complex Grassmannians. 

\begin{example}\label{re:PLE}
The \emph{Plücker embedding}
$G(k,m)\to \proj(\Lambda^k\R^{k+m})$
is a morphism of homogeneous spaces in sense of \cref{def:morph-HS}. 
In order to see this, we write 
$G(k,m)= G_1/H_1$ with $G_1=\OO(k+m)$ and the subgroup $H_1= \OO(k)\times \OO(m)$, 
see~\eqref{eq:G(kn)}. We consider the (injective) group homomorphism
$$
 \rho \colon G_1\to G_2:=\OO(\Lambda^k \R^{k+m}),\ g \mapsto g^{\wedge k}
$$
given by the wedge power. $H_1$ is mapped into 
$H_2:= \{\phi\in G_2 \mid \phi (e_1\wedge\ldots \wedge e_k) = e_1\wedge\ldots \wedge e_k \}$. 
Note that 
$\proj(\Lambda^k\R^{k+m})=G_2/H_2$. It is immediate to check that the 
resulting morphism $G_1/H_1 \to G_2/H_2$ indeed is the Plücker embedding.
\end{example}

We now show that morphisms of homogeneous spaces induce morphisms between 
probabilistic intersection rings in a functorial way. 
In~\cref{re:pullback-dR}, we will see later that this generalizes 
the pullback of de Rahm cohomology classes
in the case where $M_1,M_2$ are symmetric spaces and $G$ is connected.
The pullbacks of the embeddings 
$\RP^m \hookrightarrow\RP^n$  and $\CP^m \hookrightarrow\CP^n$ 
of real and complex projective spaces, respectively, will be investigated in~\cref{sec:complex}.

\begin{proposition}\label{prop: pullback of pir}
A morphism $f:M_1\to M_2$ induces a graded algebra homomorphism
\begin{equation}
    f^*:\HE(M_2)\to \HE(M_1) ,
\end{equation}
which preserves the decomposition \cref{eq:H-decomp} 
and is composed of homomorphisms of graded algebras (denoted by the same symbol)
$f^* \colon\HEo(M_2)\to \HEo(M_1)$ and $f^*:\HdR(M_2) \to \HdR(M_1)$.

The morphism $f^*:\HE(M_2)\to \HE(M_1)$, which we call pullback, satisfies the following properties:
\begin{enumerate}
    \item  For a sequence of morphisms
$\begin{tikzcd}
    M_1\ar[r,"f_1"]& M_2 \ar[r,"f_2"]& M_3 \end{tikzcd}$, we have
\begin{equation}\label{eq:compo pullback}
        (f_2\circ f_1)^*=f_1^*\circ f_2^*.
\end{equation}

\item For any integrable random simple vector $\xi\in \Lambda^d(T^*_\mathbf{o} M_2)$
   with $H_2$--invariant distribution, we have 
\begin{equation}\label{eq: pullback is pullback}
        f^*\left[Z(\xi)\right]=\left[Z(f^*\xi)\right] ,
\end{equation}
where $f^*$ on the right hand side denotes the pullback of skew symmetric forms. 

\item If $f$ is surjective, then $f^*$ is injective. 

\item If $f$ is injective and $\rho(H_1) = H_2$, then $f^*$ is surjective.
\end{enumerate}
\end{proposition}

\begin{proof}
Let us abbreviate $V_1:=T^*_\bo M_1$ and $V_2:=T^*_\bo M_2$. 
The derivative of $f\colon M_1 \to M_2$ at $\bo$ defines the linear map
$D_\bo f\colon V_1^* \to V_2^*$, which is $H_1$--equivariant,  
where the $H_1$-action on $V_2$ is defined via~$\rho$. 
The dual map $\phi:=(\mathrm{d}_\bo f)^*\colon V_2 \to V_1$ is $H_1$--equivariant as well.
Now we apply the functor $Q$ of~\cref{pro:funct}, obtaining the  $H_1$-equivariant 
homomorphism of graded algebras 
\begin{equation}\label{eq:Q(f^*)}
 Q(\phi)\colon Q(V_2) \to Q(V_1) .
\end{equation}
Restricting to the subalgebras of $H_1$-invariants, we obtain the homomorphism 
of graded algebras 
$$
 Q(\phi)\colon Q(V_2)^{H_1} \to Q(V_1)^{H_1} .
$$ 
Since $\rho(H_1)\subseteq H_2$, we get the inclusion 
$Q(V_2)^{H_2} \subseteq Q(V_2)^{H_1}$.
Composing this inclusion with the above~$Q(\phi)$, 
we obtain the stated homomorphism of graded algebras, that we denote $f^*$:
$$
 { f^*} \colon \HE(M_2)= Q(V_2)^{H_2} \to Q(V_1)^{H_1} = \HE(M_1).
$$ 
Clearly, this restricts to the map $\HEo(M_2)= Q_o(V_2)^{H_2} \to Q_o(V_1)^{H_1}=\HEo(M_1)$ 
of the centered parts and the map 
$\HdR(M_2)= \Lambda (V_2)^{H_2} \to \HEo(M_1)= \Lambda(V_1)^{H_1}$ of the centers.

The functoriality stated in (1) for $\HEo$ follows from the construction and the functoriality of $Q$.
(For~$\HdR$, it just expresses the functoriality of the exterior power.) 

For (2), we recall that $Q(\phi)$
is defined by the linear maps $\VGZ^d(V_2)\to \VGZ^d(V_1)$,  
which arise by mapping zonoids in $\Lambda^d V_2$ to zonoids in $\Lambda^d V_1$ 
via the exterior power map $\Lambda^d\phi:\Lambda^d V_2\to \Lambda^d V_1$. 
The stated property in (2) thus follows directly from the construction.

For (3), assume $f$ is surjective. Thus $f$ is a submersion, hence $\phi$ is injective, 
therefore $Q(\phi)$ is injective by~\cref{cor:inj-surj}. 
After composing with an inclusion, it remains injective, hence $f^*$ is injective. 

For (4), assume $f$ is injective. Thus $f$ is an immersion, hence $\phi$ is surjective, 
therefore $Q(\phi)$ is surjective by~\cref{cor:inj-surj}. 
Since we assume $\rho(H_1) = H_2$, the inclusion $Q(V_2)^{H_2} \subseteq Q(V_2)^{H_1}$
is an equality. Hence  $f^*$ is surjective.
\end{proof}

The proof of~\cref{prop: pullback of pir} shows that the pullback of 
$f\colon M_1 \to M_2$ is also defined before passing from 
$\VGZ(V)$ to the the quotient algebra $Q(V)$.
Indeed, $f$ also induces a graded algebra homomorphism
$f^*\colon \Lambda^d T^*_\bo M_2 \to \Lambda^d T^*_\bo M_1$,  
inducing linear maps
$$
 f^*\colon \VGZ^d(T^*_\bo M_2 ) \to \VGZ^d (T^*_\bo M_1 )
$$ 
which allow to map the Grassmann zonoids $K(Y), Z(Y) \in \VGZ^d(T^*_\bo M_2 )$.

\begin{remark}
\cref{re:PLE} reveals that in (4), the surjectivity of $f^*$ may fail without the condition on subgroups. 
For an example of this we consider again the Plücker embedding 
\[
 f\colon G(k,m)\to \mathrm{P}(\Lambda^k(\R^{k+m})).
 \] 
 It induces
$f^*\colon \HE(\mathrm{P}(\Lambda^k(\R^{k+m}))) \to \HE(G(k,m))$.  
However, $\HE(G(k,m))$ has infinite vector space dimension if $k,m >1$
(\cref{cor:Grassinfinitedim}), 
while the probabilistic intersection algebra of projective spaces has finite vector space dimension, 
see \cref{sec_HE_sphere}.
\end{remark}

\subsection{Pull back of Grassmann zonoids associated with stratified sets}

Recall that we associated a centered Grassmann zonoid $K(Y)$ with a stratified set $Y\subseteq M_2$ of 
a homogeneous space $M_2$ (see \cref{def:classY}). 
If $Y$ is cooriented, we also associated the noncentered Grassmann zonoid $Z(Y)$ with~$Y$.
Let $f\colon M_1\to M_2$ be a morphism of homogeneous spaces. 
We next analyze the geometric meaning of the pull back of these Grassmann zonoids and their classes 
with respect to $f$. For this, we need two lemmas. First, we have to clarify when the inverse image of $Y$ is stratified. 
This is the purpose of the next lemma.

\begin{lemma}\label{lem_preimage}
Suppose that $f\colon M_1\to M_2$ is a morphism and $Y\subseteq M_2$ is a stratified set of codimension~$d$. 
Then: 
\begin{enumerate}
\item $f^{-1}(g_2 Y)$ is a stratified set of codimension~$d$, for almost all $g_2\in G_2$.
\item If $f$ is a submersion, then $f^{-1}(Y)\subseteq M_1$ is a stratified set with codimension~$d$. 
\end{enumerate}
Moreover, if in these cases $Y$ is cooriented, then $f^{-1}(Y)$ is cooriented in a natural way. 
\end{lemma}

\begin{proof}
For the first statement, observe first that for almost every $g_2\in G_2$ 
the map $f$ is transversal to all strata of $g_2Y$. 
Then $f^{-1}(g_2Y)$ is a stratified set of the same codimension \cite{thom}.
For the second item, if $f$ is a submersion, then it is transversal to every stratified set $Y$, 
whose preimage is a stratified set of the same dimension \cite{thom}.\end{proof}

We will work with the following explicit description of the support function 
of Grassmann zonoids. 

\begin{lemma}\label{le:h_K(Y)}
Let $M$ be a Riemannian homogeneous space and $Y\subseteq M$ be a stratified set 
of codimension~$d$. 
The support function of the zonoid $K(Y)$ satisfies, for 
$w\in \Lambda^d V$, 
\begin{equation}
 h_{K(Y)} (w)  = \frac{1}{2\,\vol(G)} \, \int_{Y\times H} 
  |\langle \nu_{h\tau(y)Y}(\bo),w\rangle| \, \mathrm dy\, \mathrm dh  .
\end{equation}
Moreover, if $Y$ is cooriented, the center of $Z(Y)$ satisfies 
$$
 c(Z(Y))  = \frac{1}{2 \, \vol(G)} \, \int_{Y\times H} \nu_{h\tau(y)Y}(\bo) \, \mathrm dy\, \mathrm dh  .
$$
If $Y$ is $H$--invariant, then the integration over~$H$ 
can be dropped and the above formulas simplify to
\begin{equation}
 h_{K(Y)} (w)  = \frac{1}{ 2\,\vol(M)} \, \int_{Y} 
  |\langle \nu_{\tau(y)Y}(\bo),w\rangle| \, \mathrm dy , \quad 
   c(Z(Y))  = \frac{1}{2 \,\vol(M)} \, \int_{Y} \nu_{\tau(y)Y}(\bo) \, \mathrm dy. 
\end{equation}
\end{lemma}

\begin{proof}
Note that $\bo \in h\tau(y)Y$ for $h\in H$, $y\in Y$ by the definition of the map~$\tau$ in~\eqref{eq:def-tau}. 
The above~\cref{def:KZ_Y} of $K(Y)$ translates via~\eqref{eq:defZK} for the support function to
\begin{align}
 2 h_{K(Y)}(w)  &= \frac{\vol(Y)}{\vol(M)} \cdot \frac{1}{\vol(Y)\vol(H)} \int_{Y\times H}
                        |\langle \nu_{h\tau(y)Y}(\bo),w\rangle| \, \mathrm dy \, \mathrm dh   \\
                    &=\frac{1}{\vol(G)} \int_{Y\times H}
                         |\langle \nu_{h\tau(y)Y}(\bo),w\rangle| \, \mathrm dy \, \mathrm dh ,
\end{align}
where we used $\vol(G)=\vol(M) \vol(H)$. 
For the center we argue similarly, using $ 2\, c(Z(Y))  = \EE (\xi^+_Y)$; see~\eqref{eq:c=12E}. 
Finally, in the case where $Y$ is $H$--invariant, the stated simplification is immediate. 
\end{proof}

For stating the next result we define the notion of a \emph{random virtual zonoid}
as a Borel measurable map $\Omega \to \VZ_o(V) \oplus V,\, \omega\mapsto Z_\omega$, 
where $\Omega$ is a probability space and $V$ is a Euclidean space. 
Here we use the weak-$\ast$ topology on $\VZ_o(V)$, which is defined via \eqref{eq:VZ0-to-cM}.
One can show that the length functional $\ell\colon\VZ_o(V) \to\R$ 
is continuous with respect to the weak-$\ast$ topology 
(see \cite[Thm.~2.26]{BBLM}).
We call the random virtual zonoid $Z_\omega$ \emph{integrable} 
if $\EE(\ell (Z_\omega))$ is finite. We define the \emph{expectation} $\EE(Z_w)$ of an 
integrable $Z_\omega$ as the virtual zonoid, 
whose support function $h$ satisfies $h(v) =\EE_\omega( h_{Z_\omega}(v))$ for all $v\in V$, 
see~\cite{vitale}. 
 
Suppose $Z_\omega$ is an integrable, random virtual Grassmann zonoid $Z_\omega$ taking values in $\VGZ^d(V)$.
This defines a random Grassmann class $[Z_\omega] \in Q^d(V)$. 
We note that the expectation $\EE(Z_\omega)$ only depends on the class $[Z_\omega]$ since 
the virtual zonoids in the ideal $\M(V)$ have length zero, see~\cref{propo:loriker}.

The next result shows that the expectation of the Grassmann zonoid of 
the inverse image of a randomly moved stratified set $Y\subseteq M_2$
under a morphism $f_1\colon M_1 \to M_2$ has a natural description in terms of the pullback of 
the corresponding Grassmann zonoids. 
This even applies before passing to probabilistic intersection rings.
Moreover, we point out that no randomness is necessary in the statement 
if the structural morphism $\rho$ behind $f$ satisfies $\rho(G_1)=G_2$.

\begin{theorem}\label{th:pullback}
Let $f:M_1\to M_2$ be a morphism and $Y\subseteq M_2$ be a stratified set. 
\begin{enumerate}
\item We have for uniformly random $g_2\in G_2$,
$$
f^*\big(K(Y)\big) =\EE_{g_2\in G_2} K(f^{-1}(g_2Y)) .
$$ 
In particular,
$f^*\big([Y]_\EE\big) =\EE_{g_2\in G_2} [f^{-1}(g_2 Y)]_\EE$. 

\item If $\rho(G_1)=G_2$, then the Grassmann zonoid of the inverse image $f^{-1}(Y)$ 
is described as 
$$
 f^*\big(K(Y)\big) = K(f^{-1}(Y)) , \quad
 f^*\big([Y]_\EE\big) = [f^{-1}(Y)]_\EE . 
$$

\item If $Y$ is cooriented, then the pullback of noncentered Grassmann zonoids satisfies
$$
 f^*\big(Z(Y)\big) =\EE_{g_2\in G_2} Z(f^{-1}(g_2Y)) ,\quad 
 f^*\big([Y]^+_\EE\big) =\EE_{g_2\in G_2} [f^{-1}(g_2 Y)]^+_\EE 
$$
and if $\rho(G_1)=G_2$, then
$$
 f^*\big(Z(Y)\big) = Z(f^{-1}(Y)) , \quad
 f^*\big([Y]^+_\EE\big) = [f^{-1}(Y)]^+_\EE.
$$
\end{enumerate}
\end{theorem}

\cref{le:bg-consistent} contains an application of this 
result to the embeddings $\CP^m \hookrightarrow \CP^n$.

We remark that in the case where $M_1,M_2$ are symmetric spaces and $G$ is connected,
the pullback property for the classes of centers coincides the well known 
characterization of the pullback of cohomology classes; see~\cref{re:pullback-dR}.

The proof proceeds by first showing the assertion 
in two special cases: 
for a canonical projection $\pi \colon G \to G/H$ 
and for the inclusion $G_1\hookrightarrow G_2$ 
of a closed subgroup in a compact Lie group.
We will see that the general case follows from these two cases 
by diagram chasing. 

\begin{lemma}\label{lemma:pulback proj}
Let $M=G/H$ be a homogeneous space, $\pi:G\to M$ be the canonical projection,  
and $Y\subseteq M$ be a stratified set of codimension~$d$. Then 
$$
 \pi^*\big(K(Y)\big) = K(\pi^{-1}(Y)) , \quad 
 \pi^*([Y]_\EE) = [\pi^{-1}(Y)]_\EE .
$$
If $Y$ is cooriented, we also have 
$\pi^*(Z(Y)) = Z(\pi^{-1}(Y))$.
\end{lemma}

\begin{proof}
We note that $\pi^{-1}(Y)$ is a stratified set by \cref{lem_preimage}.

We first deal with the displayed formula. Here it clearly suffices to prove the left equation. 
For this, we will show the equality of the corresponding support functions. 
We have the linear map $\pi^*\colon \Lambda^d (T^*_{\bo}M )\to \Lambda^d (T^*_{e}G)$, where $e\in G$ is the identity. By construction, $\pi^*$ is an isometric embedding.
By \cite[Prop.~2.1(4)]{BBLM} we have for $w\in \Lambda^d T^*_{e}G$,
\begin{equation}\label{eq:321321}
 h_{\pi^*(K(Y))}(w) = h_{K(Y)}( (\pi^*)^T w ) ,
\end{equation}
where $(\pi^*)^T$ is the map transposed to $\pi^*$.
By~\cref{le:h_K(Y)}, the Grassmann zonoid $K(Y)\subseteq \Lambda^d (T^*_{\bo}M)$ associated with $Y$
satisfies 
\begin{align}\label{eq:hKY}
h_{K(Y)}((\pi^*)^T w)  
  &=\frac{1}{\kappa}\int_{H\times Y} |\langle \nu_{h\tau(y)Y}(\bo),(\pi^*)^T w\rangle| \, \mathrm d h\, \mathrm d y \\
  &=\frac{1}{\kappa} \int_{H\times Y} |\langle \pi^*\big(\nu_{h\tau(y)Y}(\bo)\big),w\rangle| \, \mathrm d h\, \mathrm d y,
\end{align}
where $\kappa := 2\, \vol(G)$.
Now we compare the conormal spaces  of $Y$ with those of $\pi^{-1}(Y)$. 
If $\bo \in Y$, we have 
$$
 \pi^* (\nu_Y(\bo)) = \nu_{\pi^{-1}(Y)}(e) .
$$
This implies that for all $y\in Y$, 
\begin{equation}\label{eq:321}
 \pi^*\big( \nu_{h\tau(y)Y}(\bo) \big) = \nu_{\pi^{-1}(h\tau(y) Y)}(e)  =\nu_{h\tau(y)\pi^{-1}(Y)}(e) ,
\end{equation}
where $\tau\colon G\to M$ is the measurable map from~\eqref{eq:def-tau}, which satisfies $\tau(y)y=\mathrm o$.
We therefore get from~\eqref{eq:321321} and~\eqref{eq:hKY}
\begin{equation}\label{eq:2hKy}
 h_{\pi^*(K(Y))}(w) =
 h_{K(Y)}((\pi^*)^T w)  = 
  \frac{1}{\kappa} \int_{y\in Y} \int_{h\in H} |\langle \nu_{h\tau(y)\pi^{-1}(Y)}(e) , w \rangle | \, \mathrm d h\,\mathrm d y .
\end{equation}
On the other hand, applying \cref{le:h_K(Y)} to the $H$--invariant inverse image $\pi^{-1}(Y)$, we get 
$$
  h_{K(\pi^{-1}Y)}(w) = \frac{1}{\kappa} \int_{g\in \pi^{-1}(Y)} 
      |\langle \nu_{g^{-1}\pi^{-1}(Y)}(e) , w\rangle| \, \mathrm d g .
$$
We rewrite this integral with the help of the coarea formula 
applied to $\pi\colon G \to M$ over~$Y$. 
Taking into account the isometry of the fibers of $\pi$, 
$$
 H \simto \pi^{-1}(y) = \{ \tau(y)^{-1} h \mid h \in H \} ,\, 
  h \mapsto \tau(y)^{-1} h^{-1} ,
$$
and abbreviating the integrand $\varphi(g) := |\langle \nu_{g^{-1}\pi^{-1}(Y)}(e) , w\rangle|$, 
we obtain
$$
  \int_{g\in \pi^{-1}(Y)} \varphi(g) \, \mathrm dg 
 =  \int_{y\in Y} \int_{g\in \pi^{-1}(y)} \varphi(g)\, \mathrm d g \, \mathrm dy 
 =  \int_{y\in Y} \int_{h\in H} \varphi(\tau(y)^{-1} h^{-1}) \, \mathrm d h\, \mathrm dy .
$$
But
$$
 \varphi(\tau(y)^{-1} h^{-1}) = |\langle \nu_{h\tau(y)\pi^{-1}(Y)}(e) , w\rangle| 
$$
equals the integrand appearing on the right of~\eqref{eq:2hKy}. We have thus shown that 
$$
  h_{\pi^*(K(Y)}(w) = h_{K(\pi^{-1}Y)}(w) ,
$$ 
which completes the proof of the displayed equation.

We now assume that $Y$ is cooriented and deal with the stated formula on the centers. 
By~\cref{le:h_K(Y)}, 
\begin{align}
\pi^*\big(c(Z(Y)) \big)
  =\frac{1}{\kappa}\int_{H\times Y} \pi^*\big(\nu_{h\tau(y)Y}(\bo)\big) \, \mathrm d h\, \mathrm d y .
\end{align}
Using \eqref{eq:321} and arguing as above, we show that 
$$
 \pi^*\big(c(Z(Y)) \big) = \pi^*\Big(  \frac{1}{\kappa} \int_{g\in \pi^{-1}(Y)} \nu_{g^{-1}Y}(e) \, \mathrm dg \Big) 
 =  c(Z(\pi^{-1}(Y))),
$$  
where the right equation holds due to~\cref{le:h_K(Y)}. 
We conclude using the isomorphism \cref{eq:Z-iso} 
(namely, a zonoid is uniquely determined by its centered zonoid and its center).
\end{proof}

The next lemma covers the most difficult part of the proof of~\cref{th:pullback}. 
The proof relies on the basic integral geometry formula in~\cite{Howard}, 
which is also the main tool behind~\cref{thm:PSC}.

\begin{lemma}\label{lemma:pulback_inclusion}
Let $G_1$ be a closed subgroup of the compact Lie group $G_2$ and $\rho:G_1\hookrightarrow G_2$ be the inclusion. 
Let $Y\subseteq G_2$ be a stratified set. 
For a uniformly random $g_2\in G_2$ we denote by 
$K(G_1 \cap g_2 Y)$ the Grassmann zonoid associated with the 
stratified set $G_1 \cap g_2 Y$ of the homogenous space $G_1$. 
By contrast, $K(Y)$ denotes the 
zonoid associated with $Y$ in the homogeneous space $G_2$. 
With these conventions, we have 
$$
 \rho^*(K(Y)) = \EE_{g_2\in G_2} K\big( G_1 \cap g_2 Y \big), \quad 
 \rho^*\big([Y]_\EE\big) =\EE_{g_2\in G_2} [G_1 \cap g_2 Y]_\EE ,
$$ 
I.e., the pullback of the Grassmann class $[Y]_\EE$ of $Y$ in $G_2$ 
is obtained by averaging the Grassmann classes $[G_1 \cap g_2 Y]_\EE$ of 
the intersection of $G_1$ with the randomly moved sets $g_2Y$.  
In addition, if $Y$ is cooriented, 
$$
 \rho^*(c(Z(Y))) = \EE_{g_2\in G_2} c\big( Z(G_1 \cap g_2 Y) \big) .
$$
\end{lemma}

\begin{proof}
First, we note that $G_1\cap g_2Y = \rho^{-1}(g_2Y)$ is a stratified set 
for almost all $g_2\in G_2$ by \cref{lem_preimage}.
  
We have $\rho^{-1}(g_2Y) = G_1 \cap g_2Y$. 
Note that $\rho^*\colon T^*_e G_2 \to T^*_e G_1$ 
becomes the orthogonal projection 
$T_e G_2 \to T_e G_1$ after identifying the tangent spaces with their duals. 
With this identification, we will also have $NY\simeq TY^\perp$.

We begin with the proof of the displayed left statement. 
For this we will show the equality of support functions
\begin{equation}\label{eq:WNATED}
 h_{\rho^*(K(Y))} (w) = \EE_{g_2 \in G_2} \, h_{K(G_1 \cap g_2 Y)} (w) ,\quad w\in T_e^*G_1
\end{equation}
where $w\in \Lambda^d T^*_e G_1$ 
and $d$ is the codimension of $Y$.
As in the proof of~\cref{lemma:pulback proj}, see \eqref{eq:hKY},
we get 
\begin{equation}\label{eq:RKYA}
  h_{\rho^*(K(Y))}(w) = \frac{1}{\kappa_2} \int_{y\in Y} |\langle \rho^*\nu_{y^{-1}Y}(e) , w \rangle | \, \mathrm d y ,
\end{equation}
where here $\kappa_2 := 2\vol(G_2)$.
Note that $\tau(y)=y^{-1}$ since we are working in the group $G_2$ 
(and the stabilizer is trivial). 

Assume that $Y$ and $G_1$ intersect transversally at a smooth point $y$ of $Y$.
The orthogonal projection of $T_yY^\perp\simeq N_yY$ onto $T_eG_1$, after moving to the origin~$e$, 
is captured by 
$$\rho^*(\nu_{y^{-1}Y}(e)) = \pm \sigma(T_yY^\perp,T_{g_1}G_1^\perp)\;  \nu_{G_1\cap y^{-1}Y}(e),$$
where the scalar factor equals
$$
 \sigma(T_yY^\perp,T_{g_1}G_1^\perp) := \| \nu_{G_1}(e) \wedge  \nu_{y^{-1}Y}(e)\| ,
$$
where the notation is from Howard~\cite[(2-1)]{Howard}. 
It quantifies the relative position of the normal spaces (identified with the conormal spaces)
$T^\perp_y Y$ and $T^\perp_{g_1}G_1$ at $y$. 
We therefore obtain the crucial geometric relation 
$$
 |\langle \rho^*(\nu_{y^{-1}Y}(e)) , w \rangle | 
   =  \sigma(T_yY^\perp,T_{g_1}G_1^\perp) \; \, |\langle \nu_{G_1 \cap y^{-1}Y}(e) ,w \rangle | .
$$ 
We continue with \eqref{eq:RKYA}, and introduce a dummy integration over $G_1$, 
$$
  h_{\rho^*(K(Y))}(w) = \frac{1}{\kappa_2\, \vol(G_1)} \int_{(g_1,y)\in G_1\times Y}  
    \sigma(T_yY^\perp,T_{g_1}G_1^\perp) \; \, |\langle \nu_{G_1 \cap y^{-1}Y}(e) ,w \rangle | \, \mathrm dg_1\,\mathrm d y . 
$$
Now we abbreviate the integrand
$h(y) := |\langle \nu_{G_1 \cap y^{-1}Y}(e) ,w \rangle |$
and apply Howard's basic integration formula \cite[(2-9)]{Howard} to 
the submanifold $G_1$ and the stratified\footnote{Howard states this for submanifolds, 
but the extension to stratified sets is straightforward.} 
subset~$Y$ of $G_2$. This yields
$$
 \int_{(g_1,y)\in G_1 \times Y}  \sigma(T_yY^\perp,T_{g_1}G_1^\perp) \;\,  h(y) \mathrm dg_1\, \mathrm dy 
 = \int_{g_2\in G_2} \int_{g_1 \in G_1 \cap g_2 Y} h(g_2^{-1}g_1) \, \mathrm dg_1 \  \mathrm dg_2 .
$$
By definition of $h$, we have 
$
 h(g_2^{-1}g_1) = |\langle \nu_{G_1 \cap g_1^{-1}g _2Y}(e) ,  w \rangle |.
$
Now we observe that by~\cref{le:h_K(Y)}, setting $\kappa_1:=2\vol(G_1)$,  
$$
  h_{K(G_1\cap g_2 Y)}(w) 
  =\frac{1}{\kappa_1}\int_{g_1\in G_1 \cap g_2 Y} |\langle \nu_{g_1^{-1}(G_1\cap g_2 Y)}(e), w \rangle | \,\mathrm d g_1  
  = \frac{1}{\kappa_1}\int_{g_1\in G_1 \cap g_2 Y} h(g_2^{-1}g_1) \,\mathrm d g_1.
$$
We conclude that, using that $\kappa_2\vol(G_1) = \kappa_1\vol(G_2)$, 
$$
 h_{\rho^*(K(Y))}(w) = \frac{1}{\vol(G_2)} \int_{g_2 \in G_2}  h_{K(G_1\cap g_2 Y)}(w) \, \mathrm dg_2 
   = \EE_{g_2 \in G_2 } h_{K(G_1\cap g_2 Y)}( w) ,
$$
which proves~\eqref{eq:WNATED}. 

It remains to show the statement on the centers 
if $Y$ is cooriented. The proof is analogous (compare also the proof of~\cref{lemma:pulback_inclusion}). 
We note that 
the sign in the crucial geometric relation equals~$1$
when using a suitable convention for the coorientation of $G_1\cap g_2Y$.
\end{proof}

\begin{proof}[Proof of \cref{th:pullback}]
(1) We prove the first assertion of the theorem in two steps. 
\cref{lemma:pulback_inclusion} covers injective Lie group morphisms. 
We can extend it to any Lie group homomorphism $f\colon G_1 \to G_2$. 
For this, we factor $f= p \circ F$, with the injective group morphism 
$$
 F\colon G_1 \to G_1\times G_2,\, g_1 \mapsto (g_1,f(g_1)) 
$$ 
and the projection $p\colon G_1\times G_2 \to G_2,\, (g_1,g_2) \mapsto g_2$. 
Writing $G_2= (G_1\times G_2)/(G_1\times \{e\})$, we see that 
$p$~is a morphism of homogeneous spaces, to which we can apply~\cref{lemma:pulback proj}. 
For a stratified set $Y\subseteq G_2$, we therefore get 
$$
 p^*(K(Y)) = K(G_1\times Y) . 
$$
On the other hand, we can apply~\cref{lemma:pulback_inclusion} 
to the injective group morphism $F$ and the stratified set $G_1\times Y$, 
which gives 
$$
 F^*(K(G_1\times Y)) = \EE_{(g_1,g_2)\in G_1\times G_2} K\big( F^{-1}(g_1G_1\times g_2Y) \big) 
 = \EE_{g_2\in G_2} K\big( f^{-1}(g_2Y) \big) .
$$
Combining these two equations, using $f^*= F^* \circ p^*$, we arrive at 
$$
 f^*(K(Y)) = F^*(p^*(K(Y))) = F^*(K(G_1\times Y)) = \EE_{g_2\in G_2} K\big( f^{-1}(g_2Y) \big) ,
$$
which proves the first assertion for the Lie group morphism~$f$. 
We remark that the expectation can be dropped if $f$ is surjective, 
since in this case $G_2$ can be viewed as a Riemannian homogeneous space under 
the action of $G_1$ and $f\colon G_1 \to G_2$ as a projection so that \cref{lemma:pulback proj} applies.

We can now derive the first statement in the general case 
by a diagram chase in the commutative diagram~\cref{diagram:morphism}
and its dual diagram:
\begin{equation}\label{diagram:comorphism}
\begin{tikzcd}
\HEo(G_1) & \HEo(M_1) \ar[l, "\pi^*_1"] \\
\HEo(G_2) \ar[u, "\rho^*"]  & \HEo(M_2) \ar[u, "f^*"]  \ar[l, "\pi^*_2"] 
\end{tikzcd} .
\end{equation}
We will also make use of the fact that $\pi^*_1$ is injective since 
$\pi_1$ is surjective. 

Let $Y\subseteq M_2$ be stratified. 
\cref{lemma:pulback proj} applied to $\pi_2$ yields
$\pi^*_2(K(Y)) = K(\pi^{-1}_2(Y))$. 
We already showed the assertion for the Lie group homomorphism $\rho$.
Applying it to $\pi_2^{-1}(Y)$, we further get 
$$
 \rho^*(\pi_2^*(K(Y))) = 
 \rho^*(K(\pi_2^{-1}(Y))) = \EE_{g_2\in G_2} K(\rho^{-1}(g_2 \pi_2^{-1}(Y)))
        = \EE_{g_2\in G_2} K(\rho^{-1}(\pi_2^{-1}(g_2 Y))) .
$$
Here \cref{lemma:pulback proj} was applied for the second equality.
Using $\pi_2\circ\rho = f\circ \pi_1$, we rewrite the argument on the right hand side 
$$
 \rho^{-1}(\pi_2^{-1}(g_2 Y)) = (\pi_2\circ\rho)^{-1}(g_2 Y) = (f\circ\pi_1)^{-1}(g_2 Y) 
  = \pi_1^{-1}(f^{-1}(g_2 Y)) .
$$
The corresponding Vitale zonoid satisfies 
$$
 K(\pi_1^{-1}( f^{-1}(g_2 Y))) = \pi^*_1( K(f^{-1}(g_2 Y)) ) ,
$$
where we applied~\cref{lemma:pulback proj} to~$\pi_1$.
We have arrived at 
$$
\pi_1^*(f^*(K(Y))) = \pi^*_1\big(\EE_{g_2\in G_2} ( K(f^{-1}(g_2 Y)) ) \big) ,
$$ 
where we used $\rho^*\circ\pi_2^* =  \pi_1^*\circ f^*$ 
to rewrite the left-hand side and interchanged $\pi^*_1$ with the expectation. 
The injectivity of $\pi_1^*$ finally implies the first assertion
$f^*(K(Y)) = \EE_{g_2\in G_2} K(f^{-1}(g_2 Y))$. 
The statement 
$f^*\big([Y]_\EE\big) =\EE_{g_2\in G_2} [f^{-1}(g_2 Y)]_\EE$ 
on the classes is an immediate consequence.

(2) We now assume that $\rho$ is surjective. The second assertion follows by tracing the above proof, 
taking into account that no expectations are necessary.

(3) The third assertion also directly follows by tracing the above proof.
\end{proof}

%%%%%%%%
\bigskip
\section{Connection to cohomology}\label{sec:classical}

Compared with classical intersection theory, the probabilistic intersection ring~$\HE(M)$ 
of a homogeneous space~$M=G/H$ is a ring whose elements are equivalence classes of zonoids, 
and not cohomology classes. Under further assumptions on~$M$, 
a fundamental results due to Elie Cartan~\cite{CartanE:29} implies that the 
de Rham cohomology ring $\HDR(M)$ of $M$ can be obtained as a subring of $\HE(M)$, 
namely as the ring $\HdR(M)=\Lambda(V)^H$ of $H$--invariants. 
For this, we assume that $M$ is a \emph{symmetric space} (see \cref{def_symmetric_space}) 
and that $G$ is connected.
The \emph{signed} count of intersection points of transversal manifolds of~$M$ 
can be expressed in the cohomology ring: by contrast with the \emph{unsigned} count, 
no randomness is required for this. 
We remark that~\cite{Mathis_Stecconi} contains a formula for expected sign count 
in a more general context.

%%%
\subsection{De Rham cohomology of symmetric spaces}\label{sec:DR}

We recall here fundamental results going back to Elie Cartan, 
which allow to describe the de Rham cohomology of 
symmetric spaces in terms of invariant forms.
For a detailed treatment of de Rham cohomology, 
we refer to~\cite{BottTu}.

In the following, we adopt the general assumptions made at the beginning of~\cref{sec:prob_int_th}, 
and additionally assume that $G$ is connected and a symmetric space \cite[\S 11.3.3]{goodman-wallach:09}.

\begin{definition}\label{def_symmetric_space}
Let $M=G/H$ be a Riemannian homogeneous space, 
where $G$ is a compact Lie group~$G$ with a closed subgroup $H$. 
If there is an involutive Lie group automorphism $\sigma\colon G\to G$ 
(Cartan involution) such that 
$\{g\in G \mid \sigma(g) = g\}$ equals the identity component of $H$, 
we call $M$ a~\emph{symmetric space}.
\end{definition}

The involution in \cref{def_symmetric_space} induces the involution 
$\sigma_\bo\colon M \to M,\, gH \mapsto \sigma(g)H$ of the homogeneous space $M$,
which fixes the point $\bo=H$ and whose derivative at $\bo$ satisfies
\begin{equation}\label{eq:der-symm}
 D_\bo \sigma_{\bo} = -\id .
\end{equation}
We have the equivariance property 
$\sigma_\bo(gx) = \sigma(g) \sigma_\bo(x)$ for $g\in G$, $x\in M$. 
(More generally, the involution at the point $gH$, 
$\sigma_{gH}(xH) := g\sigma(g^{-1}x)H$, fixes $gH$ and satisfies 
$D_{gH} \sigma_{gH} = -\id$.) 

\begin{example}
Let us define the Grassmannian as the 
homogeneous space $G(k,m):=G/H$ with the 
connected Lie group $G=\SO(k+m)$ and its subgroup
$H:= \{(g,h)\in \OO(k) \times \OO(m) \mid \det(g) \det(h)= 1\}$.
Then $G(k,m)$ is a  symmetric space,
as seen by the involution 
$\sigma\colon \SO(k+m) \to \SO(k+m),\, g\mapsto J g J$ with 
$J=(\begin{smallmatrix} I_k & 0 \\ 0 & -I_m \end{smallmatrix})$, 
which satisfies 
$$
 \{g\in \SO(k+m) \mid \sigma(g) = g\} = \SO(k) \ti \SO(m) .
$$
In particular, real projective spaces (and spheres) are symmetric spaces. 
Similarly, one shows that complex projective spaces and 
complex Grassmannians are symmetric spaces.
In~\cref{sec:psc} we use the more convenient 
definition $G(k,m) = \OO(k+m)/(\OO(k)\ti \OO(m))$, see~\eqref{eq:G(kn)}, 
which yields the same underlying manifold, 
However, see~\cref{re:Grassm-alternative}.
\end{example}

We need to recall some basic facts on differential forms. We denote by 
$\Omega^d (M)$ the vector space of smooth differential forms on $M$ of order $d$. 
The direct sum 
$\Omega(M) := \bigoplus_{d=0}^n \Omega^d(M)$
is a graded algebra with respect to the wedge product. 
The $G$--action on $M$ induces a $G$--action on $\Omega(M)$ 
by algebra homomorphisms, defined by
$g \alpha(v_1,\ldots,v_d) := \alpha(g^{-1}v_1,\ldots, g^{-1}v_d)$, $g\in G$, $\alpha\in \Omega^d(M)$.
We are interested in the graded subalgebra 
$$\Omega(M)^G := \bigoplus_{d=0}^n \Omega^d(M)^G$$
of $G$--invariants. 
It is obtained as the image of the $G$--invariant projection 
\begin{equation}\label{eq:average-form}
 \Omega^d (M) \to \big(\Omega^d (M)\big)^G,\quad 
  \alpha\mapsto \overline{\alpha} := \int_G g\alpha\; \mathrm dg ,
\end{equation}
given by averaging over $G$ with respect to the normalized Haar measure.

The following results are not new. We present proofs for lack of a suitable refererence.

\begin{lemma}\label{le:pass2average-form}
If $\alpha \in \Omega^d (M)$ is closed, $d\ge1$, then there is $\beta\in \Omega^{d-1} (M)$ such that 
$\alpha -\overline{\alpha} = \mathrm d\beta$.
\end{lemma}

\begin{proof}
Let $c$ be a $d$-cycle in $M$, $g\in G$ and $gc$ be the transformed cycle.  
Since $G$ is connected, we can continuously deform $c$ to $gc$, which implies 
$\langle\alpha, gc\rangle = \langle\alpha, c\rangle$.
Therefore, 
$\langle g\alpha, c\rangle = \langle\alpha,  gc\rangle = \langle\alpha, c\rangle$.
Taking the expectation over $G$ implies 
$\langle \overline{\alpha} , c\rangle = \langle\alpha, c\rangle$.
The assertion follows, since the we have an isomorphism 
$\HDR^d(M)\to \Hom(H_d(M),\R)$ defined by integrating closed differential forms  
over cycles (de Rham's theorem; orientability of $M$ is not required). 
\end{proof}

The assumption that $M$ is a symmetric space has far reaching consequences for cohomology. 

\begin{proposition}\label{pro:invariant-forms-are-closed}
If $\alpha\in\Omega^d (M)^G$ is $G$--invariant, then $\mathrm d\alpha=0$. 
\end{proposition}

\begin{proof}
Compare~\cite[p.~227, p.~564]{helgason:78} and~\cite[\S 7.4.2]{nicol:22}.
Recall the involution 
$\sigma_\bo\colon M \to M$ which fixes the point $\bo$.
We consider the pullback $\beta := \sigma_\bo^* \, \alpha$.
The equivariance property of $\sigma_\bo$ implies that $\beta$ is $G$--invariant. 
Since $D_\bo \sigma_{\bo} = -\id$ by \eqref{eq:der-symm}, we have at the point~$\bo$
$$
 \beta_{\bo}(x_1,\ldots,x_d) = \alpha_{\bo}(D_\bo \sigma_{\bo} x_1,\ldots, D_\bo \sigma_{\bo} x_d)
  = (-1)^d \alpha_{\bo}(x_1,\ldots, x_d) .
$$
The $G$--invariance of $\alpha$ and $\beta$ and the transitivity of the $G$--action 
imply that $\beta =  (-1)^d \alpha$ on $M$. 
Taking exterior derivatives, we get $\mathrm d\beta =  (-1)^d \mathrm d\alpha$. 
On the other hand, by the functoriality of the exterior derivative, we have 
$\mathrm d\beta= \sigma_\bo^* \, \mathrm d\alpha$, which is a form of degree $d+1$. 
Arguing as above, we get $\mathrm d\beta =  (-1)^{d+1} \mathrm d\alpha$. 
Comparing the two equations for $\mathrm d\beta$, we see that we must have $\mathrm d\alpha=0$. 
\end{proof}

Suppose now the $G$--invariant form $\alpha\in\Omega^d (M)^G$ is exact, say 
$\alpha=\mathrm d\beta$ for $\beta\in\Omega^{d-1} (M)^G$. 
The $G$--invariant projection commutes 
with the exterior derivative; see  \cite[Lemma~12.2]{bredon:93}.
Therefore, we have $\alpha= \overline{\alpha} = \overline{\mathrm d\beta} = \mathrm d\, \overline{\beta}$. 
\cref{pro:invariant-forms-are-closed} implies that  $\alpha=\mathrm d\, \overline{\beta} = 0$.

Summarizing, we showed that if $M$ is a {symmetric space}, then
every de Rham cohomology class 
contains a unique $G$--invariant form, which is automatically closed 
and obtained by the projection~\eqref{eq:average-form}.
Therefore, we can identify the de Rham cohomology space $\HDR^d(M)$ 
with the space of $G$--invariant forms $\Omega^d (M)^G$
via the linear isomorphism
$$
 \Omega^d (M)^G \to \HDR^d(M),\; \alpha \mapsto [\alpha] ,
$$  
where $[\alpha]\in \HDR(M)$ denotes the cohomology class of $\alpha$. 
The resulting map $\Omega(M)^G \simto \HDR(M)$ is an isomorphism  of graded algebras.

As above, we define $V:=(\Lg/\Lh)^*$ as the cotangent space of~$M$ 
at the distinguished point $\mathbf o$ \eqref{def_distinguished_point}. 
Thus the exterior algebra $\Lambda(V)$ is the 
stalk of $\Omega(M)$ over $\mathbf o$. 
The evaluation map $\Omega(M)^G \to \Lambda(V)^H$ at~$\mathbf o$ 
that sends $\omega\mapsto \omega(\mathbf o)$  
 defines an  isomorphismus of graded algebras, whose inverse 
\begin{equation}\label{eq:Gamma-map}
\Gamma\colon \Lambda (V)^H \to \Omega(M)^G
\end{equation}
maps $\alpha\in\Lambda (V)^H$ to the $G$--invariant extension 
$\Gamma(\alpha)\in \Omega(M)^G$, which is well defined by the 
$H$--invariance of $\alpha$. 

Let us highlight our findings in the following important result.
(For the proof, just compose the above two algebra isomorphisms.) 

\begin{theorem}\label{th:HdOC}
Under the above assumptions, 
we have the graded algebra isomorphisms
\begin{equation}\label{eq:cartaniso}
 \HdR(M) =\,  \Lambda (V)^H \stackrel{\Gamma}{\longrightarrow}\Omega(M)^G\stackrel{[\cdot]}{\longrightarrow} 
 \HDR(M) .
\end{equation}
The inverse of $\Gamma$ is the evaluation at $\mathbf o$. 
The inverse of $[\cdot]$ is given by the averaging map~\eqref{eq:average-form}.
\end{theorem}

\begin{remark}\label{re:pullback-dR}
Suppose that $f\colon M_1 \to M_2$ is a morphism of homogeneous spaces in the sense of~\cref{def:morph-HS}.  
Then $f^*:\HdR(M_2) \to  \HdR(M_1)$, as introduced in~\cref{prop: pullback of pir},  
is consistent with the pullback of de Rahm cohomology classes with 
regard to the isomorphisms in \cref{th:HdOC}
\end{remark} 

We finally remark that 
if one drops the assumption that $M$ is a symmetric space, the situation is 
considerably more complicated. Then the cohomology algebra of~$M$
can be expressed in terms of Lie algebra cohomology, as invented by Chevalley and Eilenberg
for this very purpose in their seminal work~\cite{chevalley-eilenberg:48}, 
which marked the birth of homological algebra.

%%%
\subsection{Poincar\'e duality}\label{sec:cohomological}

In addition to the previous assumptions, we now assume that the symmetric space~$M$ 
is \emph{oriented}. 
We assume that $\orf \in \Omega^n(M)^G$ is a $G$--invariant form defining the orientation. 
So $\orf(\mathbf o)$ is an orientation of $\Lambda^n V$ as in \cref{se:EP-HS}.
We note that, for $\alpha,\beta\in\Omega(M)^G$,
$$
 \int_M \alpha\wedge\beta = 
  \vol(M)\cdot \langle \alpha(\mathbf o) \wedge \beta(\mathbf o),\ \varpi(\mathbf o) \rangle_{\mathbf o} .
$$ 
Poincar\'e duality states that the bilinear pairing, 
for $0\le m\le n$ and $d:=n-m$,  
$$
 \HDR^m(M) \times \HDR^d(M) \to \HDR^n(M) \simto \R , \quad ([\alpha],[\beta]) \mapsto \int_M \alpha\wedge\beta 
$$
is nondegenerate. 
Thus we can assign to a closed oriented submanifold $Y\subseteq M$ of dimension~$m$ its 
\emph{Poincar\'e dual}, which is defined as 
the unique class $[\eta_Y]\in \HDR^d(M)$ such that 
\begin{equation}\label{eq:PDdef}
 \int_M \alpha\wedge\eta_Y = \int_Y \alpha 
\end{equation}
for every $[\alpha]\in \HDR^m(M)$, see~\cite[(5.13)]{BottTu}. 
The Poincar\'e dual of a point $x\in M$ equals 
$\eta_{x} = \frac{1}{\vol(M)}\, \orf$. Moreover, $\eta_M =1$.

In view of applications (e.g., to Schubert Calculus), 
we need a slightly more general definition of the Poincar\'e dual, 
that works also for stratified sets (see \cref{def:stratified}).  
We first record the following useful fact.

\begin{lemma}
Let $M$ be a smooth orientable manifold and $Y\subseteq M$  be a stratified subset of dimension~$m$, 
of finite volume and orientable
(i.e., $Y^{\mathrm{sm}}$ is orientable). 
Assume that the set $Y^{\mathrm{sing}}$ of singularities has codimension at least two. 
Then the following map $\HDR^m(M)\to \R$ is well defined:
\begin{equation}\label{eq:IY}
    [\alpha]\mapsto \int_{Y} \alpha := \int_{Y^{\mathrm{sm}}} \alpha .
\end{equation}
\end{lemma}

\begin{proof}
To see that this map is well defined, it is enough to prove that 
for every $\gamma\in \Omega^{m-1}(M)$ we have
$\int_{Y^{\mathrm{sm}}}\mathrm d\gamma=0.$
The proof of this identity follows by arguing as in the proof for Stokes' theorem 
for analytic varieties in \cite{gri-ha:94}, using the fact that the singularites of~$Y$ 
have codimension two.
\end{proof}

We are now ready for the following definition. In the case $Y$ is a smooth closed submanifold, 
this recovers the usual definition of Poincar\'e dual.

\begin{definition}
\label{def:PD}
Let $M$ be a smooth orientable manifold and $Y\subseteq M$  
be a stratified subset of codimension $d$, orientable and of finite volume. 
The \emph{Poincar\'e dual} of $Y$ is given by the unique class $[\eta_Y]\in \HDR^d(M)$ 
such that \eqref{eq:PDdef} holds for every element $[\alpha]\in \HDR^{n-d}(M)$.
\end{definition}

Note that, in the previous definition, the orientability of $M$ is needed 
in order to have Poincar\'e duality and the orientability of $Y$ is needed 
in order to define integration on it. 
In this case we coorient $Y$ by the Hodge star $\star \orf_Y$ of the orientation of $Y$.
Then we associate with $Y$ the zonoid 
$Z_Y =Z(\xi_Y^+)\in \GZ^d(V)^H$, which is 
the (noncentered) Vitale zonoid of the 
random variable $\xi_Y^+ \in \Lambda^d V$, 
see~\cref{def:KZ_Y}. 
Note that, by this definition,
\begin{equation}\label{eq:Yc}
  [Y]_c = \frac{\vol(Y)}{\vol(M)}\cdot \EE(\xi^+_Y)\in (\Lambda^dV)^H.
\end{equation}
since the distribution of $\xi_Y^+$ is $H$--invariant. 
We now show that 
$[Y]_c$ corresponds to the Poincar\'e dual of $Y$ under the isomorphism in~\cref{th:HdOC}.
This also explains why the volume ratio $\vol(Y)/\vol(M)$ enters naturally.

\begin{proposition}\label{propo:poincare}
Under the above assumptions, 
the class associated with $Y\subseteq M$ 
(see \cref{def:KZ_Y}) satisfies 
$$
   \left[ \Gamma([Y]_c)\right] = [\eta_Y] \in \HDR^d(M) .
$$ 
\end{proposition}

\begin{proof}
Let us abbreviate $z:= [Y]_c \in (\Lambda^d V)^H$.
In order to show that $[\Gamma(z)]$ equals the Poincar\'e dual of $Y$ 
(see \cref{def:PD}), we need to verify that for all $[\alpha]\in \HDR^m(M)$, 
$$
 \int_{M}\alpha \wedge \Gamma(z) = \int_Y \alpha .
$$
It is enough to verify this for every 
$G$--invariant closed form $\alpha\in \Omega^m(M)^G$. 
For such $\alpha$, we have 
$\alpha\wedge \Gamma(z)\in \Omega^n(M)^G$.
Hence, setting  
$$
 t := \langle \alpha(\mathbf o)\wedge z,\ \orf(\mathbf o)\rangle_{\mathbf o} , 
$$
we have 
\begin{equation}
 \int_M \alpha \wedge \Gamma(z)  = t\,\mathrm{vol}(M) .
\end{equation}
We now compute $t$. Recall from \cref{eq:Yc} that $[Y]_c = (\vol(Y)/\vol(M))\, \EE(\xi^+_Y),$ so that 
$$
 t = \frac{\vol(Y)}{\vol(M)}\; \langle \alpha(\mathbf o) \wedge \EE(\xi^+_Y),\ \orf(\mathbf o) \rangle_{\mathbf o}.
$$
Hence it suffices to prove  
\begin{equation}\label{eq:alpha-int}
  t\; \frac{\vol(M)}{\vol(Y)}= 
  \langle \alpha(\mathbf o) \wedge \EE(\xi^+_Y),\ \orf(\mathbf o) \rangle_{\mathbf o} 
     = \frac{1}{\vol(Y)}\; \int_Y\alpha .
\end{equation}

By the definition of $\orf_Y$, the Hodge dual 
$\star \orf_Y(y)$ at $y\in Y$ represents 
the positively oriented normal space $N_yY$ of $Y$ at~$y$.
Therefore
$$
 \xi^+_Y = h\, g(y) \,(\star\orf(y)) \in \Lambda^d V 
$$
where we recall that $g(y) y = \mathbf o$. As a consequence,
$$
 \EE(\xi^+_Y) = \EE_{y\in Y}\, \EE_{h\in H}\, h\, g(y) \, (\star\orf)(y) .
$$
We are going to compute the inner product
$\langle \alpha(\mathbf o) \wedge \EE(\xi^+_Y), \orf(\mathbf o) \rangle_{\mathbf o}$.
Fix  $y\in Y$ and consider the isometry 
$H\to\mathrm{Stab}(y),\; h\mapsto k(y) :=g(y)^{-1}hg(y)$, 
and note that $k(y) (\star\orf)(y) = k(y)$.  
By the $G$--invariance of of $\alpha,\orf$ and the inner product, we obtain 
\begin{align}
  \langle \alpha(\mathbf o) \wedge hg(y) (\star\orf)(y),\ \orf(\mathbf o) \rangle_{\mathbf o} 
  &= \langle g(y)^{-1}\alpha(\mathbf o) \wedge g(y)^{-1}hg(y) (\star\orf)(y),\ g(y)^{-1}\orf(\mathbf o) \rangle_y \\
  &=\langle \alpha(y) \wedge (\star\orf)(y),\ \orf(y) \rangle_y  ,
\end{align}
which does not depend on~$h$. Moreover, 
by the property~\eqref{def_hodge_star} of the Hodge star, we get
$$
 \langle \alpha(y) \wedge \star\orf_Y(y),\ \orf(y) \rangle_y = \langle \alpha(y),\ \orf_Y(y)\rangle_y .
$$ 
This implies 
$$
 \langle \alpha(\mathbf o)\wedge \EE(\xi^{+}_Y) ,\ \orf(\mathbf o) \rangle_{\mathbf o}
 =\EE_{y\in Y} \langle \alpha(y),\ \orf_Y(y)\rangle_y 
 = \frac{1}{\vol(Y)}\; \int_Y \alpha ,
$$
which shows~\eqref{eq:alpha-int} and completes the proof.
\end{proof}

%%%
\subsection{Counting with signs in probabilistic intersection ring}\label{se:sign-count}

We return to the problem of counting intersection points. 
But unlike in~\cref{sec:pit}, we shall now count the intersection points with a sign.
Let us first recall a few basic facts from classical intersection theory in smooth manifolds. 

Suppose $Y_1,\ldots,Y_s$ are 
oriented stratified subsets of $M$
with transversal intersection, which means that they intersect only along their smooth points 
and the intersection is transversal.
Their Poincar\'e duals satisfy the following key property with regard to intersection 
(see \cite[(6.31)]{BottTu}): 
\begin{equation}\label{eq:PD-intersect}
 [\eta_{Y_1\cap\cdots\cap Y_s}] = [\eta_{Y_1}] \wedge\ldots\wedge [\eta_{Y_s}] ,
\end{equation}
when using the following convention 
$\star\orf_Y := \star\orf_{Y_1}\wedge\ldots\wedge\orf_{Y_s}$
for the (co)orientation of $Y:=Y_1\cap\cdots\cap Y_s$.
Suppose now further that the codimensions $d_i$ of $Y_i$ add up to $n$.
By compactness,  there are only finitely many intersection points $x\in Y_1\cap\cdots \cap Y_s$.
At such $x$ one defines the \emph{intersection index} (see \cite[\S 6]{BottTu}) 
\[\label{eq:intind}
  i_x(Y_1, \ldots, Y_s;M) = \begin{cases}
  + 1, & \text{ if  $(\star\orf_{Y_1}\wedge\cdots\wedge\star\orf_{Y_s})(x)$ is a positive multiple of 
  $\orf(x)$}\\
   -1 & \text{ otherwise}.
\end{cases}
\] 
From \eqref{eq:PD-intersect} one can deduce that 
\begin{equation}\label{eq:i-sum}
\sum_{x\in Y_1\cap \cdots \cap Y_s} i_x(Y_1,\ldots, Y_s; M) 
 = \int_{M}\eta_{Y_1}\wedge\cdots\wedge \eta_{Y_s} 
\end{equation}
Summarizing:  one first associates to each submanifold $Y_i$ an element of the graded ring $\HDR(M)$, 
namely the Poincar\'e dual $[\eta_{Y_i}]$, then takes the product $[\eta_{Y_1}]\wedge\cdots \wedge[\eta_{Y_s}]$ 
of these elements, and finally evaluates by taking the integral over $M$, 
to get the number of intersection points, counted with their signs.

\begin{remark}\label{re:homotopy-invariance}
By \cref{le:pass2average-form}, %
the Poincar\'e duals satisfy $[\eta_{g_iY_i}] = [\eta_{Y_i}]$ 
for all choices of $g_i\in G$. Hence 
$\int_{M}\eta_{g_1Y_1}\wedge\cdots\wedge \eta_{g_sY_s} = \int_{M}\eta_{Y_1}\wedge\cdots\wedge \eta_{Y_s}$.
Therefore, the signed count of intersection points is the same integer, 
for almost all choices of $g_1,\ldots,g_s\in G$. 
This observation states a special case of the homotopy invariance of the left hand 
quantity; see\cite[\S 5.2]{hirsch:93}. 
The left hand side of \eqref{eq:i-sum} stays the same when 
moving the $Y_i$, while maintaining transversality. 
\end{remark}

The next theorem complements~\cref{thm:PSC} and gives an interpretation of the multiplication 
in the probabilistic intersection ring $\HE(M)$ of the noncentered Grassmann zonoids $Z_{Y_i}$ 
associated with stratified subsets~$Y_i$. 
The key insight is that one 
can express the signed count of intersection points of the $Y_i$ 
by their multiplication in the subring $\HdR(M)$ of $\HE(M)$. 
Recall that in this section, we assume that $M$ is an oriented symmetric space and $G$ is connected. 
These assumptions are not needed for~\cref{thm:PSC}.

\begin{theorem}\label{thm:general2}
Let $Y_1, \ldots, Y_s\subseteq M$ be closed oriented stratified subsets 
of codimensions $d_i$ 
such that $d_1+\cdots+d_s=n$ and $g_1,\ldots,g_s\in G$. 
Let 
$$ 
 e_{+}(g) := \#\{x \mid x\in  g_1Y_1\cap\ldots \cap g_sY_s \mid i_x(g_1Y_1,\ldots,g_sY_s;M) = +1\}
$$
denote the number of intersection points of $g_1Y_1,\ldots,g_sY_s$ with positive index. 
Similarly, denote by~$e_{-}(g)$ the number of intersection points with negative index. 
Then
\begin{equation}\label{eq:ZHE}
   \vol(M)\,  [Y_1]_{\EE}^+\cdots [Y_s]_\EE^+ 
   =  [\,-\EE(e_{-})\, \orf,\, \EE(e_{+})\, \orf\,] ,
\end{equation}
where $[Y_1]_\EE^+, \ldots, [Y_s]_\EE^+ \in \HE (M)$ 
are the oriented Grassmann classes of 
the stratified sets $Y_1,\ldots,Y_s$, as in~\cref{def:KZ_Y}.
\end{theorem}

\begin{remark}\label{rem_difference_constant}
While $e_{+}(g)$ and $e_{-}(g)$ may vary with $g$, their difference 
$e_{+}(g) - e_{-}(g)$ is almost surely constant by~\cref{re:homotopy-invariance}. 
\end{remark}

\begin{proof}[Proof of~\cref{thm:general2}] 
The assertion~\eqref{eq:ZHE} concerns (classes of) zonoids in $\VGZ^n(V)=\VZ(\Lambda^n V)$. 
Since $\Lambda^n V\simeq \R$, such zonoids are segments. Therefore,
it suffices to prove that both zonoids in \eqref{eq:ZHE} have the same length and the same center; 
see also~\cref{ex:ZR}.

We abbreviate $Z_i := Z(Y_i)$ and $K_i := K(Y_i)$. 
By definition of the product in $\HE(M)$ we have 
$$[Y_1]_{\EE}^+\cdots [Y_s]_\EE^+ = [Z_{1}\wedge\cdots\wedge Z_{s}].$$
Note that $Z_{1}\wedge\cdots\wedge Z_{s}$ is an interval in $\VGZ^n(V)$ 
whose centered version is, by \cref{le:prod-well-defined}, 
the interval 
$K_1\wedge\cdots\wedge K_s$, which has the same length.
This length is described by~\cref{thm:PSC} as 
$$
 \vol(M)\,\ell\big(K_{1}\wedge\cdots\wedge K_{s} \big) = \EE (e_{+}(g) -e_{-}(g)). 
$$
Note that the right hand interval in~\eqref{eq:ZHE} also has this length. 

Now we compare the centers. 
The right hand interval has the center 
$\tfrac12  \EE (e_{+}(g) -e_{-}(g))\, \orf$.
It is therefore sufficient to prove that that the left hand interval has the same center. 
From \cref{def_mult_HE} we get
$$
  [Y_1]_{\EE}^+\cdots [Y_s]_\EE^+ = [Y_1]_{\EE}\cdots [Y_s]_\EE + \tfrac{1}{2} [Y_1]_c\cdots [Y_s]_c.
$$ 
Hence, we have to prove
\begin{equation}\label{eq:oribile}
 \vol(M)\, \ori\big( [Y_1]_c\cdots [Y_s]_c  \big) = \EE (e_{+}(g) -e_{-}(g)),
\end{equation}
where $\ori\colon\Lambda^n V \simto\R$ denotes the linear isomorphism defined by the orientation of~$M$.

By \cref{propo:poincare}, $[Y_i]_c$ corresponds to the 
Poincar\'e dual $[\eta_{i}]:=[\eta_{Y_i}]$ of $Y_i$ 
under the algebra isomorphism in~\cref{th:HdOC}.
Therefore, 
$[Y_1]_c \cdots  [Y_s]_c$
corresponds to 
$[\eta_{1} \wedge \cdots \wedge \eta_{s}]$
under this isomorphism. 
This implies
$$
 \vol(M)\,  \ori\big( [Y_1]_c \cdots [Y_s]_c \big)
   = \int_M \eta_{1} \wedge \cdots \wedge \eta_{s} .
$$
However, \cref{eq:i-sum} and $[g\eta_i] = [\eta_i]$ imply that,
\begin{equation}\label{eplus_eminus}
 \int_M \eta_{1} \wedge \cdots \wedge \eta_{s}
 =\int_M g \eta_{1} \wedge \cdots \wedge g\eta_{s} 
 = e_{+}(g) -e_{-}(g) = \EE (e_{+}(g) -e_{-}(g)).
\end{equation}
where $g_1,\ldots,g_s\in G$ are such that the $g_iY_i$ intersect transversally  (which is almost surely the case 
by Sard's lemma, e.g., see \cite[Prop.~A.18]{condition}) and
where the last equality follows from the fact that $e_{+}(g) -e_{-}(g)$ is almost surely constant (see  \cref{rem_difference_constant}). This  implies \eqref{eq:oribile} and finishes the proof.
\end{proof}

Writing $E := \EE(e_{-}+ e_{+})$ for the expected number of intersection point,
we can rephrase~\cref{thm:PSC} in analogy with~\cref{thm:general2} as follows:
\begin{equation}\label{eq:KHE}
   \vol(M)\, \cdot\, [Y_1]_\EE\cdots [Y_s]_\EE 
   =  \tfrac12 [\, -E \cdot \orf,\, E \cdot \orf\,] .
\end{equation}

We also obtain the following corollary, which does not involve randomness. 

\begin{corollary}\label{coro:general2}
Under the assumptions of \cref{thm:general2}, 
if $Y_1,\ldots, Y_s$ intersect transversely, we have 
\begin{equation}\label{eq:oriE}
 \sum_{x\in Y_1\cap\ldots\cap Y_s }i_x(Y_1,\ldots,Y_s;M) 
  = \vol(M) \cdot \ori \Big( 2\,c(Z_1)\wedge\cdots \wedge 2\,c(Z_s) \Big),
\end{equation}
where $\ori\colon\Lambda^n V \simto\R$ is the linear isomorphism defined by the orientation of $M$, and $Z_i := Z(Y_i)$. 
\end{corollary}

\begin{proof} 
We have by \cref{eq:i-sum}, \cref{eplus_eminus} and \cref{{re:homotopy-invariance}} that 
$$\sum_{x\in Y_1\cap\ldots\cap Y_s }i_x(Y_1,\ldots,Y_s;M)  = \vol(M)\, \ori\big( [Y_1]_c\cdots [Y_s]_c  \big) .$$
The statement then follows from \cref{thm:general2} using that $[Y_1]_c\cdots [Y_s]_c = 2\,c(Z_1)\wedge\cdots \wedge 2\,c(Z_s)$ (see \cref{def:classY}).
\end{proof}

%%%%%%%%
\bigskip
\section{Probabilistic intersection ring of complex projective space}\label{sec:complex}

We illustrate here the general framework developed so far
at the important example of the complex projective space~$\CP^n$.
This is the Riemannian homogeneous space characterized by
\begin{equation}\label{eq:def-CPn}
 \CP^n=\UU(n+1)/(\UU(1)\times \UU(n)) .
\end{equation}
The induced action of $\UU(1)\times \UU(n)$ on $V=(T_\bo\CP^n)^*$ amounts to study the standard action 
of $\UU(n)$ on $\C^n$, since we can identify $(T_\bo\CP^n)^* \simeq \C^n$, 
compare the general setting in \cref{se:setting}.
We shall describe in detail the structure of the centered probabilistic intersection ring
\begin{equation}\label{eq:def-Rn}
 R_n:=\HEo(\CP^n) =  \CGZ_o(\C^n)^{\UU(n)} ,
\end{equation}
see~\cref{def:HE}.
This is a graded commutative algebra over $\R$: we write 
$R_n = \bigoplus_{d=0}^{2n} R_n^d$ for its decomposition into 
homogeneous parts. This situation is particularly simple since 
the algebra~$R_n$ is a finite dimensional real vector space.
The reason is that $\UU(n)$ acts transitively on the unit sphere of~$\C^n$, 
see~\cref{prop:findimtransact}. Let us briefly recall the connection to 
valuations: in this particular setting, 
the Crofton map defines an isomorphism 
$$
 R_n \simto  \sval(\C^n)^{\UU(n)} = \val(\C^n)^{\UU(n)} .
$$ 
of graded algebras, see~\cref{eq:CGT-in-Val} and \cref{re:alesker-fin-dim}.

Our characterization of the basic structure of $R_n$ 
(bases, algebra generators) are reformulations of~results obtained 
by Alesker~\cite{alesker:01,aleskerHLT}, Fu~\cite{FuUnval}, and 
Bernig and Fu~\cite{bernig-fu:11} in the context~of valuations. 
However, our derivation of the Hard-Lefschetz property 
and Hodge-Riemann relations for~$\CP^n$ via the positive definiteness of 
the matrix of moments of the Beta distribution is new. It relies on the 
positive definiteness of the Hankel matrices built with 
central binomial coefficients $\binom{2n}{n}$; see \cref{pro:Y}.
As a consequence of the general theory, we also derive an interesting formula 
on the expected number of intersection points of $n$~randomly moved 
real submanifolds of codimension two~(\cref{cor:selfintcodim2CPn}).

Our treatment is self-contained, except that we 
heavily rely on the following result, 
first obtained by Alesker~\cite[Thm.~6.1]{alesker:01},  
using representation theory:
\begin{equation}\label{eq:A-dim-F}
\dim R_n^d = 1+ \min\left\{\left\lfloor \frac{d}{2}\right\rfloor,\, \left\lfloor \frac{n-d}{2} \right\rfloor\right\} 
\quad \mbox{for $0\le d \le 2n$} .
\end{equation}
In fact, we only rely on the upper bound on the dimension.
A more direct approach to prove this upper bound via multiple K\"ahler angles
is outlined in~\cref{se:Kangle}: this relies on work by Tasaki~\cite{tasaki:03}.

%%%
\subsection{The two generators of the algebra}\label{se:two-generators}

The degree one part $R_n^1$ of $R_n$ is one-dimensional by~\cref{eq:A-dim-F}. 
It generated by the class of the unit ball $B_{2n}$ in $\C^n$:
$$
  \b_n := [B_{2n}] \in R_n^1 , 
$$
Consequently, the zonoid associated with any real hypersurface $Y\subseteq \CP^n$
is a multiple of $\b_n$. More precisely, 
$$
 K(Y) = \frac{\vol(Y)}{\vol(\CP^n)} \, \frac{\b_n}{\ell(\b_n)} .
$$
For seeing this, recall \cref {def:KZ_Y} and note that $Y$ is cohomogeneous,
since $\UU(n)$ acts transitively on the unit sphere of~$\C^n$.

The zonoid associated with the \emph{complex} hyperplane $\CP^{n-1}$ of $\CP^{n}$ 
defines the degree two class 
$$
 \g_n := [\CP^{n-1}]_\EE \in R_n^2  .
$$ 
In \cref{th:basis-An} below will we show that $\b_n$ and $\g_n$ generate the algebra $R_n$.

In the next proposition, we show that the classes of complex submanifolds are multiples of $\g_n$.  
For this, it is helpful to introduce the Vitale zonoid 
\begin{equation}\label{eq:def-P_n}
 P_n := K(g(e_1\wedge \sqrt{-1} \, e_1)) \in \GZo^2(\C^n)^{\UU(n)} ,
\end{equation}
where $g\in \UU(n)$ is uniformly random.
(Here and in the sequel, we identify $\C^n=\R^{2n}$: 
all wedge products are over~$\R$.) We note that the zonoid~$P_n$ corresponds to the 
uniform probability distribution on the space of complex lines in $\C^n$, 
seen as a subspace of the Grassmannian $G(2, \C^n)$ 
of real planes in $\C^n=\R^{2n}$; see \cref{se:zon-meas-top}. ). 

For later use, we observe that the volume $\kappa_{2n}$ of the $2n$-dimensional unit ball \cref{volballsphere}
happens to be the volume of $\CP^n$ with respect to the Fubini-Study metric 
(see also \cite[(17.9)]{condition}): 
\begin{equation}\label{eq:K2m}
 \kappa_{2n} = \frac{\pi^n}{n!}  = \vol(\CP^n) .
\end{equation}

We next show that the classes associated with complex submanifolds of $\CP^n$ 
are powers of $\g_n$ and determine the length of these powers. 

\begin{proposition}\label{propo:subm_CPn}
\begin{enumerate}
\item We have 
$$\g_n = \frac{n}{\pi} \, [P_n]$$
and $\ell(\g_n^j) =  \pi^{-j} \frac{n!}{(n-j)!}$ for $0\le j \le n$.
    \item 
Every complex submanifold $Y\subseteq \CP^n$ is cohomogeneous.
If $Y$ has complex codimension~$j$, then the class of its associated zonoid equals 
\begin{equation}
     [Y]_\EE=\frac{\vol(Y)}{\vol(\CP^{n-j})} \,  \g_n^j.
\end{equation}
In particular,
$[\CP^{n-j}]_\EE = \g_n^j$.
\end{enumerate}
\end{proposition}

\begin{proof}
For the first part, we use that 
$\vol(\CP^{n-1})/ \vol(\CP^n) = n/\pi$ 
by~\eqref{eq:K2m}. This gives $\g_n = \frac{n}{\pi} \, [P_n]$. 
The length of $\g_n^j$ will be computed during the proof of the second part.

The cohomogeneity stated in the second item follows since $\UU(n)$ acts transitively on the 
complex Grassmannian $G^\C(d,\C^n)$. 
From~\eqref{formula_cohom_zonoid} 
and the fact that the random vectors $\xi_Y$ and 
$\xi_{\CP^{n-d}}$ have the same law, it follows that 
$$
 [Y]_\EE=\frac{\vol(Y)}{\vol(\CP^{n-j})} \,  [\CP^{n-j}]_\EE.
$$
For determining the length of $\g_n^j$, we use that by~\cref{thm:PSC} 
$$
 \vol(\CP^{n-j}) = \EE\, \vol(g_1\CP^{n-1}\cap\cdots\cap g_j\CP^{n-1})
    = \vol(\CP^{n})\; \ell\big( [\CP^{n-1}]_\EE\big)^{j} 
    = \vol(\CP^{n})\, \ell(\g_n^j) .
$$
Together with~\eqref{eq:K2m}, this implies 
$$
 \ell(\g_n^j) = \frac{\vol(\CP^{n-j})}{\vol(\CP^n)} = \pi^{-j} \frac{n!}{(n-j)!} .
$$ 
For proving that $[\CP^{n-j}]_\EE = \g_n^j $, 
we observe that both sides  of this equation are represented by zonoids, which are 
invariant under the action of $\SO(2n)$.
Therefore, both are scalar multiples of the unit ball in $\Lambda^{2d} \C^{n}$
and it suffices to prove that both zonoids have the same length. 
However, we have 
$\ell([\CP^{n-j}]_\EE) = \frac{\vol(\CP^{n-j})}{\vol(\CP^n)}$  
by~\cref{def:KZ_Y} and just showed that this equals $\ell(\g_n^j)$. 
\end{proof}

Following \cref{prop: pullback of pir}, the inclusions 
$\iota_{n,m}\colon \CP^m \hookrightarrow \CP^n$ for $m\le n$ 
induce the graded algebra homomorphism 
$\iota^*\colon \HEo(\CP^n) \to \HEo(\CP^m)$. 
We show now that they are compatible with the $\beta_n,\g_n$. 

\begin{lemma}\label{le:bg-consistent}
We have $\iota_{n,m}^*(\b_n) = \b_m$ and $\iota_{n,m}^*(\g_n) = \g_m$.
\end{lemma}

\begin{proof}
The inclusion $\iota_{n,m}$ defines the inclusion of tangent spaces 
$\C^m \simeq T_\bo\CP^m \hookrightarrow T_\bo\CP^n \simeq \C^n$. Its dual map is the 
orthogonal projection $ \C^n \to \C^m$. 
The assertion $\iota_{n,m}^*(\b_n) = \b_m$ follows, 
since the orthogonal projection maps the unit ball $B_{2n}$ to the unit ball $B_{2m}$ of $\C^m$.

For the second assertion, we note that 
$\iota^{-1}_{n,m}(g\CP^{n-1}) = \CP^m \cap g\CP^{n-1} \simeq \CP^{m-1}$
for almost all $g\in \UU(n+1)$. Thus 
$[\iota^{-1}_{n,m}(g\CP^{n-1})]_\EE = [\CP^{m-1}]_\EE =\g_{m} $. 
\cref{th:pullback} implies 
$$
 \iota_{n,m}^*(\g_n) = \EE_{g\in \UU(n+1)}[\iota^{-1}_{n,m}(g\CP^{n-1})]_\EE = \g_{m}  ,
$$
which proves the second assertion.
\end{proof}

The following explicit computation of the length of the binomial
products $\g_n^j \b_n^{2(d-j)}$ will be of great importance. 
For this we define the positive quantities 
\begin{equation}\label{def:rho-n}
 \rho_n(j) := \pi^{-2j} \, \binom{2(n - j)}{n- j} \quad \mbox{for $0\le j \le n$} .
\end{equation}

\begin{proposition}\label{pro:X}
For $0\le j \le d \le n$ we have 
$$
 \ell(\g_n^j \b_n^{2(d-j)}) = \pi^d \, \frac{n! (n-d)!}{(2(n-d))!} \, \rho_n(j) .
$$
In particular, $$\ell(\g_n^j \b_n^{2(n-j)}) = \pi^n \rho_n(j).$$
\end{proposition}

\begin{proof} 
We first recall from~\eqref{eq:K2m} and \cref{propo:subm_CPn} that for $0\le j \le n$
$$
 \kappa_{2n} = \frac{\pi^n}{n!} ,\quad \ell(\g_n^j) =  \pi^{-j} \frac{n!}{(n-j)!} .
$$
Assume $0\le j \le d \le n$. Recall that $\b_n$ is the class of the unit ball $B_{2n}$ 
and from \cref{propo:subm_CPn} that $\g_n = (n/\pi) \, [P_n]$. 
Therefore, we can apply~\cref{lem:lengthwithballs} 
to obtain an equality between the lengths of~$P_n^{\wedge j}\in\GZ^{2j}(\R^{2n})$ 
and $P_n^{\wedge j}B_{2n}^{\wedge (2(d-j))}$. Namely,
\begin{align}
 \ell\big(\g_n^j \b_n^{2(d-j)}\big) 
 & = \frac{(2(n-j))! \, \kappa_{2(n - j)}}{(2(n-d))! \, \kappa_{2n -2d}} \, \ell(\g_n^j) \\
 & = \frac{(2(n-j))!}{(2(n-d))!} \, \frac{\pi^{n-j}}{(n-j!} \, \frac{(n-d)!}{\pi^{n-d}} \,\, \pi^{-j} \, \frac{n!}{(n-j)!} \\
 & = \pi^{d-2j} \binom{2(n-j)}{n-j} \, \frac{n!\, (n-d)!}{(2(n-d))!} \\
 & = \pi^d\, \frac{n!\, (n-d)!}{(2(n-d))!} \, \rho_n(j) ,
\end{align}
which proves the assertion.
\end{proof}

\subsection{The structure of the algebra} 
The main goal of this section is to prove \cref{th:basis-An}. 
For the proof we need the following two auxiliary results. 

We first identify the sequence $\binom{2n}{n}$ of central binomial coefficients 
as the sequence of moments of a probability distribution on $\R$. 

\begin{lemma}\label{le:cbc}
The probability density 
$p(x) := \pi^{-1}\, (x(4-x))^{-\tfrac12}$ on the interval $(0,4)$ has the moments 
$\int_0^4 x^n p(x)\, dx = \binom{2n}{n}$ for all  $n\in\N$.
\end{lemma}

\begin{proof}
The Beta distribution $B(1/2,1/2)$ has the density 
$f(\xi) = \pi^{-1} (\xi(1-\xi))^{-1/2}$ on $(0,1)$. 
Its moments satisfy
$$
 M_n := \int_0^1 \xi^n f(\xi) \, d\xi = \prod_{r=0}^{n-1} \frac{\tfrac12+r}{r+1} .
$$
This can be rewritten as 
$$
 M_n = \frac{1}{n!} \prod_{r=0}^{n-1} \frac{2r +1}{2} 
  = \frac{1}{n! \, 2^n}\, 1 \cdot 3 \cdot 5 \cdots (2n-1) 
  = \frac{1}{n! \, 2^n}\, \frac{(2n-1)!}{(n-1)!2^{n-1}} = \frac{1}{4^n} \binom{2n}{n} .
$$
Finally, writing $x=4y$, we have $4p(x) = f(y)$ and hence 
$$
 \int_0^4 x^n p(x) \, \mathrm dx = 4^n \int_0^1 y^n f(y) \, \mathrm dy= 4^n M_n = \binom{2n}{n} ,
$$
which completes the proof.
\end{proof}

The following insight will have far reaching consequences. 
The idea of proof originates from the Hamburger moment problem.

\begin{proposition}\label{pro:Y}
The Hankel matrix 
$\big[ \rho_n(j_1 +j_2) \big]_{0\le j_1,j_2 \le \lfloor \tfrac{n}{2}\rfloor}$
is positive definite. 
\end{proposition}

\begin{proof} 
Suppose first that $n=2m$ is even. Take a nonzero $(z_0,\ldots,z_m)\in\R^{m+1}$. 
By the definition~\eqref{def:rho-n} of $\rho_n$ we have
$$
 J := \sum_{j_1,j_2=0}^m \rho_n(j_1+ j_2) \, z_{j_1} z_{j_2} 
    =  \sum_{j_1,j_2=0}^m \binom{2(n-j_1-j_2)}{n-j_1-j_2} \, \pi^{-2j_1} z_{j_1}\, \pi^{-2j_2}  z_{j_2} .
$$
We make the transformation of variables 
$k_1 := m- j_1$, $k_2 := m- j_2$ and define 
$\zeta_k := \pi^{-2(m-k)} z_{m-k}$. 
Then the above sum equals 
$$
 J = \sum_{k_1,k_2=0}^m \binom{2(k_1+k_2)}{k_1+k_2}  \zeta_{k_1} \zeta_{k_2} .
$$
Now we use \cref{le:cbc} to express this as 
\begin{align}
 J &= \sum_{k_1,k_2=0}^m \zeta_{k_1} \zeta_{k_2} \int_0^4 x^{k_1+k_2} p(x) \, \mathrm dx 
 = \int_0^4 \Big(\sum_{k_1=0}^m \zeta_{k_1} x^{k_1}\Big) \Big(\sum_{k_2=0}^m  \zeta_{k_2} x^{k_2}\Big) \, p(x)\, \mathrm dx \\
   &= \int_0^4 \Big(\sum_{k=0}^m \zeta_{k}x^k\Big)^2 p(x) \, \mathrm dx > 0.
\end{align}
If $n$ is odd, one argues in an analogous way.
\end{proof}

Based on these two results and~\eqref{eq:A-dim-F}, 
we can now easily describe a fundamental result.
While this result is known in the framework of valuations,
the proof in our framework is simpler than the previous ones and 
has a new twist; see the discussion in~\cref{re:PDEF}. 
Out proof is self-contained, except that it relies on 
the upper bound of \eqref{eq:A-dim-F}. 
An alternative proof of the latter is outlined in~\cref{se:Kangle}.

\begin{theorem}\label{th:basis-An}
Let $0\le d \le n$.
\begin{enumerate}
\item The classes $\g_n^j \b_n^{d-2j}$, 
for $0\le j \le \min\big\{\lfloor\frac{d}{2}\rfloor, \lfloor\frac{2n-d}{2}\rfloor\big\}$, 
form a basis of~$R_n^d$.
\item 
The symmetric bilinear pairing
$\Phi: R_n^d \ti R_n^d,\, (x,y) \mapsto \ell(\b_n^{2(n-d)}\, xy)$
is positive definite if $d\le n$.
\item 
The multiplication map
$R_n^d \to R_n^{2n -d},\, x \mapsto \b_n^{2(n-d)}\, x$
is a linear isomorphism if $d\le n$. 
\end{enumerate}
\end{theorem}

\begin{proof}
We prove the first assertion under the assumption $d\le n$ 
by induction on $d$. The start $d=0$ is clear. 
Let us denote 
$$
 \cB^d := \{ \g_n^j \b_n^{d-2j} \mid 0 \le j \le \lfloor d/2 \rfloor \} .
$$ 
Let $0\le d <n$ and assume the induction hypothesis stating that $\cB^d$ is a basis of $R_n^d$.
Consider the linear map 
$\varphi\colon R_n^d\to R_n^{d+1},\,  x\mapsto \b_n x$
given by the multiplication with $\b$. 
We observe that by definition of $\cB^d$, we have 
$$
 \cB^{d+1} = \left\{
\begin{array}{ll} 
 \varphi(\cB^d) & \text{if $d$ is even} \\
 \varphi(\cB^d) \cup \{ \g_n^{(d+1)/2}\}& \text{if $d$ is odd} 
\end{array}\right. .
$$
We are going to show that $\varphi$ is injective. 
With respect to the basis $\cB^d$ of $R_n^d$, the bilinear map 
$$
 \Pi\colon R_n^d \ti R_n^d \to \R, \,  (x,y) \mapsto \ell(\b_n^{2(n-d)} xy)
$$
has the matrix elements 
$m_{j_1,j_2} := \Pi(\g_n^{j_1} \b_n^{d-2j_1},\, \g_n^{j_2} \b_n^{d-2j_2})$. 
We have
$$ 
 m_{j_1,j_2}  = \ell\big( \b_n^{2(n-d)} \, \g_n^{j_1} \b_n^{d-2j_1} \cdot \g_n^{j_2} \b_n^{d-2j_2}\big) 
        = \ell\big( \g_n^{j_1+j_2} \b_n^{2n - 2j_1 -2 j_2} \big) 
        = \pi^n \, \rho_n(j_1 +j_2) ,
$$
where we used \cref{pro:X} for the last equality. 
By~\cref{pro:Y}, the principal submatrix 
$[m_{j_1,j_2}]_{0\le j_1,j_2 \le \lfloor \tfrac{d}{2}\rfloor}$ 
of the Hankel matrix 
$[m_{j_1,j_2}]_{0\le j_1,j_2 \le \lfloor \tfrac{n}{2}\rfloor}$ is positive definite, 
which already shows the second assertion for $d$. 
In particular, this matrix is invertible and the bilinear map $\Pi$ is nondegenerate. 
If $x\in R_n^d$ satisfies $\varphi(x)=0$, then $\b_n^{2(n-d)} x=0$, hence 
$\b_n^{2(n-d)} xy=0$ for all $y\in R_n^d$. Since $\Pi$ is nondegenerate, 
we infer that $x=0$, which proves that $\varphi$ is injective.
It follows that 
$\varphi(\cB^d) \subseteq R_n^{d+1}$ is linearly independent. 

If $d$ is even, we have $\cB^{d+1}= \varphi(\cB^d)$ and 
$$
 \dim \cB^{d+1} = \dim \cB^d = \tfrac{d}{2} + 1 
  = \lfloor \tfrac{d+1}{2} \rfloor + 1 = \dim R_n^{d+1} ,
$$ 
where we used~\eqref{eq:A-dim-F} 
for the last equality.
Hence $\cB^{d+1}$ is a basis of $R_n^{d+1}$. 

If $d=2m-1$ is odd, by an analogous argument, it suffices to show that 
$\g_n \not\in \Span(\cB^{d+1})$. By contradiction, assume that 
$$
 \sum_{j=0}^m c_j \g_n^j \b_n^{d+1-2j} = 0
$$ 
for some $c_0,\ldots,c_{m-1},c_m=-1$. 
Multiplying with $\g_n^k \b_n^{n - 2k - d - 1}$
and taking lengths yields
$$
 \sum_{j=0}^m c_j \ell\big( \g_n^{j+k} \b_n^{n-2j -2k}  \big) = 0, \quad 
\text{ for $0\le k \le m$} .
$$
Again, by \cref{pro:X} and \cref{pro:Y}, 
the coeffient matrix of this system of linear equations in $c$ is invertible, 
which leads to the contradiction $c_0= \ldots= c_m= 0$. 
This shows that indeed $\g_n \not\in \Span(\cB^{d+1})$ 
and completes the proof of the first statement for $d\le n$. 

We already showed the second assertion. For the third assertion, 
it suffices to prove that the linear map 
$R_n^d \to R_n^{2n -d},\, x \mapsto \b_n^{2(n-d)}\, x$
is injective, since  $\dim R_n^d = \dim R_n^{2n-d}$ 
by Hodge duality. 
We assume $d<n$ without loss of generality and 
essentially repeat the previous argument: 
let $x\in R_n^d$ such that~$\b_n^{2(n-d)}\, x= 0$. Then 
$\b_n^{2(n-d)} xy=0$ for all $y\in R_n^d$ and hence $x=0$ 
since $\Pi$ is nondegenerate, 

Finally, in order to complete the proof of the first assertion in the case $d'>n$,
we write $d' = 2n - d$ with $d <n$ and 
note that the stated basis of $R_n^d$ is obtained from the elements of $\cB^d$ 
by multiplication with $\b_n^{2(n-d)}$: indeed,
$d' -2j = 2(n -d ) + d-2j$. 
Assertion three completes the proof.
\end{proof}

\begin{remark}\label{re:PDEF}
The first assertion is reformulation of a result by Alesker~\cite{aleskerHLT} 
in our framework. 
The third property, which is called Hard Lefschetz property,
already appeared in considerably more general form in this paper by Alesker. 
The positive definiteness of the bilinear pairing was first proved by Bernig and Fu~\cite[Cor.~5.13]{bernig-fu:11},
thereby answering an open question in Fu~\cite[\S4.1]{FuUnval}. 
The strategy of their proof was to explicitly diagonalize the pairing 
with respect to the Lefschetz decomposition~see \cref{le:Lef-diag}. 
Our proof, which relies on~\cref{pro:Y}, appears to be considerably simpler.
We note that the positive definiteness can also be deduced from 
from Kotrbat\'{y}'s general result~\cite{kotrbaty:21}.
\end{remark} 

Let us draw some further conclusions. Recall that we have shown in~\cref{le:bg-consistent} 
that the homomorphisms $\iota_{n,m}^*\colon R_n \to R_m$, for $n\ge m$, satisfy 
$\iota_{n,m}^*(\b_n) = \b_m$ and $\iota_{n,m}^*(\g_n) = \g_m$.
Since $\b_n$ and $\g_n$ generate the algebra $R_n$, it follows that $\iota_{n,m}$, for $n\ge m$, is surjective. 
We now form the inverse limit $\varprojlim R_n$. 
Due to \cref{le:bg-consistent} we have generators $\b,\g\in \varprojlim R_n$ 
of degree one and two mapping to $\b_n$ and $\g_n$, respectively. 
By \cref{th:basis-An}, $\b_n$ and $\g_n$ have no homogenous polynomial relation in degree $d\le n$, 
which implies that $\b$ and $\g$ are algebraically independent over $\R$. Therefore, 
\begin{equation}\label{eq:invlim-HECPn}
\varprojlim \HEo(\CP^n) = \varprojlim R_n \simeq \R[\b,\g]
\end{equation}
is a polynomial algebra in the variables $\b$ and $\g$. 
Similarly, one shows for the real projective spaces that (see~\eqref{eq:HE SO}) 
\begin{equation}\label{eq:invlim-HESn}
\varprojlim \HEo(\proj^n) \simeq \R[\b] 
\end{equation}
is a polynomial algebra in the variable $\b$.

Next, we define for $0\le d\le n$ the \emph{primitive space} 
$$
 \cP_n^d := \{ x\in R_n^d \mid \b_n^{2(n-d)+1} x = 0 \}. 
$$ 
The primitive space $\cP_n^d$ is $\SO(2n)$--invariant. 

\begin{corollary}\label{cor:prim-space}
We have  $\cP_n^0=R_n^0$. 
For $1\le d\le n$ we have 
$R_n^d = \cP_n^d \oplus \b_n R_n^{d-1}$.
Moreover $\cP_n^d$ is one dimensional if $d$ is even 
and $\cP_n^d=0$ otherwise. 
\end{corollary}

\begin{proof}
The fact that $\cP_n^0=R_n^0$ is immediate. 
The map 
$R_n^{d-1}\to R_n^{2n-d +1},\, x\mapsto \b_n^{2(n-d)+2} x$
is an isomorphism by~\cref{th:basis-An}.
Let now $x\in R_n^{d-1}$ such that $\b_n x \in \cP_n^d$. Then 
$\b_n^{2(n-d) +2} x =0$ and hence $x=0$. 
Thus we get the direct sum 
$P_n^d \oplus \b_n R_n^{d-1}\subseteq R_n^d$.
We compare dimensions using \eqref{eq:A-dim-F} 
and see that equality holds. The assertion on $\dim \cP_n^d$ 
follows from this.
\end{proof}

\cref{cor:prim-space} gives the $\SO(2n)$--invariant 
\emph{Lefschetz decomposition} ($0\le d \le n$)
\begin{equation}\label{eq:LD}
 R_n^d = \cP_n^d \oplus \b_n \cP_n^{d-1} \oplus \ldots \oplus \b_n^d \cP_n^0 ,
\end{equation}
in which the contributions $\b_n^{d-i} \cP_n^{i}$  for odd~$i$ vanish. 

\begin{lemma}\label{le:Lef-diag}
The symmetric bilinear map 
$\Phi\colon R_n^d\times R_n^d \to \R, \, (x,y) \mapsto \ell( \b_n^{2(n-d)} xy)$ from \cref{th:basis-An}
diagonalizes with respect to the Lefschetz decomposition~\eqref{eq:LD}. 
That is,  
$\Phi(\b_n^i \cP_n^{d-i} \times \b_n^j \cP_n^{d-j}) =0$ for~$i\ne j$.
\end{lemma}

\begin{proof}
Assume $0\le i < j\le d$, $x\in \cP_n^{d-i}$ and $y\in \cP_n^{d-j}$. Then,
$$
 \Phi(\b_n^i x, \b_n^j y) = \b_n^{2(n-d)} \, \b_n^{i+j} \, xy = \b_n^{2(n-d) +i+j} \, xy 
  = \b_n^{2(n-d) +i+1} x \, \b_n^{j-1}  y = 0 ,
$$
since $x\in \cP_n^{d-i}$ means that $\b_n^{2n-2(d-i) +1} x = 0$.
\end{proof}

\begin{remark}\begin{enumerate}
    \item In the paper~\cite[Prop~5.5]{bernig-fu:11}, generators of the primitive spaces $\cP_n^d$
are determined and the coefficients of the diagonal pairing matrix are explicitly computed.
This is how the positive definiteness is shown there.
\item With some representation theory, 
one can show that the Lefschetz decomposition is \emph{orthogonal} with 
respect to any $\SO(2n)$--invariant inner product on $R_n^d$.
\end{enumerate}
\end{remark}

\cref{th:basis-An} also immediately implies that $\b_n$ and $\g_n$ 
generate $R_n$ as an algebra. (See \cref{se:genFU} for 
a description of the ideal of relations.)

\begin{corollary}\label{cor:bg-generate-algebra}
The classes $\b_n$ and $\g_n$ generate the algebra $R_n$. 
If $R(\b_n,\g_n)=0$ for a nonzero homogeneous polynomial $R(\b,\g)$ of 
degree $d$ (where $\deg\b=1$ and $\deg\g=2$), then $d>n$. 
\end{corollary}

From the above structural decription of $\HEo(\CP^n)$ along with the 
description of $\HEo(\proj^n)$ (see \cref{sec_HE_sphere}) 
we recognize the injective graded algebra homomorphism
$$
 \HEo(\proj^{2n}) \to \HEo(\CP^n),\, \b_n \mapsto \b_n . 
$$ 
The existence of this algebra homomorphisms is directly evident: 
it expresses the inclusion of algebras 
$$
 \CGZ_o(\R^{2n})^{\SO(2n)} \subseteq \CGZ_o(\C^n)^{\UU(n)} . 
$$
Moreover, we have a surjective graded algebra homomorphism
$$
 \HEo(\CP^{n}) \to \HEo(\proj^{2n}),\, \b_n \mapsto \b_n ,
$$
which results from the projection 
$$
 \CGZ_o(\C^n)^{\UU(n)} \to \CGZ_o(\R^{2n})^{\SO(2n)} 
$$
onto the algebra of $\SO(2n)$--invariants. 
In degree~$d$, this is the projection onto the right component in 
the Lefschetz decomposition~\eqref{eq:LD}.

%%%
\subsection{The ideal of relations}\label{se:genFU}

Consider the graded algebra homomorphism
$\varphi_n\colon \R[\b,\g] \to R_n$ sending $\b$ to $\b_n$ and $\g$ to $\g_n$. 
If we denote $I_n:= \ker \varphi_n$, 
we have the descending chain of ideals 
$$
 \R[\b,\g] \supseteq I_1 \supseteq I_2 \supseteq \ldots ,
$$ 
since  $\varphi_{n-1} = \iota_{n,n-1} \circ \varphi_{n}$
by \cref{le:bg-consistent}.
We next construct, for each $n\ge 1$, a homogeneous polynomial $F_n \in I_n$ of degree~$n+1$. 

Suppose first $n=2p-1$ with $p\ge 1$. 
\cref{th:basis-An} tells us that 
$R_n^{n-1}$ has the basis $\g_n^j \b_n^{n-1-2j}$ for $0\le j \le p-1$. 
With the isomorphism 
$R_n^{n-1}\to R_n^{n+1},\, x\mapsto \b^2 x$, 
coming from the Hard Lefschetz property (\cref{th:basis-An}), 
we deduce that  $R_n^{n+1}$ has the basis 
$\g_n^j \b_n^{n+1-2j}$ for $0\le j \le p-1$. 
If we expand~$\g_n^p$, which is a class of degree $n+1$, in this basis, we get 
$$
 \g_n^p = \sum_{j=0}^{p-1} c_{n,j} \g_n^j \b_n^{n+1-2j} \quad 
\text{in $R_n^{n+1}$} 
$$
(the real coefficients $c_{n,j}$ are uniquely determined). 
We define the homogeneous polynomial $F_n \in\R[\b,\g] $ of degree~$n+1$ by 
$$
 F_n := \g^p -\sum_{j=0}^{p-1} c_{n,j} \g^j \b^{n+1-2j}
$$
and note that $F_n(\b_n,\g_n) = 0$ by construction.

In the case $n=2p$ with $p\ge 1$ we proceed similarly:
again, $R_n^{n+1}$ has the basis 
$\g_n^j \b_n^{n+1-2j}$ for $0\le j \le p-1$. 
We now expand 
$$
 \g_n^p\b_n = \sum_{j=0}^{p-1} d_{n,j} \g_n^j \b_n^{n+1-2j} \quad 
\text{in $R_n^{n+1}$} ,
$$
with uniquely determined coefficients $d_{n,j}$ 
and define the homogeneous polynomial of degree~$n+1$ 
$$
 F_n := \g^p\b -\sum_{j=0}^{p-1} d_{n,j} \g^j \b_n^{n+1-2j} .
$$
Again, we have $F_n(\b_n,\g_n) = 0$ by construction. 
Applying the algebra homomorphism $R_n \to R_{n-1}$, which sends
$\b_n$ to $\b_{n-1}$ and $\g_n$ to $\g_{n-1}$
(see \cref{le:bg-consistent}), we obtain 
$F_n(\b_{n-1},\g_{n-1}) = 0$.
So we have $F_n,F_{n+1} \in I_n$ for all $n\ge 1$. 

\begin{example}\label{ex:smallF}
The coefficients $c_{n,j}, d_{n,j}$ involved in the definition of 
the polynomials $F_n$ can be in principle computed by 
solving a linear system, as outlined at the end of this subsection. 
This way, one obtains 
$$
 F_1 = \g - \tfrac{1}{2} \pi^{-2}\b^2, \quad 
 F_2 = \g\b - \tfrac{1}{3}\pi^{-2}\b^3, \quad  
 F_3 = \g^2 - 2 \pi^{-2}\g\b^2 + \tfrac{1}{2}\pi^{-4}\b^4 .
$$
\end{example}

We next show that the ideal $I_n$ is generated by the polynomials $F_n,F_{n+1}$. 
For this, we use the same elegant argument as in Fu~\cite{FuUnval}, 
that we include for the sake of completeness. 

We first observe that \eqref{eq:A-dim-F} can be concisely expressed in terms 
of the Hilbert-Poincar\'e series of the graded algebra $R_n$ as follows:
\begin{equation}\label{eq:PS-An}
 \sum_{d\in\N} \dim(R_n^d)\,  T^d = \frac{(1-T^{n+1})(1-T^{n+2})}{(1-T)(1-T^2)} .
\end{equation}

The following is a special case of a well-known fact in commutative algebra.

\begin{lemma}\label{le:ps2-d}
Let $F,G\in\R[\b,\g]$ be coprime homogeneous polynomials of degrees $D$ and~$D+1$, 
respectively, where $\deg\b=1, \deg \g=2$. Then the graded algebra 
$B:= \R[\b,\g]/(F,G)$ satisfies
$$
 \sum_{d\in\N} \dim(B^d)\,  T^d = \frac{(1-T^{D+1})(1-T^{D+2})}{(1-T)(1-T^2)} .
$$ 
This is a polynomial in $T$ of degree~$2D$. 
\end{lemma}

\begin{proof}
If $F$ and $G$ are coprime, they form a regular sequence. Now 
apply~\cite[Thm.~11.]{AM:16}. 
\end{proof}

\begin{proposition}\label{pro:gen-alg}
The ideal $I_n$ of relations of the generators $\b_n,\g_n$ of the algebra $R_n$ 
is generated by the polynomials $F_n$ and $F_{n+1}$. 
Hence $R_n \simeq \R[\b,\g]/(F_n,F_{n+1})$.
\end{proposition}

\begin{proof}
We first argue that it is sufficient to show that $F_n$ and $F_{n+1}$ 
are coprime. Indeed, the inclusion $(F_n, F_{n+1}) \subseteq I_n$ of 
homogeneous ideals induces the surjective graded algebra homomorphism
$$
 \psi\colon B:= \R[\b,\g]/(F_n,F_{n+1}) \longrightarrow \R[\b,\g]/I_n = R_n .
$$ 
If $F_n$ and $F_{n+1}$ are coprime, then~\cref{le:ps2-d} applies.
Hence, using~\eqref{eq:PS-An}, the graded algebras $R_n$ and $B$ 
have the same Hilbert-Poincar\'e series. Therefore, the surjective 
linear maps $B^d\to R_n^d$ induced by $\psi$ on degree~$d$ parts 
are isomorphisms for all $d\in\N$. 
This implies that $\psi$ is an isomorphism 
and hence~$I_n=(F_n,F_{n+1})$ as claimed.

To finish the proof, we assume by contradiction that $F_n$ and $F_{n+1}$ 
have a greatest common divisor~$H_n$ of degree $1\le k \le n$, say 
$F_n = \widetilde{F}_n H$ and $F_{n+1} = \widetilde{F}_{n+1} H$ 
with (necessarily homogeneous) polynomials $\widetilde{F}_{n}$, $\widetilde{F}_{n+1}$ 
of degrees $n+1-k$ and  $n+2-k$, respectively. 
Note that $(F_n,F_{n+1})\subseteq (\widetilde{F}_{n},\widetilde{F}_{n+1})$. 
We have~$H\not\in I_n$, since $I_n$ does not contain a nonzero polynomial of 
degree~$k$ by \cref{th:basis-An}. In particular, this implies~$H\not\in (F_{n}, F_{n+1})$. 

We now consider the surjective graded algebra homomorphism
$$
 B= \R[\b,\g]/(F_n,F_{n+1}) \longrightarrow \R[\b,\g]/(\widetilde{F}_{n},\widetilde{F}_{n+1}) =:C .
$$ 
By \cref{le:ps2-d}, the Hilbert-Poincar\'e series of $C$ is a polynomial 
of degree 
$\deg F_{n} + \deg F_{n+1} - 3 = 2n-2k$. 
Therefore, $C^{2n-k}=0$, as $2n -k > 2n -2k$.
This implies that $G\in (\widetilde{F}_{n},\widetilde{F}_{n+1})$ 
for any homogenous polynomial $G\in\R[\b,\g]$ of degree~$2n-k$.
Therefore, $HG \in (F_{n}, F_{n+1})$ and hence $HG \in I_n$
for all such~$G$. 
This shows that the bilinear pairing 
$
 R_n^k \times R_n^{2n-k} \to R_n^{2n},\, (x,y) \mapsto xy
$
is degenerate, contradicting \cref{th:basis-An}.
\end{proof}

We explain now how the multiplication table of $R_n$ with regard to the 
basis $\cB_n$ of~\cref{th:basis-An}
can be in principle computed
by solving linear systems, hereby relying on the lengths $\ell(\g_n^j \b_n^{2d -j})$,
as given by \cref{pro:X}.
We will illustrate this in a few small examples.
In fact, when doing so, one realizes that it is of advantage
to consider the rescaled algebra generators 
\begin{equation}\label{eq:def-t-s}
t_n := \pi^{-\tfrac{2}{3}} \, \b_n, \quad s_n := \pi^{\tfrac{2}{3}}\,  \g_n ,
\end{equation}
and to work with the basis $(s_n^j t_n^{d-2j})_{j\in J(n,d)}$ of $R_n^d$, 
where
$J(n,d) := \{ j\in\N \mid 0\le j \le \min\big\{\lfloor\frac{d}{2}\rfloor, \lfloor\frac{2n-d}{2}\rfloor\big\}$.
\cref{pro:X} then can be rephrased as ($0\le j \le d \le n$)
\begin{equation}\label{eq:pro-X-mod}
  \ell(s_n^j t_n^{2(d-j)}) = \pi^{-\tfrac{d}{3}} \, \frac{n! (n-d)!}{(2(n-d))!} \, \binom{2(n-j)}{n-j}  .
  \end{equation}
Now suppose $0\le d_1,d_2 \le 2n$ such that $n < d':=d_1+d_2 \le 2n$, say 
$d' = 2n - d$ with $0\le d<n$. Note that $J(n,d')=J(n,d)$. 
We fix $j_1 \in J(n,d_1)$, $j_2 \in J(n,d_2)$ and expand the product of basis elements 
in the basis in $R_n^{d'}$, 
$$
 s_n^{j_1} t_n^{d_1 - 2 j_1} \cdot s_n^{j_2} t_n^{d_2 - 2 j_2} 
  = \sum_{j\in J(n,d)} c_j\, s_n^j t_n^{d'-2j} .
$$
Multiplying with $s_n^k t_n^{d-2k}$ for $0\le k \le \lfloor d/2 \rfloor$ 
and taking lengths yields
$$
 \ell\big( s_n^{j_1+j_2 +k} t_n^{2n- 2 j_1 - 2 j_2 -2k} \big) = 
 \sum_{j\in J(n,d)} \ell\big(s_n^{j+k} t_n^{2n- 2j -2 k} \big) \, c_j
 \quad \text{for $0\le k \le \lfloor d/2 \rfloor$} .
$$
This is a square system of linear equations in the unkowns $c_j$. 
Plugging in the values from~\cref{eq:pro-X-mod}, 
one sees that the powers of $\pi$ cancel, so that we 
are left with a rational coefficent matrix. This is the 
reason for the transformation in~\eqref{eq:def-t-s}.
The resulting matrix is positive definite by \cref{pro:X} and \cref{pro:Y}. 
Hence the $c_j\in\Q$ are uniquely determined by this linear system.

\begin{example}
For $\CP^2$ one computes the two relations
\begin{equation}\label{eq:rel-n=2}
 s_2 t_2 = \tfrac{1}{3}\, t_2^3, \quad 
 s_2^2 = \tfrac{1}{6}\, t_2^4 ,
\end{equation}
in degree three and four. These relations
dermine the whole multiplication tensor of the algebra~$A_2$.
For $\CP^3$ one computes the three relations 
\begin{equation}\label{eq:rel-n=3}
 s_3^2 =  2\, s_3 t_3^2 - \tfrac{1}{2}\, t_3^4, \quad 
 s_3 t_3^3 = \tfrac{3}{10}\, t_3^5 , \quad 
 s_3^2 t_3 = \tfrac{1}{10}\, t_3^5 
\end{equation}
in the degrees four and five. 
These relations determine the multiplication tensor of the algebra~$A_3$.
\end{example}

\begin{remark}\label{re:fu-gens} 
Fu~\cite{FuUnval} provided more detailed information on the generators of the ideal $I_n$.
To compare this with our setting, we make the 
substitution in~\eqref{eq:def-t-s}
and transform the polynomials 
$F_n(\b,\g)$ in~\cref{ex:smallF} to the new variables $s,t$.
A computation similarly as above yields 
\begin{equation}
\begin{split}
 F_1(\b,\g) = \pi^{-\tfrac23} \big(s-\tfrac12 t^2\big) =  \pi^{-\tfrac23} f_2(s,t) , \quad 
 F_2(\b,\g) = \g\b - \tfrac{1}{3}\pi^{-2}\b^3 = st - \tfrac{1}{3} t^3  = -f_3(s,t), \\
 F_3(\b,\g) =  \g^2 - 2 \pi^{-2}\g\b^2 + \tfrac{1}{2}\pi^{-4}\b^4 
            = -2\pi^{-\tfrac43} \big(-\tfrac{1}{2} s^2 + st^2 - \tfrac{1}{4} t^4 \big) = -2\pi^{-\frac43} f_4(s,t).
\end{split}
\end{equation}
The polynomials $f_n(s,t)$ arising here are the generators in Fu~\cite{FuUnval}.
In general, $f_{n+1}(s,t)$ has rational coefficients and is a nonzero multiple of $F_n(\b,\g)$. 
Fu~\cite{FuUnval} derived a linear recursion for obtaining 
$f_{n+2}(s,t)$ from $f_{n+1}(s,t)$ and $f_{n}(s,t)$. Moreover, 
he gave an explicit formula for the $f_n(s,t)$.
The polynomials \eqref{eq:rel-n=2} and \eqref{eq:rel-n=3} 
can be derived from Fu's generators using 
$f_5(s,t) = s^2 t -st^3 + \tfrac15 t^5$. 
\end{remark}

%%%
\subsection{Probabilistic intersection numbers}\label{se:PIN}

For $0\le j \le n$ we define for $d\le 2(n-j)$ 
the linear functionals
\begin{equation}\label{eq:PIS-def}
\ell_j\colon R_n^d \to\R, \ \ell_j(\alpha) := \ell (\alpha \g_n^{j})  .
\end{equation}
We call them \emph{probabilistic intersection numbers}
(for intersections with complex subspaces) 
due to the following geometric interpretation.
To a submanifold $X\subseteq \CP^n$ of real codimension~$d$ 
we assign the class 
$\alpha := [X]_\EE \in \HEo^{d}(\CP^n) =R_n^d$. 
By the integral geometry result~\cref{thm:PSC} we have 
\begin{equation}\label{eq:PIS}
 \ell_j([X]_\EE) = \ell\big( [X]_\EE \, \g_n^{j} \big)  
   = \frac{1}{\vol(\CP^n)}\, \EE_{g\in \UU(n)} \vol(X\cap g \CP^{n-j}) .
\end{equation}
For instance, $\ell_0([X]_\EE)) = \frac{\vol(X)}{\vol(\CP^n)}$ 
and $\ell_1([X]_\EE)$ gives the normalized volume of the intersection 
of $X$ with a complex hyperplane in random position.

Suppose $X\subseteq\CP^n$ is a complex irreducible algebraic subvariety 
of complex codimension $d_\C:= 2d$. 
Then it is known that for almost all $g\in \UU(n+1)$
(see \cite[Section~5.C]{mumfordAG})
\begin{equation}\label{eq:degisvol}
    \vol(X \cap g\CP^{n-j}) = \deg(X)\, \vol(\CP^{n-d_\C-j}) ,
\end{equation}
where $j\le n -d_\C$. Taking expectations, \eqref{eq:PIS} implies that 
$$
 \frac{1}{\vol(\CP^{n-d_\C-j})}\, \ell_j([X]_\EE) = \frac{1}{\vol(\CP^{n})} \, \deg(X)  
$$
is independent of $j\le n-d_\C$. Thus the 
$\ell_0,\ldots,\ell_{n-d_\C}$ are proportional on complex 
submanifolds of codimension $d_\C$. However, this is not true for real 
submanifolds!

\begin{proposition}
We have 
$\dim\Span \{ \ell_j\mid 0 \le j \le \lfloor d/2 \rfloor \} 
   = \dim R_n^d = 1 + \lfloor d/2 \rfloor$
for $d \le n$.
\end{proposition}

\begin{proof}
Assume that $\sum_k c_k \ell_k(x) =0$ for all $x\in R_n^d$, where the sum is over 
$0\le k \le \lfloor d/2 \rfloor$. 
For $x= \g_n^j \b_n^{d-2j}$ we have 
$\ell_k(x) = \ell(\g_n^k x) = \ell(\g_n^{j+k} \b_n^{d-2j})$. 
As before, we get the linear system 
$$
 \sum_k \ell(\g_n^{j+k} \b_n^{d-2j}) c_k = 0 \quad \mbox{for $0\le j \le \lfloor d/2 \rfloor$} ;
$$ 
\cref{pro:X} and \cref{pro:Y} imply that $c_k=0$ for all $k$.
\end{proof}

In the special case $d=2$ we have $\Span\{\ell_0,\ell_1\} = (R_n^2)^*$.
In \cref{se:codim-2-submf} we will study this situation in more detail.
Set $r:=\lfloor d/2 \rfloor$ and $s:=\lceil d/2 \rceil$.  
It would be interesting to know which probabilistic intersections numbers 
$\ell_{k_0},\ldots, \ell_{k_r}$, for $0\le k_0 < k_1 < \ldots < k_{r} \le s$, form a basis of $R_n^d$.
This amounts to ask which maximal minors of the matrix 
$\big[\ell(\g_n^{j+k}\, \b_n^{d-2j})\big]_{0\le j\le r, 0\le k\le s}$
are invertible. 

Of course, one could also study probabilistic intersection numbers
for intersections with real subspaces or mixed versions thereof.

\subsection{Intersection of codimension two submanifolds}
\label{se:codim-2-submf}

We make here the general reasonings of~\cref{se:PIN} more concrete by 
studying the intersection of $n$~many real submanifolds of codimension two. 
The results of this section originally appeared in the PhD thesis of 
the fourth author~\cite[Chap.~3]{mathis-thesis:22}.

Let  $X\subseteq \CP^n$ be a real submanifold of codimension two 
(assume $n\ge 2$).  
By~\cref{th:basis-An}, the associated class $[X]_\EE\in R_n^2 =\HEo^2(\CP^n)$ 
is a linear combination of $\b_n^2$ and $\g_n=[\CP^n]_\EE$, say
\begin{equation}\label{eq:LK2}
 \alpha := [X]_\EE =x_\R\, \b_n^2 + x_\C \, \g_n .
\end{equation}
\cref{prop:xrxc} below shows 
that the real coefficients $x_\R$ and $x_\C$ can be computed 
in terms of the volume $\vol(X)$ and the \emph{expected degree} $d_X$, 
which we define as 
\begin{equation}\label{eq:def-d_X}
  d_X:=\EE\#(X\cap g\CP^1) = \vol(\CP^{n}) \, \ell([X]_\EE\, [\CP^1]_\EE) 
     = \vol(\CP^{n}) \, \ell( \alpha \, \g_n^{n-1}) .
\end{equation} 
The expectation is over uniformly random $g\in \UU(n+1)$ and 
the right equality is a consequence of~\cref{thm:PSC}.
It will be convenient to introduce
\begin{equation}\label{eq:def-Delta_X}
 \Delta_X := d_X-\frac{ \vol(X)}{ \vol(\CP^{n-1})}.
\end{equation}
Note that $d_X=\deg X$ and $\Delta_X=0$ if $X$ is a complex irreducible hypersurface 
of degree~$\deg X$; see~\eqref{eq:degisvol}. 

The goal of this section is to prove the following  formula on 
the expected number of intersection points of $n$ randomly shifted 
codimension two submanifolds of $\CP^n$ of the same Grassmann class.
The formula in~\cref{cor:selfintcodim2CPn} could also be obtained by means of a kinematic formula in $\CP^n$ 
as developed by Bernig, Fu and Solanes in \cite{BFGcpxspaceforms}.
(The formula was indeed verified later by Bernig using this method).

\begin{theorem}\label{cor:selfintcodim2CPn}
Let $X\subseteq \CP^n$ be a submanifold of real codimension two ($n\ge 2$).
Let $d_X$ and $\Delta_X$ be defined as in~\eqref{eq:def-d_X} and~\eqref{eq:def-Delta_X}, respectively. 
Then, if $g_1,\ldots,g_n\in \UU(n+1)$ are chosen independently and uniformly at random, we have 
\begin{equation}
  \EE \#(g_1 X\cap\cdots\cap g_n X) =\sum_{k=0}^{\left\lfloor\tfrac{n}{2}\right\rfloor}
      \binom{n}{2k}\binom{2k}{k}\left(\frac{n}{2(n-1)}\right)^{2k} \, d_X^{n-2k} \Delta_X^{2k}.
\end{equation}
\end{theorem}

\begin{remark}
If $X$ is a complex irreducible hypersurface of degree~$\deg X$, 
then $d_X=\deg X$ and $\Delta_X=0$ and the theorem expresses 
$ \EE\#(g_1 X\cap\cdots\cap g_n X)=(\deg X)^n$, which is true by Bézout's Theorem.
When $X$ is close to a complex irreducible hypersurface, 
the quantity $\Delta_X$ is small and we may interpret~\cref{cor:selfintcodim2CPn} 
as a perturbation of Bézout:
\begin{equation}
    \EE\#(g_1 X\cap\cdots\cap g_n X)=(d_X)^n+\frac{1}{4}\frac{n^3}{n-1} (d_X)^{n-2}(\Delta_X)^2+O((\Delta_X)^4).
\end{equation}
\end{remark}

The proof of \cref{cor:selfintcodim2CPn} relies on the following lemma.

\begin{lemma}\label{prop:xrxc}
If $X\subseteq \CP^n$ be a real submanifold of codimension two, then 
$[X]_\EE =x_\R\, \b_n^2 + x_\C \, \g_n$ with 
\begin{equation}\label{eq:xrxcsyst}
x_\R = -\frac{n}{2\pi^2(n-1)}\Delta_X , 
\quad 
 x_\C= d_X + \frac{n}{n-1}\Delta_X .
\end{equation}
\end{lemma}

\begin{proof}
Using the linearity of $\ell$, and 
by multiplying the second row with $\g_n^{n-1}$,
we obtain from~\eqref{eq:LK2}
\begin{equation}\label{eq:SysEq}
\left. 
\begin{array}{lllllll}
\ell(\alpha) &=& x_\R & \ell(\b_n^2) &+& x_\C& \ell(\g_n)   \\[0.3em]
\ell(\alpha \g_n^{n-1})  &=& 
     x_\R &\ell(\b_n^2 \g_n^{n-1}) &+& x_\C & \ell(\g_n^n).
\end{array}\right.
\end{equation}
We compute the coefficients of this system with~\cref{pro:X}, obtaining 
\begin{equation}
 \ell(\b_n^2) =  2\pi(2n-1) ,\quad 
 \ell(\g_n) = \frac{n}{\pi}, \quad 
 \ell( \b_n^2 \g_n^{n-1})  = \frac{2 n!}{\pi^{n-2}} , \quad 
 \ell(\g_n^{n}) =  \frac{n!}{\pi^{n}} .
\end{equation}
By~\eqref{eq:def-d_X}, the left hand sides are 
$$ 
\ell(\alpha)=\frac{\vol(X)}{\vol(\CP^n)} 
 = \frac{n!}{\pi^n} \, \vol(X), \quad  
  \ell(\alpha \g_n^{n-1}) = \frac{n!}{\pi^n} \, d_X .
$$
We obtain the stated formulas for $x_\R$ and~$x_\C$ by solving the linear system 
and using the definition of~$\Delta_X$ in~\eqref{eq:def-Delta_X}. 
\end{proof}

The proof~\cref{cor:selfintcodim2CPn} also requires the  technical lemma below. 
The combinatorial proof we present was proposed to us by Mateusz Michalek 
(another method is in \cite[Appendix~B]{mathis-thesis:22}).

\begin{lemma}\label{prop:computealphak}
The quantity 
\begin{equation}\label{eq:defalphak}
 f_k:=\sum_{j=0}^k\binom{k}{j}\binom{2j}{j}(-1)^j2^{k-j}.
\end{equation}
vanishes for odd $k$ and equals 
$f_k=\binom{k}{k/2}$ if $k$ is even. 
\end{lemma}

\begin{proof}
We rewrite 
$\binom{2j}{j} = \sum_{t=0}^j\tbinom{j}{t}\tbinom{j}{j-t}$ and $2^{k-j} = \sum_{c=0}^{k-j} \tbinom{k-j}{c}$. 
Then,
$$ 
 f_k = \sum_{j=0}^k\sum_{c=0}^{k-j} \sum_{t=0}^j  
    \binom{k}{j}\, \binom{k-j}{c}\, \binom{j}{t}\, \binom{j}{j-t}\, (-1)^j.
$$
This is the sum of $(-1)^j$ over all possibilities to choose a subset $J\subseteq [k]$ of size $j$, 
and $T\subseteq J$ of size $t$ and $T'\subseteq J$ of size $j-t$, 
and choosing $C\subseteq [k]\setminus J$. 
Let us write $B_1 := T\setminus T'$, $B_2 := T'\setminus T$, 
$A_1 := J\setminus (T\cup T')$ 
and $A_2 := T\cap T'$ (see \cref{fig_combinatorics}). 
Then, we must have $\vert A_1\vert = \vert A_2\vert$, 
since $\vert T\vert + \vert T'\vert = \vert J\vert $. 
So $j$ is congruent to $\vert B_1\vert + \vert B_2\vert$ modulo 2. 
Then, we can instead sum over the choices of $A_1,A_2,B_1,B_2,C$ and get 
$$
 f_k =  \sum_{}(-1)^{j} =\sum_{}(-1)^{\vert B_1\vert + \vert B_2\vert}  = : \psi(k), 
$$
where the sum is over all disjoint subsets $A_1,A_2,B_1,B_2,C\subseteq [k]$, 
such that $\vert A_1\vert = \vert A_2\vert$. 

Now we show that $\psi(k)$ is zero when $k$ is odd and $\tbinom{k}{k/2}$ when $k$ is even. 
Recall that for any set $S\neq \emptyset$ one has 
$\sum_{A\subseteq S} (-1)^{\vert A\vert} = 0$. 
Thus, if we first fix $A_1,A_2, B_1,C$ and sum over $B_2$, we get $0$ unless 
$A_1\cup A_2\cup B_1\cup C = [k]$. 
In this case, we have a partition of $[k]$, 
so all possibilities reduce to summing over choices of $A_1,A_2,B_1$. 
Summing over all possibilities of $B_1$ we again get $0$ unless $A_1\cup A_2 = [k]$. 
Therefore $\psi(k)$ is the number of possibilities to divide $[k]$ into two subsets of the same size. 
There are zero such possibilities if $k$ is odd, and $\tbinom{k}{k/2}$ if $k$ is even.
\end{proof}

\begin{figure}[h]
\begin{center}
\begin{tikzpicture}[dot/.style = {circle, inner sep=0pt, minimum size=1mm, fill,node contents={}}]
\def\A{(1,0) coordinate (a) circle (1cm)}
\def\B{(2,0) coordinate (b)  circle (1cm)}
\def\C{(-2.5,-0.5) coordinate (c)  circle (1.2cm)}
\def\j{(1.5,0) coordinate (c)  circle (1.6cm)}
\fill[violet!30] \j;
\fill[blue!30] \A;
\fill[teal!30] \B;
\fill[red!30] \C;
\draw[thick] \j;
\draw[thick] \A;
\draw[thick] \B;
\draw[thick] \C;
\draw[thick] (-4,-2) -- (4,-2) -- (4,2.25) -- (-4,2.25) -- (-4,-2) -- cycle;
\begin{scope}
    \clip \B;
    \fill[brown!30] \A;
\end{scope}
\draw (1.5,1.2) node{$A_1$};
\draw (1.5,0) node{$A_2$};
\draw (1.5,1.9) node{$J = A_1\cup A_2 \cup B_1\cup B_2$};
\draw \A node[left]{$B_1$};
\draw \B node[right]{$B_2$};
\draw \C node{$C$};
\end{tikzpicture}
\end{center}
\caption{\label{fig_combinatorics}\small The disjoints sets $A_1,A_2,B_1,B_2,C\subset [k]$ from the proof of \cref{prop:computealphak}. 
In the proof, these sets come from partition $J=A_1\cup A_2\cup B_1\cup B_2$ and $T= B_1\cup A_2$ and $T' = B_2\cup A_2$.}
\end{figure}
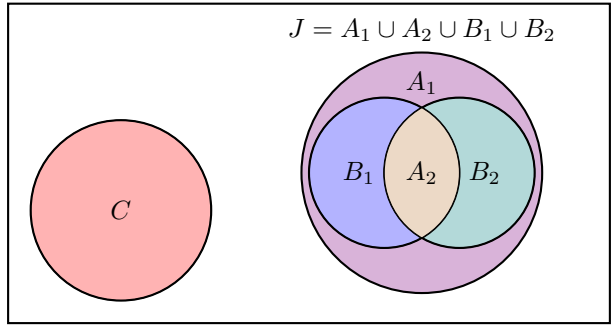

\begin{proof}[Proof of~\cref{cor:selfintcodim2CPn}]
By~\cref{thm:PSC}, the task reduces to compute the length of the $n$-th power of the class $\alpha:=[X]_\EE$, namely
\begin{equation}\label{cor:selfintcodim2CPn_eq1}
\EE \#(g_1 X\cap\cdots\cap g_n X) = \vol(\CP^{n})\, \ell(\alpha^{n}) 
     =  \frac{\pi^n}{n!}\, \ell(\alpha^{n}).
\end{equation} 
From \cref{prop:xrxc}, we have 
\begin{align}
 \alpha &= -\frac{n}{2\pi^2(n-1)}\Delta_X\, \b_n^2 + \Big(d_X + \frac{n}{n-1}\Delta_X\Big)\, \g_n\\
          & = d_X \, \g_n  +\frac{n}{n-1}\Delta_X\Big( \g_n -\frac{1}{2\pi^2} \b_n^2 \Big). 
\end{align}
Raising this to the $n$-th power and taking the length, we get 
\begin{align} \notag
   \ell(\alpha^{n}) 
    &=\sum_{k=0}^{n}\binom{n}{k} d_X^{n-k}\Delta_X^k \left(\frac{n}{n-1}\right)^k\ell
         \left(  \g_n^{n-k} \Big( \g_n -\frac{1}{2\pi^2} \b_n^2 \Big)^{k}\right)\\
    &=\sum_{k=0}^{n}\binom{n}{k}d_X^{n-k}\Delta_X^k 
    \left(\frac{n}{n-1}\right)^k\sum_{j=0}^k\binom{k}{j}\frac{(-1)^j}{2^j\pi^{2j}} \, 
    \ell \left( \g_n^{n-j} \b_n^{2j} \right) .  \label{eq:thiseqlxrxcdelta} 
\end{align}
Now we apply~\cref{pro:X} to get 
\begin{equation}
   \ell \left( \g_n^{n-j} \b_n^{2j} \right) = \pi^{2j-n} \, n! \, \binom{2j}{j} .
\end{equation}
Introducing this in~\eqref{eq:thiseqlxrxcdelta} above, and plugging the result into \eqref{cor:selfintcodim2CPn_eq1}, 
we can express this using the quantity $f_k$ defined in~\cref{prop:computealphak}:   
\begin{equation}
    \EE \#(g_1 X\cap\cdots\cap g_n X) 
      = \sum_{k=0}^n\binom{n}{k}d_X^{n-k}\Delta_X^k\left(\frac{n}{2(n-1)}\right)^kf_k .
\end{equation}
By~\cref{prop:computealphak}, 
$f_k=\binom{k}{k/2}$ if $k$ is even and $f_k=0$ otherwise.
This completes the proof.
\end{proof}

%%%
\subsection{Multiple K\"ahler angles and moments of their distribution}\label{se:Kangle}

This subsection complements the previous ones by providing more detailed information
on the construction of the probabilistic intersection algebra~$\HEo(\CP^n)$. 
Recall from~\eqref{eq:KMV} and \cref{re:GCOS} the ideal $\M(\C^n)$ of the algebra $\VGZ_o(\C^n)$, 
which can be interpreted as the kernel of the generalized cosine transform. 
By~\cref{def:HE}, the centered probabilistic intersection algebra~$\HEo(\CP^n)$
is obtained by factoring out from $\VGZ_o(\C^n)^{\UU(n)}$ the ideal $\M(\C^n)^{\UU(n)}$.
This is a significant complexity reduction, which even leads to finite dimensional vector spaces 
in the special case of $\CP^n$. 
We analyze here in more detail which information is retained, 
when passing from a zonoid (or measure) to the corresponding class in~$\HEo(\CP^n)$. 
In particular, we outline a direct proof of the fundamental upper bound in~\eqref{eq:A-dim-F}. 

For this purpose, let us recall the notion of 
\emph{multiple K\"ahler angles} from~\cite{tasaki:01}.
As usual, $e_1\ldots,e_n$ denotes the standard basis of $\C^n$. 
Let $0\le d\le n$ and put 
$k :=\lfloor d/2 \rfloor$. 
For $\theta =(\theta_1\ldots,\theta_k)\in [0,\pi/2]^k$, 
we consider the real $d$--dimensional subspace
\begin{equation}\label{eq:def-Vtheta}
 V_\theta := \sum_{i=1}^{k} \Span\{ e_{2i-1},\, \cos\theta_i \sqrt{-1}\, e_{2i-1} 
   + \, \sin\theta_i e_{2i}\} \quad (\ +\, \R e_{d}\ ) ,
\end{equation}
where the last term is only added when $d$ is odd.
Te spaces~$V_\theta$ 
form a system of representatives for the 
orbits of $G(d,\C^n)$ under the action of $\UU(n)$ \cite{tasaki:01}. Moreover, 
$V_\theta$ and $V_{\vartheta}$ are in the same orbit 
iff $\vartheta$ is a permutation of $\theta$.  
One calls 
$\theta$ the vector of \emph{multiple K\"ahler angles}
(or \emph{Tasaki angles}) 
of $V\in G(d,\C^n)$ if $V$ is in the same orbit as $V_\theta$. 
One defines the Tasaki angles of a real subspace $V\subseteq\C^n$ 
of dimension $n < d \le 2n$ as the Tasaki angles of $V^\perp$. 

\begin{remark}
Let $0\le j\le k$. Consider the  $\UU(n)$-orbit of $\C^j\oplus\R^{d-2j}$ in $G(d,\C^n)$,
which is denoted $E^{d,j}$ in \cite{bernig-fu:11}. 
This corresponds to the vectors $\theta$ in which $0$~occurs $j$ many times 
and $\pi/2$ occurs $d-2j$ many times. 
In particular, if $d=2k$, then the complex subspaces of $\C^n$ with complex dimension~$k$ 
form the single $\UU(n)$--orbit corresponding to $\theta=0$. 
Note that for $d=1$ we have $k=0$ and there are no Tasaki angles. 
This is consistent with the fact that $\UU(n)$ acts transitively on the 
space of real lines in $\C^n$. 
\end{remark}

It will be convenient to parametrize the vector $\theta$ of angles 
by the vector $x\in [0,1]^k$ of squares of their cosines: 
we define $$x_j := \cos^2\theta_j.$$ 
Abusing notation, we also denote $V_x:= V_\theta$. 
The map $x\mapsto V_x$ is continuous. 
Denoting the symmetric group by $S_k$, 
the situation can be briefly summarized by the bijection 
\begin{equation}\label{eq:tasaki-orbits}
 [0,1]^k/S_k \stackrel{\sim}{\longrightarrow} G(d,\C^n)/\UU(n),\quad S_k x \mapsto \UU(n) V_x , 
\end{equation}
which is a homeomorphism of compact orbit spaces.
Hereby, the $\UU(n)$-orbit of complex subspaces of real dimension $2k$ is encoded 
by the uniform vector $x=(1,\ldots,1)$. 
This bijection induces a linear isomorphism of spaces of measures:
\begin{equation}\label{eq:CPn-measures-iso}
 \cM([0,1]^k)^{S_k} \stackrel{\sim}{\longrightarrow} \cM(G(d,\C^n))^{\UU(n)} \simeq \ \VGZ_o^d(\C^n)^{\UU(n)} ,
 \quad \mu \mapsto Z_\mu ,
\end{equation}
where the right isomorphism is explained in~\cref{se:zon-meas-top}, 
see in particular~\cref{pro:X2measure}. 
Accordingly, if $\mu$ is a probability distribution on $[0,1]^k$, 
then $Z_\mu$ is defined as the Vitale zonoid of the random variable 
$hV_x \in \Lambda^d \C^n$,
where $h\in \UU(n)$ is uniformly distributed, 
$x\in [0,1]^k$ is distributed according to~$\mu$, 
and $h$ and $x$ are independent.

For $d=2$, the Dirac measure at the endpoint $x=1$ of the interval $[0,1]$ 
corresponds under the bijection~\eqref {eq:CPn-measures-iso} to the 
uniform measure supported on $G^\C(1,\C^n)\subseteq G(2,\C^n)$.  
The corresponding centered Grassmann zonoid was denoted 
$P_n \in \GZ_o^2(\C^n)^{\UU(n)}$ in~\eqref{eq:def-P_n}. 
More generally, the Dirac measure at $(1,\ldots,1)^k$ corresponds to 
the uniform measure supported on 
$G^\C(k,\C^n)\subseteq G(2k,\C^n)$. 
The corresponding Grassmann zonoid is a multiple of $P_n^{\wedge k}\in \GZ_o^{2k}(\C^n)^{\UU(n)}$ 
and has length one. The constant factor can be read off~\cref{propo:subm_CPn}.
We note that the measure on $[0,1]^k$ resulting from $\b_n^{k}$ is quite complicated, 
so that we will not describe it here.

\begin{remark}
The $S_k$--invariant measures on $[0,1]^k$ can be seen as ``mixtures'' 
of the discrete measures supported on the $S_k$--orbits of the vectors 
$(1,\ldots,1,0,\ldots,0)\in [0,1]^k$, where there are $j$ many ones. 
They encode the $\UU(n)$-orbit $E^{d,j}$ of $\C^j\oplus\R^{d-2j}$ in $G(d,\C^n)$.
Via the bijection~\eqref {eq:CPn-measures-iso}, they define 
$\UU(n)$--invariant Grassmann zonoids. 
Their classes form a vector space basis of $\HEo(\CP^n)$.
The  corresponding valuations defined via the Crofton map (\cref{def:Phi}) 
are called \emph{Hermite intrinsic volumes} in~\cite{bernig-fu:11}. 
\end{remark}

Recall  the symmetric bilinear pairing 
$\VGZ_o^d(\C^n) \times \VGZ_o^d(\C^n) \to \R, \, (K,L)\mapsto \langle K,L\rangle$
on $\VGZ_o^d(\C^n)$, which was introduced in~\eqref{eq:def-pairing}. 
The (left-) kernel $\M^d(\C^n)$ of this bilinear map forms the degree~$d$ part of the ideal $\M(\C^n)$
of the algebra $\VGZ_o(\C^n)$; see~\eqref{eq:KMV}.
According to~\cref{re:GCOS}, we can interpret $\M^d(\C^n)$ 
as the kernel of the generalized cosine transform. 
When restricting to $\UU(n)$--invariant measures and using~\eqref{eq:CPn-measures-iso}, 
this pairing translates to a pairing on $\cM([0,1]^k)^{S_k}$,
that we can describe in more explicit form.

\begin{lemma}\label{le:pairing-CPn}
Let $Z_\mu,Z_\nu \in \VGZ_o^d(\C^n)$ be encoded by the real 
measures $\mu,\nu\in \cM([0,1]^k)^{S_k}$ as in~\eqref{eq:CPn-measures-iso}.
Then 
$$
 \langle Z_\mu , Z_\nu \rangle = \int_{[0,1]^k\times [0,1]^k} K_n^d(x,y) \, \mathrm d(\mu\times\nu)(x,y) 
$$
with the continuous kernel $K_n^d\colon [0,1]^k\times [0,1]^k \to \R$
defined by 
$$
 K_n^d(x,y) := \EE_{h\in \UU(n)} \, \| h V_x \wedge Y_y\| .
$$
\end{lemma}

\begin{proof}
By scaling, we may assume without loss of generality
that $\mu,\nu$ are probability distributions.
By the above definition~\cref{eq:CPn-measures-iso},  
$Z_\mu$ is the Vitale zonoid of the random variable 
$hV_x \in \Lambda^d \C^n$ with independent uniformly distributed $h\in \UU(n)$ 
and $x\in [0,1]^k$ is distributed according to $\mu$. 
Analogously, $Z_\nu$~is the Vitale zonoid of the random variable 
$gV_y \in \Lambda^d \C^n$ with independent uniformly distributed $g\in \UU(n)$ 
and $y\in [0,1]^k$ is distributed according to $\nu$.
By the definition~\eqref{eq:def-pairing} of the pairing, we have 
$$
 \langle Z_\mu , Z_\nu \rangle = \EE_{h,g,x,y} \langle hV_x , g V_y \rangle 
  =  \EE_{x,y} K_n^d(x,y) ,
$$
since for fixed $x,y$
$$
 \EE_{h,g} \,\langle hV_x , g V_y \rangle = \EE_{h} \,\langle hV_x , V_y \rangle = K_n^d(x,y)
$$
by the definition of the kernel $K_n^d$.
\end{proof}

From the definition of $V_x$ it follows that that $K_n^d(x_1,\ldots,x_k,y_1,\ldots,y_k)$ 
is invariant under permutations of the $x_i$ and $y_j$. Moreover, $K_n^d(x,y)$ 
it is invariant under exchanging $x$ and $y$. 
Tasaki~\cite{tasaki:01} showed that for $d=2$, we have
$K_n^2(x_1,y_1) = \tfrac14 \big( (1+x_1)(1+y_1) + \tfrac{n}{n-1} (1-x_1)(1-y_1) \big)$,
which describes the random intersection of real surfaces in $\CP^n$.
More generally, Tasaki~\cite{tasaki:03} proved the remarkable fact that 
$K_n^d(x,y)$ is a multilinear {\em polynomial}, 
see also Bernig and Fu~\cite{bernig-fu:11}. 

\begin{proposition}\label{pro:tasaki}
For $0\le d\le n$ , there is a matrix $[C^{d,n}_{i,j}]$ of size  
$k +1  = 1+ \lfloor d/2 \rfloor$ 
such that the kernel $K_n^d(x,y)$ has the form
$$
 K_n^d(x,y) = \sum_{i,j=0}^k C^{d,n}_{i,j} e_i(x) e_j(y) ,
$$
where $e_i(x),e_j(y)$ denote the elementary symmetric polynomials  
of degree $i$ and $j$ in the $x$ and $y$ variables, respectively.
\end{proposition}

\cref{th:basis-An} actually implies that the matrix $[C^{d,n}_{i,j}]$ in \cref{pro:tasaki} is positive definite,
see the notes in~\cref{re:PDEF}.

Let us draw some interesting conclusions. 
We assign to a real measure $\mu\in\cM([0,1]^k)^{S_k}$ the following 
\emph{elementary symmetric moments}
\begin{equation}
 m_i(\mu) := \int_{[0,1]^k} e_a(x) \, \mathrm d\mu(x) \quad \mbox{$0\le i \le k$} .
\end{equation}
\cref{le:pairing-CPn} and \cref{pro:tasaki} imply that 
$$
 \langle Z_\mu , Z_\nu \rangle =\sum_{i,j=0}^k C^{d,n}_{i,j} \, m_i(\mu) \, m_j(\nu) .
$$ 
With this, we conclude for $\mu\in \cM([0,1]^k)^{S_k}$ 
\begin{equation}\label{eq:Zmuequiv}
\begin{aligned}
 Z_\mu \in \M^d(\C^n) \Longleftrightarrow 
  \forall \nu\  \langle Z_\mu , Z_\nu \rangle =0 
  &\Longleftrightarrow 
  \forall \nu:  \sum_{j=0}^k \Big(\sum_{i=0}^k C^{d,n}_{i,j} \, m_i(\mu) \Big) \, m_j(\nu) = 0\\
  &\Longleftrightarrow 
  \forall j: \sum_{i=0}^k C^{d,n}_{i,j} \, m_i(\mu) = 0 ,
\end{aligned}
\end{equation}
where for the last equivalence, we used that the linear map 
$$
 \cM([0,1]^k)^{S_k} \to \R^{k+1},\, \mu\mapsto (m_0(\mu),\ldots,m_k(\mu))
$$ 
is surjective. 
This implies that 
$$
 \dim \HEo(\CP^n) \ \le \ k +1 
$$
when we recall that 
$\HEo(\CP^n) = \VGZ_o(\C^n)^{\UU(n)}/\M(\C^n)^{\UU(n)}$,  see~\cref{def:HE}.
Thus we obtained an alternative proof of the upper bound in~\eqref{eq:A-dim-F}.

Since we know by \cref{th:basis-An} that the matrix $[C^{d,n}_{i,j}]$ is invertible, 
we can derive from~\eqref{eq:Zmuequiv} the following interesting result.
It tells us that what is preserved,
when replacing an $\UU(n)$--measure on $G(d,\C^n)$ by its class 
in $\HEo(\CP^n)$, are the elementary symmetric moments 
in the cosine squares of the distribution of multiple K\"ahler angles.

\begin{corollary}\label{cor:KCST-moments}
For $\mu,\nu\in \cM([0,1]^k)^{S_k}$, we have 
$$
 Z_\mu \equiv Z_\nu \bmod \M(\C^n) \Longleftrightarrow 
  m_0(\mu) = m_0(\nu) ,\ldots, m_k(\mu) = m_k(\nu) .
$$
\end{corollary}

%%%
\subsection{The cohomology algebra of $\CP^n$}\label{sec:cohomologyCPn}

Let us return to the beginning of this section, 
where we introduced the Riemannian homogeneous space $\CP^n$ by~\cref{eq:def-CPn}. 
We recall that the elements of the probabilistic intersection ring $\HE(\CP^n)$ are classes 
of Grassmann zonoids in the real exterior algebra $\Lambda(\C^n)$, which are invariant under 
the action of $\UU(n)$; see \cref{sec:probabilisticintersectionring}.
According to~\eqref{eq:H-decomp}, 
the probabilistic intersection ring $\HE(\CP^n)$ decomposes as 
$\HE(\CP^n) = \HEo(\CP^n) \oplus \HdR(\CP^n)$.
Here $\HEo(\CP^n)$ consists of the classes of centrally symmetric Grassmann zonoids.
Moreover, the centers of the $\UU(n)$--invariant Grassmann zonoids span the space 
$$ 
  \HdR(\CP^n) = \Lambda(\C^n)^{\UU(n)} .
$$
This is the de Rham cohomology algebra of $\CP^n$, 
since $\CP^n$ is a symmetric space; see \cref{sec:DR}. 

After having described $\HEo(\CP^n)$ in some detail, we end this section 
by giving a brief account on how to determine $\HdR(\CP^n)$ 
by applying tools from representation theory. 
This is well known, but we think it is worthwhile to present it here 
to complement the overall picture. 
Indeed, the structure of the cohomology algebra has been 
a guiding principle in our investigations of the Grassmannians
in~\cref{sec:psc}. 

Recall from~\cref{cor:bg-generate-algebra} that the algebra $\HEo(\CP^n)$ is 
generated by the classes of the ball $B_{2n}$ and the zonoid $P_n = K(g(e_1\wedge \sqrt{-1} \, e_1))$. Let us also denote 
$$
 \widetilde P_n:=  Z(g(e_1\wedge \sqrt{-1} \, e_1)) = P_n + c(\tilde P_n).
$$
Remarkably, the zonoid $\widetilde P_n$ 
is not centered, because $\tfrac{1}{2}\EE g(e_1\wedge \sqrt{-1} \, e_1)\neq 0$ 
and the latter is the center of $\widetilde P_n$ by~\cref{eq:c=12E}. 
As we will see, this essentially is the reason why 
the cohomology algebra $\HdR(\CP^n)$ is nontrivial.

Consider the standard hermitian inner product 
$H(z,w) := \sum_{j=1}^n z_j \overline{w}_j$ on $\C^n$. 
It is invariant under the action of $\UU(n)$ on $\C^n$. 
The real part of $H$ is the standard Euclidean inner product on $\C^n\simeq\R^{2n}$.
The imaginary part of $H$, 
$$
 \omega := -\tfrac12\, \mathrm{Im} \, H \in\big(\Lambda^2 \C^n \big)^{\UU(n)} ,
$$ 
is a nondegenerate, skew symmetric real bilinear form, called the
\emph{standard symplectic form} of $\C^n$.
(Here and in the following, 
we identify $\C^n$ with its real dual space using the standard Euclidean inner product on $\C^n$.)
It is clear that $\omega$ is invariant under the action of $\UU(n)$, since $H$ is so. 
Therefore, the power $\omega^{\wedge k}$ is a nonzero $\UU(n)$--invariant in 
$\Lambda^{2}\C^n$ for $0\le k\le n$.  
We will see in a moment that all $\UU(n)$--invariants in $\Lambda^{2k}\C^n$
are a multiple of $\omega^{\wedge k}$. 

\begin{lemma}\label{eq:def-omega}
We have 
$\omega = \sum_{j=1}^n e_j \wedge \sqrt{-1}\, e_j$
if $e_1,\ldots,e_n$ denotes the standard basis of $\C^n$. 
Moreover,  $\|\tfrac{1}{k!}\, \omega^{\wedge k}\|^2 = \binom{n}{k}$ 
for $0\le k \le n$.
Finally,
$\EE(g (e_1\wedge\sqrt{-1}e_1) = \tfrac{1}{n}\, \omega$, 
thus $\tfrac{1}{2n}\, \omega$ equals the center of the zonoid $\widetilde P_n$, and we have 
$$[\CP^{n-1}]_c = \frac{1}{\pi}\, \omega.$$
\end{lemma}

\begin{proof}
We consider the standard basis $e_1,f_1,\ldots,e_n,f_n$ of $\R^{2n}$, where $f_j := \sqrt{-1} \, e_j$.
By definition, 
$e_1\wedge f_1 = \frac12 (e_1 \otimes f_1 - f_1 \otimes e_1)$.
Writing $z=(x_1,y_1,\ldots,x_n,y_n)$, $w=(u_1,v_1,\ldots,u_n,v_n)$, we get 
$$
 \sum_{j=1}^n(e_j\wedge f_j) (z,w)
  = \tfrac12 \sum_{j=1}^n (x_jv_j - y_j u_j ) 
  =  -\tfrac12\, \sum_{j=1}^n \mathrm{Im} \big((x_j+ \sqrt{-1}\, y_j)\cdot (u_j-\sqrt{-1}\, v_j)\big) 
  = \omega (z,w), 
$$
which proves the first assertion. 

For the second assertion, 
$e_{j_1} \wedge f_{j_1}\wedge\ldots\wedge  e_{j_k} \wedge f_{j_k}$  are orthonormal for 
$1\le j_1 <\ldots < j_k \le n$, hence 
$$
 \omega^{\wedge k} = \sum_{j_1,\ldots,j_k=1}^n e_{k_1} \wedge f_{j_1}\wedge\ldots\wedge  e_{j_k} \wedge f_{j_k}
  = k! \, \sum_{1\le j_1 <\ldots < j_k \le n}  e_{j_1} \wedge f_{j_1}\wedge\ldots\wedge  e_{j_k} \wedge f_{j_k} .
$$ 
This implies $\|\omega^{\wedge k}\|^2 = (k!)^2 \binom{n}{k}$, showing the second assertion.   
For the third assertion, we calculate 
$$
 \EE (g (e_1\wedge f_1)) = \frac{1}{n} \sum_{j=1}^n \EE (g(e_j \wedge f_j)) 
  = \frac{1}{n} \EE \Big( g \sum_{j=1}^n e_j \wedge f_j \Big) 
  = \frac{1}{n} \EE (g\omega)  = \frac{1}{n} \omega ,
$$
where we used the first assertion and the $\UU(n)$--invariance of~$\omega$.
Hence $c(\widetilde P_n) = \tfrac{1}{2n} \omega$ by~\cref{eq:c=12E}. 
Finally, as for~\cref{propo:subm_CPn},  
$$
 [\CP^{n-1}]_c = 2 \, c(Z(\CP^{n-1})) =  2 \, \frac{\vol(\CP^{n-1})}{\vol(\CP^{n})}\, c(\widetilde P_n) 
  = 2\, \frac{n}{\pi} \, c(\widetilde P_n) = \frac{1}{\pi} \omega .
$$
where the first equality is due to~\eqref{eq:[Y]_c}.
\end{proof}

\begin{proposition}\label{pro:HCPn}
Let $0\le d\le 2n$. If $d$ is odd, then 
$\HdR^{d}(\CP^n) =0$. If $d=2k$ is even, then 
$$
 \HdR^{2k}(\CP^n) = \big(\Lambda^{2k}\C^n \big)^{\UU(n)} = \R\,\omega^{\wedge k} .
$$
\end{proposition}

\begin{proof}
We write $\C^n = E \oplus F$ with $E:= \R^n$, $F:= \sqrt{-1}\,\R^n$
and let $\UU(1)$ act on on $E\oplus F$ via
$t\cdot (e,f) := (te,t^{-1}f)$, for $t\in \UU(1)$. 
This leads to the weight space decomposition 
(called Hodge decomposition) 
$$
 \Lambda(E\oplus F) = \bigoplus_{k,\ell} \Lambda^k E \otimes \Lambda^\ell F ,
$$
where $\UU(1)$ acts on $\Lambda^k E \otimes \Lambda^\ell F$ 
by multiplication with $t^{k-\ell}$. 
Thus nonzero $\UU(n)$-invariants can only occur in the middle space, 
when $k=\ell$, in which case $d=2k$ is even. This proves the first assertion. 

For the second assertion, 
we recall from \cref{eq:def-omega} that $(\Lambda^{2k} \C^n)^{\UU(n)}$ 
contains $\omega^{\wedge k} \neq 0$. 
Hence, it suffices to prove that 
$(\Lambda^{2k} \C^n)^{\UU(n)}$ 
has dimension at most one.
Consider the simultaneous action of $\OO(n)$ on $E$ and $F$.  
Note that $E^* \simeq E \simeq \R^n$ as $\OO(n)$-modules, since 
$(g^{-1})^T =1$ for $g\in \OO(n)$. Therefore, 
$$
 \big(\Lambda^k E \otimes \Lambda^k F\big)^{\OO(n)} 
   \simeq \big( \End\, \Lambda^k \R^n \big)^{\OO(n)} 
   \simeq \End_{\OO(n)} \Lambda^k \R^n . 
$$
It is known that
$\Lambda^k \R^n$ is a simple $\OO(n)$-module;
e.g.; see~\cite[Satz~15.1]{boerner:67}.  
Therefore, the complexification 
$\C\otimes_\R \Lambda^k \R^n$ is a simple $\OO(n)$-module as well. 
Schur's lemma implies that every $\OO(n)$-equivariant complex 
linear endomorphism is a multiple of the identity. 
This implies 
$\dim_\R \Lambda^k \R^n \le 1$ and 
completes the proof. 
\end{proof}

%%%%%%%%
\bigskip
\section{Probabilistic Schubert calculus}\label{sec:psc}

Over the complex numbers, a Schubert problem is translated into 
an intersection theory problem 
solved in the cohomology ring of a complex Grassmannian.
Expanding ideas from \cite{PSC}, 
we will study \emph{probabilistic} intersection problems in the 
\emph{real} Grassmannian using the probabilistic intersection ring. 

We consider the Grassmannian $G(k,m)$ as a homogeneous space under 
the action of the orthogonal group as follows:
\begin{equation}\label{eq:G(kn)}
    G(k,m)=\OO(k+m)/\left(\OO(k)\times \OO(m)\right) .
\end{equation}
The tangent space of $G(k,m)$ at the distinguished point~$\bo$ 
(corresponding to $I_{k+m}\in\OO(k+m)$) is given by 
the quotient Lie algebra (see \cref{se:setting})
$$
 T_{\bo} G(k,m) = \oo(k+m)/\left(\oo(k)\times \oo(m)\right) \simeq \Big\{ 
   \begin{bmatrix} 0 & A \\ -A^T & 0\end{bmatrix}  \mid A \in \R^{k\times m}\Big\} 
   \simeq \R^{k}\otimes \R^{m} .
$$ 
The stabilizer group of $\bo$ equals
$H := \OO(k)\times \OO(m)$ and
the tensor action of $H$ on $\R^{k}\otimes \R^{m}$ 
corresponds to the induced action 
of~$H$ on the cotangent space $T_{\bo}^* G(k,m)$.
Thus the centered probabilistic intersection ring of $G(k,m)$ 
is given by (see \cref{def:HE}) 
\begin{equation}\label{eq:HEG(k,n)} 
 \HEo(G(k,m)) =   \CGZ_o(V)^H ,
 \quad \mbox{ where} \ V = \R^k\otimes\R^{m}, \
 \quad H = \OO(k) \times \OO(m) .
\end{equation}
Note that the orbits of the action of $H$ on the space of matrices $\R^{k\times m}\simeq V$
given by $(h_1,h_2)\,A := h_1 A h_2^T$
are determined by the singular values of~$A$. 
In particular, this action is not transitive on $\mathbb{P}(V)$ when $k,m>1$. 
\cref{prop:findimtransact} thus implies the following.

\begin{corollary}\label{cor:Grassinfinitedim}
$\HEo(G(k,m))$ is an infinite dimensional real vector space 
if $k,m>1$. 
\end{corollary}

The fact that $\HEo(G(k,m))$ is infinite dimensional leads to considerable 
technical complications. It clearly indicates that concepts and tools 
from functional analysis are required for its analysis.

We next recall the definition of Schubert varieties, which are subvarieties of $G(k,m)$
encoding Schubert conditions.
These subvarieties are associated with partitions contained 
in a rectangle $[k]\times [m]$. 
For explaining this, let us introduce some notation. 

By a \emph{partition} $\la=(\la_1,\ldots,\la_a)$ we understand a decreasing 
tuple of nonnegative integers.
A partition can be visualized by a \emph{Young diagram}: it consists of rows of boxes, 
stacked one upon the other, aligned on the left, 
and with the top row containing $\lambda_1$ boxes, 
the second row containing $\lambda_2$ boxes, and so on. 
More formally, the Young diagram of $\lambda$ consists of the pairs ($i,j)\in \N^2$ 
such that~$j \leq \lambda_i$. 
The \emph{size} $|\lambda|:=\sum_{i}\lambda_i$ of $\lambda$ is 
the number of boxes of the diagram.  
The \emph{length} $\ell(\lambda)$ is the number of rows of 
the diagram, which equals $a$ if $\lambda_a \ne 0$.
We shall identify a partition $\lambda$ with its Young diagram. 
It fits into the $k \times m$ rectangle if $\ell(\lambda)\leq k$ and $\lambda_1\leq m$: 
we briefly express this by $\lambda\subseteq [k]\times[m]$.
An \emph{outer corner} of $\lambda$ is a pair $(i,j)$ such that $j=\lambda_i$ 
and $\lambda_{i+1}<\lambda_i$,
see~\cref{fig_YD}.
We note that $\lambda$ has exactly one outer corner iff it is of rectangular shape.

\begin{figure}[!ht]
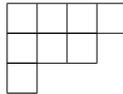

\begin{center}
$\yng(4,3,1)$
\end{center}
\caption{The Young diagram of the partition $(4,3,1)$ contained in a $3\times 4$ rectangle. 
The outer corners are $(1,4),(2,3),(3,1)$.\label{fig_YD}}
\end{figure}

The following definition is standard. 
To ease notation it will be convenient to set $n:=k+m$.
Throughout, we fix a \emph{complete flag} of $V=\R^n$, 
\begin{equation}\label{eq:flag}
 \mathcal{F}: \{0\}\subset V_1\subset V_2\subset \cdots\subset V_{n-1}\subset V_n=V ,\quad 
 \dim V_i = i ,
\end{equation}
and associate with a Young diagram $\lambda\subseteq [k]\times [m]$ 
the \emph{Schubert variety} defined as 
\begin{equation}\label{eq:schubertcycles} 
\Omega_\lambda := \Omega_\lambda(\mathcal{F}) := 
   \big\{L\in G(k,m)\,\mid\, \forall (i,j)\in\lambda: \dim(L\cap V_{m-j+i})\geq i \big\} .
\end{equation}
The rationale for this choice of notation will become clear in~\cref{le:EES}.  
We observe that only the outer corners of $\lambda$ matter in these conditions: 
the other are redundant. 
If ($p_1,q_1),\ldots,(p_r,q_r)$ denote the outer corners, %
ordered such that
$1\le p_1 < p_2 < \ldots < p_r \le k$ and $m \ge q_1 > q_2 > \ldots > q_r\ge 1$, then
\begin{equation}\label{eq:schubertcyclesshort} 
  \Omega_\lambda = \big\{L\in G(k,m) \,\mid\, 
   \forall i=1,\ldots,r:\dim (L\cap V_{m- q_i + p_i})\ge p_i \big\}.
\end{equation}
We call $\Omega_\lambda$ a \emph{simple Schubert variety} if the diagram~$\la$ 
is rectangle, that is, there is a exactly one outer corner $(p,q)$. 
The choice of the reference flag is not essential:
since two complete flags can be mapped to each other by 
an orthogonal transformation $g\in \OO(V)$, 
using a different flag just amounts to replacing 
$\Omega_\lambda$ by $g\Omega_\lambda$. 

We refer to~\cite{EH:16, fulton:97, milnor-stasheff} for motivation and explanations, 
as well as for proofs of the following facts:
the Schubert variety~$\Omega_\lambda$ is the closure 
of the~\emph{Schubert cell} $\Osm_\lambda$, 
obtained by replacing the inequalities in~\eqref{eq:schubertcells} 
(or in~\eqref{eq:schubertcycles}) by equalities.
\begin{equation}\label{eq:schubertcells} 
  \Osm_\lambda := \big\{L\in G(k,m) \,\mid\, 
   \forall i=1,\ldots,r:\dim (L\cap V_{m - q_i + p_i}) = p_i \big\}.
\end{equation}
Moreover, $\Osm_\lambda$ is the smooth part of $\Omega_\lambda$; 
it is a topological cell of codimension $|\lambda|$. 
In~\cref{thm_schubert_zonoid_simple} 
we will show that $\Osm_\lambda$ is cohomogeneous 
(see \cref{def:cohomogeneous}).
The collection of Schubert cells~$\Osm_\mu$, 
taken over all Young diagrams $\lambda\subseteq [k]\times [m]$, 
provides the Grassmannian $G(k,m)$ with the structure of a CW--complex structure.
Moreover, the collection of Schubert cells~$\Osm_\mu$, 
taken over those Young diagrams $\mu$ containing $\la$, 
stratifies $\Omega_\lambda$ in the sense of~\cref{def:stratified}. 

The definition~\eqref{eq:schubertcycles} carries over the complex setting: 
in an analogous way one defines the 
\emph{complex} Schubert varieties $\Omega^{\C}_\lambda$ 
in the \emph{complex} Grassmannian $G^\C(k,m)$. 
A \emph{classical Schubert problem} asks the following: 
given partitions $\lambda_1, \ldots, \lambda_s\subseteq [k]\times [m]$ 
with $\sum_i |\lambda_i|=km$, 
count the number of \emph{complex} points in the intersection
$g_1\Omega^\C_{\lambda_1}\cap\cdots\cap g_s\Omega^\C_{\lambda_s}$, 
where $g_1, \ldots, g_s\in {U}(n)$ are generic in the unitary group~$\UU(n)$. 

Motivated by this, we now define the probabilistic version of these problems 
over the \emph{real} numbers. 

\begin{definition}[Probabilistic Schubert problem]\label{def:PS-P}
Given partitions $\lambda_1, \ldots, \lambda_s\subseteq [k]\times [m]$ satisfying 
$\sum_i |\lambda_i|=km$, $n:=k+m$, 
we ask for the expectation
\begin{equation}\label{eq:intersection}
 \EE_{g_i\in \OO(n)}\#\,\left(g_1\Omega_{\lambda_1}\cap\cdots\cap g_s\Omega_{\lambda_s}\right),
\end{equation}
where the $g_i\in \OO(n)$ are independent uniformly distributed. 
\end{definition}

In~\cite{PSC}, the case of intersecting 
$\dim G(k,m)$ many randomly rotated hypersurfaces $\Omega_{\Box}$
was investigated and the resulting expectation was called the 
\emph{expected degree} $\mathrm{edeg}(k,n)$ of~$G(k,m)$. 

\begin{example}\label{example:fourlines}
For $k=m=2$,
computing the expected degree amounts to 
counting the expected number of lines in $\R\mathrm{P}^3$ intersecting four random lines 
$L_1, L_2, L_3, L_4\subseteq \R\mathrm{P}^3$. 
Indeed, the space of lines in $\RP^3$ can be identified with $G(2,2)$ 
and the set of lines intersecting a given line $L$ can be identified (up to a rotation) with 
the Schubert variety $\Omega_{\Box}$. 
The expected degree was analyzed in \cite{PSC} 
and shown to be $\mathrm{edeg}(2,4) \approx 1.726$. 
Note that the corresponding classical Schubert problem over $\C$ has the exact answer $2$, 
which expresses that the number of complex lines intersecting four complex lines in $\CP^3$ 
in generic position equals two.
\end{example}

Due to~\cref{thm:PSC}, we can formulate probabilistic Schubert problems in terms 
of the multiplication of the Grassmann classes 
of Schubert varieties in the probabilistic intersection ring $\HEo(G(k,m))$. 
Let us now explain this in more detail.
We just noted that $\Omega_{\lambda}$ is stratified set in a natural sense. 
According to \cref{def:KZ_Y}, we can associate with $\Omega_{\lambda}$ 
a Grassmann zonoid, which we call Schubert zonoid.
(This zonoid does not depend on the choice of the reference flag.)

\begin{definition}\label{def:schubertzonoids}
Let $\lambda\subseteq [k]\times [m]$ be a Young diagram with $|\lambda|=d$ boxes. 
The \emph{Schubert zonoid} for $\lambda$ 
$$
 K_\lambda := K(\Omega_\lambda)\in \GZo^{d}(\R^k\otimes \R^m)^{\OO(k)\times \OO(m)} .
$$ 
is defined as the Grassmann zonoid associated with 
the Schubert variety $\Omega_{\lambda}\subseteq G(k,m)$, 
see~\cref{def:KZ_Y}. 
The \emph{Schubert span} $V_\lambda$ for $\lambda$ 
is defined as the linear span of $K_\lambda$:  
$$
 V_\lambda := \Span (K_\lambda) \subseteq \Lambda^d(\R^k\otimes \R^m) .
$$
\end{definition}

The special case of the zonoid~$K_\Box$ associated with the 
Schubert variety $\Omega_\Box$ 
of codimension one was studied and named \emph{Segre zonoid} in~\cite{PSC}.  
This zonoid corresponds to the uniform distribution 
on the compact space of rank one matrices in $\R^{k\times m}$ having Frobenius norm one, 
compare~\cref{re:cohomog-distr}. 

Clearly, $K_\lambda$ and $V_\lambda$ are invariant under 
the action of $H=\OO(k)\times \OO(m)$. 
We will give an explicit description of Schubert zonoids in the next subsection.
\cref{se:span-S} is devoted to a thorough investigation of Schubert spans 
with tools of representation theory. 

Let us also name the classes corresponding to Schubert zonoids 
in the probabilistic intersection ring. 

\begin{definition}\label{def:Grassmann-Schubert-classes}
We call \emph{Grassmann Schubert class} for $\la$ 
the Grassmann class of the Schubert zonoid~$K_\la$
in the probabilistic intersection ring (see see~\cref{def:classY})
$$
 [\Omega_\la]_\EE := [K_\lambda] \in \HEo^{d}(G(k.m)) = \CGZ_o(V)^H .
$$ 
\end{definition}

The multiplication of Grassmann Schubert classes 
in the probabilistic intersection ring $\HEo(G(k,m))$
provides a conceptual answer to the probabilistic Schubert problems 
in the sense of \cref{def:PS-P}.
Indeed, \cref{thm:PSC} specializes as follows. 

\begin{corollary}\label{cor: int geo of schubert}
Let $\lambda_1, \ldots, \lambda_s\subseteq [k]\times [m]$ be Young diagrams such that 
$\sum_i |\lambda_i|=km$. Then 
$$
  \EE_{g_i\in \OO(n)}\#\,\left(g_1\Omega_{\lambda_1}\cap\cdots\cap g_s\Omega_{\lambda_s}\right)
   =\vol(G(k,m))\cdot\ell\left([K_{\lambda_1}]_\EE \cdot \ldots\cdot [K_{\lambda_s}]_\EE\right)
$$
where the $g_i\in \OO(n)$ are independent uniformly distributed.  
\end{corollary}

\begin{example}
If the partitions are all equal to~$\lambda$ and $s|\lambda| = k m$,
then \cref{coro:hypersurfaces}, 
together with \cite[Thm.~5.2]{BBLM}, 
implies that the expected number of intersection points can be expressed in terms of the 
$s$-th intrinsic volume $V_s (K_\lambda)$ 
of the Schubert zonoid $K_\lambda$: 
$$
 \EE_{g_i \in \OO(n)}
   \#\,\left( g_1\Omega_{\lambda}\cap\cdots\cap g_s\Omega_{\lambda}\right)
    = \vol(G(k,m))\cdot\ell\left([K_{\lambda}]_\EE^{s}\right)
    =\vol(G(k,m))\cdot s!\cdot V_s (K_\lambda).
$$
For instance, in the setting of ~\cref{example:fourlines}, 
the expectation equals $\vol(G(2,2)\cdot 4!\cdot\vol(K_\Box)$.
\end{example}

We denote by $\ast\lambda$ the \emph{dual} Young diagram of~$\la$
consisting of the boxes in the rectangle $[k]\times[m]$ 
that are not contained in $\lambda$. The dual diagram
$\ast\lambda$ has the complementary size $|\ast\lambda|= km-|\la|$.

\begin{lemma}\label{le:dualSchubert}
Suppose $|\lambda_1| + |\lambda_2| = km$. 
Then 
$\EE_{g_i \in \OO(n)} \#\,\left( g_1\Omega_{\lambda_1}\cap g_2\Omega_{\lambda_2}\right)$ 
is positive iff  $\lambda_1$ and $\lambda_2$ are dual. 
\end{lemma}

\begin{proof}
For the statement for complex Schubert varieties see, e.g., \cite[\S 9.4]{fulton:97}. 
The argument given in \cite[\S 9.4]{fulton:97} also 
proves the assertion over the reals.
\end{proof}

%%%%
\subsection{Schubert zonoids}\label{sec:schubert}

We provide here a specific description of Schubert zonoids.
Given a Young diagram $\lambda \subseteq [k]\times [m]$ with $|\lambda|=d$, 
we define the vector 
\begin{equation}\label{def:vla}
        v_\lambda := \bigwedge_{(i,j)\in\lambda} e_i \ot f_j \ \in\ \Lambda^d(\R^k\ot\R^{m}),
\end{equation}
where $e_1,\ldots, e_k$ is an orthonormal basis for $\R^k$ and 
$f_1,\ldots,f_{m}$ is an orthonormal basis for $\R^{m}$. 
Here we tacitly assume an ordering of the boxes of $\la$:
note that a change of this ordering only results in a change of sign. 
Using different orthonormal bases amounts to replacing $v_\lambda$
by an element in its $\OO(k)\times \OO(m)$-orbit. 
As we will see below, our constructions only depend on this orbit.

Our goal is the following specific description of Schubert zonoids and spans.

\begin{theorem}\label{thm_schubert_zonoid_simple}
\begin{enumerate}
\item The smooth part $\Osm_\lambda$ of the Schubert variety $\Omega_\lambda$ is cohomogeneous. 
\item The Schubert zonoid~$K_\lambda\subseteq \Lambda^d(\R^k\otimes \R^{m})$ 
equals the centered Vitale zonoid
$$
  K_\lambda=\frac{\vol(\Omega_{\lambda})}{\vol(G(k,m))}\; \cz (h\cdot v_\lambda),
$$
with the random variable 
$ h\cdot v_\lambda\in \Lambda^d(\R^k\otimes \R^{m})$ 
defined by a uniformly random $h\in \OO(k)\times \OO(m)$.
\item The Schubert span $V_\lambda$ is the linear span of the 
$\OO(k)\times \OO(m)$-orbit of $v_\la$.
\end{enumerate}
\end{theorem}

The proof will follow readily from a concrete description of the tangent spaces 
at smooth points of Schubert varieties. 
Let $L\in\Osm_\lambda$ be a smooth point of $\Omega_\lambda$. 
By \eqref{eq:schubertcells}, this means that 
$\dim(L\cap V_{m-q_i+p_i}) = p_i$ for all~$i$. 
\cref{le:EES} below implies that 
$\dim(L^\perp \cap V_{m-q_i+p_i})= q_i$ for all~$i$.
(This also explains the reason behind the choice of notation 
in~\cref{eq:schubertcycles}.)
Using the chosen ordering of the outer corners and the flag inclusions, 
one shows that there exist orthonormal bases $e_1,\ldots, e_k$ of $L$ 
and $f_1,\ldots, f_m$ of $L^\perp$ such that for all~$i$: 
\begin{equation}\label{emspan}
    L\cap V_{m-q_i+p_i}=\Span\{e_1,\ldots,e_{p_i}\} , \quad 
    L^\perp \cap V_{m-q_i+p_i}=\Span\{f_1,\ldots,f_{q_i}\} .
\end{equation}
We call such an orthonormal bases $(e_1,\ldots, e_k, f_1, \ldots, f_{m})$ of~$V$ 
\emph{adapted} to $L\in\Osm_\lambda$.

\begin{lemma}\label{le:EES}
If $L,L'$ are subspaces of $V$, then
$\dim(L^\perp \cap L') =  \dim(L\cap L') -\dim L' + \dim L^\perp$.
\end{lemma}

\begin{proof}
This follows from the exact sequence 
$0 \to L \cap L' \to L' \to L^\perp \to L^\perp\cap L' \to 0$.
\end{proof}

Recall that the tangent bundle to the Grassmannian $G(k,m)$ can be identified with 
$$
 TG(k,m)=\Hom(\tau, \tau^\perp),
$$ 
where 
$\tau :=\{(L, x)\in G(k,m)\times \R^n \mid x\in L\}$ 
denotes the tautological bundle
(see \cite{milnor-stasheff} for details). 
Thus at each point $L\in G(k,m)$ we have
\begin{equation}\label{tgtatA}
	T_LG(k,m)=\Hom(L,L^\perp).
\end{equation} 
Using this identification, we can characterize the tangent spaces 
to Schubert varieties at their smooth points
by the following result. It is inspired by \cite[Proposition 4.3]{FM}, 
which covers the case of simple Schubert varieties~$\Omega_{\lambda}$.

\begin{lemma}\label{lemma:ts}
For $L\in\Omega_\lambda^{\mathrm{sm}}$ we have 
\begin{equation}\label{tgtOmega}
 	T_L\Omega_\lambda^{\mathrm{sm}}=\set{\varphi : L\to L^\perp}{\forall\ i=1,\ldots,r: 
         \varphi\left(L\cap V_{m-q_i+p_i}\right)\subseteq L^\perp \cap V_{m-q_i+p_i}} ,
\end{equation}
where $(p_1,q_1),\ldots, (p_r,q_r)$ are the outer corners of $\lambda\subseteq [k]\times [m]$. 
\end{lemma}

We postpone the somewhat technical proof and first draw conclusions from it. 

\begin{corollary}\label{coro:normalschubert}
If $(e_1,\ldots, e_k,f_1, \ldots, f_{m})$ is an orthonormal bases 
adapted to $L\in\Osm_\lambda$, then 
$$
  T_L\Osm_\lambda = \Span \set{e_i\otimes f_j}{(i,j)\notin \lambda} ,   
 \quad
  N_L\Osm_\lambda = \Span \set{e_i\otimes f_j}{(i,j)\in \lambda}.
$$
\end{corollary}

\begin{proof}
We define for $(i,j) \in [k]\times [m]$ 
the linear map $\varphi_{ij}\in \Hom(L,L^\perp)$ by 
$\varphi(e_i)=f_j$, $\varphi(e_l)=0$ for all $l\neq i$.
The $\varphi_{ij}$ 
form an orthogonal basis of $\Hom(L,L^\perp)$. 
We note that 
$\varphi_{ij} = e_i \otimes f_j$ 
via the identification 
$\Hom(L,L^\perp) \simeq L \otimes L^\perp$ 
using the inner product. \cref{lemma:ts} expresses the desired characterization of the tangent spaces
$ T_L\Omega_\lambda^{\mathrm{sm}}
     =\Span \set{\varphi_{ij}}{ (i,j) \in [k]\times [m] \setminus \lambda}$
and the stated description of 
the normal space $N_L\Omega_\lambda^{\mathrm{sm}}$ follows from this.
\end{proof}

We now prove the stated description of Schubert zonoids.

\begin{proof}[Proof of \cref{thm_schubert_zonoid_simple}]
1. For proving the cohomogeneity, 
let $L,L'\in \Omega_\lambda^{\mathrm{sm}}$ and
$(e_1,\ldots, e_k, f_1, \ldots, f_{m})$,  
$(e'_1,\ldots, e'_k, f'_1, \ldots, f'_{m})$ 
be orthonormal bases adapted to $L$ and $L'$, respectively. 
We define $g\in \OO(k+m)$ by $g(e_i)=e'_i$ and $g(f_j)=f'_j$. 
\cref{coro:normalschubert} implies that $g$ sends 
$T_L\Osm_\lambda$ to $T_{L'}\Osm_\lambda$. 
This proves the first assertion.

2. We show now the characterization of the Schubert zonoid
$K_\lambda = K(\Omega_\lambda)$, see~\cref{def:schubertzonoids}. 
Suppose that $e_1,\ldots,e_k,f_1,\ldots,f_m$ is an orthonormal basis of $V=\R^{k+m}$
adapted to $L\in\Osm_\la$. We may assume that $L=\R^k$ without loss of generality.
Then  $H=\OO(k)\times \OO(m)$ is the stabilizer group at~$L$.
By \cref{coro:normalschubert}, the orthogonal complement 
$N_L\Osm_\lambda$ of the tangent space of $\Osm_\la$ at $L$
is spanned by the tensor products $e_i\otimes f_j$, where $(i,j)\in \lambda$.
Therefore, the wedge product $v_\la\in \Lambda^d (L\otimes L^\perp$ defined in~\cref{def:vla}
describes $N_L\Osm_\lambda$. Note that $\|v_\la\|=1$; see~\cref{se:EP-HS}. 
By~\cref{def:KZ_Y}, 
$K(\Omega_\lambda)= K(\Osm_\lambda)$
is the zonoid associated with the smooth stratum $\Osm_\lambda$, 
of which we just proved cohomogeneity. 
Hence we can use the characterization of associated zonoids 
in~\cref{random_vectors_in_cohom_case} using 
the fixed point $L\in\Osm_\lambda$.
We therefore consider the random vector
$$
\xi := \varepsilon h v_\la \in \Lambda^d (\R^k\otimes \R^m) ,
$$
with independent uniformly random $h\in H$ and uniformly random $\varepsilon\in\{+1,-1\}$.
By \cref{formula_cohom_zonoid}, we have 
$$
 K_\la = K(\Omega_\la) = K(\Osm_\la) 
   = \frac{\mathrm{vol}(\Osm_\la)}{\mathrm{vol}(G(k,m))} \, K(\xi) .
$$
Since $K(-\xi)=K(\xi)$, we can omit $\varepsilon$ and the stated formula 
for the Schubert zonoid $K_\la$ follows. 

3. The assertion on the span of $K_\la$ follows from \cref{pro:spanK(Y)}
\end{proof}

\begin{proof}[Proof of \cref{lemma:ts}] 
First suppose that $\la$ is a rectangle, i.e., 
there is only one outer corner $(p,q)$. 
In this case, the smooth part 
is given by
$\Omega_\lambda^{\mathrm{sm}} = \set{L\in G(k,m)}{\dim (L\cap V_{m-q+p})= p}$.
One checks that both sides of \eqref{tgtOmega} have the same dimension. 
Therefore, it is enough to prove the inclusion $\supseteq$ in~\eqref{tgtOmega}. 
For doing so, let $\varphi \in \Hom(L,L^\perp)$ and consider 
the curve $\gamma_\varphi: (-\eps, \eps)\to G(k,m)$ given by
\begin{equation}
    \gamma_\varphi (t)=\set{a+t\varphi(a)}{a\in L}\in G(k,m).
\end{equation}
Then $\gamma_\varphi$ satisfies $\gamma_\varphi(0)=L$ and $\dot{\gamma}_\varphi (0)=\varphi$.
Now suppose $\varphi\left(L\cap V_{m-q+p}\right)\subseteq L^\perp \cap V_{m-q+p}$. 
Then, for small $t\in (-\epsilon, \epsilon)$, 
we have $L\cap V_{m-q+p}\subseteq \gamma_\varphi(t)\cap V_{m-q+p}$.
Since dimension can locally only  decrease, 
we have $\gamma_\varphi(t)\in \Omega_\lambda^{\mathrm{sm}}$
for small enough~$t$. 
Thus $\varphi\in T_L\Omega_\lambda^{\mathrm{sm}}$,  
which shows the claimed inclusion. 

Now let $r\geq 2$. Then $\Omega_\lambda^{\mathrm{sm}}=\bigcap_{i=1}^r\Omega_{i}^{\mathrm{sm}}$,
where $\Omega_{i}^{\mathrm{sm}}:=\set{L\in G(k,m)}{\dim (L\cap V_{m-q_i+p_i})= p_i}$. 
This intersection is transversal, hence 
$T_L\Omega_\lambda^{\mathrm{sm}}=\bigcap_{i=1}^rT_L\Omega_{i}^{\mathrm{sm}}$ 
for all $L\in\Omega_\lambda^{\mathrm{sm}}$. 
The assertion follows now from the case $r=1$ already dealt with. 
\end{proof}

%%%
\subsection{Grassmann classes of Schubert varieties}

We focus here on the classes of Schubert zonoids, that is, 
the Grassmann classes $[\Omega_\lambda]_\EE$ of Schubert varieties 
in the centered probabilistic intersection algebra $\HEo(G(k,m))$.
We can show that they are linearly independent.

\begin{proposition}
The Grassmann classes $[\Omega_\lambda]_\EE$ 
of Schubert varieties in the probabilistic intersection algebra $\HEo(G(k,m))$
are linearly independent over $\R$, where $\la$ runs over the 
Young diagrams contained in $[k]\times[m]$.
However, if $k,m>1$, these classes do not generate $\HEo(G(k,m))$: 
neither as a real vector space algebra nor as real algebra. 
\end{proposition}

\begin{proof}
Suppose we have a linear combination
$\sum_{|\lambda|=d} a_\lambda [\Omega_\lambda]_\EE=0$ 
in $\HEo^d(G(k,m))$, with $a_\lambda\in\mathbb{R}$. 
Choose~$\mu$ with $|\mu|=d$ and 
multiply with $[\Omega_{\ast\mu}]_\EE$. 
By the duality statement in~\cref{le:dualSchubert} 
and ~\cref{cor: int geo of schubert}, 
we see that 
$[\Omega_{\lambda}]_\EE \cdot [\Omega_{\ast\mu}]_\EE$
is nonzero iff $\la=\mu$. 
Therefore $a_\mu=0$, proving the linear independence. 

The remaining statements are a consequence of the the fact that 
$\HEo(G(k,m))$ is an infinite dimensional real vector space if $k,m>1$; 
see~\cref{cor:Grassinfinitedim}. 
The subalgebra of $\HEo(G(k,m))$ generated by the (finitely many)
$[\Omega_\lambda]_\EE$ is a finite dimensional vector space.
Indeed, it is the span of the finitely many products 
$[\Omega_{\la_1}]_\EE\cdots [\Omega_{\la_r}]_\EE$, 
where $|\la_1| + \ldots + |\la_r| \le n$.
\end{proof}

One might hope that the Grassmann classes $[\Omega_\la]_\EE$
associated to Young diagrams span a subalgebra of $\HEo(G(k,m))$. 
However, the next result shows that this is not the case.
It also shows that Pieri's rule does not hold 
in the probabilistic intersection ring.

\begin{proposition}\label{prop:no real pieri for classes}
The product of Grassmann Schubert classes 
$[\Omega_\Box]_\EE\cdot [\Omega_\Box]_\EE$
in $\HEo(G(2,2))$
is \emph{not} a linear combination of the Grassmann Schubert classes
of the Young diagrams ${\tiny\yng(2)}$ and ${\tiny\Yvcentermath1\yng(1,1)}$.
\end{proposition}

\begin{proof}
For this proof it is convenient to adopt the projective point of view and consider $G(2,2)$ 
as the space of lines in the three dimensional (real) projective space. 
Let us describe the Schubert varieties associated to our partitions, 
following the definition in \eqref{eq:schubertcycles} 
and using that only the outer corners of~$\lambda$ matter. 
Accordingly, $\Omega_{(1)}$ consists of the lines intersecting a given line 
(the projectivization of $V_2$). 
Similary, $\Omega_{(2)}$ consists of all lines passing through a given point 
(the projectivization of $V_1$),  
$\Omega_{(1,1)}$ consists of all lines contained in a given plane (the projectivization of $V_3$), 
and finally, $\Omega_{(2,2)}$ consists of a single fixed line (the projectivization of $V_2$). 

Therefore, $g_1\Omega_{(1,1)}\cap g_2\Omega_{(1)}\cap g_3\Omega_{(1)}$ 
consists of all lines intersecting two given 
lines and contained in a given plane. 
This intersection almost surely consists of exactly one line.
By~\cref{cor: int geo of schubert}, this means 
$\alpha_{1,1} \cdot \alpha_{1}^2 =  \alpha_{2,2}$  
for the multiplication in $A:=\HEo(G(2,2))$, 
where we have put 
$ \alpha_{\la} := [\Omega_{\la}]_\EE \in A$ to simplify notation.
Note that the degree four part of $A$ equals $\R \alpha_{2,2}$ 
and $\vol(M) \, \ell(\alpha_{2,2}) = 1$, where we abbreviate $M:=G(2,2)$. 
Similarly, 
we show that 
$\alpha_{2} \cdot \alpha_{1}^2 =  \alpha_{2,2}$.
Moreover, $g_1\Omega_{(1,1)}\cap g_2\Omega_{(2)}$ 
consists of all lines through 
a given point and contained in a given plane. 
This intersection almost surely empty and hence 
$\alpha_{1,1} \cdot \alpha_{2} =  0$.
Finally, with a similar reasoning, we deduce that 
$\alpha_{2}^2 = \alpha_{1,1}^2 =  \alpha_{2,2}$.

By way of contradiction, 
suppose we have 
$\alpha_{1}^2 = c_1 \alpha_{1,1} + c_2 \alpha_{2}$
in $A$ with $c_1,c_2\in\mathbb{R}$. 
Multiplying with~$\alpha_{1,1}$, 
we deduce from the above relations that $c_1=1$.
Similarly, multiplying with $\alpha_{2}$ 
we deduce that $c_2=1$. 
So we obtain that 
$\alpha_{1}^2 = \alpha_{1,1} + \alpha_{2}$.
Taking the square and using the above relations implies that 
$\alpha_{1}^4 = (\alpha_{1,1} + \alpha_{2})^2 = 2\, \alpha_{2,2}$.
Therefore, we get 
$$
 \vol(M) \, \ell(\alpha_{1}^4) = 2 \, \vol(M) \, \ell(\alpha_{2,2}) = 2.
$$
On the other hand, 
$\vol(M)\, \ell(\alpha_{1}^4) < 2$
by \cite[Proposition~ 6.7]{PSC} 
(this quantity, called expected degree, is $\approx 1.72$).
This contradiction  completes the proof. 
\end{proof}

\begin{remark}\label{re:Grassm-alternative}
We can coorient the Schubert varieties $\Omega_\lambda$ 
by fixing an ordering of the boxes of $[k]\times [m]$;
then the zonoid $Z(\Omega_\la) \subseteq \Lambda^d (\R^k \otimes \R^m)^H$
and its center is defined, see \cref{def:classY}.
With some work, one can show that 
the centers $c(Z(\Omega_\la))$ form a basis of $\HDR^d(G(k,m))$, 
when $\la$ runs over all Young diagrams in $[k]\times [m]$
whose row lengths and column lengths are even.
Moreover, $\HDR^d(G(k,m))=0$ if $d$ is odd.
We are not aware of an explicit statement of this fact 
in the literature, even though descriptions of 
the cohomology rings of oriented an nonoriented real Grassmannians 
can be found in \cite{sadykov:17} and \cite[Cor.~2.3]{carlson:21}.
\end{remark}

Let $m_1 \le m_2$ and put $n_1:=k+m_1$, $n_2:=k+m_2$.  
Viewing $\R^{n_1}\subseteq \R^{n_2}$ 
gives the inclusion $f\colon G(k,m_1) \to G(k,m_2)$ of Grassmannians,
which is a morphism of homogeneous spaces in the sense of~\cref{def:morph-HS}. 
According to \cref{prop: pullback of pir}, the pullback of~$f$ 
is a graded algebra morphism
\begin{equation}\label{eq:fGrass-pb}
 f^*\colon \HEo(G(k,m_2)) \to \HEo(G(k,m_1)) .
\end{equation}
For a Young diagram $\la\subseteq [k] \times [m_1]$ we denote by 
$\Omega^{m_s}_\la$ its Schubert variety in $G(k,m_s)$ and write 
$[\Omega^{m_s}_\la]_\EE \in \HEo(G(k,m_i))$ for its Grassmann class.
We show now that the construction of Grassmann classes of Schubert varieties 
is compatible with the pullback $f^*$.

\begin{proposition}\label{pro:pullback-Schubert}
We have 
$f^* [\Omega^{m_2}_\la]_\EE = [\Omega^{m_1}_\la]_\EE $
for $\la\subseteq [k] \times [m_1]$ and $m_1\le m_2$.
\end{proposition}

\begin{proof}
Recall from \cref{def:schubertzonoids} that the zonoid associated to $\Omega^{m_i}_\la$ 
is denoted $K^{m_i}_\la = K(\Omega^{m_i}_\la)$.
It suffices to show that $f^* (K^{m_2}_\la) = K^{m_1}_\la$.

By~\cref{th:pullback}(1), we have 
$$ 
 f^*\big(K(\Omega^{m_2}_\la) \big) = \EE_{g_2\in G_2} K\big( G(k,m_1) \cap g_2 \Omega^{m_2}_\la \big) .
$$
Let $\Omega^{m_2}_\la$ be defined with respect to 
the complete flag~$(V_i)$ of $\R^{n_2}$. 
For almost all~$g_2$,  the subspaces 
$$
 U_j := \R^{n_1} \cap g_2 V_{m_2-m_1 + j}, \quad 0\le j \le n_1
$$ 
form a complete flag of $\R^{n_1}$. 
Let $(p_i,q_i)$ denote the outer corners of~$\la$. 
From the characterization~\eqref{eq:schubertcyclesshort}  of Schubert varieties, 
we get  
\begin{align}
    G(k_1,m_1) \cap g_2\Omega_\lambda^{m_2} 
        &= \{ L \in G(k_1,m_1) \mid \forall  i:\
            \dim(L\cap \R^{n_1} \cap g_2 V_{m_2-q_i+p_i}) \ge p_i \} \\ 
        &= \{ L \in G(k_1,m_1) \mid \forall  i:\ 
            \dim(L\cap U_{m_1-q_i+p_i}) \ge p_i,  \}  = \Omega_\lambda^{m_1} (U_j) .
\end{align}
Hence 
$K \big(G(k_1,m_1) \cap g_2\Omega_\lambda^{m_2} \big) 
 = K\big( \Omega_\lambda^{m_1} (U_j)\big) = K^{m_1}_\la $, 
which does not depend on $g_2$ 
(recall that Schubert zonoids do not depend on the chosen flag). 
Hence the expectation equals $K^{m_1}_\la$, as claimed. 
\end{proof}

Fix $k$ and consider $m\ge k$. 
The graded algebra morphisms \eqref{eq:fGrass-pb} form an inverse system, which define 
the inverse limit 
\begin{equation}\label{eq:inv-lim-Grass}
 \HEo(G(k,\infty)) := \varprojlim \HEo(G(k,m)) ,
\end{equation}
for $m\to\infty$, which is a graded algebra. Moreover, by \cref{pro:pullback-Schubert},
each partition $\la$ with $\ell(\la)\le k$ 
defines an element in $\HEo(G(k,\infty))$, 
which can be seen as the limit of the $[\Omega_\la]_\EE$, when $m\to\infty$.

\begin{remark}
For real projective spaces ($k=1$) one obtains a polynomial algebra in one variable,
$\HEo(G(1,\infty))\simeq \R[\beta]$, see \cref{sec_HE_sphere}. 
We investigated in \cref{se:two-generators} this construction in detail 
for complex projective spaces, for which the resulting algebra 
turns out to be a polynomial algebra in two variables. 
The general situation appears to be considerably more complicated,  
since we have to deal with infinite dimensional vector spaces. 
This investigation will be the focus of a separate paper. 
We remark that, unlike in the 
above mentioned special cases, the pullback morphisms in \eqref{eq:fGrass-pb}
do not appear to be surjective.
However, we conjecture that  they have a dense image. 
\end{remark}

%%%
\subsection{Decomposition of exterior power into Schubert spans}\label{se:span-S}

Consider the tensor action of $\Gl(\R^k)\ti \Gl(\R^m)$ on the space 
$V=\R^k \otimes \R^m$. 
We  think of $V$ as 
endowed with the standard euclidean inner product,
which is invariant under the action of the subgroup $\OO(k)\ti \OO(m)$.
Recall the induced inner product on the exterior power $\Lambda^d(V)$
from \cref{se:EP-HS}.  

The aim of this section is to prove that 
$\Lambda^d(V)$ is an orthogonal direct sum of Schubert spans $V_\la$.

\begin{theorem}\label{th:Lambda-decomp}
We have the orthogonal decomposition
$\Lambda^d(\R^k\ot\R^m) = \bigoplus_{\lambda} V_\lambda$, 
where the sum runs over all Young diagrams $\la\subseteq [k]\times [m]$ 
with $d$ boxes.
\end{theorem}

We first prove in~\cref{cor:C-Lambda-decomp} a more precise statement 
in the analogous setting over the complex numbers. 
It turns out that the decomposition of $\Lambda^d_\C(\C^k\ot\C^m)$ into 
the complex Schubert spans in fact is the decomposition of 
the complex exterior power into simple modules under the action 
of $\GL(\C^k)\times \GL(\C^m)$, or equivalently, the action of $\UU(k)\times \UU(m)$.

\cref{th:Lambda-decomp} will be derived from the result over~$\C$, 
relying on facts around the notion of \emph{complexification}
which we now briefly recall. 

The complexification of a real finite dimensional vector space~$W$
is the complex vector space defined as $W^\C := W \otimes_\R \C$.  
Consider the complex conjugation $\tau\colon W^\C\to W^\C$, 
which sends $w\otimes z$ to $w\otimes \bar{z}$. 
We can retrieve $W= (W^\C)^\tau$ as the subspace of vectors in $W^\C$
that are invariant under $\tau$.

The complexification of $V=\R^k\ot\R^m$ equals $V^\C = \C^k \otimes_\C \C^m$.
Note that the complexification of a real exterior power 
can be identified with a complex exterior power: we have 
$\Lambda^d(V)\otimes_\R \C \simeq \Lambda_\C^d(V^\C)$. 
On $V^\C =  \C^k \otimes_\C \C^m$ we have the tensor action 
of $\Gl(\C^k)\ti \Gl(\C^k)$ on $V^\C$, 
which extends to an action on $\Lambda_\C^d(V^\C)$. 

We define now the \emph{complex Schubert span} $V^\C_\la$ as the 
$\C$-span of $V_\la$. Note that $V^\C_\la = V \otimes_\R \C$. 

\begin{lemma}\label{le:VlaSpan}
\begin{enumerate}
\item $V_\lambda$ equals the $\R$--span of the orbit of $v_\la$ under the 
action of $\Gl(\R^k)\ti \Gl(\R^m)$.

\item $V^\C_\lambda$ equals the $\C$--span of the 
the orbit of $v_\la$ under the action of $\Gl(\C^k)\ti \Gl(\C^m)$.
\end{enumerate}
\end{lemma}

\begin{proof}
We choose a refinement of the componentwise partial order on $[k]\times [m]$ 
to a total order $\preceq$ with the property 
$k \le i, \ell \le j \Longrightarrow (k,\ell) \preceq (i,j)$. 
We will think of the wedge product defining $v_\lambda$ to be in the order provided by $\preceq$. 

1. By~\cref{thm_schubert_zonoid_simple}(3), 
$V_\lambda$ is the $\R$--span of the $\OO(k)\ti \OO(m)$--orbit of $v_\lambda$. 
We can factor any 
$\widetilde{A}\in \Gl(\R^k)$ and $\widetilde{B}\in \Gl(\R^m)$ as 
$\widetilde{A} = QA$, $\widetilde{B}=RB$, where $Q,R$ are orthogonal and 
$A,B$ are upper triangular (QR--decomposition). 
Using the basic principles of the wedge product, we obtain 
$$
 (A,B) v_\lambda = \bigwedge_{(i,j)\in\lambda} A e_i \ot B f_j = 
 \bigwedge_{(i,j)\in\lambda}  a_{ii} e_i \ot b_{jj} f_j ,
$$
hence 
\begin{equation}\label{eq:hwv}
 (A,B) v_\lambda = \prod_{(i,j)\in\lambda} a_{ii} \prod_{(i,j)\in\lambda} b_{jj}  
  \bigwedge_{(i,j)\in\lambda} e_i \ot f_j 
  = \prod_i a_{ii}^{\lambda_i} \prod_j b_{jj}^{\lambda'_j} \, v_\lambda , 
\end{equation}
where 
$\lambda_i := \#\{j \mid (i,j) \in \lambda \}$ 
is the number of boxes in the $i$th row of $\la$
and $\lambda'_j := \#\{i \mid (i,j) \in \lambda \}$ 
is the number of boxes in the $j$th row of $\la$, 
Thus $(A,B) v_\lambda$ is a scalar multiple of $v_\lambda$.
From this. we see that 
$(\widetilde{A},\widetilde{B}) v_\lambda  = (Q,R) (A,B) v_\lambda$
is a scalar multiple of $(Q,R) v_\lambda$, 
which proves the first assertion. 

2. For the second assertion, put $G:=\Gl(\R^k) \ti \Gl(\R^m)$ 
and $G^\C :=\Gl(\C^k) \ti \Gl(\C^m)$. 
By part one, we have $V_\lambda  = \Span_\R G v_\lambda$. 
Hence, $V_\lambda^\C$, which was defined as the $\C$--span of $V_\lambda$, 
is contained in $\Span_\C G^\C v_\lambda$.
For the reverse inclusion, we use that $G$ is Zariski dense in $G^\C$. 
Hence 
$G v_\lambda \subseteq \Span_\C G v_\lambda$ 
implies that 
$G^\C v_\lambda \subseteq \Span_\C G v_\lambda$,
therefore $V_\lambda^\C \subseteq \Span_\C G v_\lambda$. 
So we showed that~$V_\lambda^\C = \Span_\C G v_\lambda$. 
\end{proof}

For the next result we use the following basic fact from the representation theory 
of complex general linear groups. 
Assume that $G^\C :=\Gl(\C^k)\times \Gl(\C^m)$ acts linearly 
on a complex finite dimensional vector space $W$ such that the corresponding 
representation $G^\C \to \Gl(W)$ is a polynomial map. 
Suppose $w\in W$ is a nonzero vector, which is invariant by the action of any pair 
$(A,B)\in G^\C$ of upper triangular matrices $A,B$ with ones of the diagonal.
Moreover, we assume that there are $\alpha\in\N^k, \beta\in\N^m$ 
satisfying
$$
 (\mathrm{diag}(a),\mathrm{diag}(b))\, w = 
   a_1^{\alpha_1}\cdots a_k^{\alpha_k}\,  
   b_1^{\beta_1}\cdots b_m^{\beta_m} \, w
$$
for all pairs $(\mathrm{diag}(a),\mathrm{diag}(b)) \in G^\C$ 
of complex diagonal matrices.
Then one says that $w$ is a \emph{highest weight vector} of 
\emph{highest weight} $(a,b)$ of the given representation.
One can show that the entries of $a$ and $b$ 
are monotonically decreasing, hence they are partitions. 
The relevant fact is that $\Span_\C G^\C w$ is a simple $G^\C$--module.
Moreover, every simple $G^\C$--module has a unique 
highest weight vector (up to scaling) and the corresponding highest weight
determines the module up to isomorphism. 
One often denotes by $\Sc_{a}(\C^k)$ a simple $\Gl(\C^k)$--module of highest weight $a$
and calls this a \emph{Schur-Weyl module}. 
Up to isomorphism, a simple $G^\C$--module of highest weight $(a,b)$ 
can be written as the tensor product $\Sc_{a}(\C^k) \ot \Sc_{b}(\C^m)$.
For details and proofs of these facts we refer to~\cite[\S 9.1]{hall:03}.
Another very useful reference is \cite{procesi:07}. 

Consider now the compact subgroup 
$K := \UU(k)\times \UU(m)$ of $G^\C$
and assume $W$ is endowed with an inner product for which $K$ acts isometrically.  
Suppose we have another such representations of $G^\C$ on $W'$, 
which is equivalent as a representations of $K$. This means there is a
$K$--equivariant linear isomorphism $\varphi\colon W \to W'$. 
Then $\varphi$ must be $G^\C$--equivariant, since $K$ Zariski dense in $G^\C$; 
e.g., see \cite[\S 2.2.2]{wallach:17}.
This has the following consequence. Suppose $U,U'\subseteq W$ are 
nonisomorphic simple $G^\C$--submodules. Then $U$ and $U'$ 
are nonisomorphic as $K$--modules and therefore must be orthogonal.

We denote by $\lambda'\subseteq [m] \times [k]$ 
the transposed Young diagram obtained from $\la$ by
exchanging rows and columns.

\begin{lemma}\label{le:Vla-irred}
$V_\lambda^\C$ is a simple $\Gl(\C^k)\ti \Gl(\C^m)$--module
with the highest weight~$(\lambda,\lambda')$. 
\end{lemma}

\begin{proof}
As in the proof of \cref{le:VlaSpan}(1) we show that 
\eqref{eq:hwv} holds for any complex upper triangular matrices $A,B$.
This exactly expresses that $v_\lambda$ is a highest weight vector of 
$V_\lambda^\C$ with the highest weight 
$(\lambda,\lambda')$. The assertion follows.
\end{proof}

We now prove the complex version of~\cref{th:Lambda-decomp}. 

\begin{proposition}\label{cor:C-Lambda-decomp}
We have the orthogonal decomposition
$\Lambda_\C^d(\C^k\ot\C^m) = \bigoplus_{\lambda} V^\C_\lambda$, 
where the sum runs over all Young diagrams $\la\subseteq [k]\times [m]$ 
with $d$ boxes. This is the decomposition of $\Lambda_\C^d(\C^k\ot\C^m)$ 
into simple $\Gl(\C^k)\ti \Gl(\C^m)$--modules, which is 
multiplicity--free. 
\end{proposition}

\begin{proof}
By \cref{le:Vla-irred}, $V_\lambda^\C$ is a simple $\Gl(\C^k)\ti \Gl(\C^m)$--module
with the highest weight~$(\lambda,\lambda')$. 
As we noted, it can be written as the tensor product
\begin{equation}\label{eq:V=SS}
 V_\lambda^\C\simeq \Sc_\lambda(\C^k) \ot \Sc_{\lambda'}(\C^m)
\end{equation}
of Schur-Weyl modules. 
Different simple $G^\C$--submodules of $\Lambda^d_\C (V^\C)$
are nonisomorphic as $\UU(k)\times \UU(m)$-modules and therefore  
orthogonal (see the explanations right before \cref{le:Vla-irred}).
So we have an orthogonal direct sum 
$\bigoplus_{\lambda} V^\C_\lambda \subseteq \Lambda^d_\C (V^\C)$.
The Cauchy formula~\cite[(8.4.1)]{procesi:07}, 
which states that 
\begin{equation}\label{eq:cauchy}
  \Lambda_\C^d(\C^k\ot\C^m)  \simeq 
  \bigoplus_{|\lambda|=d} \Sc_\lambda(\C^k) \ot \Sc_{\lambda'}(\C^m) ,
\end{equation}
shows that equality holds. 
\end{proof}

\begin{proof}[Proof of \cref{th:Lambda-decomp}]
Suppose we have subspaces $U_1,\ldots,U_s\subseteq W$ of a 
finite dimensional real vector space $W$ such that 
$W^\C = U^\C_1 \oplus \ldots \oplus U^\C_s$. 
Passing to the subspaces of $\tau$-invariant vectors yields 
$$
 W = (W^\C)^\tau = (U^\C_1)^\tau  \oplus \ldots \oplus (U^\C_s)^\tau 
= U_1 \oplus \ldots \oplus U_s. 
$$
The assertion follows from~\cref{cor:C-Lambda-decomp} 
by applying this general principle. 
\end{proof}

\begin{example}
Assume $\la=(d)$ consists of a single row. Then 
$\Sc_\lambda(\C^k) \simeq \mathrm{Sym}^d\C^k$ 
is a simple $\Gl(\C^k)$--module and 
$\Sc_{\lambda'}(\C^m)$ is one dimensional,  
corresponding to the multiplication with $\det^d$.
However, the real version $\mathrm{Sym}^d\R^k$ is not a simple $\Gl(\R^k)$--module.
(Its decomposition into simple modules is given by harmonic polynomials.)
This shows that the $V_\la$ are not simple 
$\Gl(\R^k)\ti \Gl(\R^m)$--modules.
\end{example}

For proving Theorem~\ref{th:wedge-decomp} below, 
will need a more concrete understanding of the 
Schur-Weyl modules and the isomorphism behind~\eqref{eq:V=SS}, 
which will be explained in the next section.

%%%
\subsection{Decomposition of wedge product of Schubert spans}\label{se:decomp-wedge}

We already explained that probabilistic Schubert problems can be rephrased in terms of 
the multiplication of the Grassmann classes of Schubert varieties in the probabilistic intersection ring. 
This amounts to the wedge multiplication of Schubert zonoids in the exterior algebra. 
In the previous section, we proved that the exterior algebra is the direct sum of 
the spans of Schubert zonoids. 
As a first step toward understanding how Schubert zonoids multiply, 
we study the splitting of the wedge product 
\[\label{eq:wedgespace}
 V_\la \wedge V_\mu := \Span\big\{ u\wedge v \mid u\in V_\lambda, v\in V_\mu \big\} 
\]
of the spans $V_\la$ and $V_\mu$ of Schubert zonoids. 

The multiplication of the cohomology classes 
$[\Omega^\C_{\lambda_2}]$ 
of \emph{complex} Schubert varieties is determined by the 
\emph{Littlewood--Richardson coefficients} $c^\lambda_{\mu \nu}$ \cite{fulton:97,EH:16}, 
which can be defined as the multiplicities of tensor products of simple $\Gl(\C^k)$--modules:
\begin{equation}\label{eq:LR}
 \Sc_\lambda(\C^k) \ot \Sc_\mu(\C^k) \simeq \bigoplus_\nu c^\nu_{\lambda \mu} \, \Sc_\nu(\C^k) ,
\end{equation}
where the sum runs over all Young diagrams~$\nu$ with $|\nu |=k+m$, 
but not necessarily $\nu\subseteq [k] \ti [m]$. 

These coefficients have a rich theory~\cite{fulton-harris:91,procesi:07} 
and prominently appear in algebraic combinatorics, 
see \cite[\S7.18 and A.1.3]{stanley_II:99} and \cite[\S 4.5]{ike:12b}.
The multiplication rule for the cohomology classes 
for complex Schubert varieties is~\cite{fulton:97,EH:16}
\begin{equation}\label{eq:mult-CSsch}
 [\Omega^\C_{\lambda_1}] \cdot [\Omega^\C_{\lambda_2}]  = 
 \bigoplus_{\nu\subseteq [k] \ti [m]} c^\nu_{\lambda \mu} \, [\Omega^\C_{\lambda_2}] ,
\end{equation}
which is the same rule as the splitting of the tensor product 
of Schur-Weyl modules in \cref{eq:LR}, 
with the additional requirement that $\nu\subseteq [k] \ti [m]$. 

We prove here that the representation theoretic splitting of $V_\lambda\wedge V_\mu$ 
is as well determined by Littlewood--Richardson coefficients, 
namely by their positivity. Let us denote by 
\begin{equation}\label{eq:def-LR-set}
 \LR_{k,m}(\la,\mu) := \{\nu\subseteq [k] \ti [m] \mid c^\lambda_{\mu \nu} > 0 \}
\end{equation}
the set of Young diagrams~$\nu\subseteq [k] \ti [m]$ with the property that 
$\Sc_\nu(\C^k)$ occurs in the tensor product $\Sc_\lambda(\C^k)\ot \Sc_\mu(\C^k)$.
The main result of this section is the following. 

\begin{theorem}\label{th:wedge-decomp}
For partitions $\lambda$ and $\mu$ contained in $[k] \ti [m]$, we have 
$$
 V_\lambda \wedge V_\mu = \bigoplus_{\nu\in \LR_{k,m}(\lambda,\mu)} V_\nu .
$$
\end{theorem}

From this we can derive an interesting conclusion. 

\begin{corollary}\label{cor:prob-of-intersecting-Schub-var}
Let $\Omega_{\lambda_1},\Omega_{\lambda_2},\Omega_{\lambda_2}\subseteq G(k,m)$ be three  
Schubert varieties associated with Young diagrams~$\lambda_i\subseteq [k] \ti [m]$ such that 
$|\lambda_1| + |\lambda_2| + |\lambda_3| = km$. 
Equivalent are:
\begin{enumerate}
\item $g_1\Omega_{\lambda_1}$, $g_2\Omega_{\lambda_2}$, and $g_3 \Omega_{\lambda_3}$ 
intersect with positive probability,
when $g_1,g_2,g_3\in \OO(n)$ are independent and uniformly random.

\item $\ast\lambda_3\in \LR(\lambda_1,\lambda_2)$.

\item The tensor product $\Sc_{\lambda_1}(\C^k) \ot \Sc_{\lambda_2}(\C^k) \ot\Sc_{\lambda_3}(\C^{k})$
contains a nonzero $\Sl(\C^k)$-invariant.

\item Generic translates of the \emph{complex} Schubert varieties 
$\Omega^\C_{\lambda_1},\Omega^\C_{\lambda_2},\Omega^\C_{\lambda_2}\subseteq G^\C(k,m)$ 
have a nontrivial intersection (in which case they intersect in a single point).
\end{enumerate}
\end{corollary}

\begin{proof}
The event 
$g_1\Omega_{\lambda_1}\cap g_2\Omega_{\lambda_2} \cap g_3 \Omega_{\lambda_3} \ne\varnothing$ 
occurs with positive probability iff 
$\EE \#\left(g_1\Omega_\lambda\cap g_2\Omega_\mu \cap g_3 \Omega_\nu\right)$ 
is positive. By \cref{thm:PSC}, this is equivalent to 
$K_{\lambda_1}\wedge K_{\lambda_2} \wedge K_{\lambda_3}$ having positive length, that is 
$K_{\lambda_1}\wedge K_{\lambda_2} \wedge K_{\lambda_3}\ne 0$. This is equivalent to 
$V_{\lambda_1}\wedge V_{\lambda_2} \wedge V_{\lambda_3}\ne 0$.
From \cref{th:wedge-decomp}, we get 
$$ 
 V_{\lambda_1} \wedge V_{\lambda_2} \wedge V_{\lambda_3}
  = \bigoplus_{\nu\in \LR_{k,m}(\lambda_1,\Lambda_2)} V_\nu \wedge V_{\lambda_3} .
$$
By the duality statement in~\cref{le:dualSchubert}, we have 
$V_\nu \wedge V_{\lambda_3} \ne 0$ iff  $K_\nu \wedge K_{\lambda_3}\ne 0$ 
iff $\nu = \ast\lambda_3$. 
Therefore, 
$V_{\lambda_1} \wedge V_{\lambda_2} \wedge V_{\lambda_3}\ne 0$ iff 
$\ast\lambda_3\in \LR(\lambda_1,\lambda_2)$.
This proves the equivalence of the first two statements. 

The equivalence of the second and third statement is well known and can be 
derived as follows.
Let us abbreviate $\Sc_{\lambda} := \Sc_{\lambda}(\C^k)$.
From \eqref{eq:LR}, we get by tensoring and passing to the space 
of $\Sl_k(\C)$--invariants, 
\begin{equation}\label{eq:TPI}
 \big(\Sc_{\lambda_1} \ot \Sc_{\lambda_2} \ot \Sc_{\lambda_3} \big)^{\Sl(\C^k)}
    \simeq \bigoplus_\nu c^{\nu}_{\lambda_1 \lambda_2} \, 
      \big( \Sc_\nu \ot \Sc_{\lambda_3} \big)^{\Sl(\C^k)} ,
\end{equation}
where the sum is over all~$\nu$ with 
$|\lambda_1| + |\lambda_2| =km - |\lambda_3| =|\ast\lambda_3|$. 
The dual of the $\Sl(\C^k)$--module $\Sc_\nu$ is given by $\Sc_{\ast\nu}$. 
Therefore, 
$$
 \big( \Sc_\nu \ot \Sc_{\lambda_3}  \big)^{\Sl(\C^k)}
  =\mathrm{\End}_{\Sl(\C^k}\big(\Sc_{\ast\nu}, \Sc_{\lambda_3} \big) ,
$$
which is nonzero iff $\nu = \ast\lambda_3$
(in which case it is one dimensional by Schur's lemma).
Thus the space in~\eqref{eq:TPI} is nonzero iff $\ast\lambda_3\in \LR(\lambda_1,\lambda_2)$. 

The equivalence of the fourth with the second and statement follows 
from~\eqref{eq:mult-CSsch} 
by the same reasonings as before, which completes the proof.
\end{proof}

The rest of the section is devoted to the proof of~\cref{th:wedge-decomp}.
It is easily verified that $V_\lambda^\C \wedge_\C V_\mu^\C$ 
is the complexification of $V_\lambda \wedge V_\mu$. 
We will prove the complex version of this theorem, namely
\begin{equation}\label{eq:C-wedge-decomp}
 V^\C_\lambda \wedge_\C V^\C_\mu = \bigoplus_{\nu\in \LR_{k,m}(\lambda,\mu)} V_\nu^\C .
\end{equation}
From this, \cref{th:wedge-decomp} follows by passing to the invariants 
under complex conjugation. 

The proof of~\cref{eq:C-wedge-decomp} relies on several facts from representation theory 
that we now recall. 
To each Young diagram $\lambda$ with $d$ boxes 
there corresponds a simple module of the symmetric group~$S_d$,
denoted by $[\lambda]$ and called {\em Specht module}. 
Up to isomorphism, these are all the simple $S_d$--modules: 
the Young diagram~$\lambda$ determines the isomorphism type.
We refer to~\cite[Chap.~4]{fulton-harris:91} for details and proofs. 

Let now $E\simeq\C^k$ be a complex vector space 
and consider the $d$--fold tensor power $\ot^d E$. 
The complex linear group $\Gl(E)$ acts on $\ot^d E$ via 
$g \cdot v_1\ot\ldots \ot v_d  := gv_1 \ot\ldots \ot g v_d$. 
Moreover, the symmetric group $S_d$ acts on $\ot^d E$ 
by permuting the tensor factors. These actions commute, so that  
$E^{\ot d}$ is a module over the product group $S_d\times \Gl(E)$. 

For each Young diagram $\la$, we define the corresponding 
{\em Schur-Weyl module} corresponding to $\lambda$ and $E$ as 
the space of $S_d$--module homomorphisms from $[\lambda]$ to $\otimes^d E$:
\begin{equation}\label{eq:defSW}
 \Sc_\lambda(E) := \Hom_{S_d}([\lambda], E^{\otimes d}) .
\end{equation}
One can show that $\Sc_\lambda(E)$ is a simple 
$\Gl(E)$--module if $\ell(\lambda) \le k$; 
otherwise $\Sc_\lambda(E)=0$.
Moreover, every simple polynomial $\Gl(E)$--module is isomorphic to some $\Sc_\lambda(E)$: 
the partition $\lambda$ determines the isomorphism type.
The module $\Sc_\lambda(E)$ contains a {\em highest weight vector}~$w$
of {\em highest weight} $\lambda$, which is unique up to a scaling factor.
This means that $w\ne 0$ and 
$gw = g_{11}^{\lambda_1}\ldots g_{kk}^{\lambda_k} w$ 
for all upper triangular matrices $g\in\Gl(E)$. 
For this we assumed $E=\C^k$, but we note that 
the highest weight $\la$ does not depend on the choice of a basis of $E$.
A short and accessible account of these facts can be found in~\cite[III.1.4]{kraft:84}. 
For further explanations, we  refer to~\cite[Chap.~6]{fulton-harris:91} and~\cite[Chap.~8]{fulton:97}.

For a Young diagram~$\la$ such that $\ell(\la)\le k$, we define 
the submodule $E^{\ot d}_\lambda\subseteq E^{\ot d}$ 
as the image of the $S_d\times \Gl(E)$-equivariant linear map 
\begin{equation}\label{eq:Ela-iso}
 [\la] \ot \Sc_\lambda(E) \simto E^{\ot d}_\la ,\ z \ot \varphi \mapsto \varphi(z) .
\end{equation}
This is an isomorphism since the map is nonzero and 
$[\la] \ot \Sc_\lambda(E)$ is a simple $S_d\times \Gl(E)$--module. 
The $E^{\ot d}_\lambda$ form a direct sum since they are nonisomorphic simple modules.
Moreover, the add up to $E^{\ot d}$, since  every $\Gl(E)$-submodule of $E^{\otimes d}$ is isomorphic to 
$\Sc_\lambda(E)$ for some $\la$ with $\ell(\la)\le k$. 
This implies the multiplicity-free (isotypical) decomposition 
\begin{equation}\label{eq:schur-weyl}
 E^{\otimes d} = \bigoplus_\la E^{\otimes d}_\lambda
\end{equation}
of the $S_d\times \Gl(E)$--module $E^{\otimes d}$.
This expresses an important link between 
the representations of~$S_d$ and~$\Gl(E)$, called \emph{Schur-Weyl duality};
see~\cite[\S 6.1]{fulton-harris:91} and~\cite[Chap.~9]{goodman-wallach:09}.
We call $E^{\otimes d}_\lambda$ the \emph{$\la$-isotypical component} of $\la$.
It is nonzero iff $\ell(\la)\le \dim E$. 

In the special case where $\la$ is the Young diagram consisting of a single column of length~$d$, 
$[\la]$ is the one-dimensional $S_d$--module corresponding to the sign character $\sgn$.  
The corresponding isotypical component equals the $d$th exterior power $\Lambda^d(E)$. It is obtained as 
as the image of the projection $E^{\ot d} \to \Lambda^d E,\, w \mapsto p_d w$, 
given by the multiplication with the group algebra element 
\begin{equation}\label{eq:def-p_d}
 p_d := \frac{1}{d!}\sum_{\sigma} \sgn(\sigma) \sigma \in \C[S_d] .
\end{equation}  

Assume $F\simeq\C^m$. 
We apply \eqref{eq:Ela-iso}--\eqref{eq:schur-weyl} to $E^{\ot d}$ and $F^{\ot d}$
and form the tensor product, which gives the decomposition
$$
  (E \ot F)^{\otimes d} =  E^{\otimes d} \ot F^{\ot d} 
     = \bigoplus_{\la,\mu} \, E^{\otimes d}_\lambda \ot F^{\otimes d}_\mu
    \simeq \bigoplus_{\la,\mu} \, 
    [\la] \ot [\mu] \ot \Sc_\lambda(E) \ot \Sc_\mu(F) .
$$
into simple $S_d\ti S_d \ti \Gl(E) \ti \Gl(F)$-modules. 
Consider the diagonally embedded subgroup
$$
 \bar{S}_d := \{ (\sigma,\sigma) \mid \sigma \in S_d \}
$$
of $S_d\ti S_d$, which describes the simultaneous permutations of the 
$E$ and $F$ factors.  
If we define $\bar{p}_d \in  \C[\bar{S}_d]$ analogously as $p_d$ in~\eqref{eq:def-p_d}, 
then $\Lambda^d(E\ot F) = \bar{p}_d ((E \ot F)^{\otimes d})$. 
From~\cref{cor:C-Lambda-decomp} we see that 
\begin{equation}\label{eq:pd-EF-lamu} 
 \bar{p}_d \big( E^{\otimes d}_\lambda \ot F^{\otimes d}_{\la'}\big)  = V^\C_\la ,\quad 
 \dim \bar{p}_d( [\la] \ot [\mu]) = \delta_{\la' \mu} .
\end{equation}

Let us recall a general construction in group representation theory.
Let $H$ be a subgroup of the finite group $G$. 
We identify $G$--modules with modules over the group algebra $\C[G]$. 
Every $\C[H]$--module~$M$ \emph{induces} a $\C[G]$--module
by scalar extension: 
$$
 M\uparrow^G := M \ot_{\C[H]} \C[G] .
$$
Let $N$ be any $\C[G]$--module, which we can view as a $\C[H]$--module.
The canonical isomorphism 
\begin{equation}\label{eq:adjoint}
  \Hom_H(M, N ) \simto \Hom_G(M\uparrow^G, N) , \, \varphi \mapsto
    \big( m \ot a \mapsto a\varphi(m) \big).
\end{equation}
expresses that scalar extension and restriction are adjoint; 
see~\cite[Prop.~3.17]{fulton-harris:91}.  

Now suppose $\lambda$ and $\mu$ are Young diagrams of size $d$ and $g$, respectively. 
We first form the tensor product $[\la]\ot [\mu]$ of the Specht modules, 
which is an irreducible $S_d\ti S_g$--module, and then consider the 
induced $S_{d+g}$--module, where we view $S_d\ti S_g$ as a subgroup of $S_{d+g}$.
This construction is called \emph{outer product}.
It is a remarkable fact that the outer product splits 
according to the Littlewood-Richardson coefficients~$c^\nu_{\lambda \mu}$,
which were defined in~\eqref{eq:LR} in terms of the splitting 
of tensor products of representations of $\GL(E)$. 
We have 
\begin{equation}\label{eq:LRagain}
 ([\la]\ot [\mu])\uparrow^{S_{d+g}} \ \simeq\ \bigoplus_\nu c^\nu_{\lambda \mu} [\nu] ,
\end{equation}
where the sum is over Young diagrams~$\nu$ of size $d+g$. 
This can be deduced via Schur-Weyl duality \eqref{eq:Ela-iso}--\eqref{eq:schur-weyl};  
we refer to~\cite[Prop.~4.5.4]{ike:12b} for the proof. 
Descriptions in terms of characters can be found in~\cite[\S 7.18]{stanley_II:99}. 
Another useful reference is~\cite[Ex.~4.43]{fulton-harris:91})  

After these preparations, we can now finally provide the proof of \cref{th:wedge-decomp}.

\begin{proof}[Proof of \cref{th:wedge-decomp}] 
We write $E:=\C^k$ and $F:=\C^m$. 
Let $\la,\mu\subseteq [k] \ti [m]$ be Young diagrams of the size~$d$ and $g$, respectively. 
By~\cref{eq:pd-EF-lamu}, we have 
$V_\lambda^\C = \bar{p}_d \big(E^{\ot d}_\la \ot F^{\ot d}_{\la'} \big)$
and 
$V_\mu^\C = \bar{p}_g \big(E^{\ot g}_\mu \ot F^{\ot g}_{\mu'} \big)$.
We can therefore express the wedge product of complex Schubert spans as follows: 
\begin{equation}\label{eq:Vwedge-char}
 V_\lambda^\C \wedge_\C V_\mu^\C = \bar{p}_{d+g} \big(  V_\lambda^\C \ot V_\mu^\C \big)
   = \bar{p}_{d+g} \Big(  \bar{p}_d \big( E^{\ot d}_\la \ot F^{\ot d}_{\la'} \big) \ot \bar{p}_g 
      \big( E^{\ot g}_\mu \ot F^{\ot g}_{\mu'} \big) \Big)  
   = \bar{p}_{d+g} \big( N_{\la\mu}  \big) ,
\end{equation}
where we have set 
$$
 N_{\la\mu} := E^{\ot d}_\la \ot F^{\ot d}_{\la'} \ot F^{\ot g}_\mu \ot F^{\ot g}_{\mu'} \subseteq (E\ot F)^{\ot(d+g)} .
$$
By applying~\eqref{eq:Ela-iso} twice, using the definition of Schur-Weyl modules in~\cref{eq:defSW}, we obtain
\begin{equation}\label{eq:Edla}
 E^{\ot d}_\la \ot F^{\ot d}_{\la'} = \big\{ \varphi(z) \mid  \varphi\in 
  \Hom_{S_d\ti S_d}([\la] \ti [\la'], E^{\ot d} \ot F^{\ot d}),\,  z \in [\la] \ti [\la'] \big\} . 
\end{equation}
We view
$H := S_d \ti S_d \ti S_g \ti S_g$ as a subgroup of $G :=  S_{d+g} \ti S_{d+g}$
and define 
the $H$--module $M := [\la] \ti [\mu] \ti [\la'] \ti [\mu']$. 
Moreover, we consider 
the $G$--module  $N := E^{\ot (d+g)} \ot  F^{\ot (d+g)}$.
From~\eqref{eq:Edla} and an analogous characterization for $F^{\ot g}_\mu \ot F^{\ot g}_{\mu'}$, we obtain
\begin{equation}\label{eq:Nlamu-char}
 N_{\la\mu} := E^{\ot d}_\la \ot F^{\ot d}_{\la'} \ot F^{\ot g}_\mu \ot F^{\ot g}_{\mu'} 
  = \big\{ \rho(m) \mid  \rho\in \Hom_{H}(M , N),\,  m \in M \big\} ,
\end{equation}
where $N\simeq E^{\ot d} \ot E^{\ot g} \ot  F^{\ot d} \ot F^{\ot g}$ 
is viewed here as an $H$-module via scalar restriction. 
Consider the scalar extension $\widetilde{M} := M \ot_{\C[H]} \C[G]$. 
Using the canonical isomorphism in~\eqref {eq:adjoint}, we get 
$$
 N_{\la\mu} = \big\{ \widetilde{\rho}(\widetilde{m}) \mid \widetilde{\rho}\in \Hom_{G}(\widetilde{M} , N),\,  
      \widetilde{m} \in \widetilde{M} \big\} . 
$$
Therefore, by~\eqref{eq:Vwedge-char}, we arrive at 
\begin{equation}\label{eq:Vwedge-Char}
 V_\lambda^\C \wedge_\C V_\mu^\C = \bar{p}_{d+g} \big( N_{\la\mu}  \big) = 
 \big\{ \bar{p}_{d+g}(\widetilde{\rho}(\widetilde{m})) \mid \widetilde{\rho}\in \Hom_{G}(\widetilde{M} , N),\, 
      \widetilde{m} \in \widetilde{M} \big\} . 
\end{equation}
Applying~\eqref{eq:LRagain} twice yields the following decomposition 
of $\widetilde{M}$ into irreducible $G$-modules: 
\begin{equation}\label{eq:tildeM-char}
 \widetilde{M} := M\uparrow^G \ \simeq \big([\la]\ot [\mu])\uparrow^{S_{d+g}} \ 
    \ot \ \big([\la']\ot [\mu'])\uparrow^{S_{d+g}}\ 
   \simeq  \bigoplus_{\nu,\pi'} c^\nu_{\la \mu}  c^{\pi'}_{\la' \mu'} \, [\nu] \ot [\pi']  .
\end{equation}
Using 
$\dim \bar{p}_{d+g}( [\nu] \ot [\pi']) = \delta_{\nu \pi}$ 
(see~\eqref{eq:pd-EF-lamu}), we conclude from~\eqref{eq:Vwedge-Char} that 
$$
 V_\lambda^\C \wedge_\C V_\mu^\C  
  \subseteq \bigoplus_{\nu\in\LR_{k,m}(\la,\mu)} V_\nu^\C .
$$
In order to  show equality, let $\nu\in\LR_{k,m}(\la,\mu)$. 
\cref{le:LRsymm} below implies that  $\nu'\in\LR_{k,m}(\la',\mu')$. 
Therefore, $[\nu]\ot [\nu']$ occurs as a $G$--submodule of $\widetilde{M}$
by the decomposition in~\eqref{eq:tildeM-char}. 
Let now $w\in V^\C_\nu$. There is 
$\widetilde{\rho}\in \Hom_{G}(\widetilde{M} , N)$
such that 
$$
 \widetilde{\rho}\big([\nu]\ot [\nu']\big) = [\nu]\ot [\nu']\ot w .
$$
Therefore  
$$w\in \bar{p}_{d+g}\big(\widetilde{\rho}([\nu]\ot [\nu'])\big)$$
by~\eqref{eq:pd-EF-lamu}. 
This completes the proof.
\end{proof}

In the above proof we used the following duality result for 
Littlewood-Richardson coefficients. 
We very briefly outline its proof for lack of a suitable reference.

\begin{lemma}\label{le:LRsymm}
We have $c^\nu_{\lambda \mu} = c^{\nu'}_{\lambda' \mu'}$, where $\lambda',\mu',\nu'$ denote 
the Young diagrams obtained from $\lambda,\mu,\nu$ by transposition.
\end{lemma}

\begin{proof}
Let $s_\la$ denote the Schur polynomial associated with~$[\lambda]$. Then 
$c^\nu_{\lambda \mu} = \langle s_\lambda s_\nu , s_\nu\rangle$; 
see~\cite[Eq.~(7.64)]{stanley_II:99}. 
(The inner product corresponds to the one obtained by 
integrating characters over the unitary group~$\UU(k)$.)
There is a linear involution $\omega$ on the space of symmetric polynomials, 
compatible with multiplication, which 
preserves the inner product, and satisfies 
$\omega(s_\la) =  s_{\la'}$ \cite[Thm.~7.14.5]{stanley_II:99}. 
This implies
$$
 c^\nu_{\lambda \mu} = \langle s_\la s_\mu , s_\nu\rangle  
  = \langle \omega(s_\la s_\mu) , \omega(s_\nu)\rangle 
  = \langle \omega(s_\la) \omega(s_\mu) , \omega(s_\nu)\rangle 
  = \langle s_{\la'} s_{\mu'} , s_\nu'\rangle 
  = c^{\nu'}_{\lambda' \mu'} ,
$$
showing the assertion. 
\end{proof}

%%%%%%%%
\bigskip
\section{Table of notation}

\begin{xltabular}{\textwidth}{p{0.19\textwidth} X}
$[a,b]$ \dotfill & for $a,b$ in a vector space $V$, the segment $ta+(1-t)b$.\\
$c(Z)$\dotfill & the center of a zonoid $Z$, \eqref{eq:centerzonoid}.\\ 
$\ZZ(V)$ \dotfill & the set of zonoids in a vector space $V$, \eqref{eq:zonoidsdef}.\\
$\ZZ_o(V)$ \dotfill & the set of centered zonoids in a vector space $V$, \eqref{eq:zonoidsdef}.\\
$h_{Z}$\dotfill & the support function of a convex body $Z$, \eqref{eq:supportfctdef}.\\
$\VZ(V)$\dotfill & the Grothendieck group of\emph{ virtual zonoids}, i.e., formal differences of elements from $\VZ(V)$.\\
$\z(\xi)$\dotfill & for an integrable random vector $\xi\in V$, the zonoid in $\ZZ(V)$ with support function $h_{\z(\xi)}(u) := \EE \max\{0, \langle \xi,u\rangle\}$, \eqref{eq:defZK}.\\
$\cz(\xi)$ \dotfill & for an integrable random vector $\xi\in V$, the zonoid in $\ZZ_o(V)$ with support function 
$h_{\cz(\xi)}(u) := \tfrac{1}{2}\EE \vert\langle \xi,u \rangle\vert$, \eqref{eq:defZK}.\\
$\ell(K)$\dotfill& the length of a zonoid $K=K(\xi),$ i.e., $\ell(K(\xi))=\EE\|\xi\|$, \eqref{eq:ell}. 
The same letter denotes the extension of this map, in a translation invariant way, on noncentered zonoids
to obtain the linear functional $\ell:\VZ(V)\to \R$.\\
$\widehat{M}$\dotfill& for multilinear map 
$M\colon V_1\times\ldots\times V_p \to W$ 
of finite dimensional real vector spaces, the 
induced  multilinear map 
$\widehat{M}\colon \VZ_o(V_1)\times\ldots\times \VZ_o(V_p) \to \VZ_o(W)$, defined by \eqref{eq:FTZCdef}.\\
$ \langle K, L \rangle$\dotfill& for two zonoids $K, L$, the pairing defined by \eqref{eq:def-pairing}.\\
$\cos(\mu)$\dotfill& for a measure $\mu$ on the projective space $\mathbb{P}(V)$, its \emph{cosine transform},  \eqref{cosine_transform}.\\
$\Lambda(V)$\dotfill& exterior algebra of real vector space $V$, \eqref {eq:def-exterior-algebra}.\\
$\Lambda^d\varphi$\dotfill& for a linear map $\varphi:V\to W$, the induced linear map on the $d$--th exterior algebras $\Lambda^d\varphi:\Lambda^dV\to \Lambda^dW,$ \eqref{eq:Lambda^d}.\\
$\orf$\dotfill& an orientation of a Euclidean vector space $V$ satisyfing $\langle \orf, \orf \rangle = 1$.\\
$\star$\dotfill& the Hodge star $\star\colon\Lambda^d V \to \Lambda^{n-d} V$, \eqref{def_hodge_star}, and its extension to zonoids \cref{hodge_on_random_vectors}.\\
$\A^d(V), \VA^d(V)$\dotfill& the semi--vector space $\A^d(V) := \ZZ(\Lambda^d V),$ and the vector space $\VA^d(V):= \VZ(\Lambda^dV)$ of virtual zonoids in $\Lambda^dV$, \eqref{eq:semivector}.\\
 $ \A^d_0(V), \VA_0^d(V)$\dotfill& the centered parts of $\A^d(V), \VA^d(V)$, \eqref{eq:semivector}.\\
 $Z_1\wedge Z_2$\dotfill& for zonoids $Z_1, Z_2$, their wedge multiplication; see \cref{se:zon-in-ext-power}.\\
$\VA(V)$\dotfill& the noncentered zonoid algebra of a vector space $V$, \eqref{eq:noncentzonoidalgebra}.\\
$\mathrm{MV}$\dotfill& the mixed volume of convex bodies.\\
$K_1\vee K_2$\dotfill&  the convolution product of zonoids, $K_1\vee K_2 := \star  \left((\star  K_1)\wedge (\star  K_2)\right)$, \cref{def:coprod}.\\
$G(d,V)$\dotfill& the Grassmannian of $d$--planes in $V$.\\
$\GZ_o^d(V)$\dotfill& the set of centered Grassmann zonoids, \cref{sec:grass_algebra}.\\
$\VGZ_o^d(V)$\dotfill& the linear span of centered Grassmann zonoids in $\VZ(\Lambda^dV)$, \cref{sec:grass_algebra}.\\
$\VGZ_o(V)$\dotfill & The algebra of centered Grassmann zonoids, $\VGZ_o(V):=\bigoplus_{d=0}^{n}\VGZ_o^d(V),$ a graded subalgebra of $\VA_o(V)$, \cref{sec:grass_algebra}.\\
$\GZ^d(V), \VGZ^d(V)$\dotfill& the set of Grassmann zonoids and its linear span, \cref{sec:grass_algebra}.\\
$\VGZ(V) $\dotfill& the Grassmann zonoid algebra, $\VGZ(V) := \bigoplus_{d=0}^{n}\VGZ^d(V)$, a subalgebra of $\VA(V)$, \cref{sec:grass_algebra}.\\
$\exp$\dotfill& the exponential map $\exp\colon \VZ_o(V) \to \VGZ_0(V), L\mapsto e^L := \sum_{d=0}\tfrac{1}{d!} \, K^{\wedge d}$, \eqref{eq:expzono}.\\
$\M^d(V)$\dotfill& for a vector space $V$, the set of virtual Grassmann zonoids $K\in \VGZ^d(V)$ such that $h_{K}(w)=0$ for all simple vector $w\in \Lambda^d(V)$, \cref{eq:KMV}.\\
$ \M(V)$\dotfill&for a vector space $V$, the ideal $\M(v):=\bigoplus_{d=0}^{n}\M^d(V) \subseteq \VGZ^d(V)$, \cref{propo:ideal}.\\
$\CGZ(V), \CGZ_o (V)$\dotfill& for a vector space $V$, the algebra of classes of Grassmann zonoids of $V$, i.e., the quotient algebra
$\CGZ(V) := \VGZ(V)/\M(V)$, and the $\Gl(V)$--invariant subalgebra 
$\CGZ_o := \VGZ_o(V)/\M(V)$ of classes of centered Grassmann zonoids, \cref{def:ACGZ}.\\
$\CGZ(V)^H, \CGZ(V)_o^H$\dotfill& given a closed subgroup $H$ of the orthogonal group of $V$, the subalgebras  of $H$--invariants of $\CGZ(V)$ and, respectively, $\CGZ_o(V)$, \cref{eq:CGZH}.\\
$\KK(V), \KK_o(V)$\dotfill& for a vector space $V$, the set of convex bodies in it and the subset of centered convex bodies, \cref{sec:valuations}.\\
$\val(V), \val^d(V)$\dotfill& for a vector space $V$, the set of
translation invariant, continuous, even valuations and its subspaces $\val^d$ 
of homogeneous degree $d$ valuations, \cref{sec:valuations}.\\
$\|\phi\|$\dotfill& the norm of a valuations $\phi$, \cref{eq:def-val-norm}.\\
$ \nu_{L_1,\ldots,L_{n-d}}$\dotfill& for convex bodies $L_1,\ldots,L_{n-d}\in \KK_0(V)$
the valuation defined by \eqref{eq:def-nu}.\\
$\Kl(\phi)$\dotfill& the Klain function $\Kl(\phi)\colon G(d,V) \to \R$ of an even valuation $\phi$, \cref{def:klain}.\\
$\phi_\mu$\dotfill& for a measure $\mu$ on the Grassmannian, the valuation defined by the Crofton map \cref{eq:defphiA}.\\
$\Phi$\dotfill& the Crofton map for Grassmann zonoids, \cref{def:Phi}.\\
$\phi_{L}$\dotfill& for a zonoid $L$, the valuation $\phi_{L}:=\phi_{\mu_L}$ associated to it via the Crofton map for Grassmann zonoids, \eqref{eq:crofton}.\\
$\phi_1 \ast \phi_2 $\dotfill& the convolution product of two valuations, defined by \eqref{eq:conv_prod}.\\
$\HE(M)$\dotfill& the probabilistic intersection ring of the homogeneous space $M=G/H$, defined by  $\HE (M) := \CGZ(V)^H$, \cref{def:HE}.\\
$\HEo (M)$\dotfill& the centered probabilistic intersection ring of the homogeneous space $M=G/H$, i.e., the set  of centered elements of $\HE(M)$, \cref{def:HE}.\\
$K(Y)$\dotfill& for a stratified set $Y\subseteq M$ of finite volume, the centered zonoid associated to it, \cref{def:KZ_Y}.\\
$Z(Y)$\dotfill& for a stratified set $Y\subseteq M$ of finite volume and cooriented, the zonoid associated to it, \cref{def:KZ_Y}.\\
$[Y]_{\EE}$\dotfill& for a stratified set $Y\subseteq M$, its  Grassmann class, \cref{def:classY}.\\
$[Y]^+_{\EE}$\dotfill &for a stratified and cooriented set $Y\subseteq M$, its oriented Grassmann class, \cref{def:classY}.\\
$[Y]_c$\dotfill &for a stratified and cooriented set $Y\subseteq M$, the element $[2c(Z_Y)]$, \cref{eq:Ydeco}.\\
$\Omega^d (M)$\dotfill& space of smooth differential forms of degree $d$ on $M$, \cref{sec:DR}.\\
$\big(\Omega^d (M)\big)^G$\dotfill & space of smooth $G$--invariant differential forms of degree $d$ on $M$, \cref{sec:DR}.\\
$\Gamma$\dotfill& $\Gamma\colon \Lambda (V)^H \to \Omega(M)^G$ the $G$--invariant extension to a differential form, \eqref{eq:Gamma-map}.\\
$\eta_Y$\dotfill&for a stratified subset $Y\subseteq M$ of an oriented manifold, a form representing its Poincar\'e  dual, \cref{def:PD}.\\
$ i_x(Y_1, \ldots, Y_s;M)$\dotfill & the intersection index of cooriented subsets in an oriented manifold $M$, \cref{eq:intind}.\\
$S^n$\dotfill & $n$-dimensional sphere.\\
$\RP^n$\dotfill & $n$-dimensional real projective space.\\
$\CP^n$\dotfill & $n$-dimensional complex projective space, \cref{eq:def-CPn}.\\
$R_n$\dotfill & centered probabilistic intersection ring $\HEo(\CP^n)$, \cref{eq:def-Rn}.\\
$G(k,m)$\dotfill & Grassmannian of real $k$-dimensional subspaces of $\R^{k+m}$, \cref{eq:G(kn)}.\\
$G^\C(k,m)$\dotfill& Grassmannian of complex $k$--planes in $\C^{k+m}$.\\
$\Omega_\lambda$\dotfill & Schubert variety associated with Young diagram, \cref{eq:schubertcycles}\\
$K_\lambda$\dotfill & Schubert zonoid associated with Young diagram, \cref{def:schubertzonoids}.\\
$V_\lambda$\dotfill & Schubert span associated with Young diagram~$\la$, \cref{def:Grassmann-Schubert-classes}.\\
$V_\lambda\wedge V_\mu$\dotfill & span of $v\wedge u$ with $v\in V_\lambda$, $u\in V_\mu$ , 
\cref{eq:wedgespace}.\\
$[\Omega_\la]_\EE$\dotfill & Grassmann Schubert class associated with Young diagram~$\la$, \cref{def:Grassmann-Schubert-classes}.\\
$c^\lambda_{\mu \nu}$\dotfill & Littlewood--Richardson coefficient, \cref{eq:LR}.\\
$\Sc_\lambda(E)$\dotfill & Schur-Weyl module associated with Young diagram~$\la$, \cref{eq:defSW}.\\
$\LR_{k,m}(\la,\mu)$\dotfill & set of $\nu\subseteq [k] \ti [m]$ such that  $c^\lambda_{\mu \nu} > 0$, \cref{eq:def-LR-set}.\\
\end{xltabular}

%%%
\bibliographystyle{alpha}
\bibliography{references}

\end{document}